\documentclass[11pt]{article}
\usepackage{geometry}
\usepackage{amsmath,amsthm,amssymb,amsfonts}
\usepackage{amsrefs}
\usepackage{latexsym}
\usepackage{xcolor}
\usepackage{graphicx}
\usepackage{epsfig}
\usepackage{subfig}
\usepackage{pdflscape}
\newcommand{\vol}{\operatorname{vol}}
\newcommand{\Btilde}{{\widetilde B}}
\newcommand{\Ktilde}{{\widetilde K}}
\newcommand{\cBtilde}{{\widetilde \cB}}
\newcommand{\btilde}{{\tilde b}}
\newcommand{\ctilde}{{\tilde c}}
\newcommand{\vbtilde}{{\tilde \vb}}
\newcommand{\vctilde}{{\tilde \vc}}

\newcommand{\cDtilde}{{\widetilde \cD}}
\newcommand{\cMtilde}{{\widetilde \cM}}
\newcommand{\WSd}{{{\rm WS}_d}}
\newcommand{\WS}{{{\rm WS}_3}}
\newcommand{\spSd}{\spS_d^{d-1}}
\newcommand{\spline}{s}
\newcommand{\s}{e}
\newcommand{\greville}{b^*}
\newcommand{\vgreville}{\vb^*}
\newcommand{\vgrevilleL}{\vb^{*L}}
\newcommand{\vgrevilleR}{\vb^{*R}}
\newcommand{\vgrevilletilde}{\tilde \vb^*}
\newcommand{\vgrevilletildeL}{\tilde \vb^{*L}}
\newcommand{\vgrevilletildeR}{\tilde \vb^{*R}}
\newcommand{\vgrevilletildeindex}{\vgrevilletilde\hspace*{-0.07cm}}
\newcommand{\mods}[2]{(#1 \text{ mod } #2)}

\newtheorem{theorem}{Theorem}

\newtheorem{definition}[theorem]{Definition}
\newtheorem{corollary}[theorem]{Corollary}
\newtheorem{proposition}[theorem]{Proposition}

\newcommand{\vb}{{\boldsymbol{b}}}
\newcommand{\vc}{{\boldsymbol{c}}}

\newcommand{\vm}{{\boldsymbol{m}}}
\newcommand{\vn}{{\boldsymbol{n}}}

\newcommand{\vp}{{\boldsymbol{p}}}
\newcommand{\vq}{{\boldsymbol{q}}}

\newcommand{\vu}{{\boldsymbol{u}}}
\newcommand{\vv}{{\boldsymbol{v}}}
\newcommand{\vw}{{\boldsymbol{w}}}
\newcommand{\vx}{{\boldsymbol{x}}}
\newcommand{\vy}{{\boldsymbol{y}}}

\newcommand{\vA}{{\boldsymbol{A}}}

\newcommand{\vC}{{\boldsymbol{C}}}

\newcommand{\vxi}{{\boldsymbol{\xi}}}

\newcommand{\spP}{{\mathbb{P}}}

\newcommand{\spR}{{\mathbb{R}}}
\newcommand{\spS}{{\mathbb{S}}}

\newcommand{\cB}{{\mathcal{B}}}

\newcommand{\cD}{{\mathcal{D}}}

\newcommand{\cM}{{\mathcal{M}}}

\newcommand{\cP}{{\mathcal{P}}}

\newcommand{\cT}{{\mathcal{T}}}

\newcommand{\NN}{{\mathbb{N}}}

\newcommand{\RR}{{\mathbb{R}}}

\newcommand{\abs}[1]{\lvert#1\rvert}

\newcommand{\makehosquare}[3]%
{\dimen0=#1\advance\dimen0 by -#3%
\vrule height#1 width#2 depth0pt \kern-#2%
\vrule height#1 width#1 depth-\dimen0 \kern-#1%
\vrule height#2 width#1 depth0pt \kern-#3%
\vrule height#1 width#3 depth0pt%
}

\begin{document}
\captionsetup[subfigure]{labelformat=empty}
\graphicspath{{figs/}}

\title{\textbf{Construction of $C^2$ cubic splines on arbitrary triangulations}}

\author{Tom Lyche\footnote{Tom Lyche, Dept. of Mathematics, University of Oslo, Norway, {\it email: tom@math.uio.no}}
,
Carla Manni\footnote{Carla Manni, Dept. of Mathematics, University of Rome Tor Vergata, Italy, {\it email: manni@mat.uniroma2.it}}
,
Hendrik Speleers\footnote{Hendrik Speleers, Dept. of Mathematics, University of Rome Tor Vergata, Italy,
{\it email: speleers@mat.uniroma2.it}}}

\date{}

\maketitle
\begin{abstract}
In this paper, we address the problem of constructing $C^2$ cubic spline functions on a given arbitrary triangulation $\cT$. To this end, we endow every triangle of $\cT$ with a Wang--Shi macro-structure. The $C^2$ cubic space on such a refined triangulation has a stable dimension and optimal approximation power. Moreover, any spline function in such space can be locally built on each of the macro-triangles independently via Hermite interpolation.
We provide a simplex spline basis for the space of $C^2$ cubics defined on a single macro-triangle which behaves like a Bernstein/B-spline basis over the triangle.
The basis functions inherit recurrence relations and differentiation formulas from the simplex spline construction, they form a nonnegative partition of unity, they admit simple conditions for $C^2$ joins across the edges of neighboring triangles, and they enjoy a Marsden-like identity. Also, there is a single control net to facilitate control and early visualization of a spline function over the macro-triangle.
Thanks to these properties, the complex geometry of the Wang--Shi macro-structure is transparent to the user. Stable global bases for the full space of $C^2$ cubics on the Wang--Shi refined triangulation $\cT$ are deduced from the local simplex spline basis by extending the concept of minimal determining sets.
\end{abstract}


\section{Introduction}
Piecewise polynomial spaces defined over polygonal partitions, usually triangulations, have applications in several branches of the sciences including geometric modeling, signal processing, data analysis, visualization, and numerical simulation; 
we refer the reader to \cites{Cohen.Riensenfeld.Elber01,Hughesbook,Lai.Schumaker07} and Section~\ref{sec:applications} for some examples.
For many of these applications, a smooth join between the different pieces is beneficial or even required; $C^2$ smoothness is often preferred. Such spaces are commonly referred to as (bivariate) spline spaces.
According to \cite[page~197]{Lai.Schumaker07}, {\it in general we would like to work with low degree splines: they involve fewer coefficients, and have less tendency to oscillate}.

An indispensable feature for a spline space to be useful in practice is having a stable dimension that only depends on the degree ($d$), the order of smoothness ($r$), and combinatorial --- or other easy to check --- properties of the partition ($\cT$). 
When $\cT$ is a triangulation, the dimension can be expressed in terms of the above quantities for spline spaces with $d\geq 3r+2$; see \cite{IbrS91} and \cite[Chapter~9]{Lai.Schumaker07}. On the other hand, instability in the dimension has been illustrated for $d=2r$ in \cite{DD}. We refer the reader to \cite{Toshniwal19} for recent results on the dimension of spline spaces on triangulations with nonuniform degrees.
Similar results are known for spline spaces over general rectilinear partitions; see \cites{Chui.Wang.83,Manni92} and references therein.

Spline spaces with too low degree compared to the smoothness are also exposed to several other shortcomings. In particular, they might lack optimal approximation power, a property strongly related to the possibility of constructing stable bases with local support for the considered spaces \cite{Lai.Schumaker07}. In this perspective, the bound $d\geq 3r+2$ plays again an important role in identifying the spline spaces with optimal approximation power on a given triangulation \cite[Chapter~10]{Lai.Schumaker07}.
Furthermore, the possibility of constructing any function of the spline space locally on each of the elements of $\cT$ is often seen as a desirable, if not imperative, property for practical purposes.
On a triangulation, a degree $d\geq 4r+1$ is necessary to admit such a local construction \cites{Ciarlet.78,Zenisek.74,Lai.Schumaker07}.

The above lower bounds on the degree can be alleviated by considering so-called macro-elements, where the partition $\cT$ is further refined in a specific manner (often referred to as splits). In case $\cT$ is a triangulation, the most famous examples are the Clough--Tocher (CT) split \cites{Clough.Tocher.65,Lai.Schumaker07,Sabl85,Ciarlet.78} and the Powell--Sabin (PS) 6 and 12 splits \cites{Alfeld.Schumaker02,Powell.Sabin77,Sabl85,Lai.Schumaker07,Schumaker.Sorokina06}. They subdivide each triangle of $\cT$ into $3$, $6$, and $12$ subtriangles, respectively.
To achieve global $C^2$ smoothness, polynomial pieces of at least degree $d=7$ are necessary for the CT split, while at least degree $d=5$ is required for both PS splits of $\cT$. All these spline spaces have a stable dimension and possess optimal approximation power \cites{Lai.Schumaker01,Lai.Schumaker03,Lai.Schumaker07}. 
Other common macro-elements also require at least degree $d=5$ to realize $C^2$ splines with the above properties on a refined partition only consisting of triangles \cite[Section~7.7]{Lai.Schumaker07}.

The Bernstein polynomial basis is the most common tool for the construction and analysis of splines on a given triangulation 
$\cT$ \cite{Lai.Schumaker07}, as it helps in localizing imposition of smoothness conditions across edges of (the refinement of) $\cT$. 
Interesting alternatives have been developed for CT and PS splits in \cites{Cohen.Lyche.Riesenfeld13,Lyche.Merrien18,Lyche.Merrien.Sauer.21,Lyche.Muntingh17}, where a simplex spline basis for the local spline space over a triangle of $\cT$ has been considered. Such a basis behaves like a Bernstein polynomial basis for imposing smoothness across edges of $\cT$ and like a B-spline basis internal to each triangle of $\cT$. 
Neither the Bernstein polynomial basis nor the simplex spline basis provide a global basis for the full spline space on (the refinement of) $\cT$. To achieve a global basis, one may apply the general framework of minimal determining sets via the local Bernstein basis; see \cite{Lai.Schumaker07}. 
Global B-spline bases have been constructed for $C^1$ PS spline spaces on triangulations \cites{Dierckx97,Groselj.Speleers17,Groselj.Speleers21,Speleers15b}, for PS spline spaces with higher smoothness \cites{Groselj16,Speleers10,Speleers13}, and for CT spline spaces \cite{Speleers10b}.

While in the univariate case $C^2$ cubics are probably the best known and most used splines, the above discussion shows that dealing with $C^2$ cubics in the bivariate setting is an arduous task.
In this paper, we address the problem of building and handling $C^2$ cubic splines on a suitable refinement of any given triangulation $\cT$.
Our wish list for the spline space consists of stable dimension, optimal approximation power, and local construction on any (refined) triangle of~$\cT$. Moreover, we want a practical construction of a stable global basis for the space at our disposal. 

Despite the high smoothness and the minimum gap between degree and smoothness, a $C^2$ cubic space can be obtained by splitting any triangle $\Delta$ in $\cT$ according to a degree-dependent scheme introduced by Wang and Shi \cite{Wang.90}. Contrarily to the well-known splits mentioned above, the family of Wang--Shi (WS) splits generates a very large number of polygonal pieces in each $\Delta$; for cubics we get a set of 75 polygons which includes triangles, quadrilaterals, and pentagons. In practice, this complex geometry hampers a piecewise treatment --- in terms of a local polynomial basis --- of spline functions on WS splits and discourages the use of such an interesting space.

To overcome this issue, we propose a simplex spline basis for the local space of $C^2$ cubics on the (cubic) WS split of any $\Delta$ in $\cT$. The basis functions enjoy the following properties:
\begin{itemize}
  \item they form a nonnegative partition of unity;
  \item they inherit recurrence relations and differentiation formulas from the simplex spline structure;
  \item for each of them, the restriction to a boundary edge of $\Delta$ reduces to a classical $C^2$ cubic univariate B-spline;
  \item they admit simple conditions for $C^2$ joins to neighboring triangles in $\cT$;
  \item cubic polynomials can be represented through a Marsden-like identity;
  \item they lead to well-conditioned collocation matrices for Lagrange and Hermite interpolation using certain sites;
  \item a control net can be formed that mimics the shape of the spline function and exhibits distance $O(h^2)$ to any one of its control points from its surface, where $h$ is the length of the longest edge.
\end{itemize}
Thanks to the characteristics of the simplex spline basis, one can avoid to consider separate polynomial representations on each of the polygonal subelements of $\Delta$. Instead, there is a single control net to facilitate control and early visualization of a spline function over each element $\Delta$ in $\cT$. 
This makes that the complex geometry of the WS split is transparent to the user.
However, the simplex spline basis is a local basis and does not provide a global basis for the full space of $C^2$ cubics on the (cubic) WS refinement of $\cT$. To this end, we extend the concept of minimal determining sets and use the simplex spline basis as a stepping stone to the construction of a stable global basis for the full space.

The remainder of this paper is divided into four sections. In Section~\ref{sec:preliminaries}, we summarize the definition and some properties of simplex splines and describe the family of WS splits. In Section~\ref{sec:simplex-bases}, we present a local simplex spline basis for the refinement of a single triangle and discuss some of its properties. To simplify some computations, an alternative basis is also provided. Smoothness conditions across the edges of the given triangulation and stable global bases for the $C^2$ cubic space on the (cubic) WS refinement of any triangulation are considered in Section~\ref{sec:spline-space}. Section~\ref{sec:conclusion} collects some concluding remarks about implementation aspects, possible application areas, and a higher-order extension of the basis.
Finally, the appendix aggregates data related to the presented simplex spline bases that might be useful for practical computations.

Throughout the paper, we use small boldface letters for vectors and capital boldface letters for matrices. Calligraphic letters like $\cB$ indicate sets, and we write $\#\cB$ for the cardinality of $\cB$. Function spaces are denoted by symbols like $\spS$. In particular, $\spP_d$ stands for the space of bivariate polynomials with real coefficients of total degree $\leq d$. The partial derivatives in $x$ and $y$ are denoted by $D_x$ and $D_y$, respectively. Given a vector $\vu$, the associated directional derivative is denoted by~$D_\vu$. The directional derivative in the direction of the vector from point $\vp_1$ to $\vp_2$ is denoted by~$D_{\vp_1\vp_2}$.

\section{Preliminaries}
\label{sec:preliminaries}
This section contains some preliminary material about simplex splines and the splits of interest in the rest of the paper.

\subsection{A short summary of simplex splines}\label{sec:summary_simplex}
For $\s\in\NN$, $d\in\NN_0$, let $n:=d+\s$ and $\Xi:=\{\vxi_1,\ldots,\vxi_{n+1}\}$ be a sequence of possibly repeated points in $\RR^\s$ called {knots}. The {multiplicity} of a knot is the number of times it occurs in the sequence.
Let $\langle\cdot\rangle$ denote the convex hull of a sequence of points.
For the sake of simplicity, we assume $\langle\Xi\rangle$ is nondegenerate, i.e., $\vol_\s(\langle\Xi\rangle)>0$. 
Let $\sigma = \langle\overline\vxi_1,\ldots,\overline\vxi_{n+1}\rangle$
be any simplex in $\RR^n$ with $\vol_n(\sigma)>0$, whose projection $\pi: \RR^n \to  \RR^\s$ onto the first $\s$ coordinates satisfies $\pi(\overline\vxi_i) = \vxi_i$ for $i = 1, \ldots, n+1$.

The {simplex spline} $M_\Xi$ can be defined geometrically by
$$
M_\Xi: \RR^\s\to  \RR,\quad M_\Xi(\vx) := \dfrac{\vol_{n-\s} \big(\sigma\cap \pi^{-1}(\vx)\big)}{\vol_n(\sigma)}.
$$
For $d=0$ we have
\begin{equation*} 
M_\Xi(\vx)=\begin{cases}
1/\vol_n(\langle \Xi\rangle), &\vx\in \text{interior of } \langle \Xi\rangle,\\ 
0,&\text{if } \vx\notin \langle \Xi\rangle,
\end{cases}
\end{equation*}
and the value of $M_\Xi$ on the boundary of $\langle \Xi\rangle$ has to be dealt with separately.
For properties of $M_\Xi$ and proofs, we refer the reader to, e.g., \cites{Micchelli79,Prautsch.Boehm.Paluszny02}. Here, we mention:
\begin{itemize}
\item \textbf{Knot dependence}: $M_\Xi$ only depends on $\Xi$; in particular, it is independent of the choice of $\sigma$ and the ordering of the knots.
\item \textbf{Support}:  
$M_\Xi$ has support $\langle \Xi\rangle$.
\item \textbf{Normalization}: 
$M_\Xi$ has unit integral.
\item \textbf{Nonnegativity}: 
$M_\Xi$ is a nonnegative piecewise polynomial of total degree $d$.
\item \textbf{Differentiation formula ($A$-recurrence)}:
For any $\vu\in\RR^\s$ and any $a_1,\ldots,a_{d+\s+1}$ such that $\sum_i a_i\vxi_i=\vu$, $\sum_i a_i=0$, we have
$$
D_\vu M_\Xi=(d+\s)\sum_{i=1}^ {d+\s+1} a_iM_{[\Xi\setminus \vxi_i]}.
$$
\item \textbf{Recurrence relation ($B$-recurrence)}:
For any $\vx\in\RR^\s$ and any $b_1,\ldots,b_{d+\s+1}$ such that $\sum_i b_i\vxi_i=\vx$, $\sum_i b_i=1$, we have
$$
M_\Xi(\vx)=\frac{d+\s}{d}\sum_{i=1}^ {d+\s+1} b_iM_{[\Xi\setminus \vxi_i]}(\vx).
$$
\item \textbf{Knot insertion formula ($C$-recurrence)}:
For any $\vy\in\RR^\s$ and any $c_1,\ldots,c_{d+\s+1}$ such that $\sum_i c_i\vxi_i=\vy$, $\sum_i c_i=1$, we have
$$M_\Xi=\sum_{i=1}^ {d+\s+1} c_iM_{[\Xi\cup \vy\setminus \vxi_i]}.
$$
\end{itemize}
If $\s=1$ then $M_\Xi$ is the univariate B-spline of degree $d$ with knots $\Xi$, normalized to have its integral equal to one.

In the bivariate case, $\s=2$, the lines in the complete graph of $\Xi$ are called {knot lines}. They provide a partition of $\langle\Xi\rangle$ into polygonal elements. The simplex spline $M_\Xi$ is a polynomial of degree $d=\# \Xi-3$ in each region of this partition, and across a knot line
\begin{equation*} 
M_\Xi\in C^{d+1-\mu},
\end{equation*}
 where
$\mu$ is the number of knots on that knot line, including multiplicities.

\subsection{The Wang--Shi splits}
Given three noncollinear points $\vp_1,\vp_2,\vp_3$ in $\RR^2$, the triangle
$\Delta:=\langle\vp_1,\vp_2,\vp_3\rangle$
with vertices $\vp_1,\vp_2,\vp_3$ will serve as our macro-triangle.
Given a degree $ d\in\NN$, we divide each edge of $\Delta$ into $d$ equal segments, respectively, resulting into $3d$ boundary points. Then, we refine $\Delta$ into a number of subelements delineated by the complete graph connecting those boundary points. This is called the $\WSd$ split of $\Delta$ as it was originally proposed by Wang and Shi \cite{Wang.90}.
We denote by $\Delta_{\WSd}$ the obtained mesh structure,
and by $\cP_d$ the set of polygons in $\Delta_{\WSd}$. All the possible intersections of the various lines connecting the boundary points are called vertices of $\Delta_{\WSd}$. In particular, the boundary points are vertices of $\Delta_{\WSd}$.
The cases $d=2,3,4$ are shown in Figure~\ref{fig:splits}. For $d=2$ we obtain the well-known PS-12 split \cite{Powell.Sabin77}, while for $d=1$ we have $\cP_1=\{\Delta\}$.
Note that for $d>2$ not all elements of $\cP_d$ are triangles.
We consider the space
\begin{equation} \label{eq:defsplinespace}
\spSd(\Delta_{\WSd}):=\{\spline\in C^{d-1}(\Delta): \spline_{|\tau} \in\spP_d,\  \forall\ \tau\in\cP_d\}.
\end{equation}
When the degree increases, the complexity of the mesh grows quickly.
There are
\begin{itemize}
\item $3d$ boundary points and $3d(d-1)$ interior lines in the complete graph;
\item the maximum number of lines intersecting at an interior vertex is $3,3,4,5,6,7,6$ for $d=2,3,\ldots,8$;
\item the number of vertices of $\Delta_{\WSd}$ is $10, 58, 178, 558,1255,2532, 4786$ for $d=2,3,\ldots,8$.
\end{itemize}

\begin{figure}[t!]
\centering
\includegraphics[width=4.3cm]{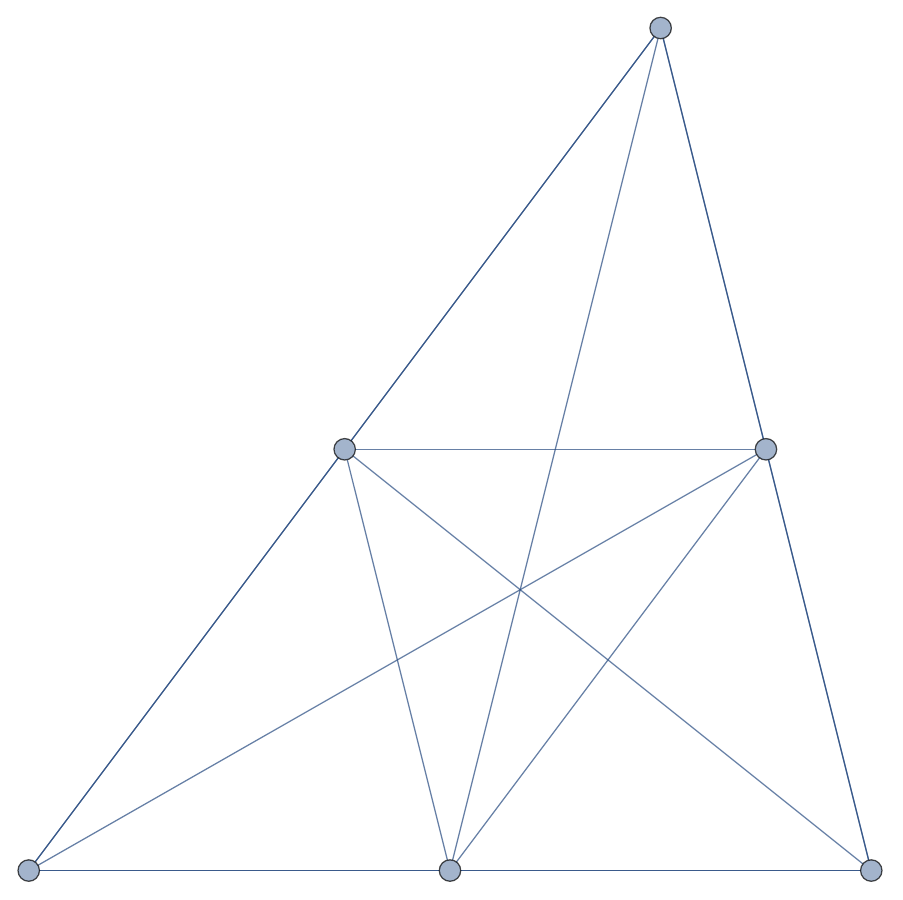} \quad
\includegraphics[width=4.3cm]{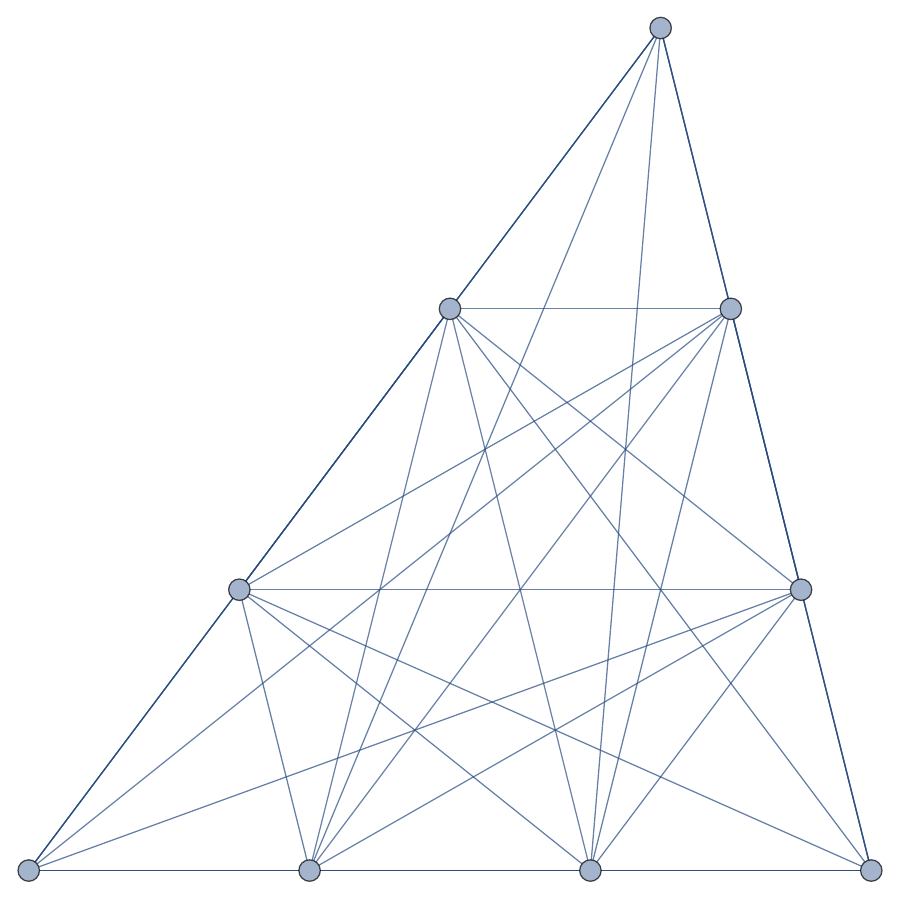} \quad
\includegraphics[width=4.3cm]{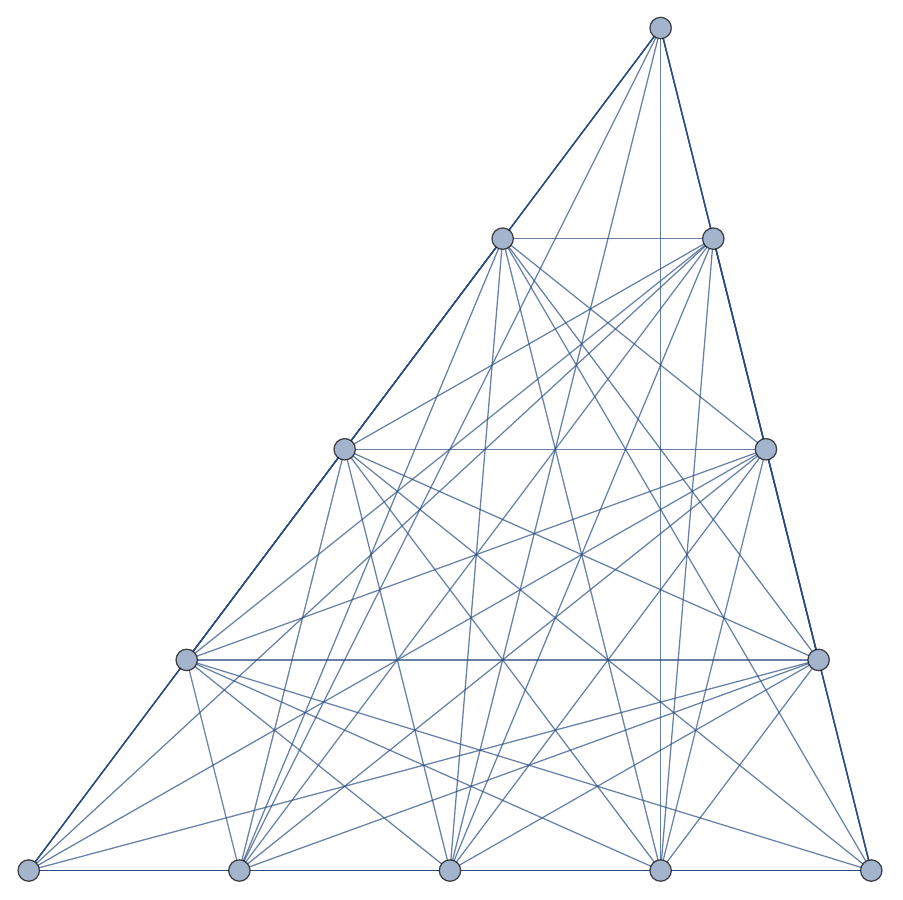} \\[0.2cm]
\caption{WS$_d$ splits for $d=2,3,4$.}
\label{fig:splits}
\end{figure}

The dimension of $\spSd(\Delta_{\WSd})$ can be computed using the general dimension formula for spline spaces over cross-cut partitions from \cite[Theorem~3.1]{Chui.Wang.83}.
A partition $\cT_c$ of a domain $\Omega$ is called a cross-cut partition if it is obtained by drawing lines across $\Omega$. Let $\spS^r_d(\cT_c)$ be the space of functions in $C^r(\Omega)$ which belong to $\spP_d$ when restricted to any polygon of $\cT_c$.

\begin{theorem} \label{thm:chui-wang}
Let $\Omega$ be a simply connected domain in $\spR^2$. Let $\cT_c$ be a cross-cut partition of $\Omega$, with $m$ cross-cuts, $n$ interior vertices $\vv_1,\ldots,\vv_n$, and $m_k$ cross-cuts intersecting at $\vv_k$, $k=1,\ldots,n$.
Then, the dimension of the spline space $\spS^r_d(\cT_c)$, $0\leq r\leq d-1$ is
\begin{equation} \label{eq:dim-crosscut}
  \dim(\spS^r_d(\cT_c))=\binom{d+2}{2} + m \binom{d-r+1}{2}+\sum_{k=1}^{n}\varsigma(m_k),
\end{equation}
where
\begin{equation*}
  \varsigma(l):=\frac{1}{2}\left(d-r-\left\lfloor\frac{r+1}{l-1}\right\rfloor\right)_+
  \left((l-1)d -(l+1)r +(l-3) + (l-1)\left\lfloor\frac{r+1}{l-1}\right\rfloor\right).
\end{equation*}
As usual, $\lfloor x\rfloor$ denotes the largest integer smaller than or equal to $x$, and $(x)_+ := \max\{x,0\}$.
\end{theorem}

\begin{theorem} \label{thm:dimension}
Assuming at most $d+1$ lines intersect at an interior vertex of $\Delta_{\WSd}$, we have
$$
\dim(\spSd(\Delta_{\WSd}))=\dim\spP_d+m,
$$
where $m=3d(d-1)$ is the number of interior lines in the complete graph.
\end{theorem}
\begin{proof}
We make use of the dimension formula \eqref{eq:dim-crosscut} in Theorem~\ref{thm:chui-wang}. In our case we have $r=d-1$, and it is easy to check that
$$
\varsigma(l)=0, \quad l=1,\ldots,d+1.
$$
Since at most $d+1$ lines cross at each interior vertex,
it immediately follows from \eqref{eq:dim-crosscut} that
$$
\dim(\spSd(\Delta_{\WSd})) = \binom{d+2}{2} + m,
$$
which completes the proof.
\end{proof}

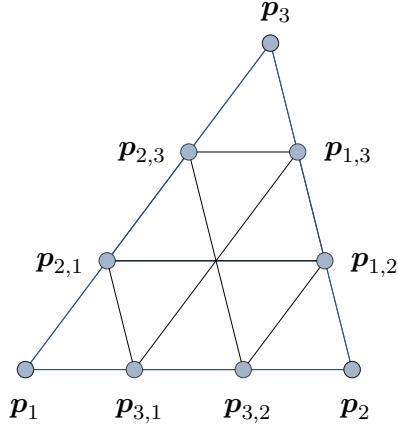
\begin{figure}[t!]
\centering
\begin{picture}(130,160)(0,-5)
\put(0,10){\includegraphics[width=4.75cm]{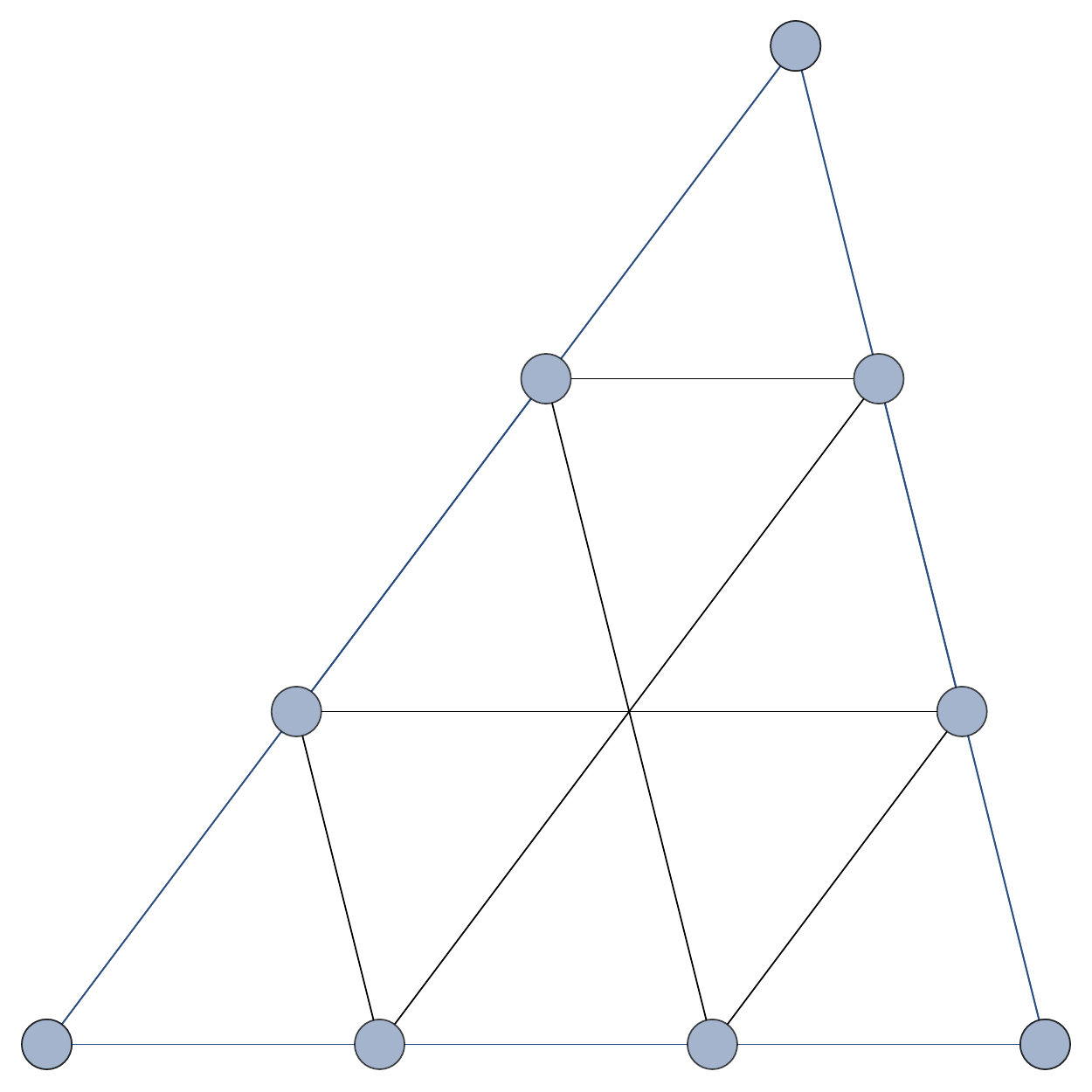}}
\put(0,0){$\vp_1$}
\put(125,0){$\vp_2$}
\put(95,150){$\vp_3$}
\put(40,0){$\vp_{3,1}$}
\put(81,0){$\vp_{3,2}$}
\put(41,97){$\vp_{2,3}$}
\put(10,55){$\vp_{2,1}$}
\put(129,55){$\vp_{1,2}$}
\put(119,97){$\vp_{1,3}$}
\end{picture}
\caption{Labeling of the knots on the boundary of the triangle $\Delta$.}
\label{fig:bvertexorder}
\end{figure}

\section{Simplex spline bases for $\spS_3^2(\Delta_{\WS})$}
\label{sec:simplex-bases}
In this section, we focus on the case $d=3$, provide two (scaled) simplex spline bases for the space $\spS_3^2(\Delta_{\WS})$ in \eqref{eq:defsplinespace}, and prove some properties of these bases.
With a slight abuse of notation we also refer to the corresponding basis functions as simplex splines.

\begin{figure}[t!]
\centering
\subfloat[1]{\includegraphics[width=1.5cm]{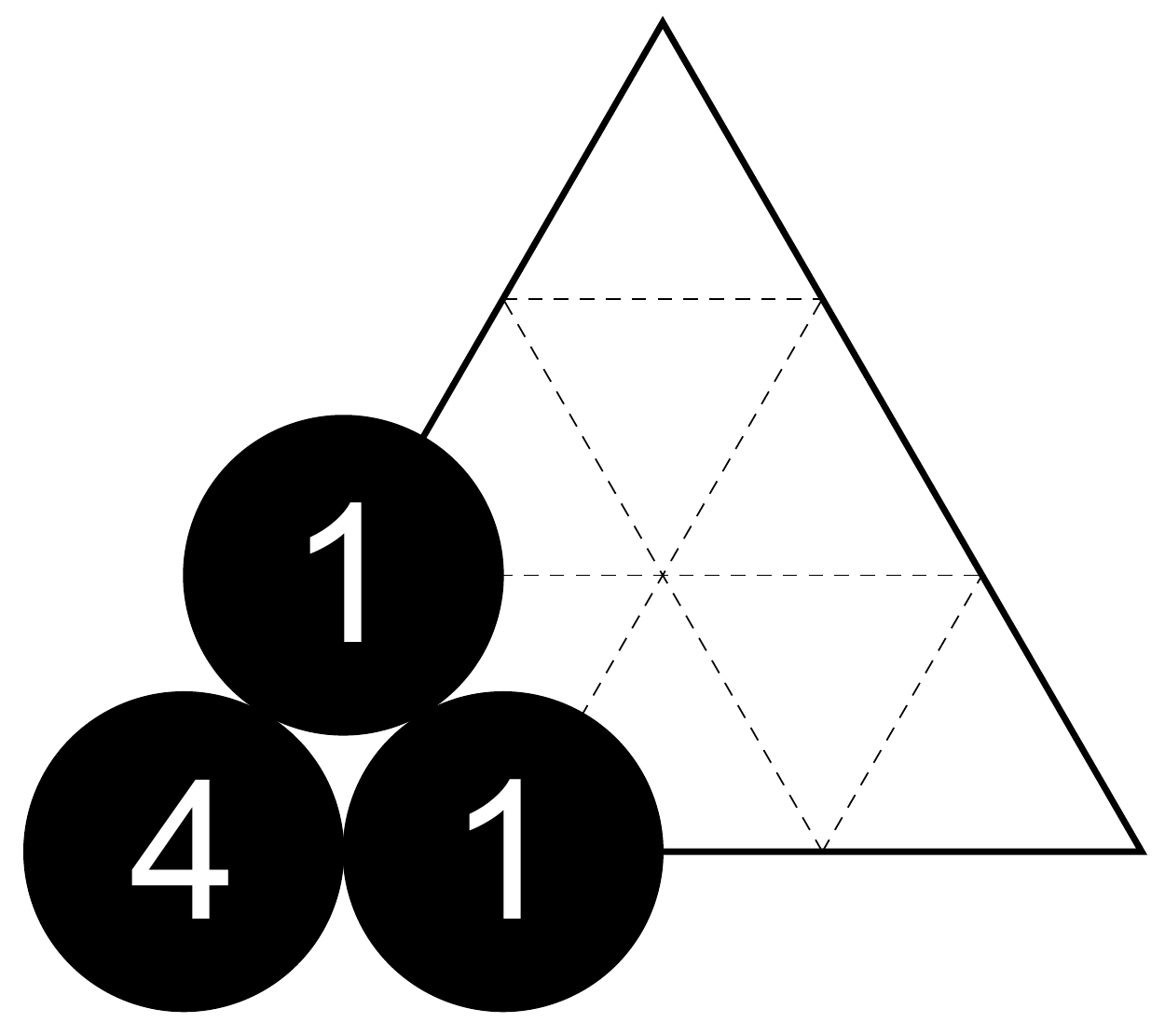}} \quad
\subfloat[2]{\includegraphics[width=1.5cm]{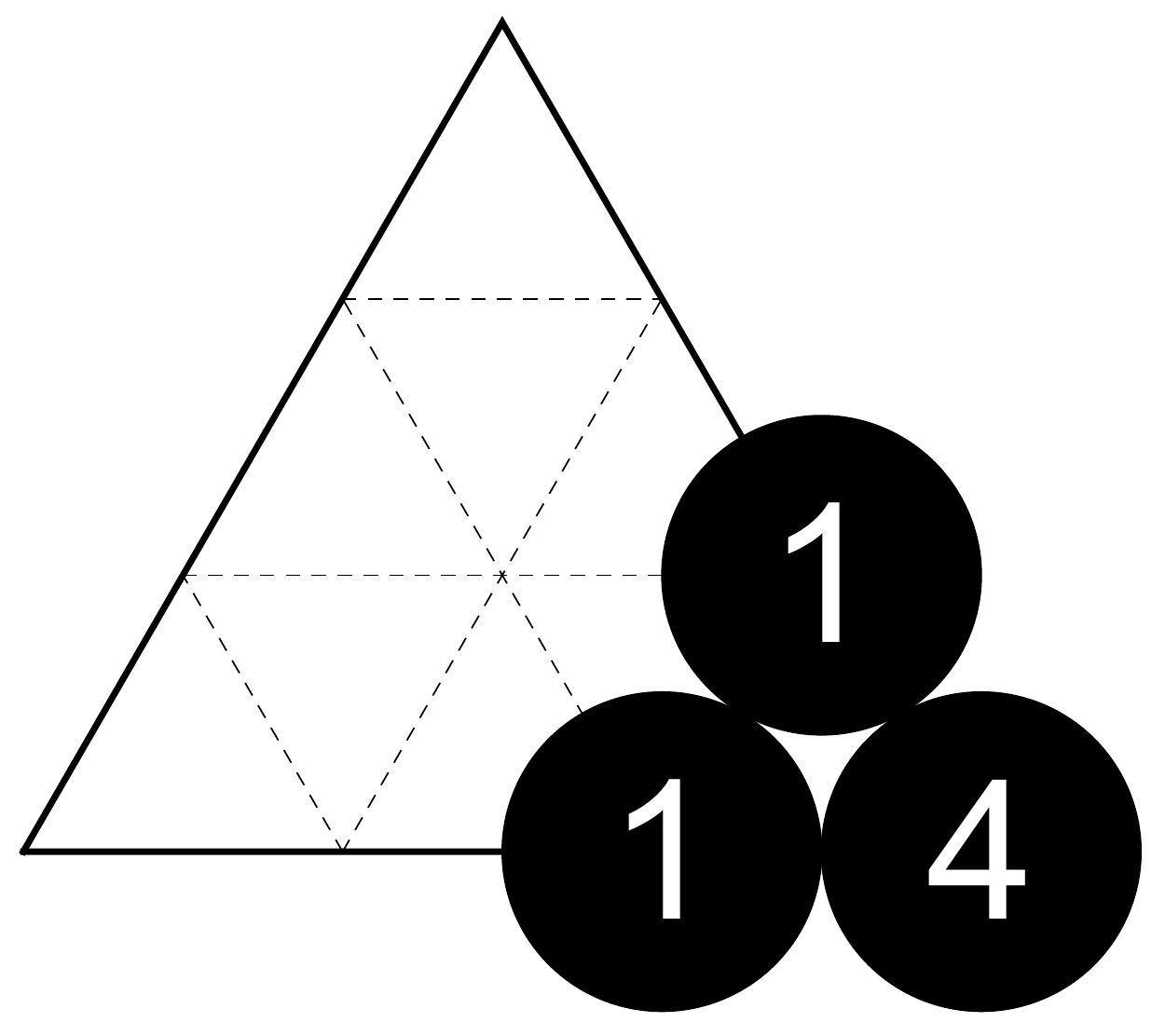}} \quad
\subfloat[3]{\includegraphics[width=1.3cm]{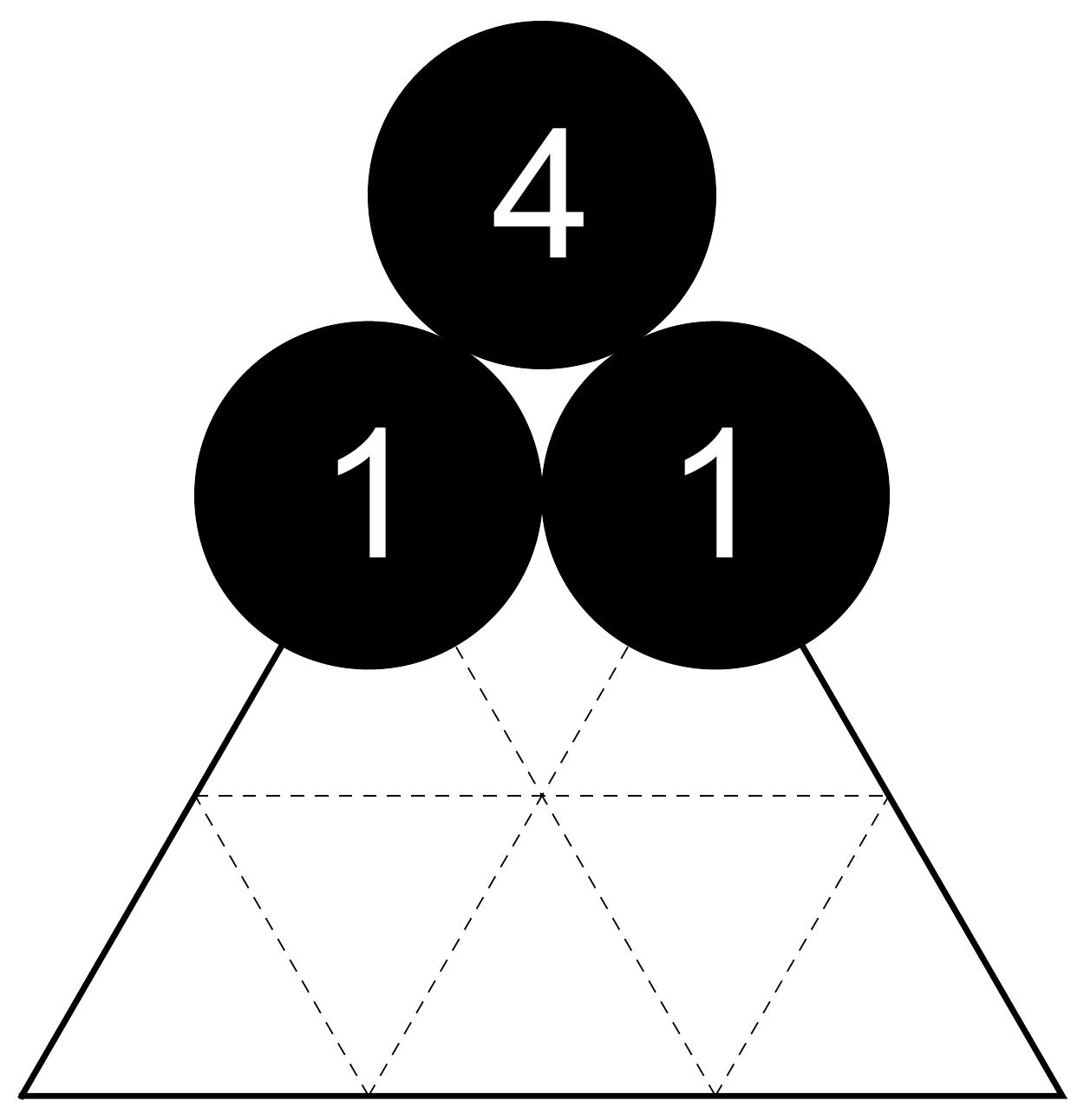}} \quad
\subfloat[4]{\includegraphics[width=1.5cm]{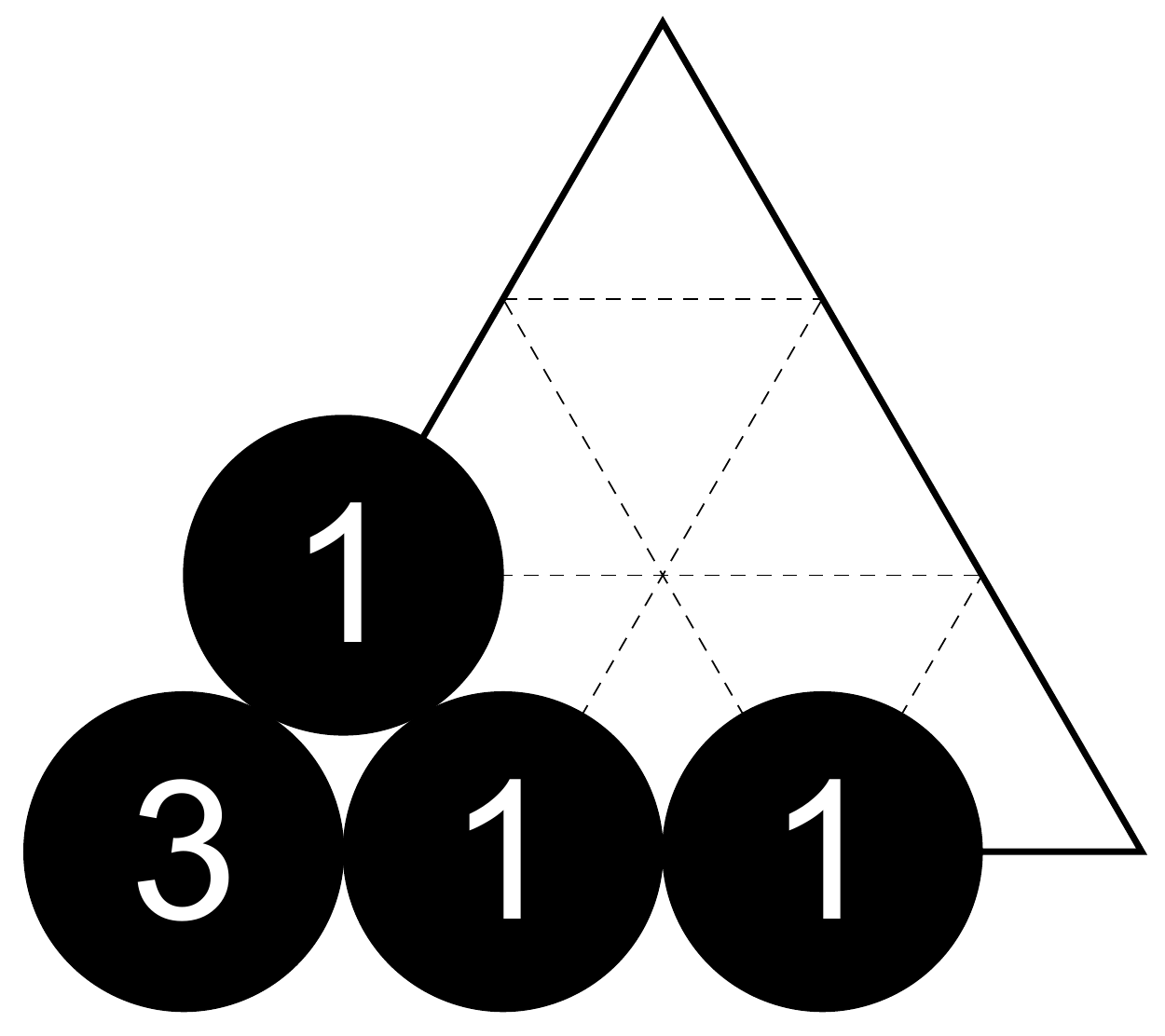}} \quad
\subfloat[5]{\includegraphics[width=1.5cm]{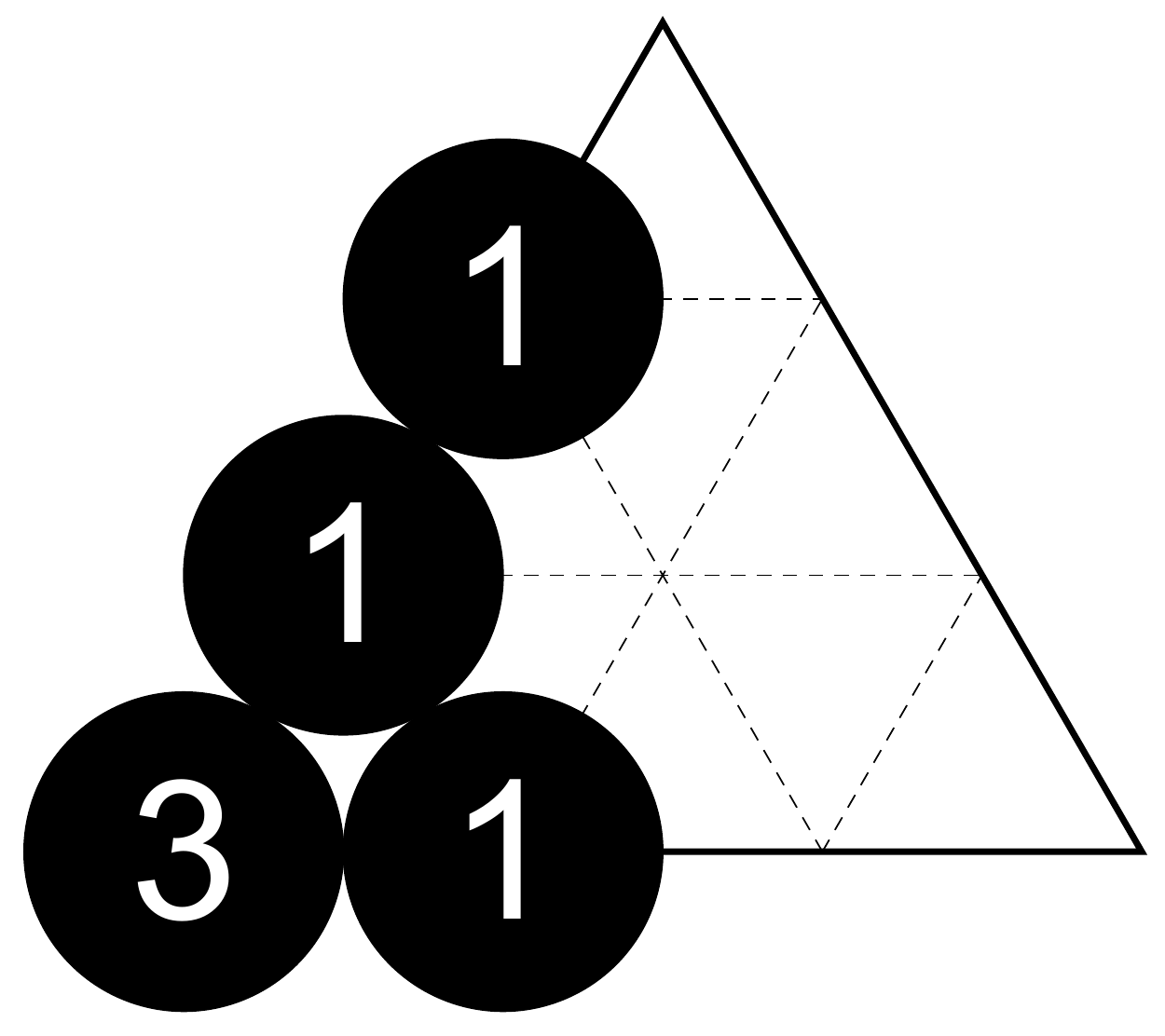}} \quad
\subfloat[6]{\includegraphics[width=1.5cm]{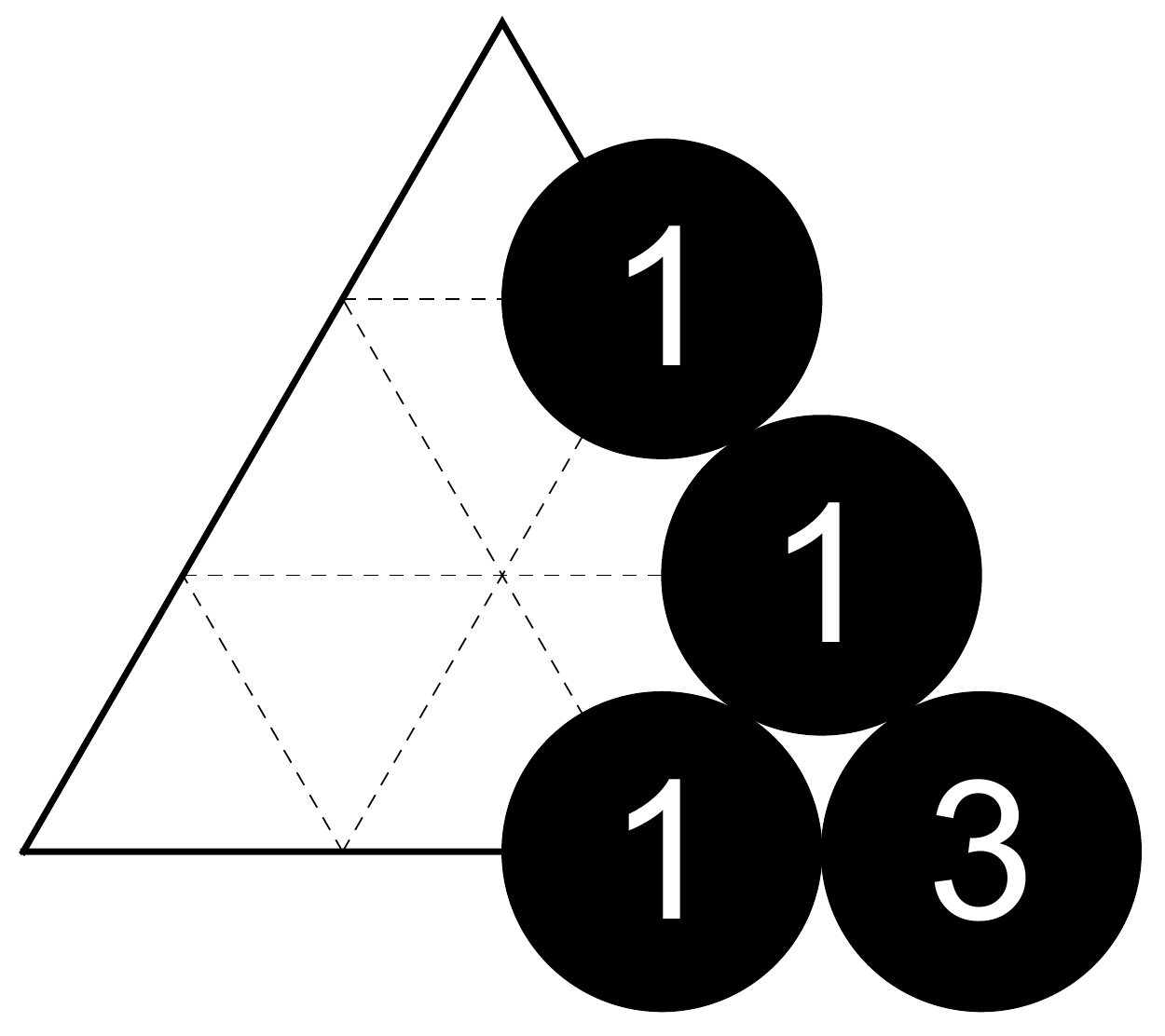}} \quad
\subfloat[7]{\includegraphics[width=1.5cm]{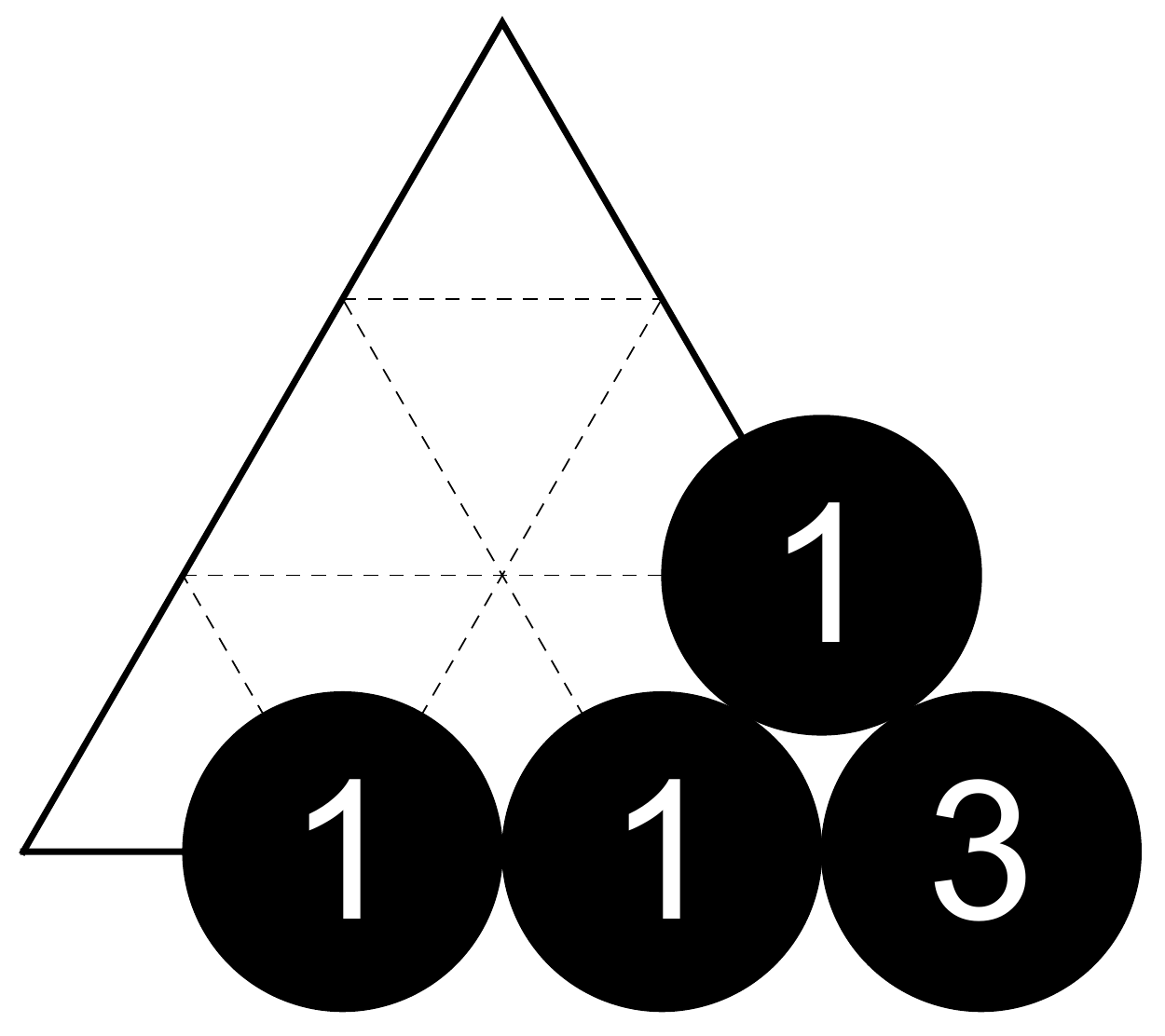}}
\\
\subfloat[8]{\includegraphics[width=1.35cm]{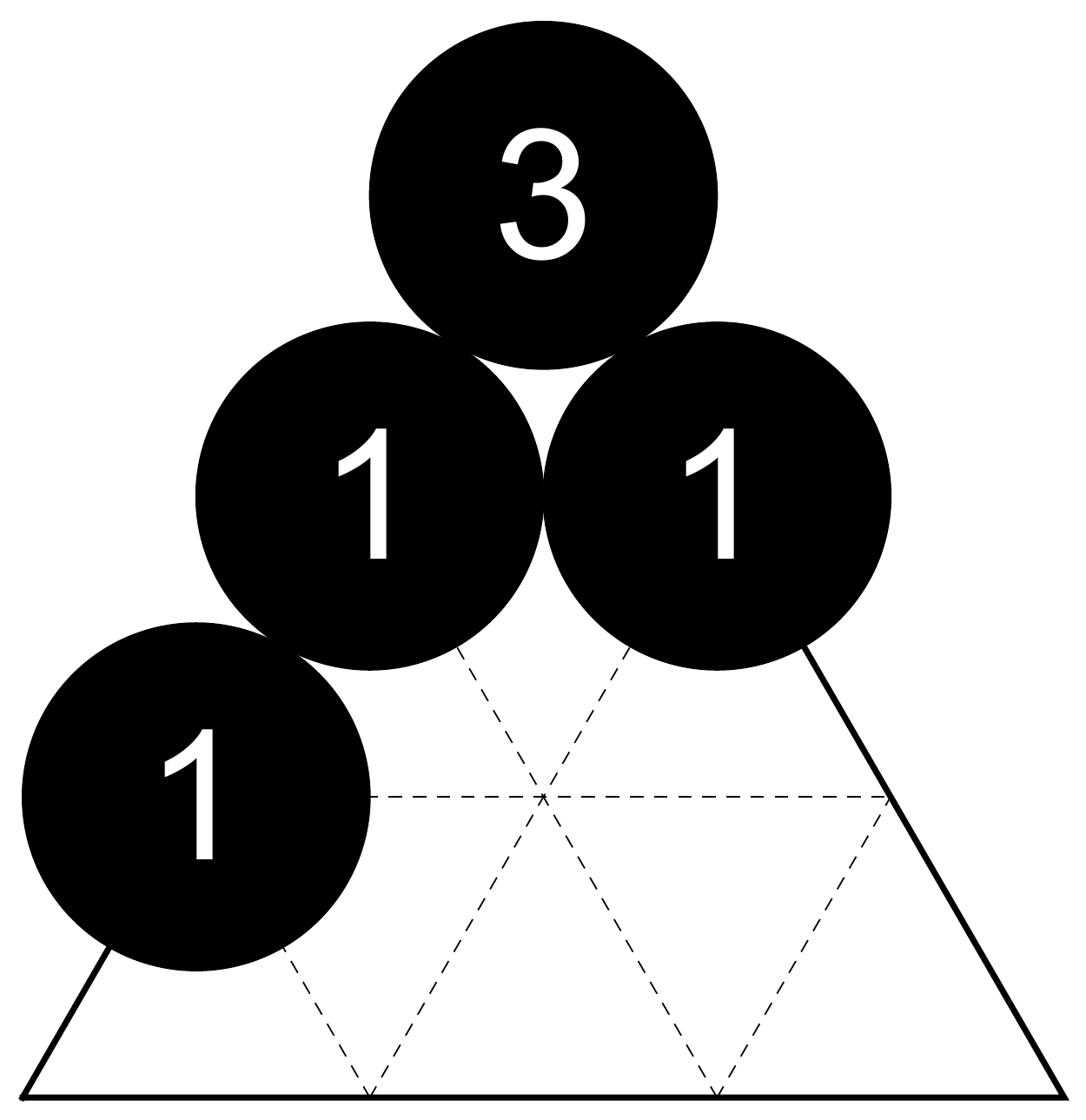}} \quad
\subfloat[9]{\includegraphics[width=1.35cm]{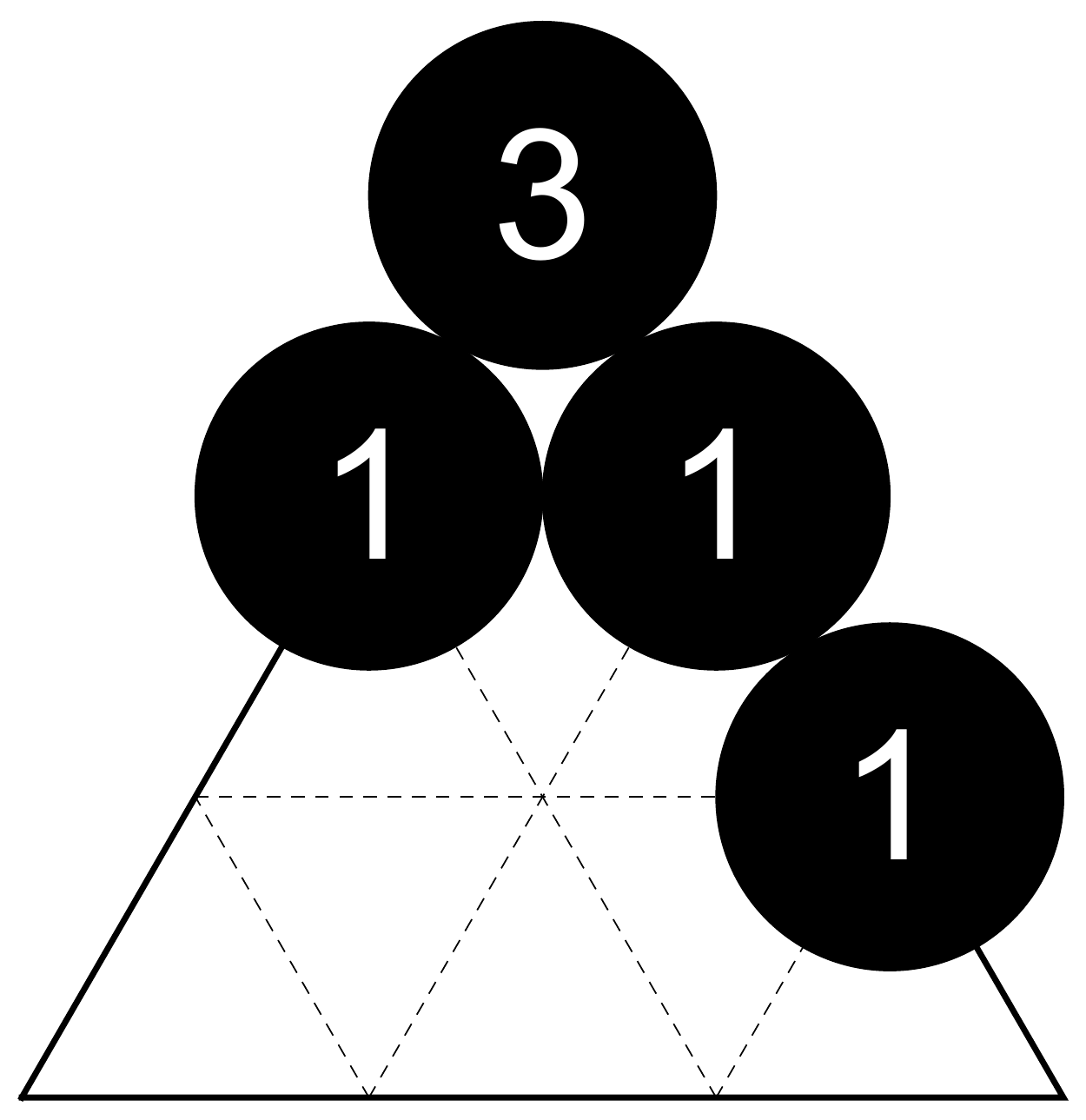}} \quad
\subfloat[10]{\includegraphics[width=1.7cm]{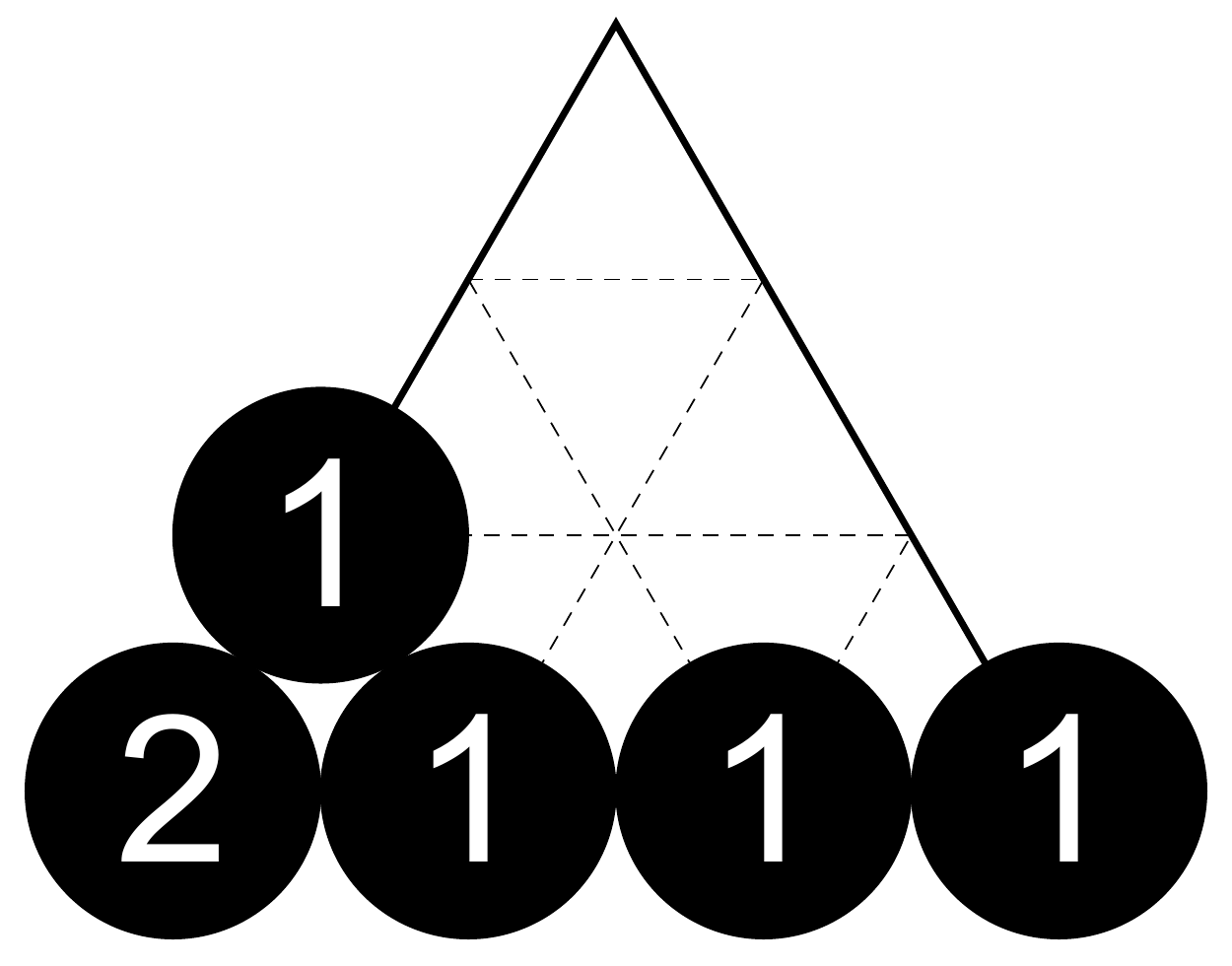}} \quad
\subfloat[11]{\includegraphics[width=1.5cm]{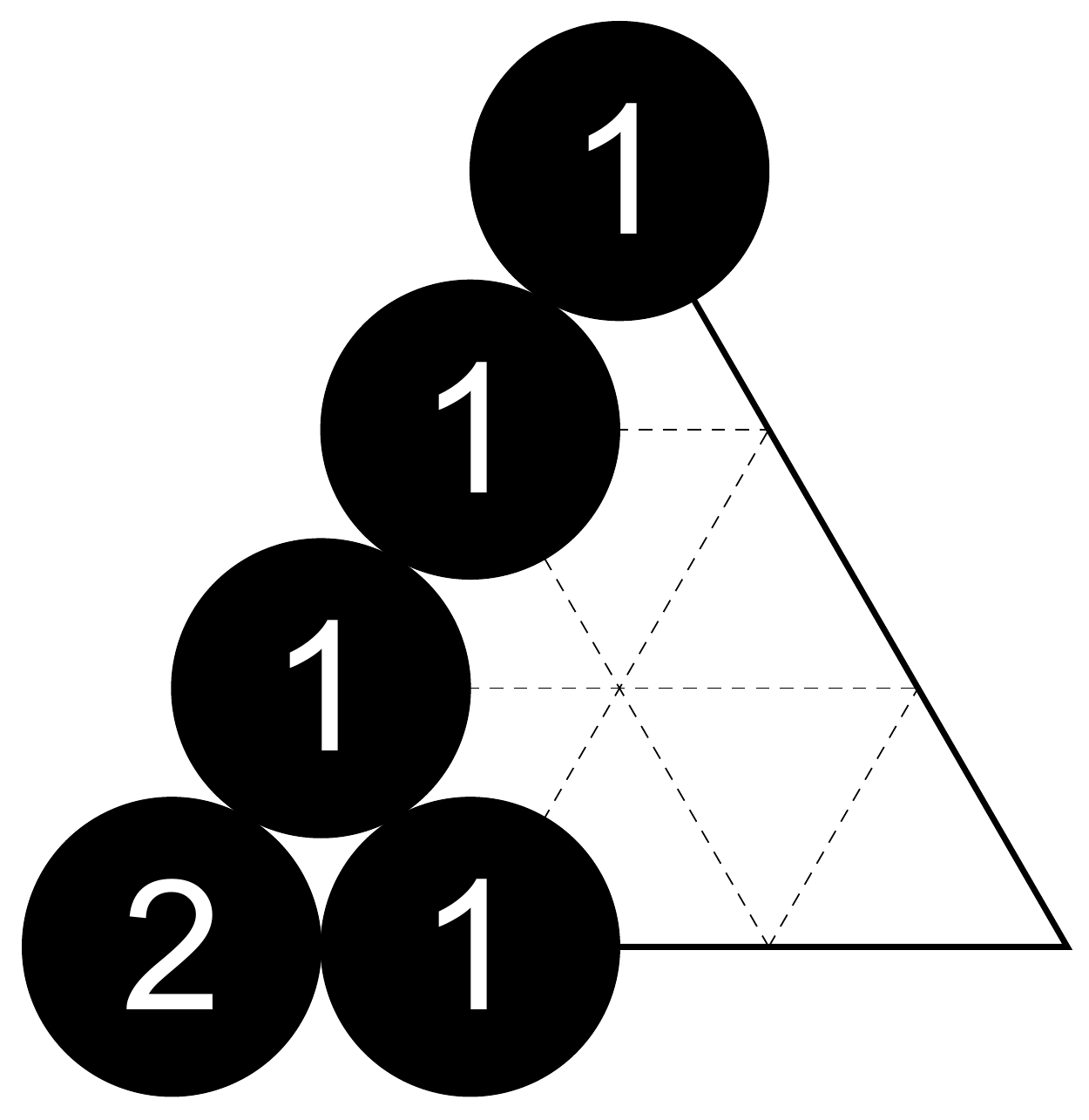}} \quad
\subfloat[12]{\includegraphics[width=1.5cm]{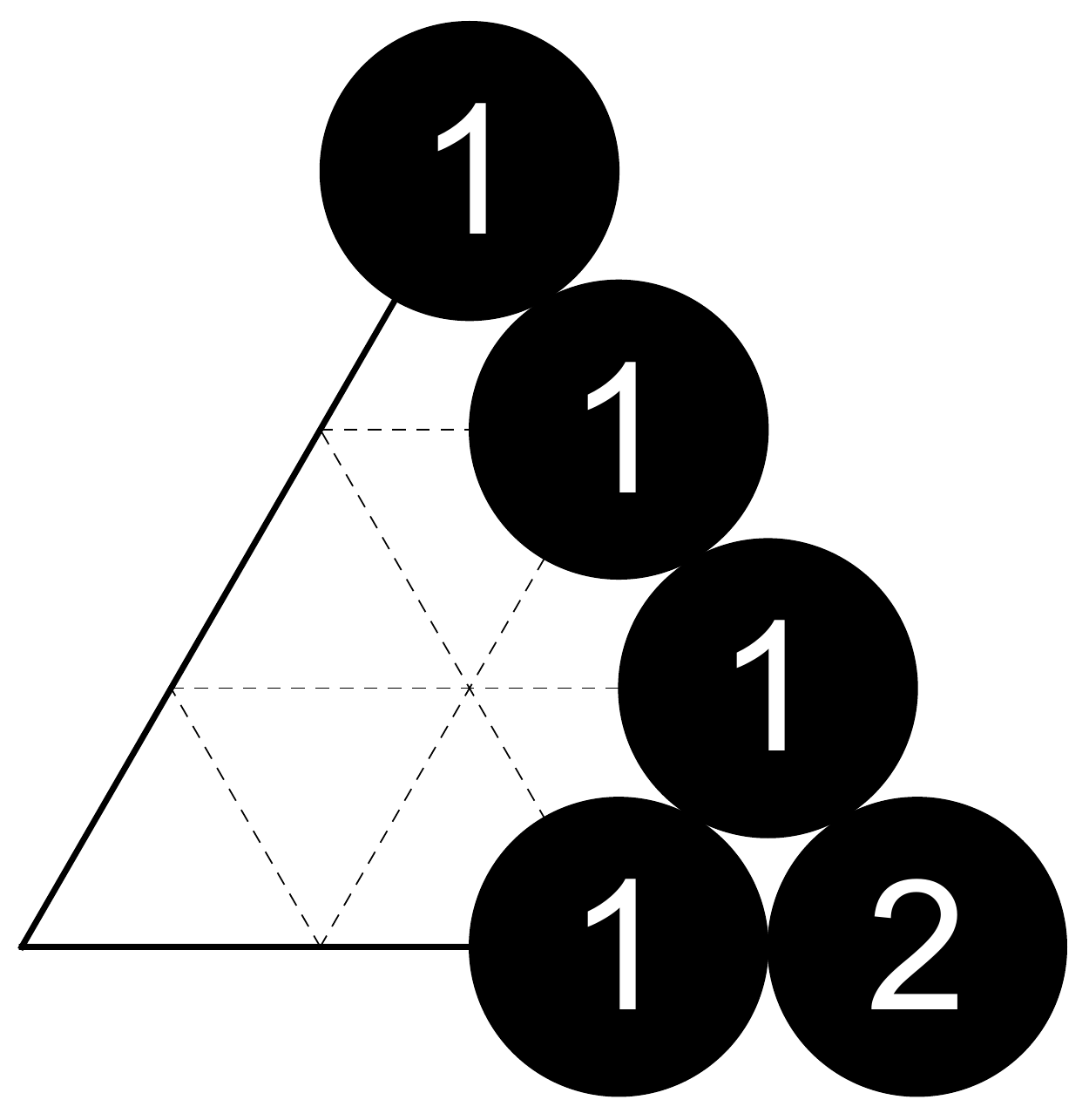}} \quad
\subfloat[13]{\includegraphics[width=1.5cm]{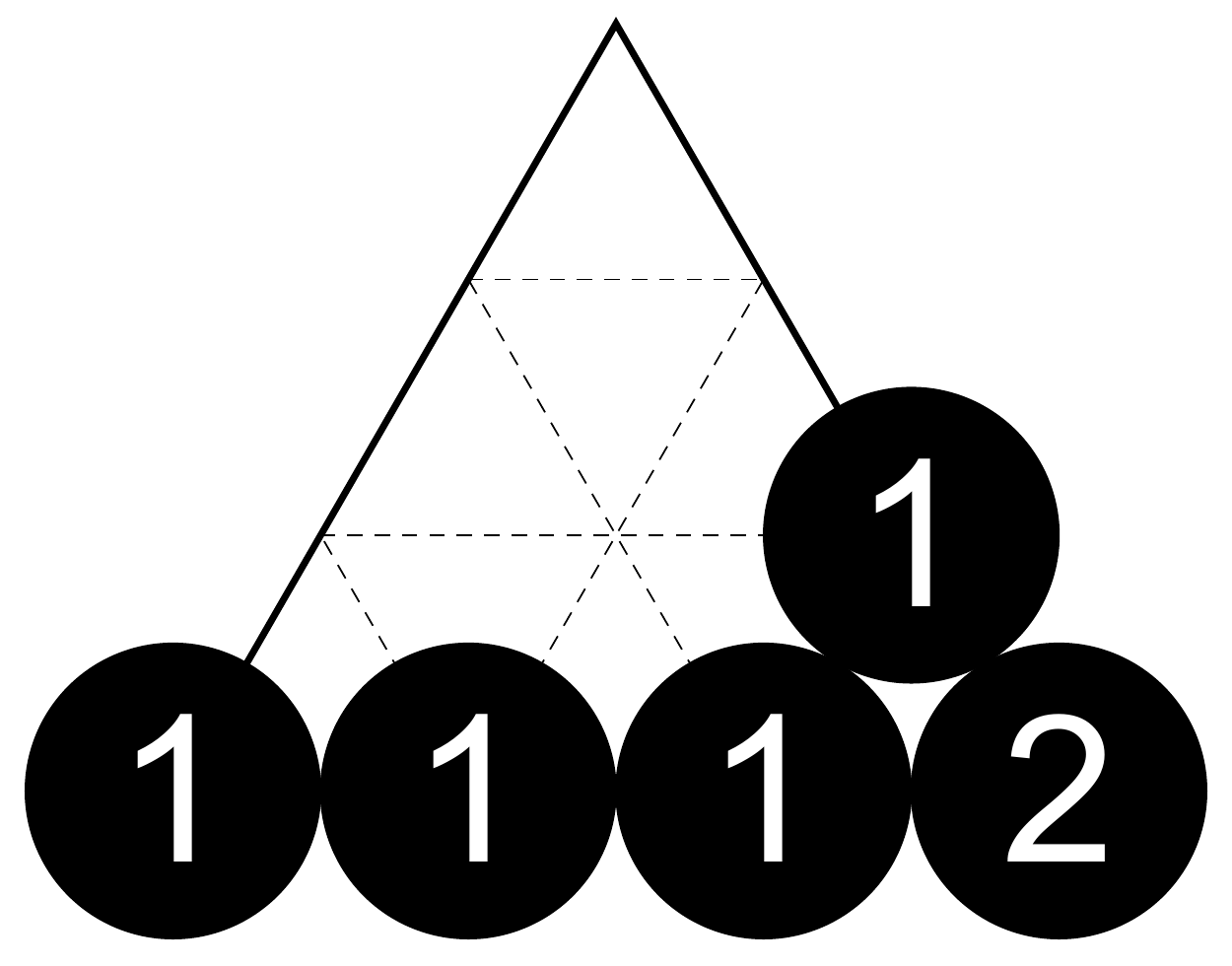}} \quad
\subfloat[14]{\includegraphics[width=1.5cm]{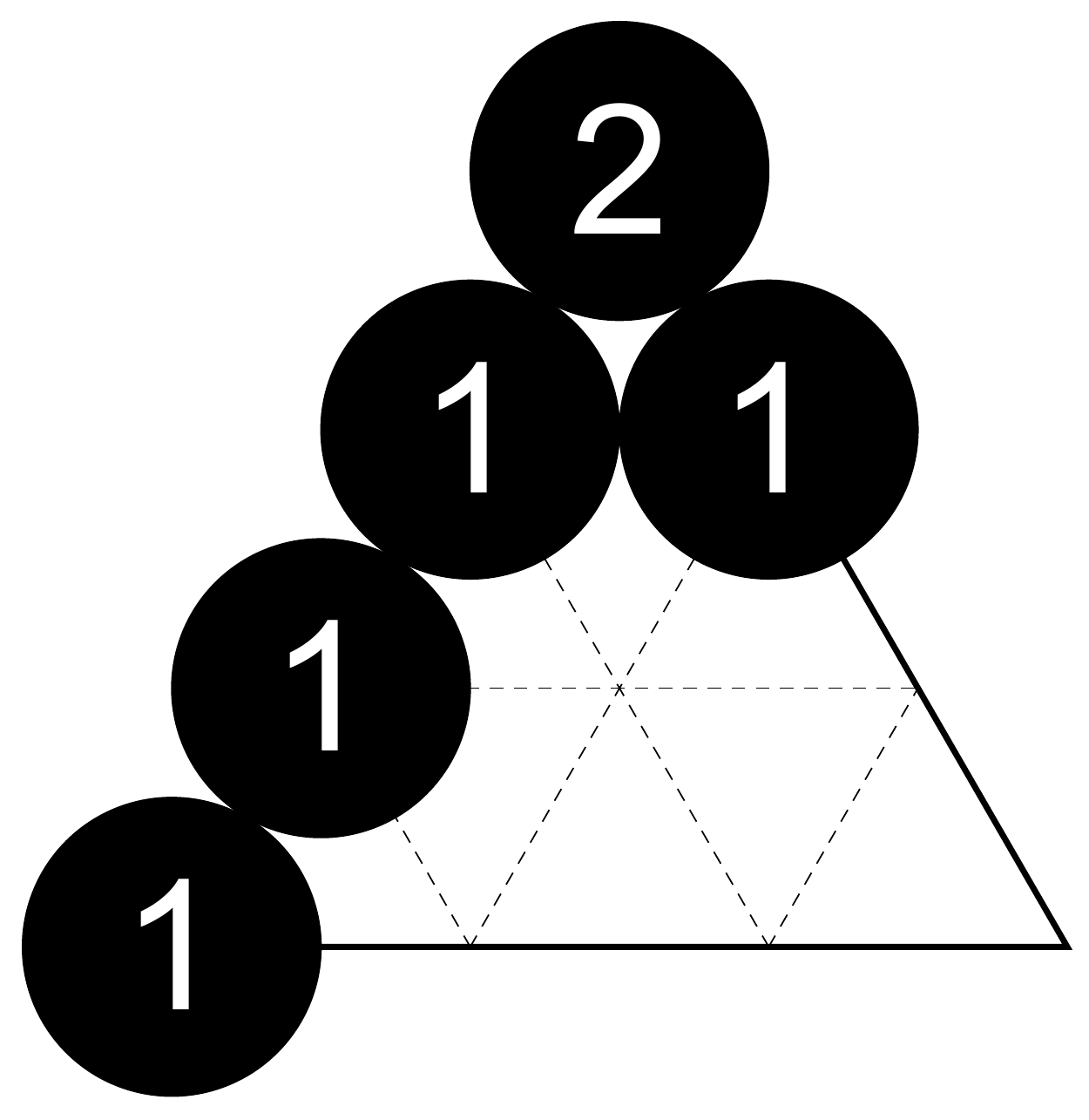}}
\\
\subfloat[15]{\includegraphics[width=1.5cm]{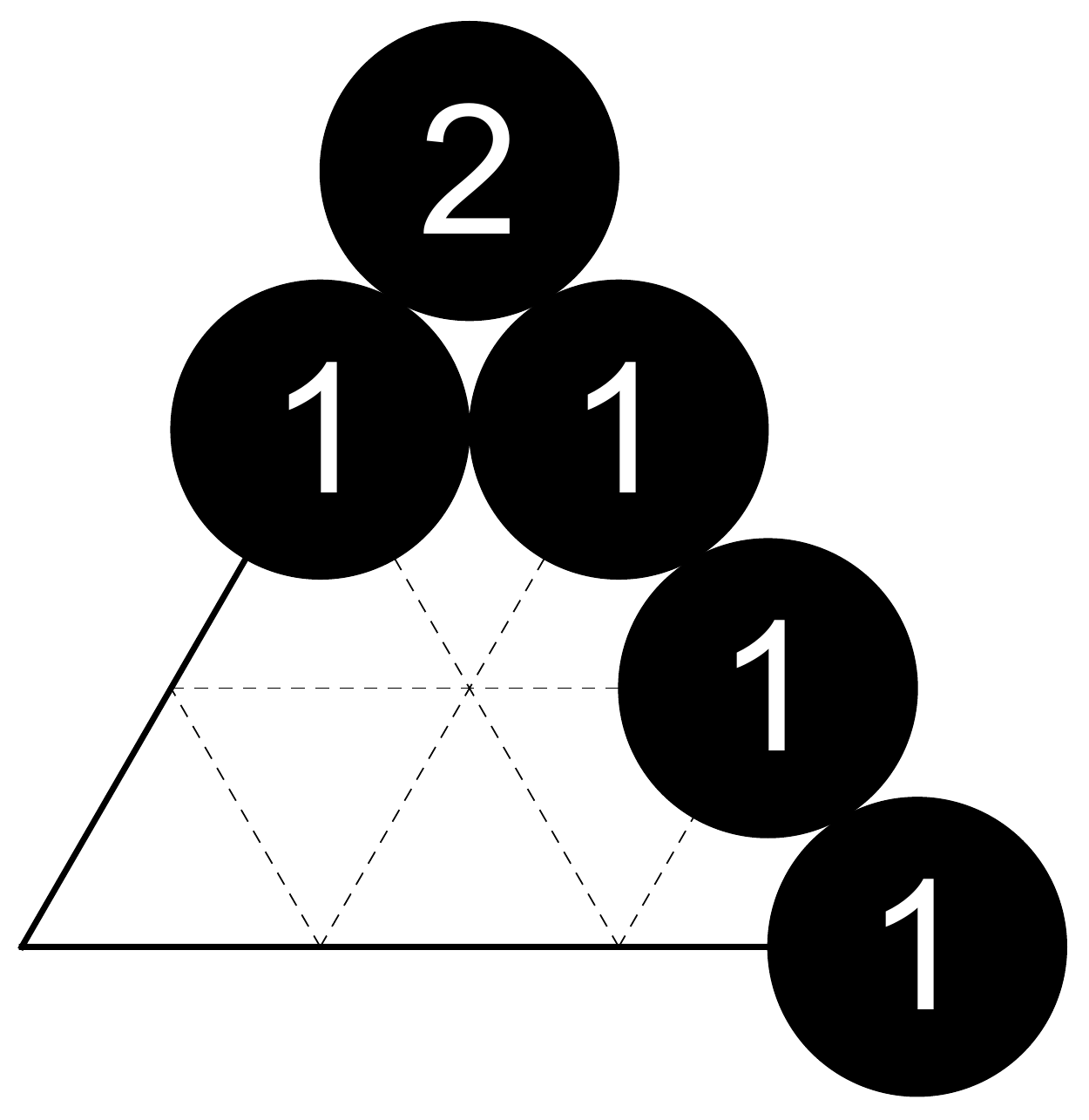}} \quad
\subfloat[16]{\includegraphics[width=1.35cm]{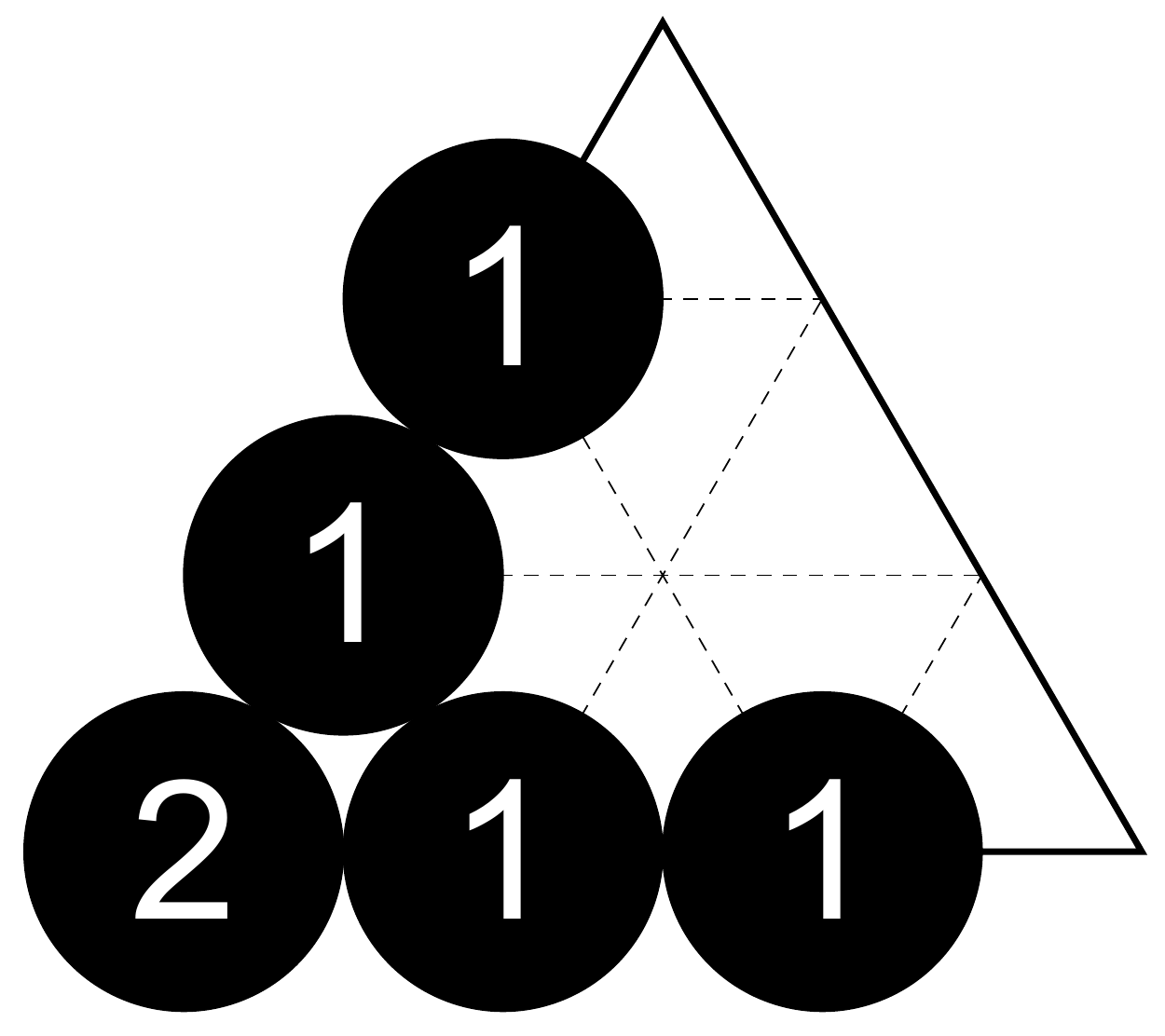}} \quad
\subfloat[17]{\includegraphics[width=1.5cm]{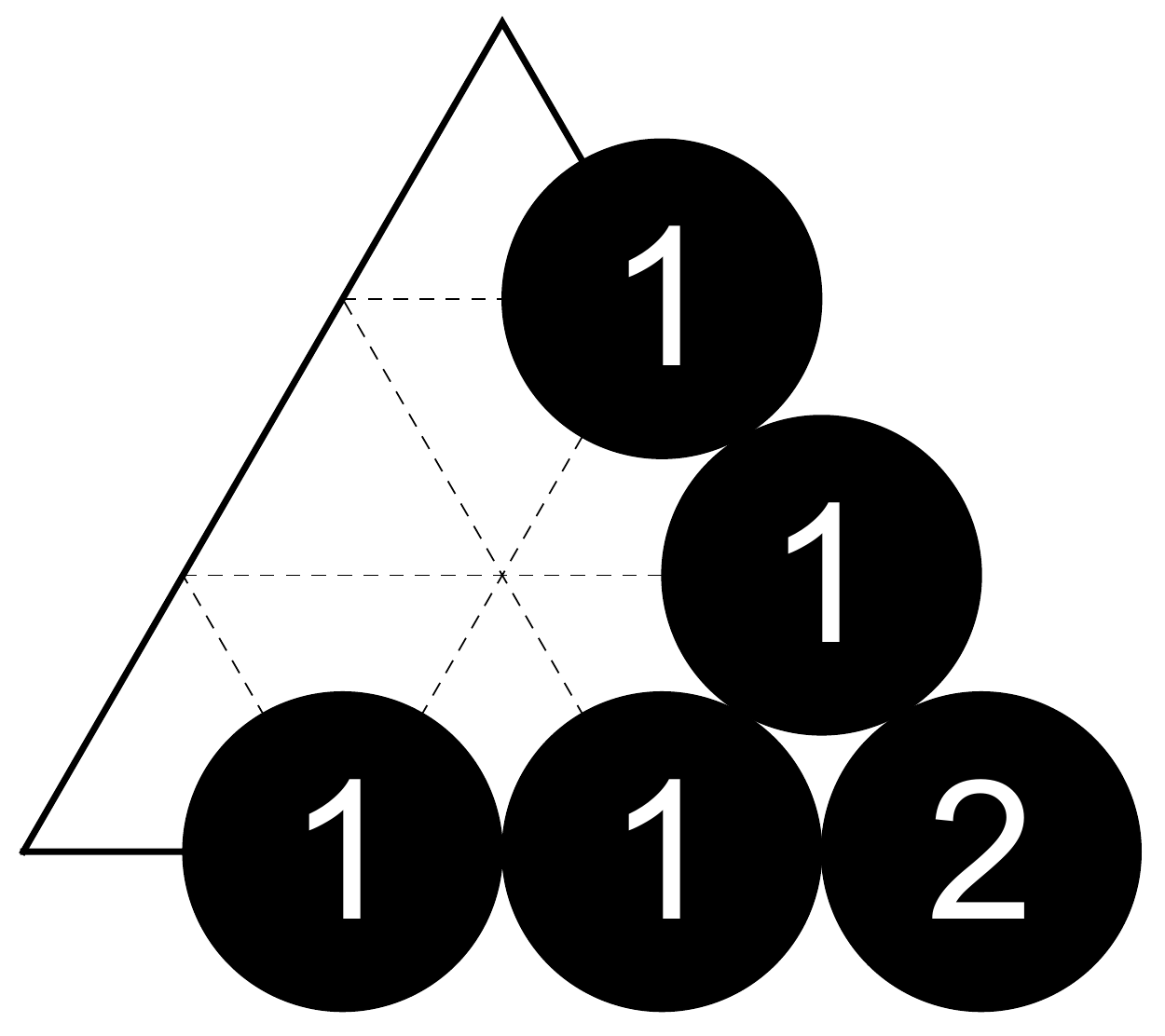}} \quad
\subfloat[18]{\includegraphics[width=1.4cm]{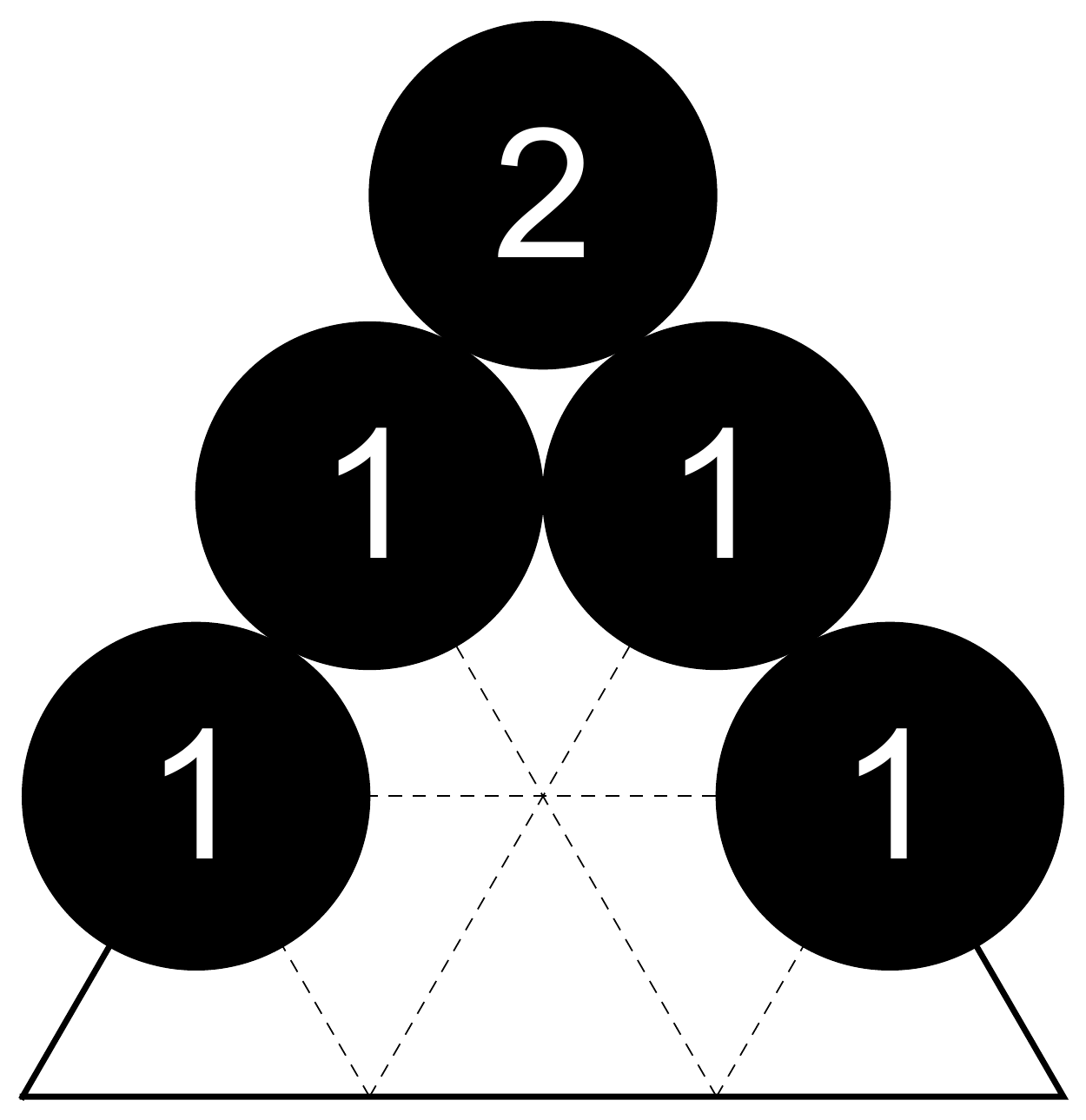}} \quad
\subfloat[19]{\includegraphics[width=1.6cm]{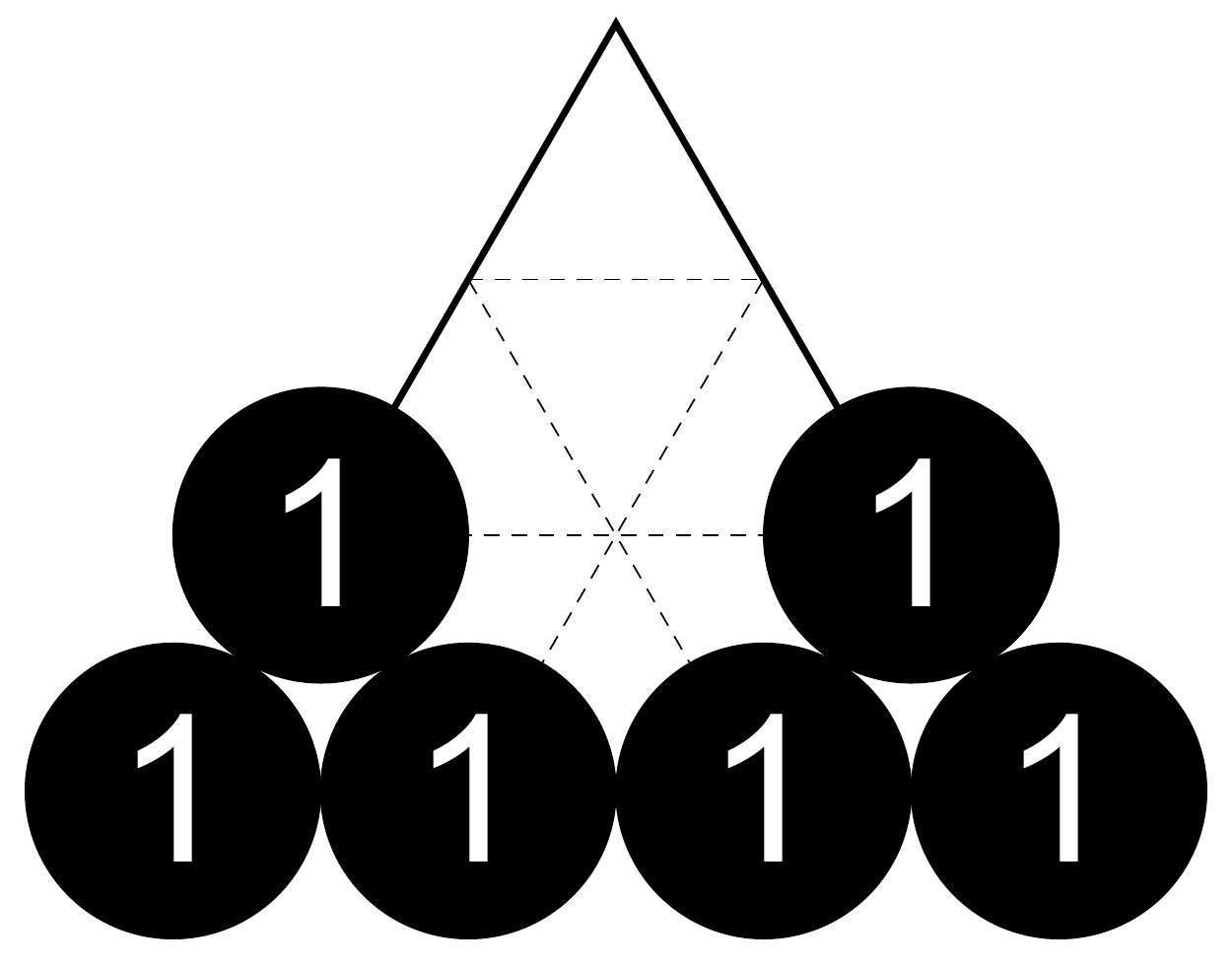}} \quad
\subfloat[20]{\includegraphics[width=1.5cm]{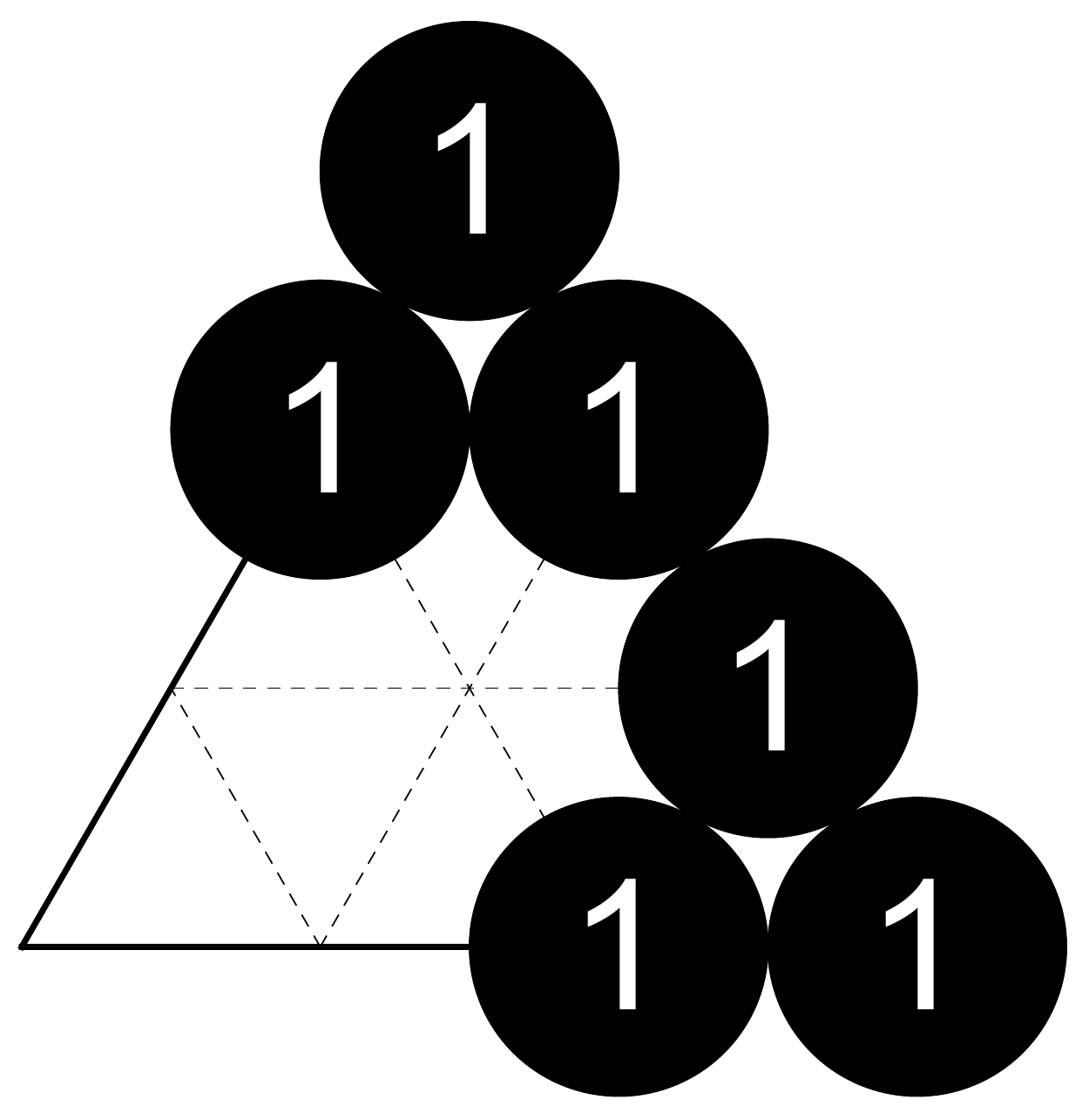}} \quad
\subfloat[21]{\includegraphics[width=1.5cm]{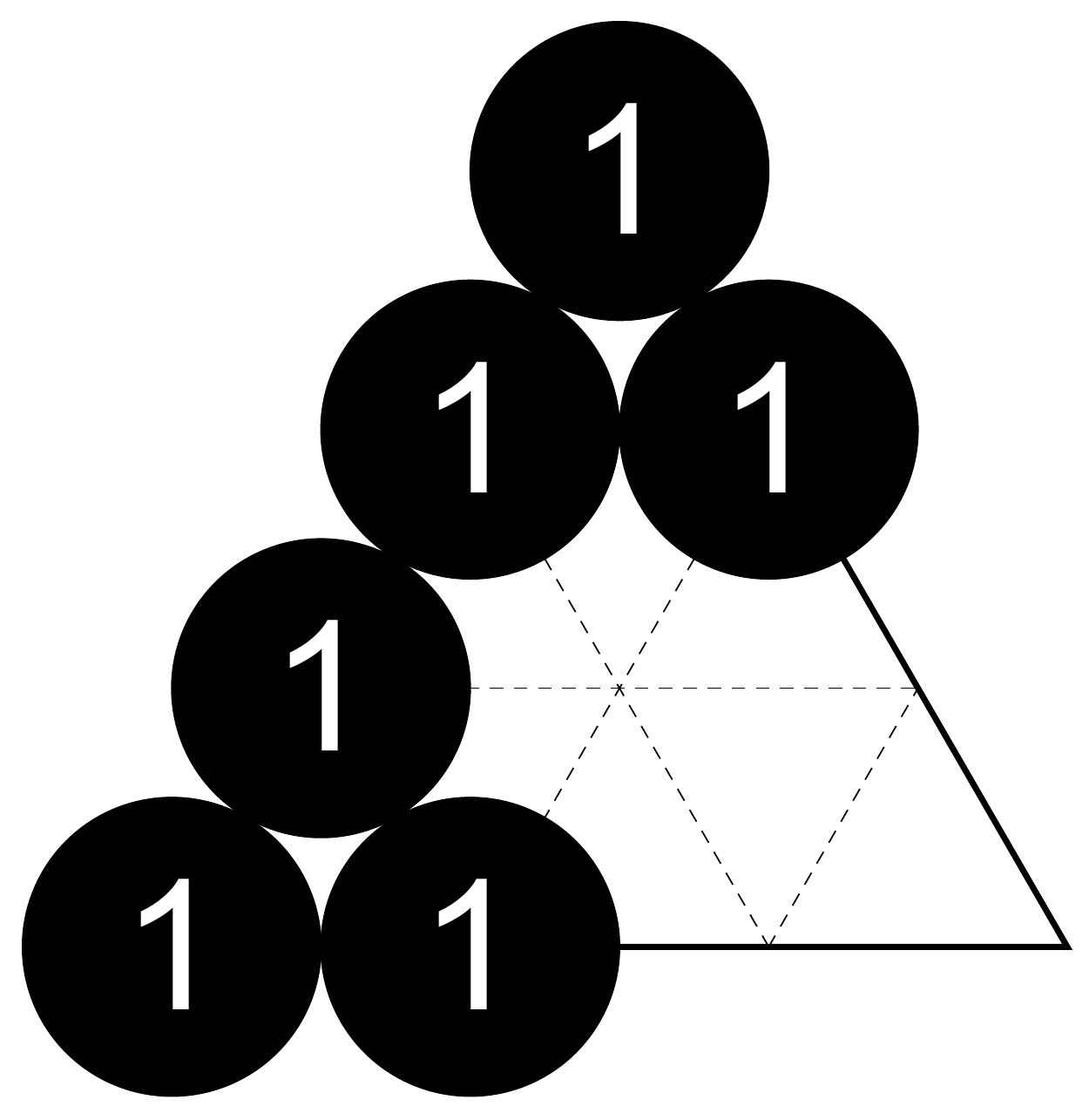}}
\\
\subfloat[22]{\includegraphics[width=1.5cm]{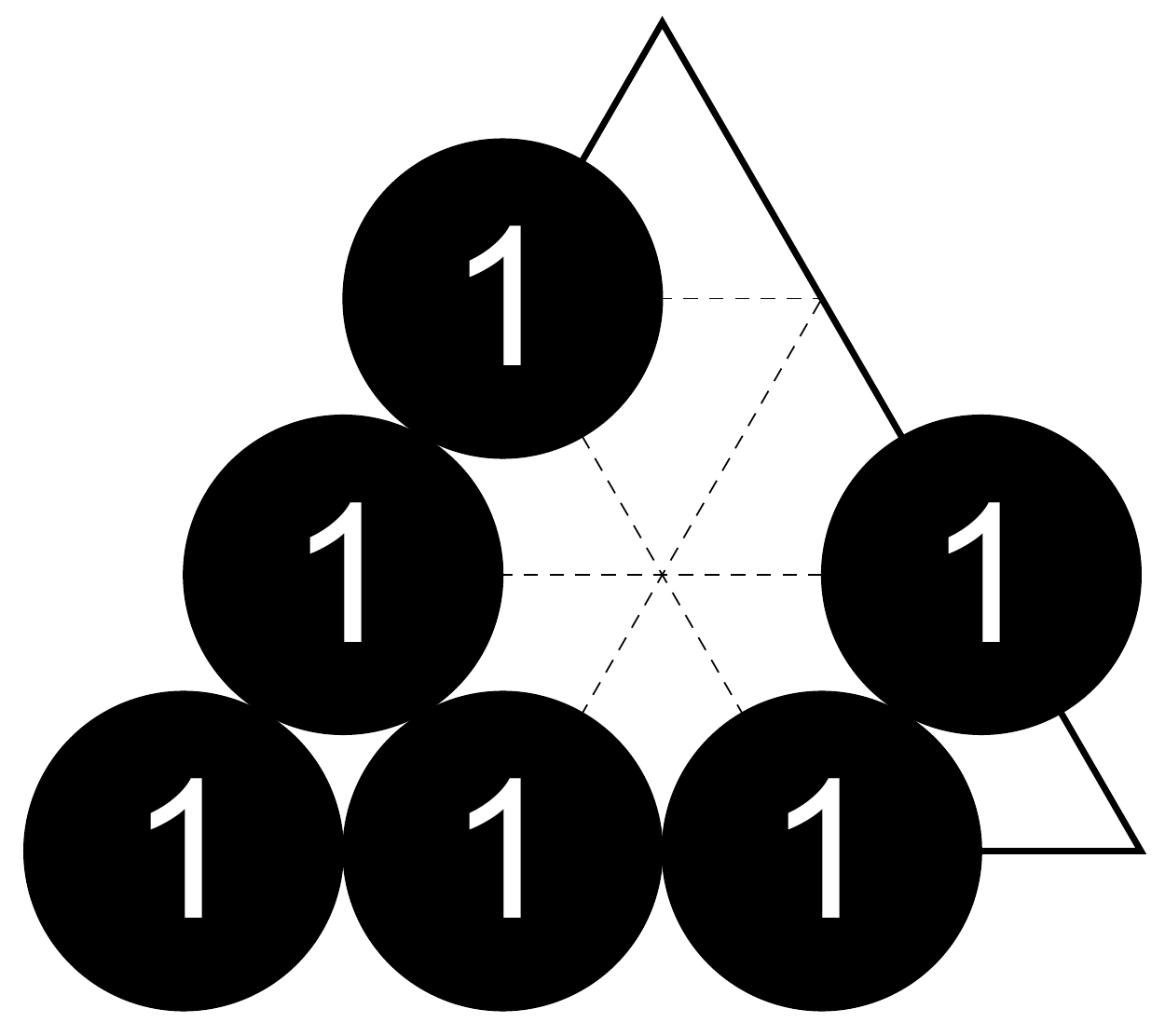}} \quad
\subfloat[23]{\includegraphics[width=1.5cm]{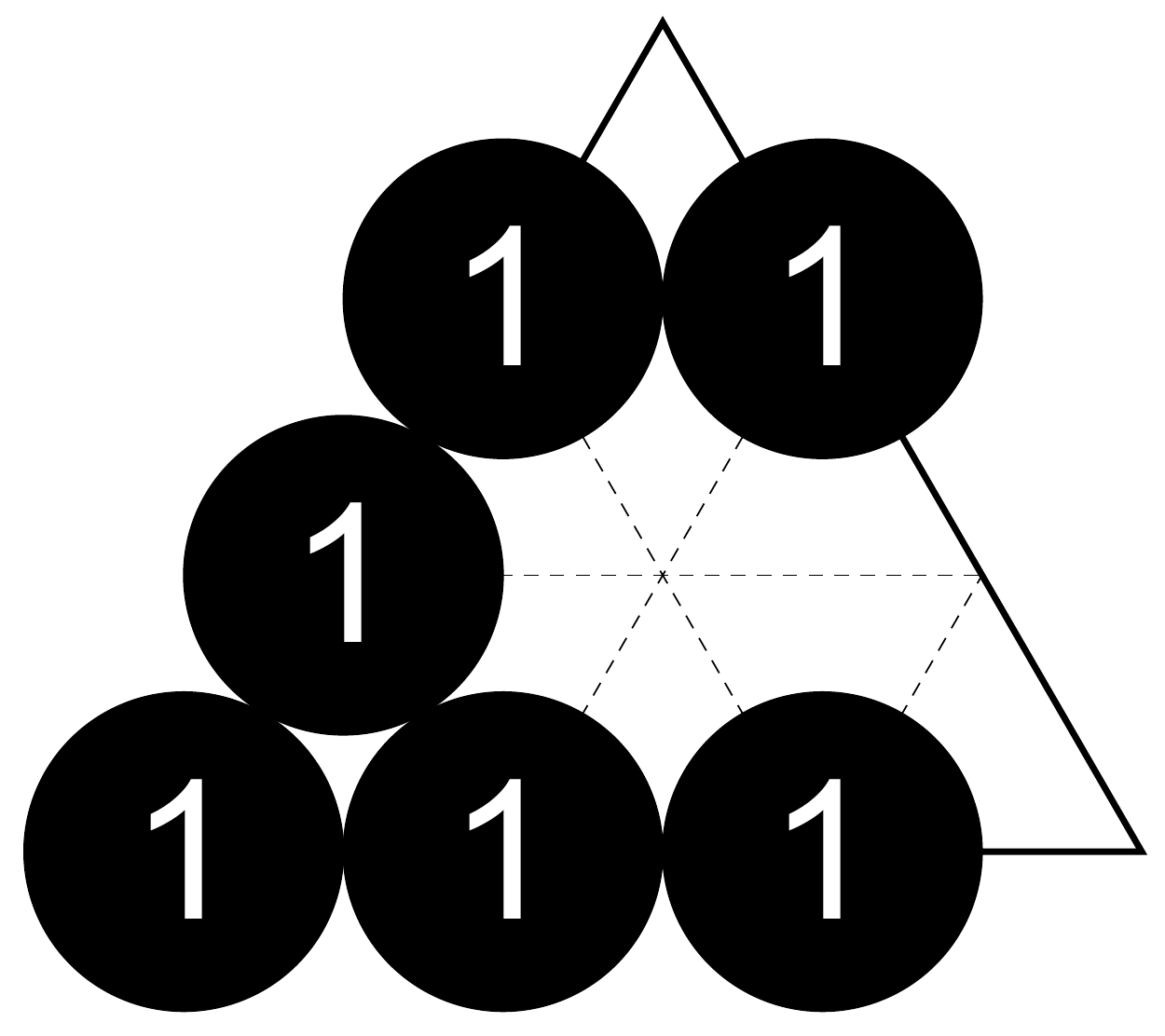}} \quad
\subfloat[24]{\includegraphics[width=1.35cm]{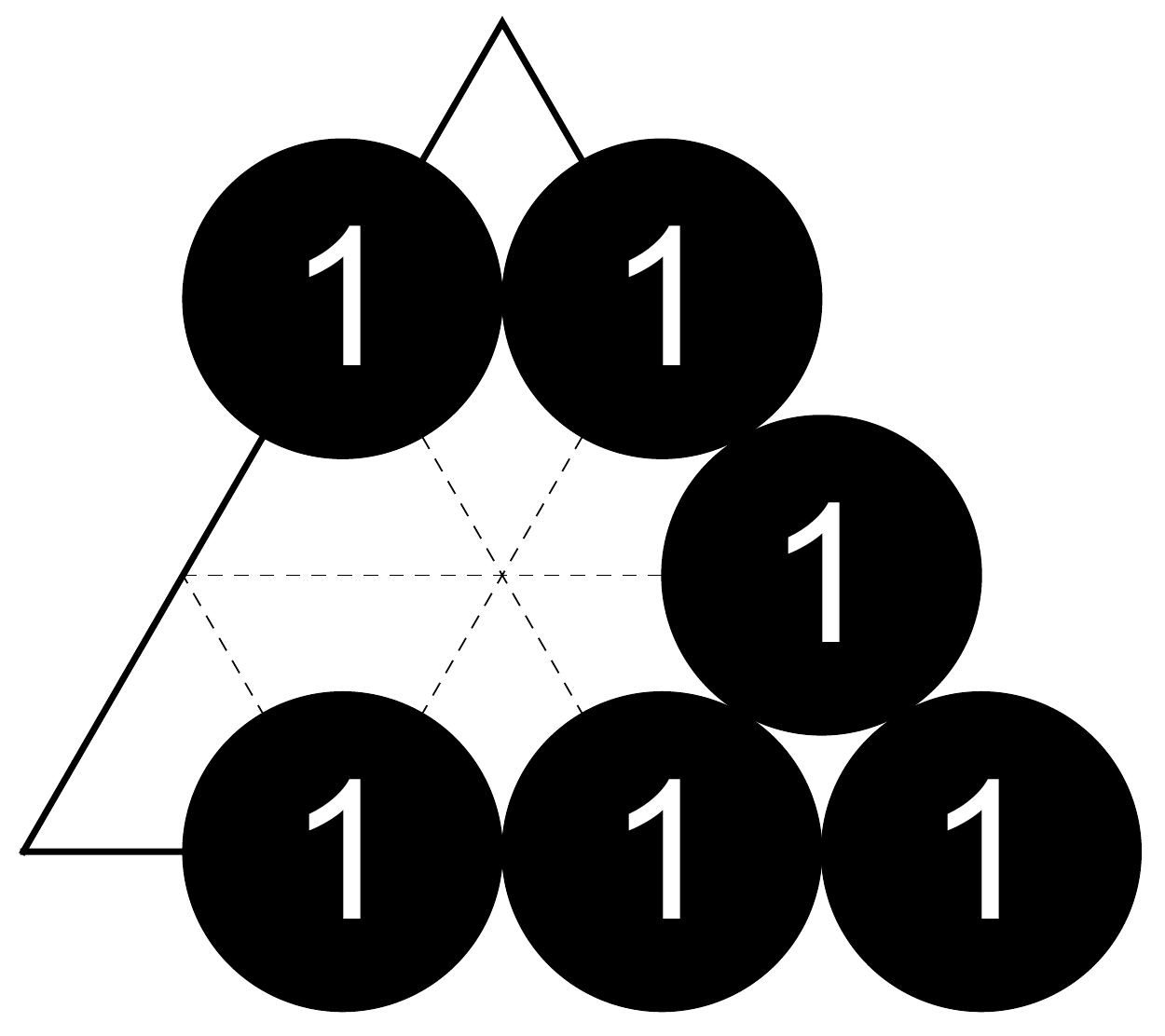}} \quad
\subfloat[25]{\includegraphics[width=1.5cm]{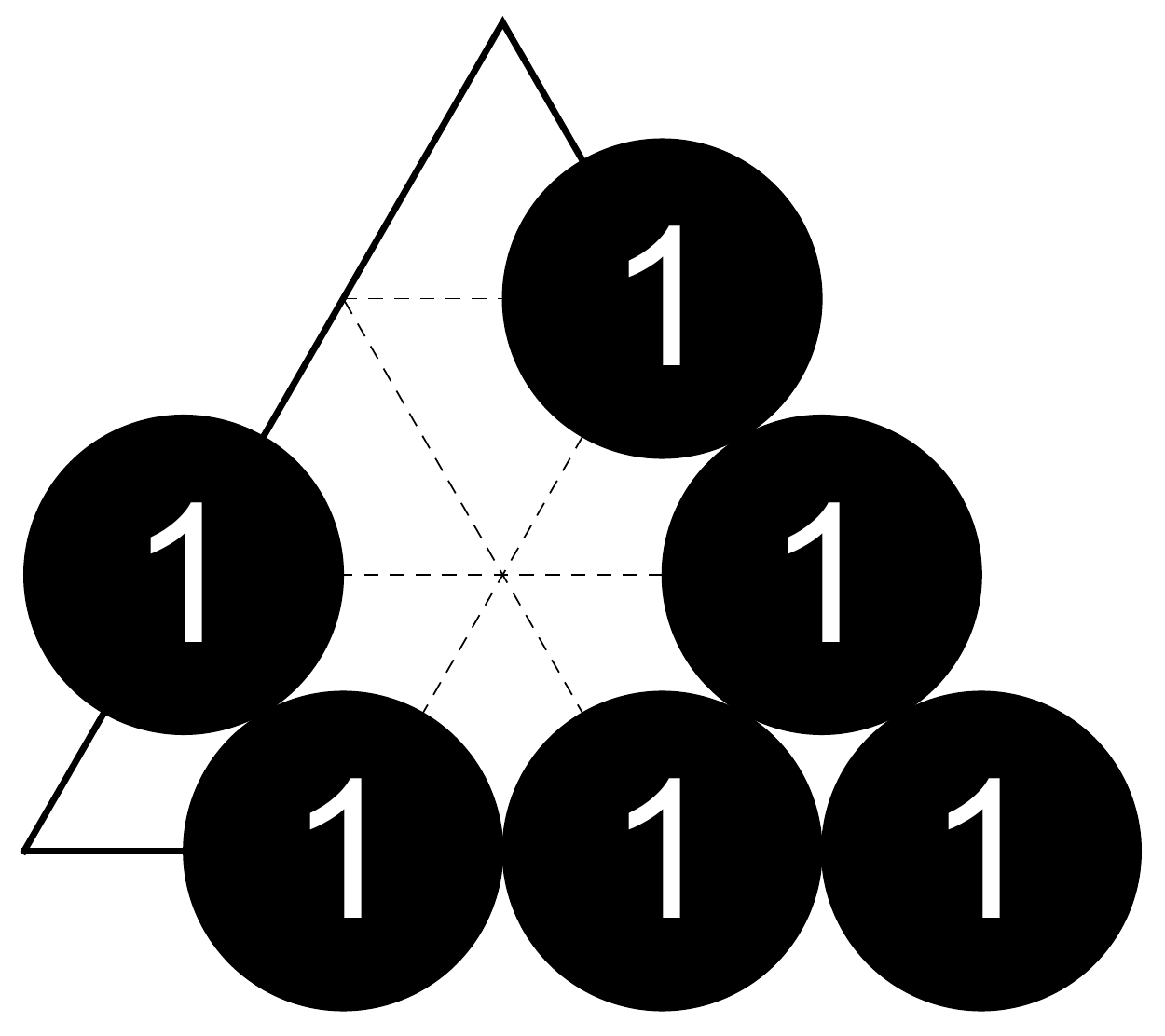}} \quad
\subfloat[26]{\includegraphics[width=1.3cm]{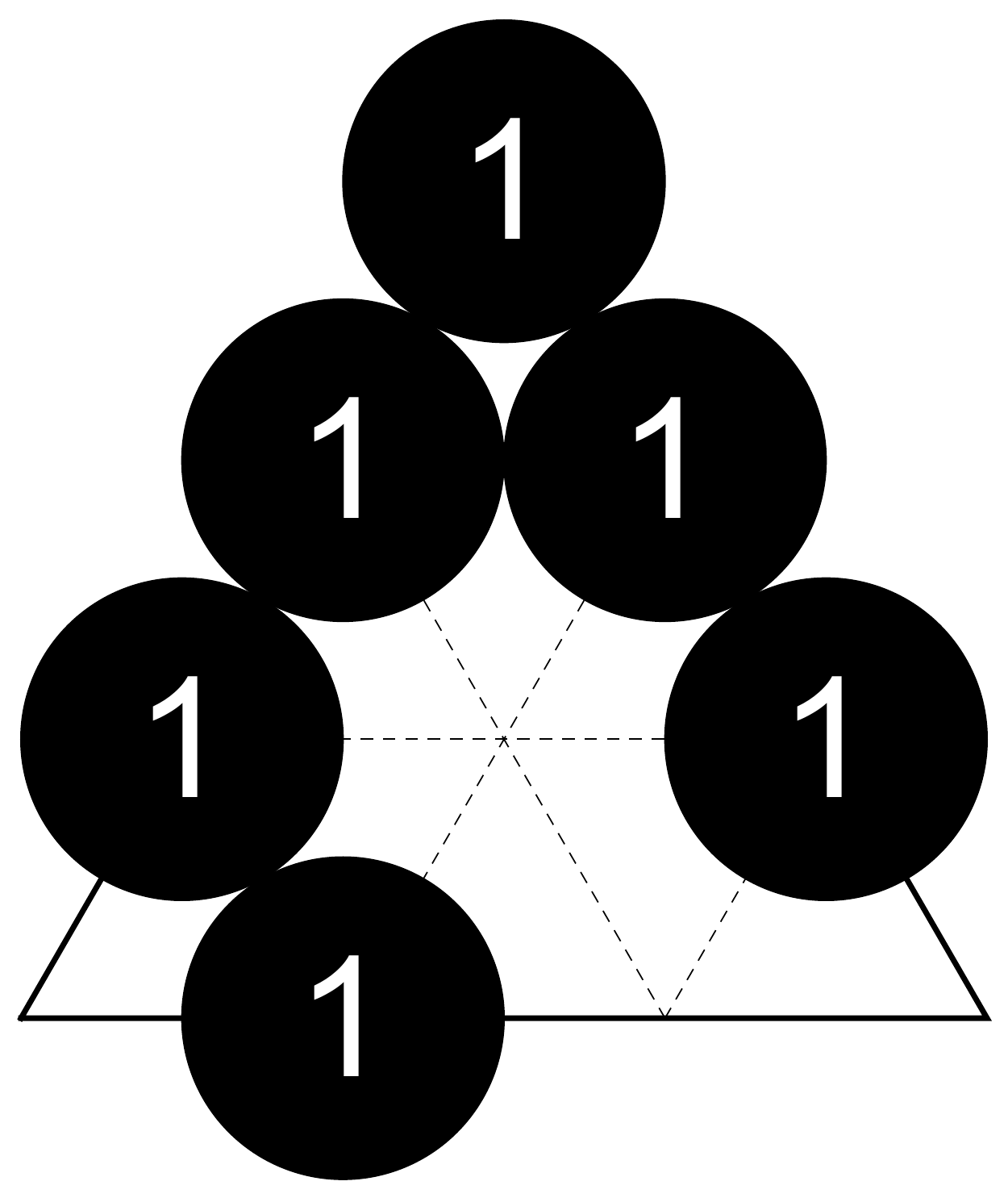}} \quad
\subfloat[27]{\includegraphics[width=1.35cm]{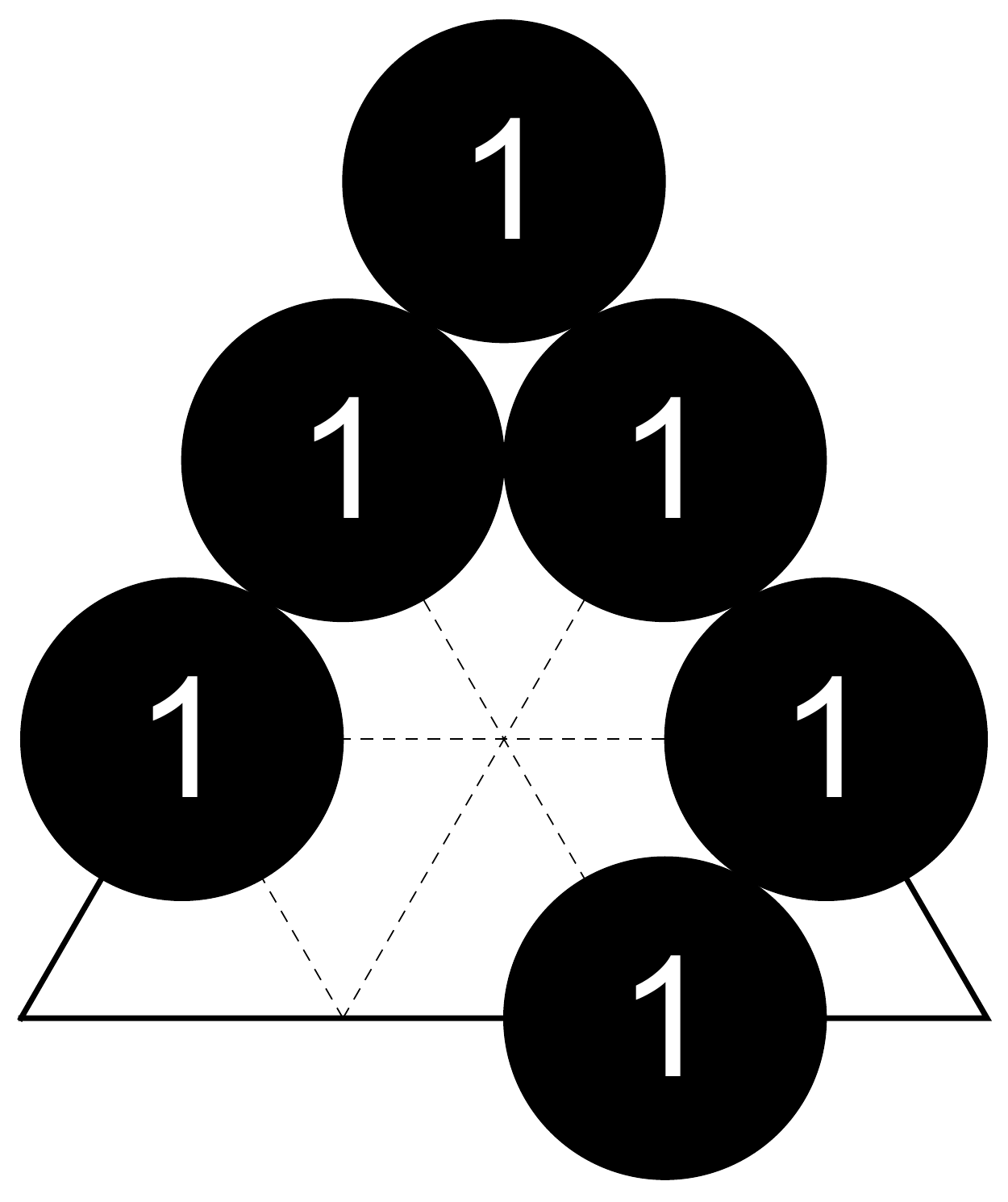}} \quad
\subfloat[28]{\includegraphics[width=1.3cm]{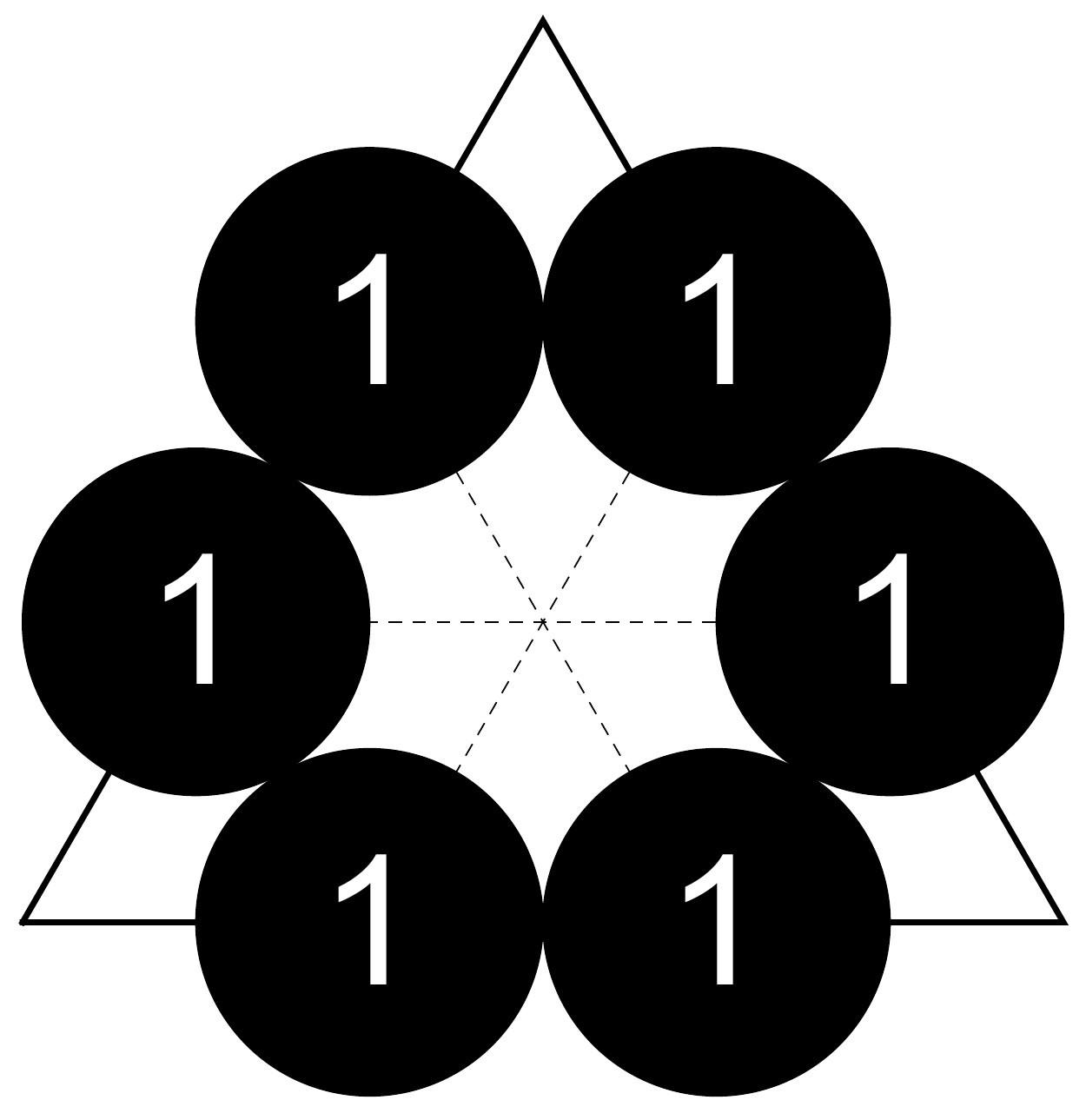}}
\caption{Sequences of knots for a set of simplex spline basis functions for $\spS_3^2(\Delta_{\WS})$. Each black disc shows the position of a knot and the number inside indicates its multiplicity.}
\label{fig:allknots}
\end{figure}

\subsection{A simplex spline basis}\label{sec:simplex-basis-1}
For a given triangle $\Delta=\langle\vp_1,\vp_2,\vp_3\rangle$, the $\WS$ split is shown in the middle plot of Figure~\ref{fig:splits}. From Theorem~\ref{thm:dimension} we know that the dimension of $\spS_3^2(\Delta_{\WS})$ is 28.
In order to construct a basis for this space, we first specify nine points along the boundary of the triangle (see Figure~\ref{fig:bvertexorder}): the three vertices $\vp_1,\vp_2,\vp_3$ and
\begin{equation}\label{eq:knot-points}
\begin{alignedat}{4}
\vp_{1,2}&:=\frac{2}{3}\vp_2+\frac{1}{3}\vp_3, &\quad
\vp_{1,3}&:=\frac{1}{3}\vp_2+\frac{2}{3}\vp_3, \\
\vp_{2,1}&:=\frac{2}{3}\vp_1+\frac{1}{3}\vp_3, &
\vp_{2,3}&:=\frac{1}{3}\vp_1+\frac{2}{3}\vp_3, \\
\vp_{3,1}&:=\frac{2}{3}\vp_1+\frac{1}{3}\vp_2,&
\vp_{3,2}&:=\frac{1}{3}\vp_1+\frac{2}{3}\vp_2.
\end{alignedat}
\end{equation}
Note that these points are part of the $\WS$ split.
We then consider the cubic simplex splines $M_1,\ldots, M_{28}$ as schematically illustrated in Figure~\ref{fig:allknots}, where each simplex spline has six (including multiplicity) knots chosen among the nine points above. For instance, $M_4$ is defined by the sequence
$\{\vxi_1,\ldots,\vxi_6\}=\{ \vp_1,\vp_1,\vp_1,\vp_{3,1},\vp_{3,2},\vp_{2,1}\}$.
Each of them can be computed using the B-recurrence relation.
We define the following set of 28 (scaled) simplex splines:
\begin{equation}\label{eq:basis}
\cB:=\{B_i:=w_iM_i, \ i=1,\ldots,28\},
\end{equation}
where, denoting by $|\Delta|$ the area of $\Delta$, the scaling factors are given by
$$
\vw:=\frac{\abs{\Delta}}{15}\left\{\frac{1}{6},\frac{1}{6},
  \frac{1}{6},\frac{1}{3},\frac{1}{3},\frac{1}{3},\frac{1}{3},
  \frac{1}{3},\frac{1}{3},\frac{1}{2},\frac{1}{2},\frac{1}{2},
  \frac{1}{2},\frac{1}{2},\frac{1}{2},\frac{2}{3},\frac{2}{3},
  \frac{2}{3},\frac{5}{6},\frac{5}{6},\frac{5}{6},\frac{2}{3},
  \frac{2}{3},\frac{2}{3},\frac{2}{3},\frac{2}{3},\frac{2}{3},1
\right\}.
$$
Note that the scaling factors sum up to $|\Delta|$.
There are seven different types of simplex splines in $\cB$. For each type, a representative $B_i$ is depicted in Figures~\ref{fig:B1}--\ref{fig:B28} in the appendix.
Explicit expressions of their polynomial pieces are given in Table~\ref{tab:polynomial-expressions} in the appendix; the remaining ones can be obtained by symmetry.

On any edge of $\Delta$, there are six basis functions nonzero. Their restrictions to that edge are nothing but the set of univariate $C^2$ cubic B-splines defined on a uniform open-knot sequence with two interior knots. For instance, for the edge $\vp_1\vp_2$, they correspond to the univariate cubic B-splines on the knot sequence specified by $\{\vp_1,\vp_1,\vp_1,\vp_1,\vp_{3,1},\vp_{3,2},\vp_2,\vp_2,\vp_2,\vp_2\}$.

\begin{theorem}\label{thm:basis}
The simplex splines $\{B_1,\ldots,B_{28}\}$ in \eqref{eq:basis} form a nonnegative partition of unity basis for the space $\spS_3^2(\Delta_{\WS})$.
\end{theorem}
\begin{proof}
Let $B$ be one of the functions $B_i$. We first prove that $B\in\spS_3^2(\Delta_{\WS})$. Since $B$ has six knots, it is a piecewise cubic polynomial. Moreover,
the knots of $B$ are a subset of the knots shown in Figure~\ref{fig:bvertexorder}. Thus, the knot lines of $B$ are a subset of the knot lines in the complete graph; see Figure~\ref{fig:splits}. Since each interior knot line contains exactly two knots, $B$ has $C^2$ smoothness according to the smoothness property of simplex splines. It follows that $B\in\spS_3^2(\Delta_{\WS})$.

We now consider linear independence. Using the recurrence relation and differentiation formula for simplex splines and the scaling factors $\vw$, we compute values and derivatives of $B$ corresponding to the following $28$ operators: $\rho_1,\ldots,\rho_{18}$ are related to the vertices,
\begin{equation} \label{eq:rho-1}
\begin{alignedat}{4}
\rho_1(f)&:=f(\vp_1), &\quad \rho_2(f)&:=f(\vp_2), &\quad \rho_3(f)&:=f(\vp_3),\\
\rho_4(f)&:=D_{\vp_1\vp_2}f(\vp_1), & \quad
\rho_5(f)&:=D_{\vp_1\vp_3}f(\vp_1), & \quad
\rho_6(f)&:=D_{\vp_2\vp_3}f(\vp_2), \\
\rho_7(f)&:=D_{\vp_2\vp_1}f(\vp_2), &
\rho_8(f)&:=D_{\vp_3\vp_1}f(\vp_3), &
\rho_9(f)&:=D_{\vp_3\vp_2}f(\vp_3), \\
\rho_{10}(f)&:=D_{\vp_1\vp_2}^2f(\vp_1), & \quad 
\rho_{11}(f)&:=D_{\vp_1\vp_3}^2f(\vp_1), & \quad 
\rho_{12}(f)&:=D_{\vp_2\vp_3}^2f(\vp_2), \\ 
\rho_{13}(f)&:=D_{\vp_2\vp_1}^2f(\vp_2), & 
\rho_{14}(f)&:=D_{\vp_3\vp_1}^2f(\vp_3), & 
\rho_{15}(f)&:=D_{\vp_3\vp_2}^2f(\vp_3), \\ 
\rho_{16}(f)&:=D_{\vp_1\vp_2}D_{\vp_1\vp_3}f(\vp_1), &
\rho_{17}(f)&:=D_{\vp_2\vp_3}D_{\vp_2\vp_1}f(\vp_2), &
\rho_{18}(f)&:=D_{\vp_3\vp_1}D_{\vp_3\vp_2}f(\vp_3);
\end{alignedat}
\end{equation}
$\rho_{19},\ldots,\rho_{27}$ are related to the edges,
\begin{equation} \label{eq:rho-2}
\begin{alignedat}{4}
\rho_{19}(f)&:=D_{\vq_3\vp_3}f(\vq_3), &
\rho_{20}(f)&:=D_{\vq_1\vp_1}f(\vq_1),&
\rho_{21}(f)&:=D_{\vq_2\vp_2}f(\vq_2), \\
\rho_{22}(f)&:=D_{\vp_{3,1}\vp_3}^2f(\vp_{3,1}), & 
\rho_{23}(f)&:=D_{\vp_{2,1}\vp_2}^2f(\vp_{2,1}), & 
\rho_{24}(f)&:=D_{\vp_{1,2}\vp_1}^2f(\vp_{1,2}), \\ 
\rho_{25}(f)&:=D_{\vp_{3,2}\vp_3}^2f(\vp_{3,2}), & \quad 
\rho_{26}(f)&:=D_{\vp_{2,3}\vp_2}^2f(\vp_{2,3}), & \quad 
\rho_{27}(f)&:=D_{\vp_{1,3}\vp_1}^2f(\vp_{1,3}), 
\end{alignedat}
\end{equation}
where
\begin{equation} \label{eq:dof-points-1}
\vq_1 :=\frac{1}{2}\vp_2+\frac{1}{2}\vp_3, \quad
\vq_2 :=\frac{1}{2}\vp_1+\frac{1}{2}\vp_3, \quad
\vq_3 :=\frac{1}{2}\vp_1+\frac{1}{2}\vp_2;
\end{equation}
and the final $\rho_{28}$ is related to the triangle,
\begin{equation}\label{eq:rho-3}
\rho_{28}(f) :=f(\vq), \quad
\vq:=\frac{1}{3}\vp_1+\frac{1}{3}\vp_2+\frac{1}{3}\vp_3.
\end{equation}
The computed values are shown in Table~\ref{tab:hermiteB} in the appendix.
Since the matrix $[\rho_j(B_i)]\in\RR^{28\times 28}$ is block upper triangular with nonsingular $2\times2$ blocks, linear independence of the set of functions $\{B_1,\ldots,B_{28}\}$ follows.
From Theorem~\ref{thm:dimension} we know that the dimension of $\spS_3^2(\Delta_{\WS})$ is 28, and thus the 28 linearly independent functions in \eqref{eq:basis} form a basis of this space. At the same time, we may conclude linear independence of the set of operators $\{\rho_1,\ldots,\rho_{28}\}$ defined on $\spS_3^2(\Delta_{\WS})$.

Simplex splines are nonnegative, so it only remains to prove that the functions in \eqref{eq:basis} sum up to one on $\Delta$. An inspection of
Table~\ref{tab:hermiteB} shows that 
$$\rho_j\biggl(\sum_{i=1}^{28}B_i\biggr)=\sum_{i=1}^{28}\rho_j(B_i)=\rho_j(1), \quad j=1,\ldots,28.$$
By linear independence of the operators $\rho_j$, $\sum_{i=1}^{28} B_i$ must be equal to the unity function which belongs to $\spS_3^2(\Delta_{\WS})$.
\end{proof}

From the proof of Theorem~\ref{thm:basis} it follows that we can formulate a Hermite interpolation problem to characterize any spline in $\spS_3^2(\Delta_{\WS})$.
\begin{corollary} \label{cor:hermite}
For given data $f_{k,\alpha,\beta}$, $g_k$, $g_{k,l}$, and $h_0$, there exists a unique spline $\spline\in\spS_3^2(\Delta_{\WS})$ such that
\begin{alignat*}{3}
D_x^\alpha D_y^\beta\spline(\vp_k) &= f_{k,\alpha,\beta}, \quad 0\leq \alpha+\beta\leq 2,\quad k=1,2,3, \\
D_{\vn_k}\spline(\vq_{k}) &= g_{k},\quad  D_{\vn_k}^2\spline(\vp_{k,l}) = g_{k,l},\quad  k,l=1,2,3, \quad k\neq l, \\
\spline(\vq) &= h_0,
\end{alignat*}
where $\vn_k$ is the normal direction of the edge opposite to vertex $\vp_k$, and the points $\vp_{k,l}$, $\vq_{k}$, and $\vq$ are defined in \eqref{eq:knot-points}, \eqref{eq:dof-points-1}, and \eqref{eq:rho-3}, respectively.
\end{corollary}
The Hermite degrees of freedom specified in Corollary~\ref{cor:hermite} are schematically visualized in Figure~\ref{fig:dof} using graphical symbols that are common in finite element literature; see, e.g., \cite{Ciarlet.78}.

\begin{figure}[t!]
\centering
\includegraphics[trim=60 30 60 20,clip,width=7.5cm]{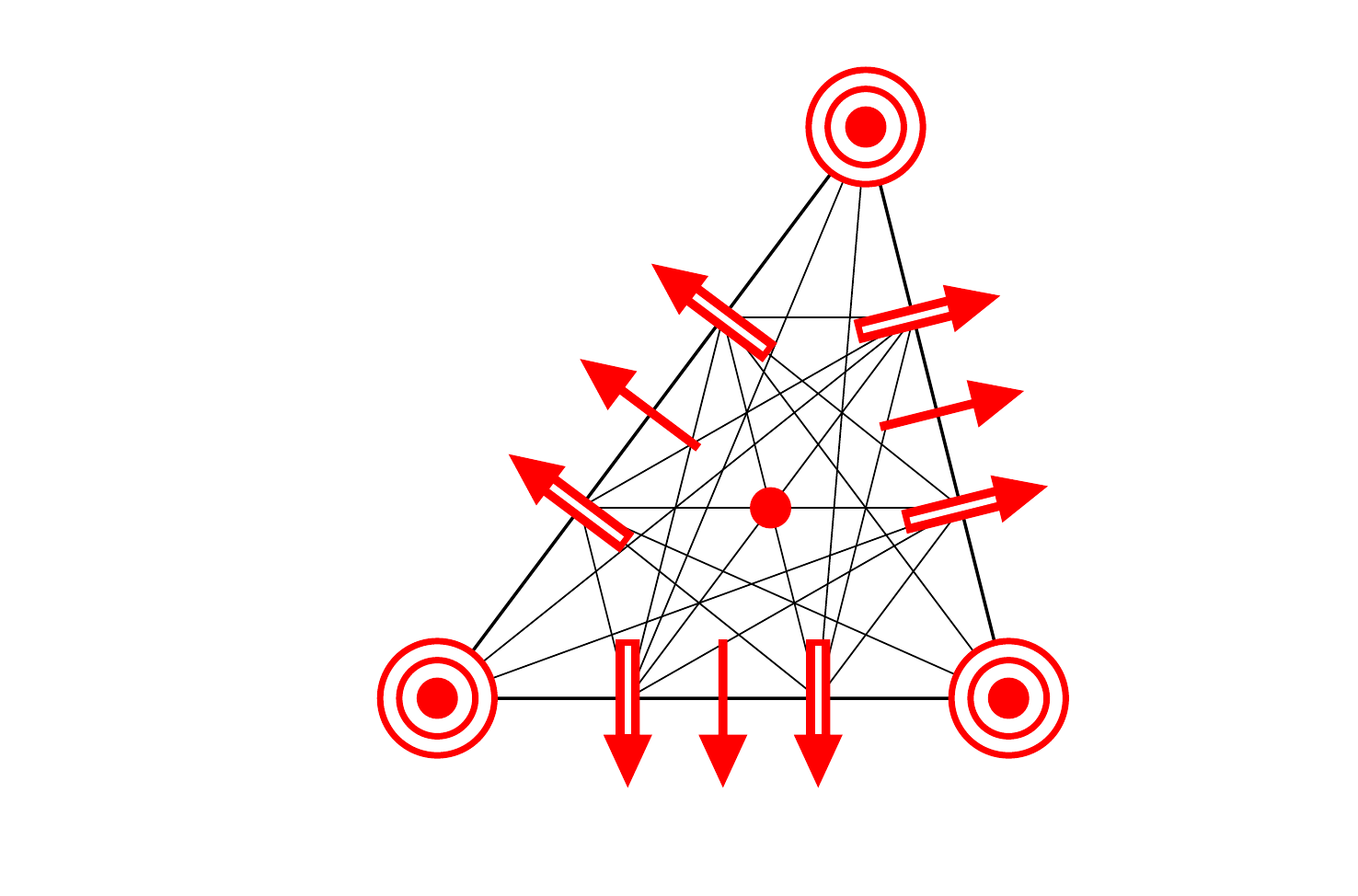}
\caption{Hermite degrees of freedom on the $\WS$ split.}
\label{fig:dof}
\end{figure}

\subsection{Domain points and condition number} \label{sec:domain-point}
We now associate a special point in $\Delta$ with each basis function $B_i$ in \eqref{eq:basis}, that plays an important role in geometric modeling.
We solve the system $\rho_j(\sum_{i=1}^{28}\greville_i(f)B_i)=\rho_j(f)$ for
the two functions $f_1(x,y):=x$ and $f_2(x,y):=y$. The points
\begin{equation} \label{eq:domain-points}
\vgreville_i:=(\greville_i(f_1),\greville_i(f_2)),\quad i=1,\ldots,28,
\end{equation}
are called the {domain points} of the basis \eqref{eq:basis}. Together with the partition of unity
the domain points provide an explicit representation of any affine function with respect to the basis \eqref{eq:basis}.
The barycentric coordinates with respect to the triangle $\Delta$ of the domain points \eqref{eq:domain-points} are given
by
\begin{equation}\label{eq:domain-points-bary}
\begin{alignedat}{5}
 \vgreville_1 &: \left(1,0,0\right), &\hspace*{0.5cm}
 \vgreville_2 &: \left(0,1,0\right), &\hspace*{0.5cm}
 \vgreville_3 &: \left(0,0,1\right), &\hspace*{0.5cm}
 \vgreville_4 &: \left(\frac{8}{9},\frac{1}{9},0\right), \\
 \vgreville_5 &: \left(\frac{8}{9},0,\frac{1}{9}\right), &
 \vgreville_6 &: \left(0,\frac{8}{9},\frac{1}{9}\right), &
 \vgreville_7 &: \left(\frac{1}{9},\frac{8}{9},0\right), &
 \vgreville_8 &: \left(\frac{1}{9},0,\frac{8}{9}\right), \\
 \vgreville_{9} &: \left(0,\frac{1}{9},\frac{8}{9}\right), &
 \vgreville_{10} &: \left(\frac{2}{3},\frac{1}{3},0\right), &
 \vgreville_{11} &: \left(\frac{2}{3},0,\frac{1}{3}\right), &
 \vgreville_{12} &: \left(0,\frac{2}{3},\frac{1}{3}\right), \\
 \vgreville_{13} &: \left(\frac{1}{3},\frac{2}{3},0\right), &
 \vgreville_{14} &: \left(\frac{1}{3},0,\frac{2}{3}\right), &
 \vgreville_{15} &: \left(0,\frac{1}{3},\frac{2}{3}\right), &
 \vgreville_{16} &: \left(\frac{7}{9},\frac{1}{9},\frac{1}{9}\right), \\
 \vgreville_{17} &: \left(\frac{1}{9},\frac{7}{9},\frac{1}{9}\right), &
 \vgreville_{18} &: \left(\frac{1}{9},\frac{1}{9},\frac{7}{9}\right), &
 \vgreville_{19} &: \left(\frac{7}{15},\frac{7}{15},\frac{1}{15}\right), &
 \vgreville_{20} &: \left(\frac{1}{15},\frac{7}{15},\frac{7}{15}\right), \\
 \vgreville_{21} &: \left(\frac{7}{15},\frac{1}{15},\frac{7}{15}\right), &
 \vgreville_{22} &: \left(\frac{5}{9},\frac{1}{3},\frac{1}{9}\right), &
 \vgreville_{23} &: \left(\frac{5}{9},\frac{1}{9},\frac{1}{3}\right), &
 \vgreville_{24} &: \left(\frac{1}{9},\frac{5}{9},\frac{1}{3}\right), \\
 \vgreville_{25} &: \left(\frac{1}{3},\frac{5}{9},\frac{1}{9}\right), &
 \vgreville_{26} &: \left(\frac{1}{3},\frac{1}{9},\frac{5}{9}\right), &
 \vgreville_{27} &: \left(\frac{1}{9},\frac{1}{3},\frac{5}{9}\right), &
 \vgreville_{28} &: \left(\frac{1}{3},\frac{1}{3},\frac{1}{3}\right).
\end{alignedat}
\end{equation}
These points are visualized in Figure~\ref{fig:domainpoints} (left). When representing a spline $\spline\in\spS_3^2(\Delta_{\WS})$ in the basis~\eqref{eq:basis},
\begin{equation}
\label{eq:rep-local}
\spline=\sum_{i=1}^{28} b_iB_i,
\end{equation}
it is common to organize the coefficients in terms of control points $(\vgreville_i,b_i)$, $i=1,\ldots,28$. There are several possibilities to connect these points into a control net. Such a control net forms a caricatural approximation for the graph of the function that is useful for geometric modeling. A viable option for connecting these points is shown in Figure~\ref{fig:domainpoints} (left); the configuration consists of a small number of regions but both triangles and quadrilaterals are involved.
As shown in Section~\ref{sec:smootheness-cond}, this choice allows for a geometric interpretation of $C^1$ smoothness conditions analogous to the classical Bernstein representation for polynomial triangular patches.
An example spline and its corresponding control net is illustrated in Figure~\ref{fig:controlnet}.

\begin{figure}[t!]
\centering
\includegraphics[width=7cm]{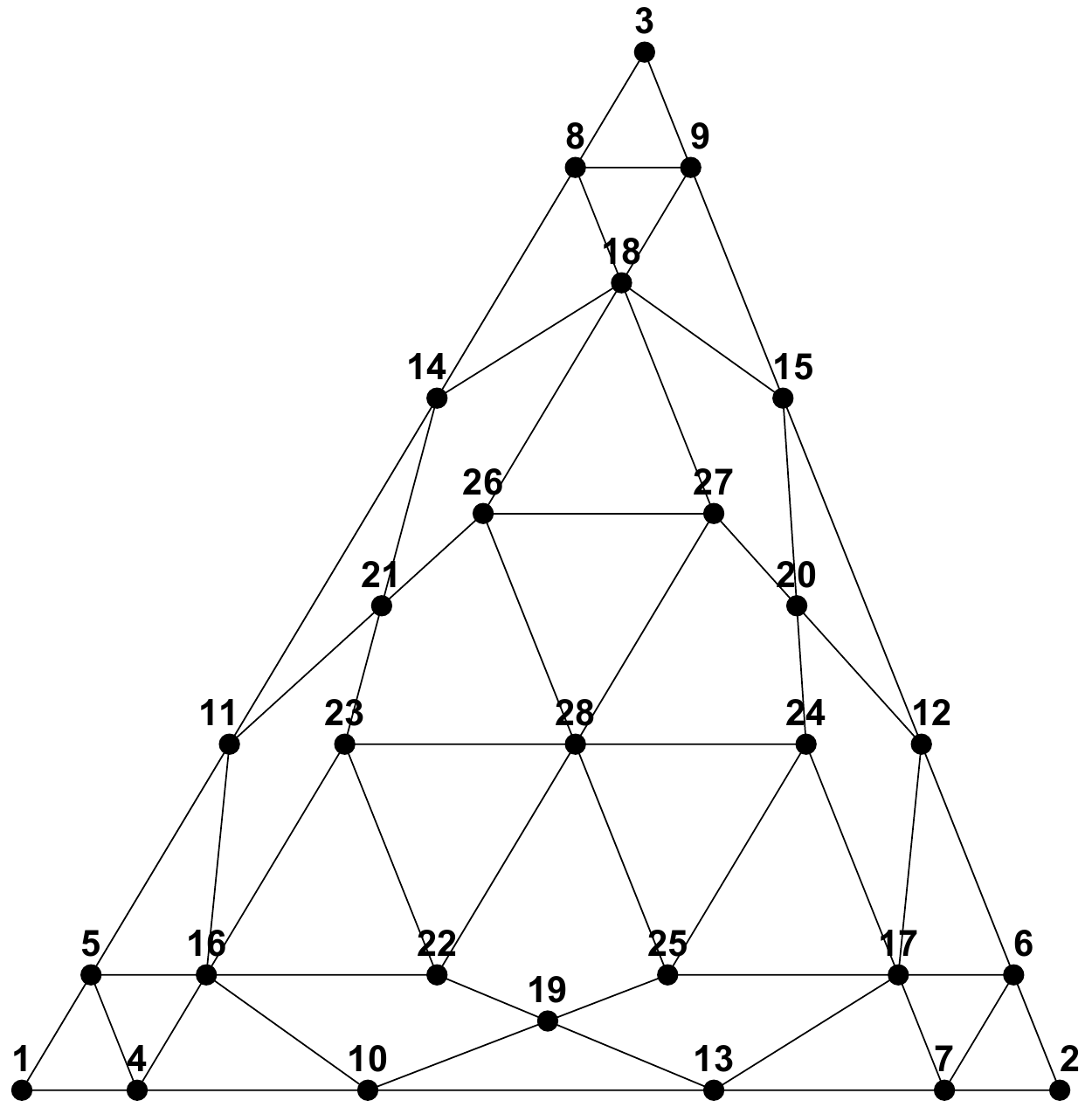}\qquad
\includegraphics[width=7cm]{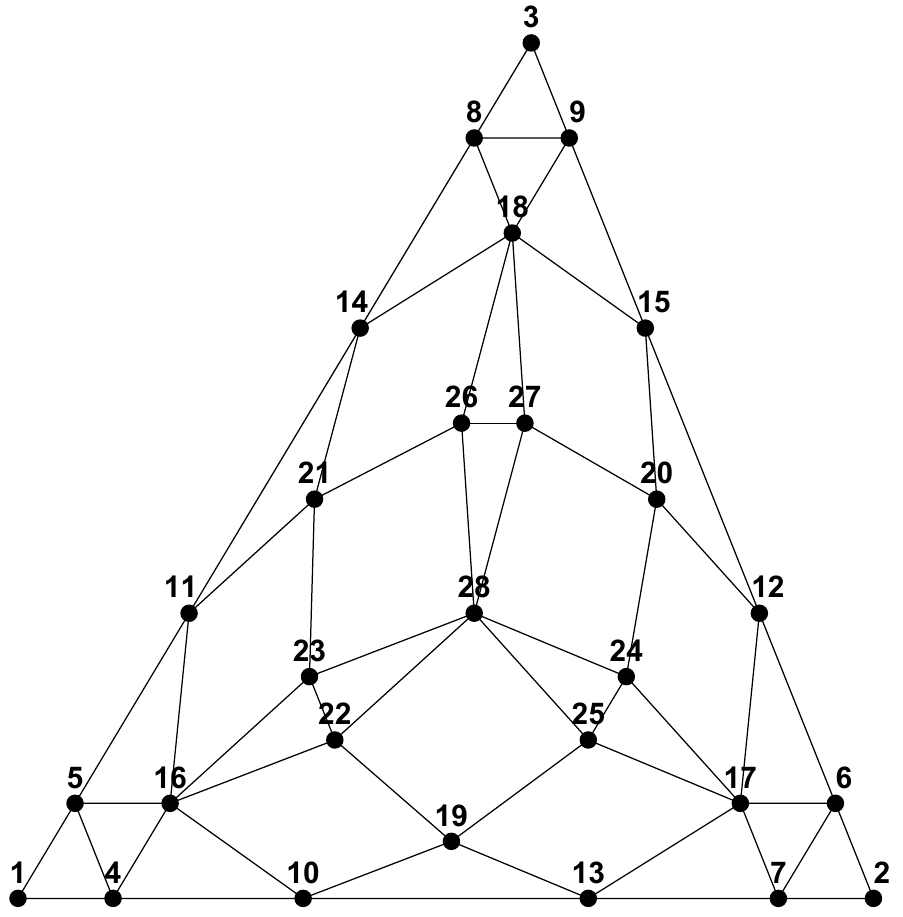}
\caption{A possible net configuration for $\spS_3^2(\Delta_{\WS})$. Left: the domain points \eqref{eq:domain-points-bary}. Right: the domain points \eqref{eq:domain-points-bary-2}.}
\label{fig:domainpoints}
\end{figure}
\begin{figure}[t!]
\centering
\includegraphics[trim=0 0 40 40,clip,width=7.1cm]{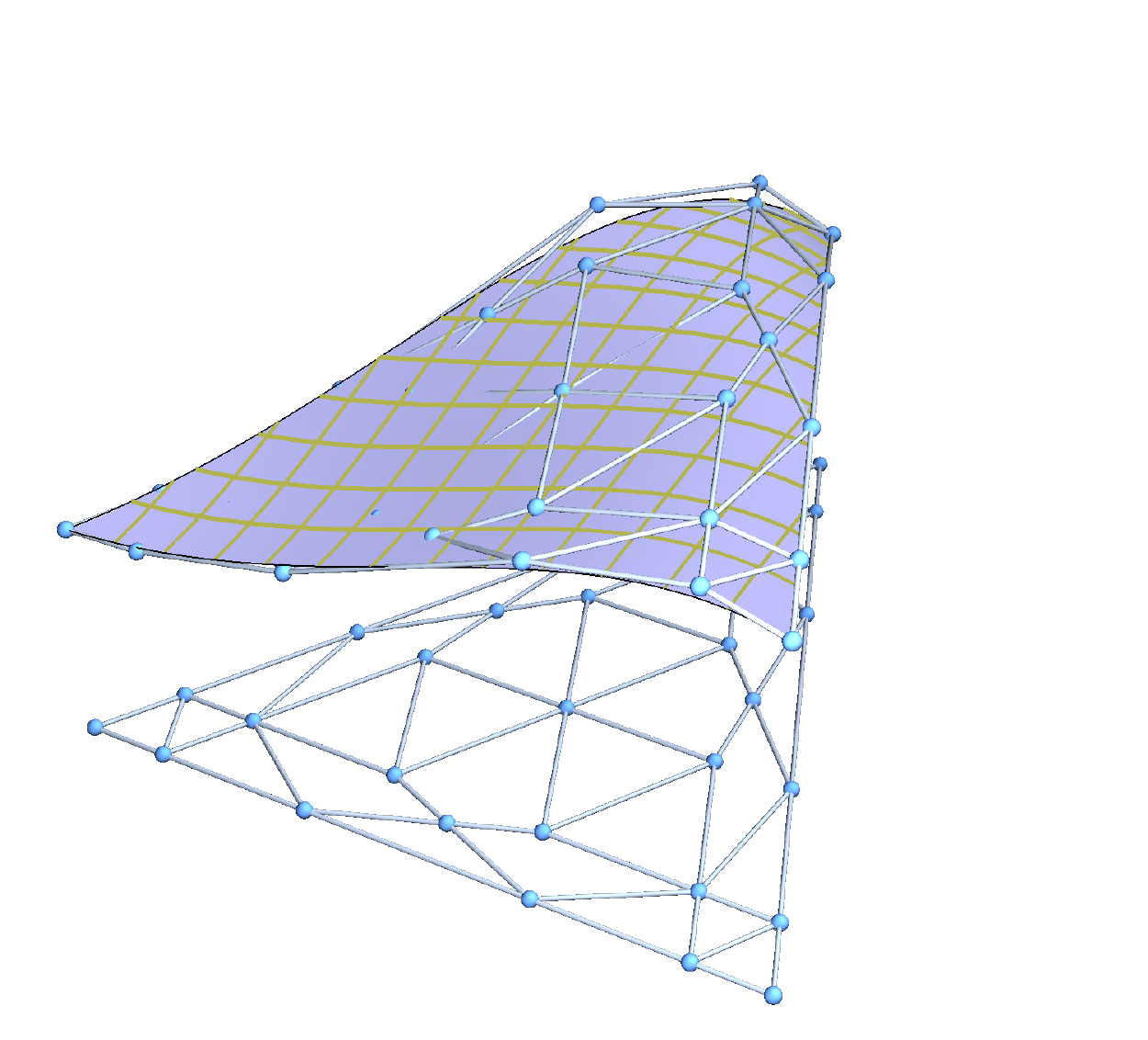}
\caption{A simplex spline surface and its control net for the domain points \eqref{eq:domain-points-bary}.}
\label{fig:controlnet}
\end{figure}

In many applications it is of interest to have a bound on the condition number of the basis we are dealing with. We consider the infinity norm, 
and we look for constants $K^{-}_\infty,K^{+}_\infty>0$ such that
for all $\vb:=(b_1,\ldots,b_{28})^T\in\RR^{28}$,
\begin{equation}
\label{eq:local-stability}
K^{-}_\infty\|\vb\|_\infty\leq\biggl\|\sum_{i=1}^{28}b_iB_i\biggr\|_\infty \leq K^{+}_\infty\|\vb\|_\infty.
\end{equation}
The condition number of the basis is then defined by
\begin{equation*}
\kappa_\infty(\cB):=\inf\{K^{+}_\infty/K^{-}_\infty : K^{-}_\infty \text{ and } K^{+}_\infty\text{ satisfy \eqref{eq:local-stability}}\}.
\end{equation*}
\begin{proposition}\label{prop:condition}
The condition number of the simplex spline basis \eqref{eq:basis} is bounded by 
$$
\kappa_\infty(\cB)<37.
$$
\end{proposition}
\begin{proof}
Since the simplex spline basis \eqref{eq:basis} forms a nonnegative partition of unity, it is clear that $K^{+}_\infty=1$ satisfies \eqref{eq:local-stability}.
Let $\vA$ be the matrix for any unisolvent Lagrange interpolation problem with respect to the basis \eqref{eq:basis}. Then,
\begin{equation*}
\|\vb\|_\infty\leq \|\vA^{-1}\|_\infty\biggl\|\sum_{i=1}^{28}b_iB_i\biggr\|_\infty.
\end{equation*}
Considering interpolation at the domain points \eqref{eq:domain-points-bary}, a direct computation gives
$\|\vA^{-1}\|_\infty<37$.
This implies that $K^{-}_\infty=1/37$ satisfies \eqref{eq:local-stability}.
\end{proof}
Note that the bound on the condition number in Proposition~\ref{prop:condition} is independent of the shape of the triangle $\Delta$. From the proof we also deduce that the condition number of $\cB$ in \eqref{eq:basis} can be computed as
\begin{equation*}
\kappa_\infty(\cB)=\inf\{1/K^{-}_\infty : K^{-}_\infty \text{ satisfies \eqref{eq:local-stability}}\}.
\end{equation*}
By means of this number we can easily obtain the following distance result.

\begin{proposition}\label{prop:distance-cp}
Let $\spline\in\spS_3^2(\Delta_{\WS})$ be given as in \eqref{eq:rep-local}. Then,
$$
  |b_i-\spline(\vgreville_i)|\leq 2 \kappa_\infty(\cB)\, h^2 \max_{\alpha+\beta=2}\|D_x^\alpha D_y^\beta\spline\|_{\infty}, \quad i=1,\ldots, 28,
$$
where $h$ is the length of the largest edge of $\Delta$.
\end{proposition}
\begin{proof}
Let $\Pi_i\in\spP_1$ be the linear Taylor approximation to $\spline$ at the domain point $\vgreville_i$. Note that $\Pi_i(\vgreville_i)=\spline(\vgreville_i)$.
Then, using the definition of domain points, we have
$$
\spline-\Pi_i=\sum_{j=1}^{28}(b_j-\Pi_i(\vgreville_j))B_j,
$$
so that from \eqref{eq:local-stability} we obtain
$$
|b_i-\spline(\vgreville_i)|=|b_i-\Pi_i(\vgreville_i)|\leq \max_{j=1,\ldots,28}|b_j-\Pi_i(\vgreville_j)|\leq \kappa_\infty(\cB) \| \spline-\Pi_i\|_{\infty}.
$$
Furthermore, since $\spline\in C^2(\Delta)$, Taylor approximation error analysis tells us that
$$
\| \spline-\Pi_i\|_{\infty} \leq 2h^2 \max_{\alpha+\beta=2}\|D_x^\alpha D_y^\beta \spline\|_{\infty},
$$
which completes the proof.
\end{proof}
Proposition~\ref{prop:distance-cp} implies that the control points of the spline $\spline$ in \eqref{eq:rep-local} converge like $O(h^2)$ to the graph of $\spline$.

\subsection{A Marsden-like identity}
In Section~\ref{sec:domain-point}, we have provided the representation of any affine function with respect to the basis \eqref{eq:basis}. We now extend this result by providing a Marsden-like identity which allows us to represent any cubic polynomial. 
In the univariate B-spline case, the Marsden identity is given by
$$
(z -x)^d =\sum_j \prod_{k=j+1}^{j+d} (z-\xi_k)B_{j,d}(x),
$$
where $B_{j,d}$ is a normalized B-spline of degree $d$ defined by the knots $\xi_j,\xi_{j+1},\ldots,\xi_{j+d+1}$; see, e.g., \cite[Theorem~2]{LycheMS08}. Dividing both sides by $z^d$ and setting
$y := z^{-1}$ we obtain a form more amenable to multivariate generalization
\begin{equation}\label{eq:Marsden-uni}
(1 - xy)^d =\sum_j \psi_{j,d}(y)B_{j,d}(x), \quad \psi_{j,d}(y):=\prod_{k=j+1}^{j+d} (1-y\xi_k).
\end{equation}
The functions $\psi_{j,d}$ are polynomials of degree $d$ and are called dual polynomials.
The following result is obtained by a direct computation.

\begin{theorem}\label{thm:Marsden}
We have
$$
(1+\vy^T\vx)^3=\sum_{i=1}^{28}\psi_i(\vy)B_i(\vx),\quad \vy\in\RR^2,\quad \vx\in\Delta,
$$
where the dual polynomials $\psi_i$, $i=1,\ldots,28$ are defined as follows.
Recalling the points in \eqref{eq:knot-points} and~\eqref{eq:rho-3},
\begin{alignat*}{5}
\psi_1(\vy) &:= (1+\vy^T\vp_1)^3, \quad
\psi_2(\vy) := (1+\vy^T\vp_2)^3, \quad
\psi_3(\vy) := (1+\vy^T\vp_3)^3, \\
\psi_4(\vy) &:= (1+\vy^T\vp_1)^2(1+\vy^T\vp_{3,1}), \quad
\psi_5(\vy) := (1+\vy^T\vp_1)^2(1+\vy^T\vp_{2,1}), \\
\psi_6(\vy) &:= (1+\vy^T\vp_2)^2(1+\vy^T\vp_{1,2}), \quad
\psi_7(\vy) := (1+\vy^T\vp_2)^2(1+\vy^T\vp_{3,2}), \\
\psi_8(\vy) &:= (1+\vy^T\vp_3)^2(1+\vy^T\vp_{2,3}), \quad
\psi_9(\vy) := (1+\vy^T\vp_3)^2(1+\vy^T\vp_{1,3}),
\end{alignat*}
and
\begin{alignat*}{5}
\psi_{10}(\vy) &:= (1+\vy^T\vp_1)(1+\vy^T\vp_{3,1})(1+\vy^T\vp_{3,2}), &\
\psi_{11}(\vy) &:= (1+\vy^T\vp_1)(1+\vy^T\vp_{2,3})(1+\vy^T\vp_{2,1}), \\
\psi_{12}(\vy) &:= (1+\vy^T\vp_2)(1+\vy^T\vp_{1,2})(1+\vy^T\vp_{1,3}), &
\psi_{13}(\vy) &:= (1+\vy^T\vp_2)(1+\vy^T\vp_{3,1})(1+\vy^T\vp_{3,2}), \\
\psi_{14}(\vy) &:= (1+\vy^T\vp_3)(1+\vy^T\vp_{2,3})(1+\vy^T\vp_{2,1}), &
\psi_{15}(\vy) &:= (1+\vy^T\vp_3)(1+\vy^T\vp_{1,2})(1+\vy^T\vp_{1,3}), \\
\psi_{16}(\vy) &:= (1+\vy^T\vp_1)(1+\vy^T\vp_{3,1})(1+\vy^T\vp_{2,1}), &
\psi_{17}(\vy) &:= (1+\vy^T\vp_2)(1+\vy^T\vp_{3,2})(1+\vy^T\vp_{1,2}), \\
\psi_{18}(\vy) &:= (1+\vy^T\vp_3)(1+\vy^T\vp_{1,3})(1+\vy^T\vp_{2,3}), &
\psi_{19}(\vy) &:= (1+\vy^T\vp_{3,1})(1+\vy^T\vp_{3,2})(1+\vy^T\vm_1), \\
\psi_{20}(\vy) &:= (1+\vy^T\vp_{1,2})(1+\vy^T\vp_{1,3})(1+\vy^T\vm_2), &\
\psi_{21}(\vy) &:= (1+\vy^T\vp_{2,3})(1+\vy^T\vp_{2,1})(1+\vy^T\vm_3), \\
\psi_{22}(\vy) &:= (1+\vy^T\vp_{3,1})(1+\vy^T\vp_{3,2})(1+\vy^T\vp_{2,1}), &
\psi_{23}(\vy) &:= (1+\vy^T\vp_{3,1})(1+\vy^T\vp_{2,3})(1+\vy^T\vp_{2,1}), \\
\psi_{24}(\vy) &:= (1+\vy^T\vp_{3,2})(1+\vy^T\vp_{1,2})(1+\vy^T\vp_{1,3}), &
\psi_{25}(\vy) &:= (1+\vy^T\vp_{3,1})(1+\vy^T\vp_{3,2})(1+\vy^T\vp_{1,2}), \\
\psi_{26}(\vy) &:= (1+\vy^T\vp_{1,3})(1+\vy^T\vp_{2,3})(1+\vy^T\vp_{2,1}), &
\psi_{27}(\vy) &:= (1+\vy^T\vp_{1,2})(1+\vy^T\vp_{1,3})(1+\vy^T\vp_{2,3}),
\end{alignat*}
and
\begin{align*}
&\psi_{28}(\vy) := (1+\vy^T\vq)\bigl[2(1+\vy^T\vq)^2\\
&\quad -\frac{1}{3}\bigl((1+\vy^T\vp_{1,3})(1+\vy^T\vp_{2,3})+
(1+\vy^T\vp_{3,2})(1+\vy^T\vp_{1,2})+(1+\vy^T\vp_{3,1})(1+\vy^T\vp_{2,1})\bigr)\bigr],
\end{align*}
where
$$
\vm_1 := \frac{2}{5} \vp_1 + \frac{2}{5} \vp_2 + \frac{1}{5} \vp_3, \quad
\vm_2 := \frac{1}{5} \vp_1 + \frac{2}{5} \vp_2 + \frac{2}{5} \vp_3, \quad
\vm_3 := \frac{2}{5} \vp_1 + \frac{1}{5} \vp_2 + \frac{2}{5} \vp_3.
$$
\end{theorem}
It is remarkable that the dual polynomials $\psi_i$, $i=1,\ldots,27$, can be written as products of three linear polynomials and mimic the classical univariate Marsden identity \eqref{eq:Marsden-uni}. It is worth noting that these functions are the polar forms (or blossoms) of $(1+\vy^T\vx)^3$ evaluated at appropriate points \cite{Ramshaw89}. Similarly, polar forms were used in \cite{Neamtu07} for the representation of polynomials in terms of simplex splines whose knots are in generic position. This is not the case for the knots in Figures~\ref{fig:allknots} and \ref{fig:allknots2}.

\begin{figure}[t!]
\centering
\subfloat[1]{\includegraphics[width=1.5cm]{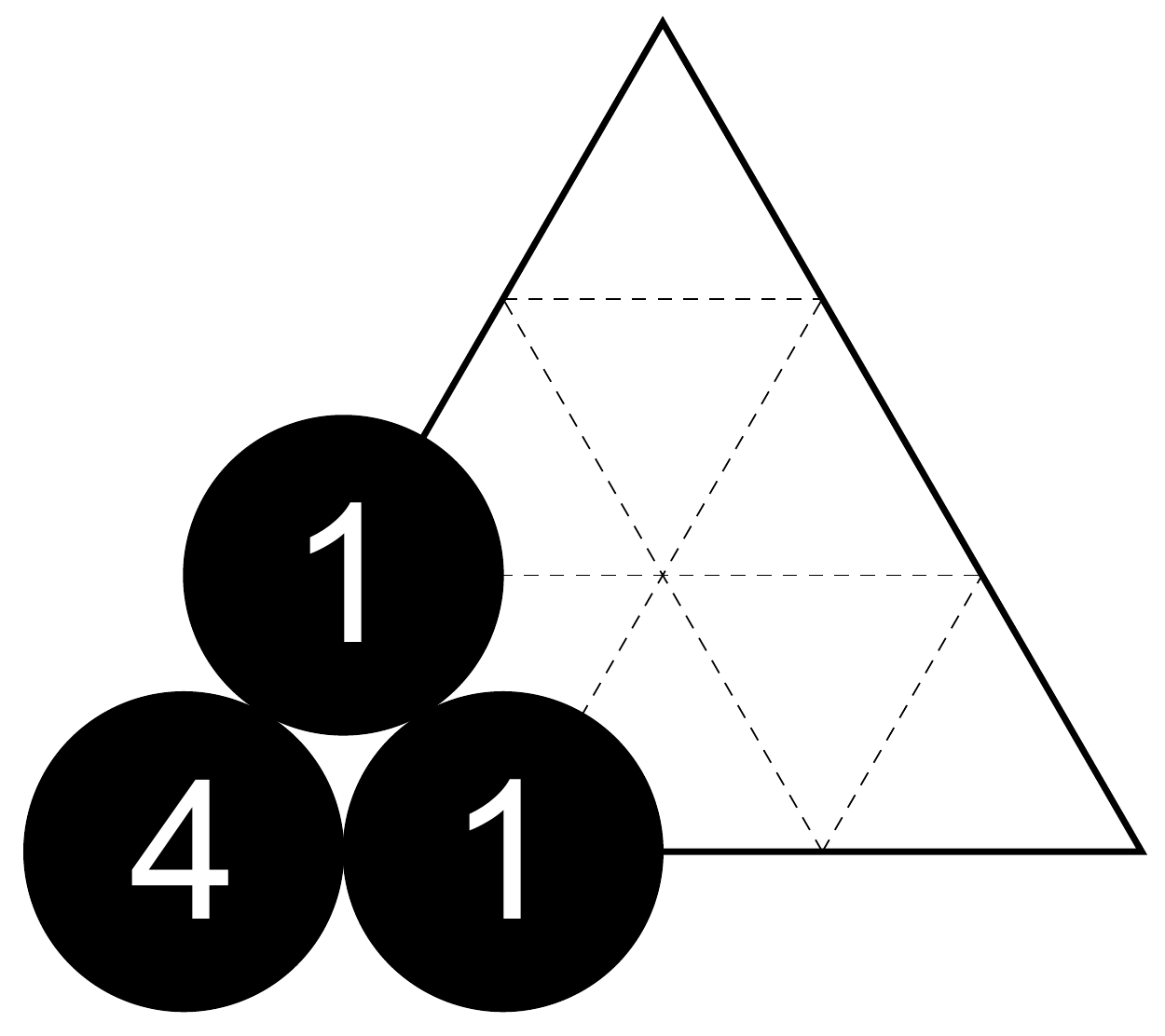}} \quad
\subfloat[2]{\includegraphics[width=1.5cm]{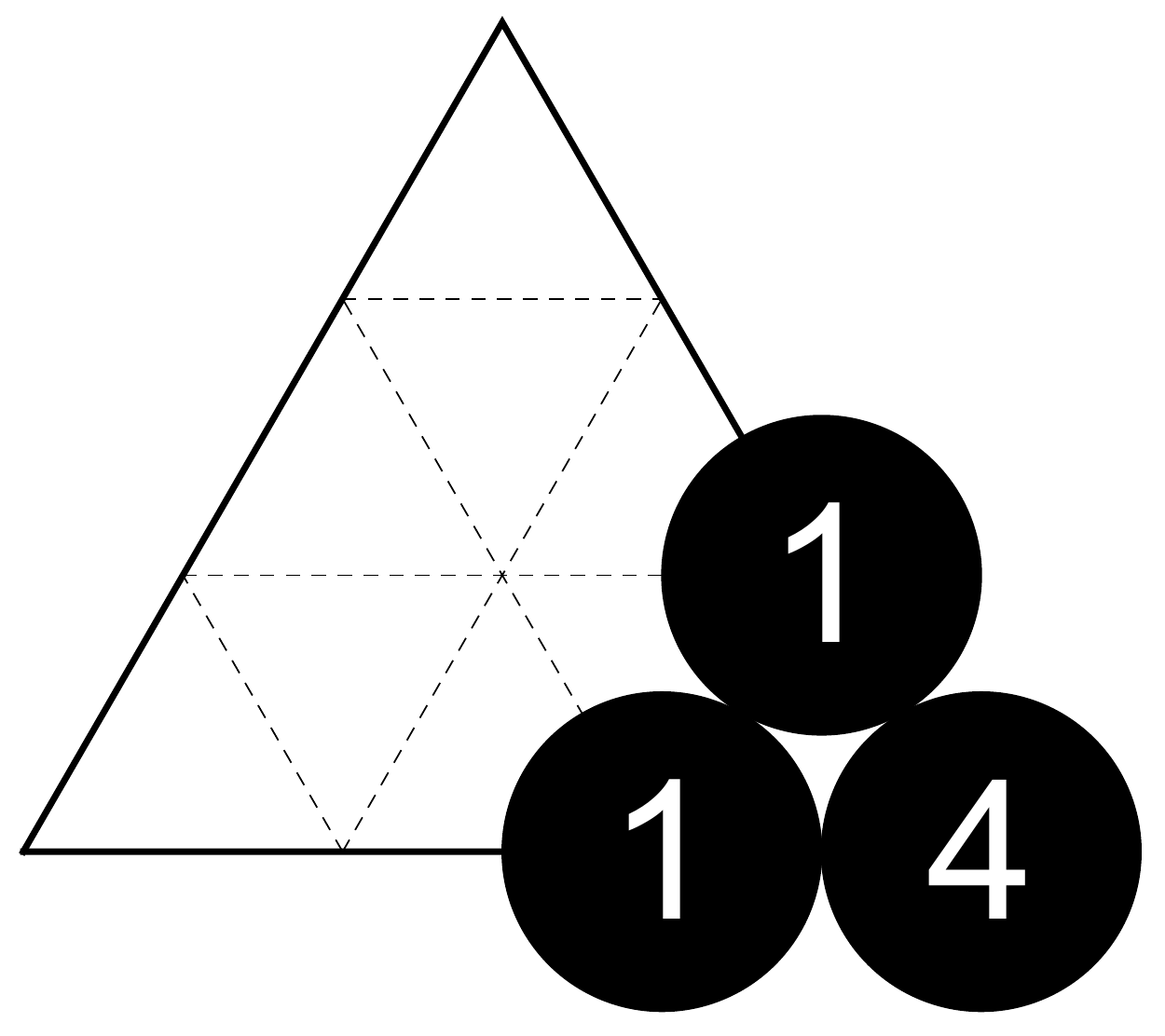}} \quad
\subfloat[3]{\includegraphics[width=1.3cm]{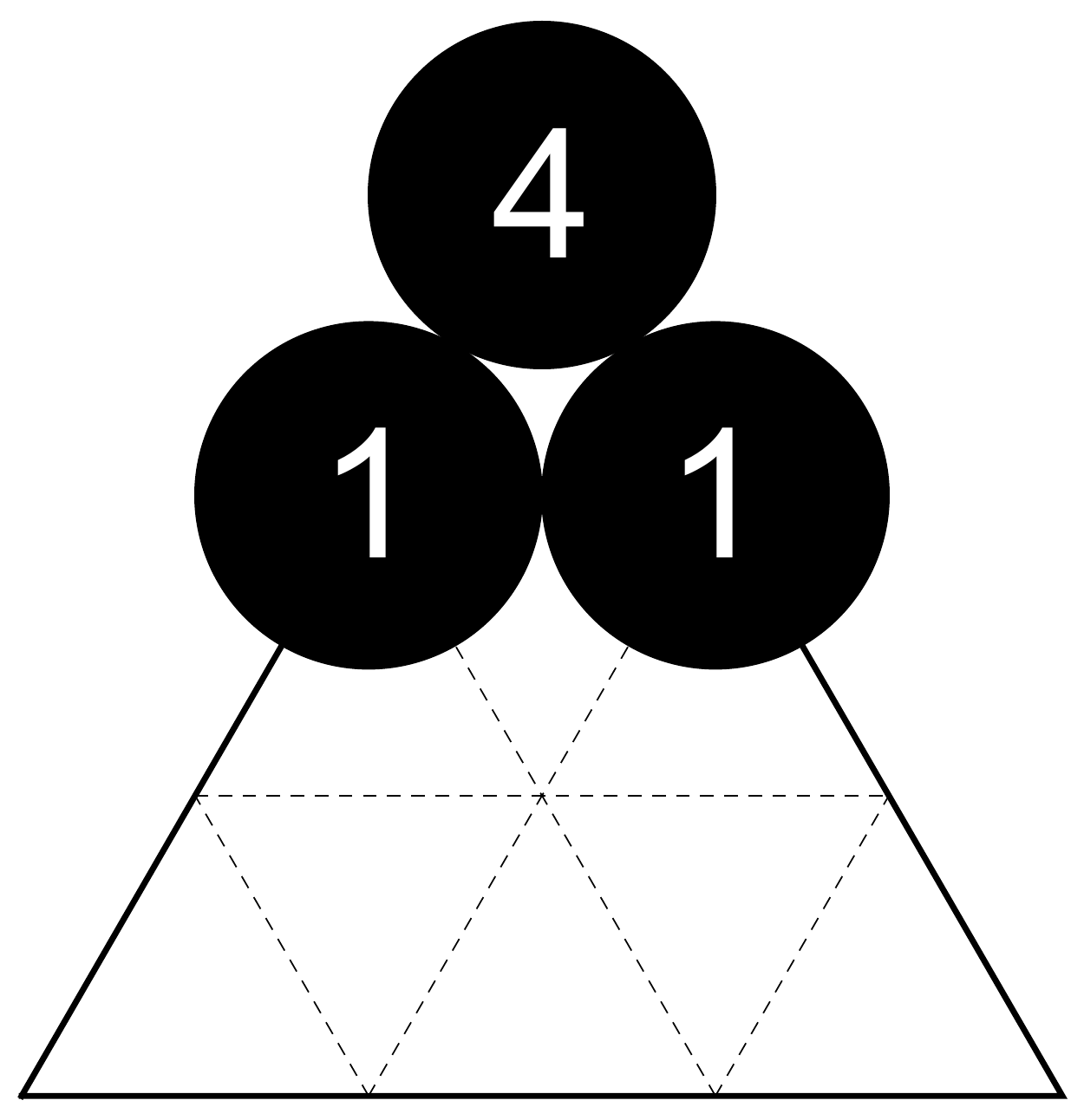}} \quad
\subfloat[4]{\includegraphics[width=1.5cm]{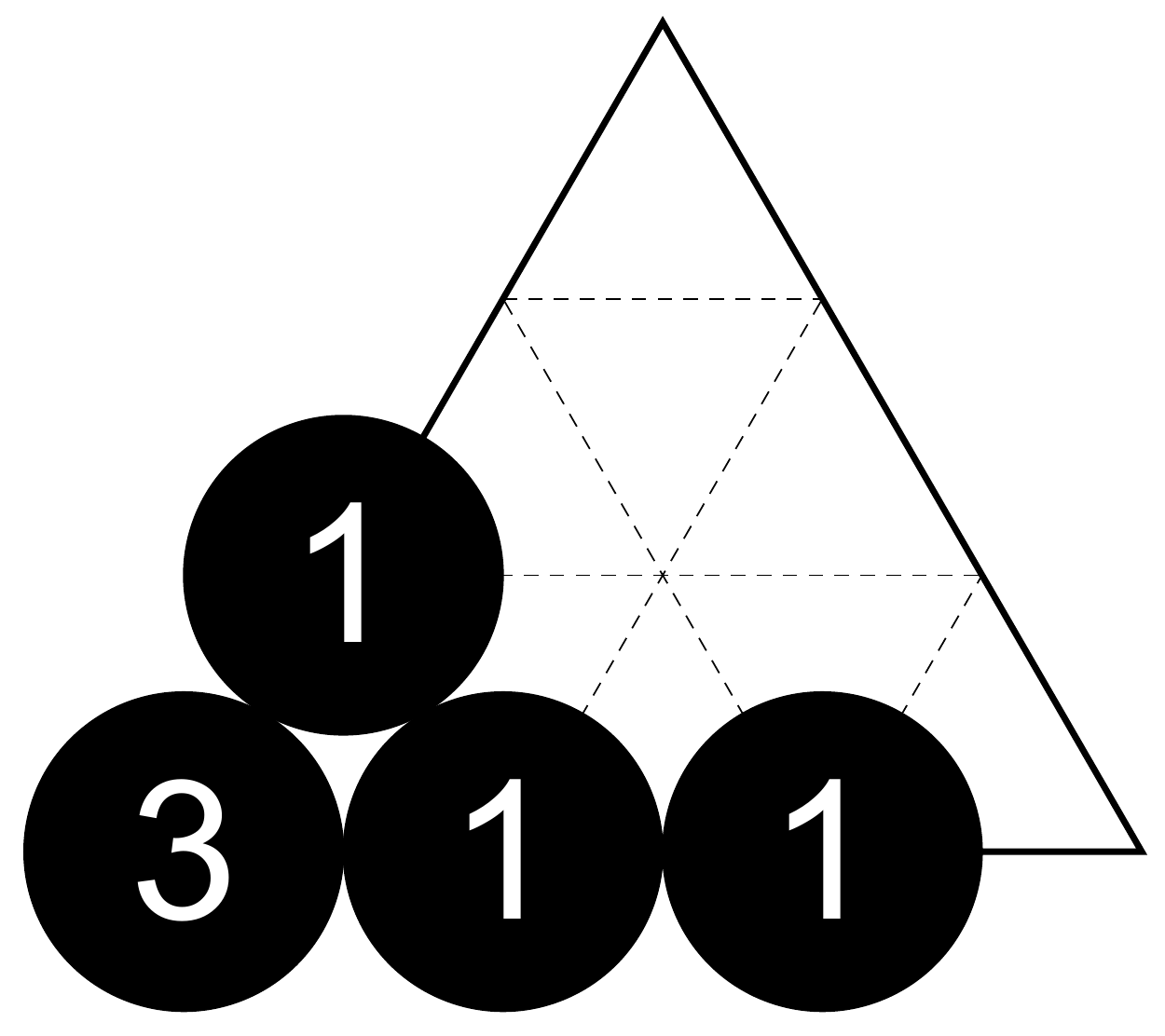}} \quad
\subfloat[5]{\includegraphics[width=1.5cm]{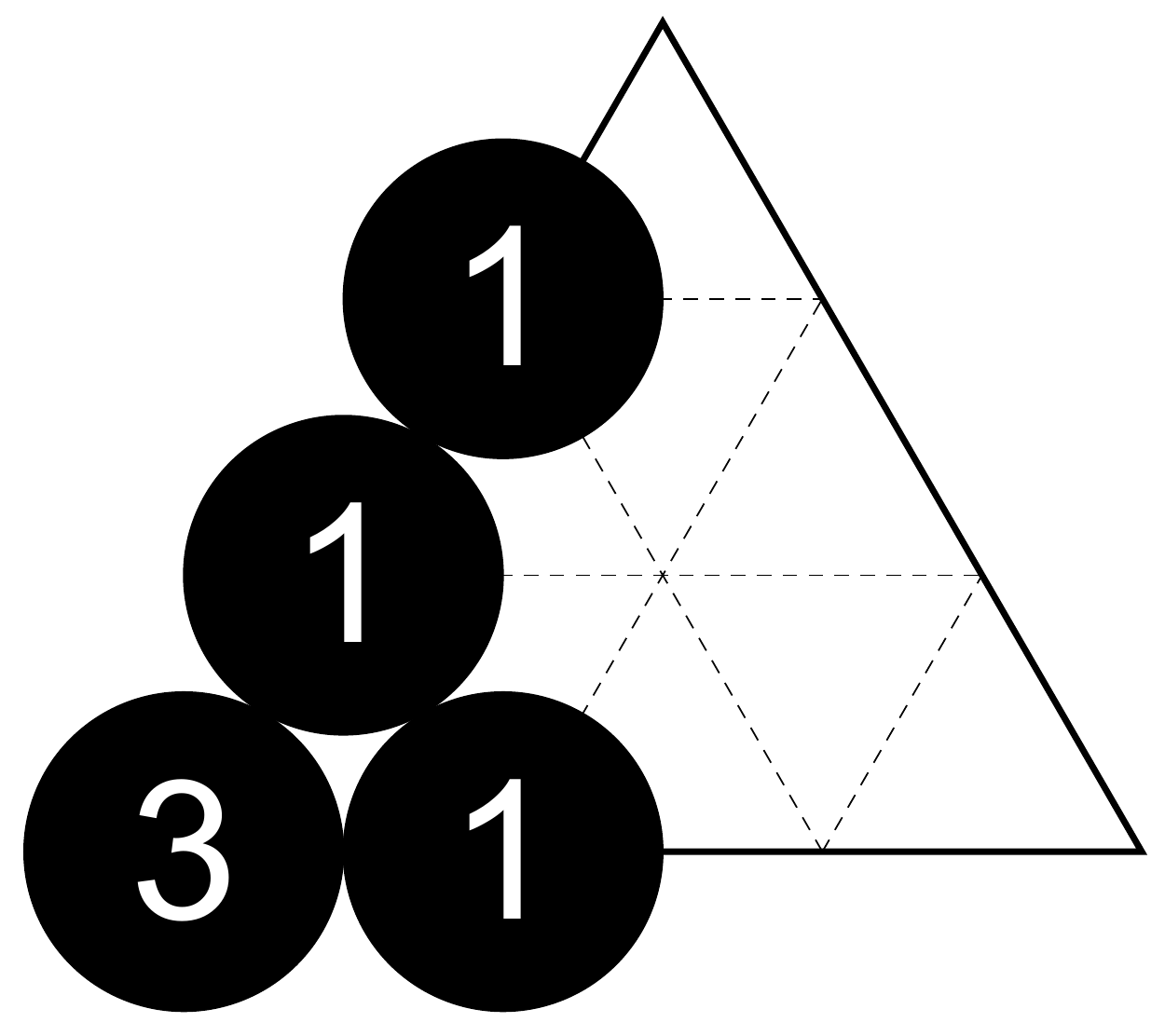}} \quad
\subfloat[6]{\includegraphics[width=1.5cm]{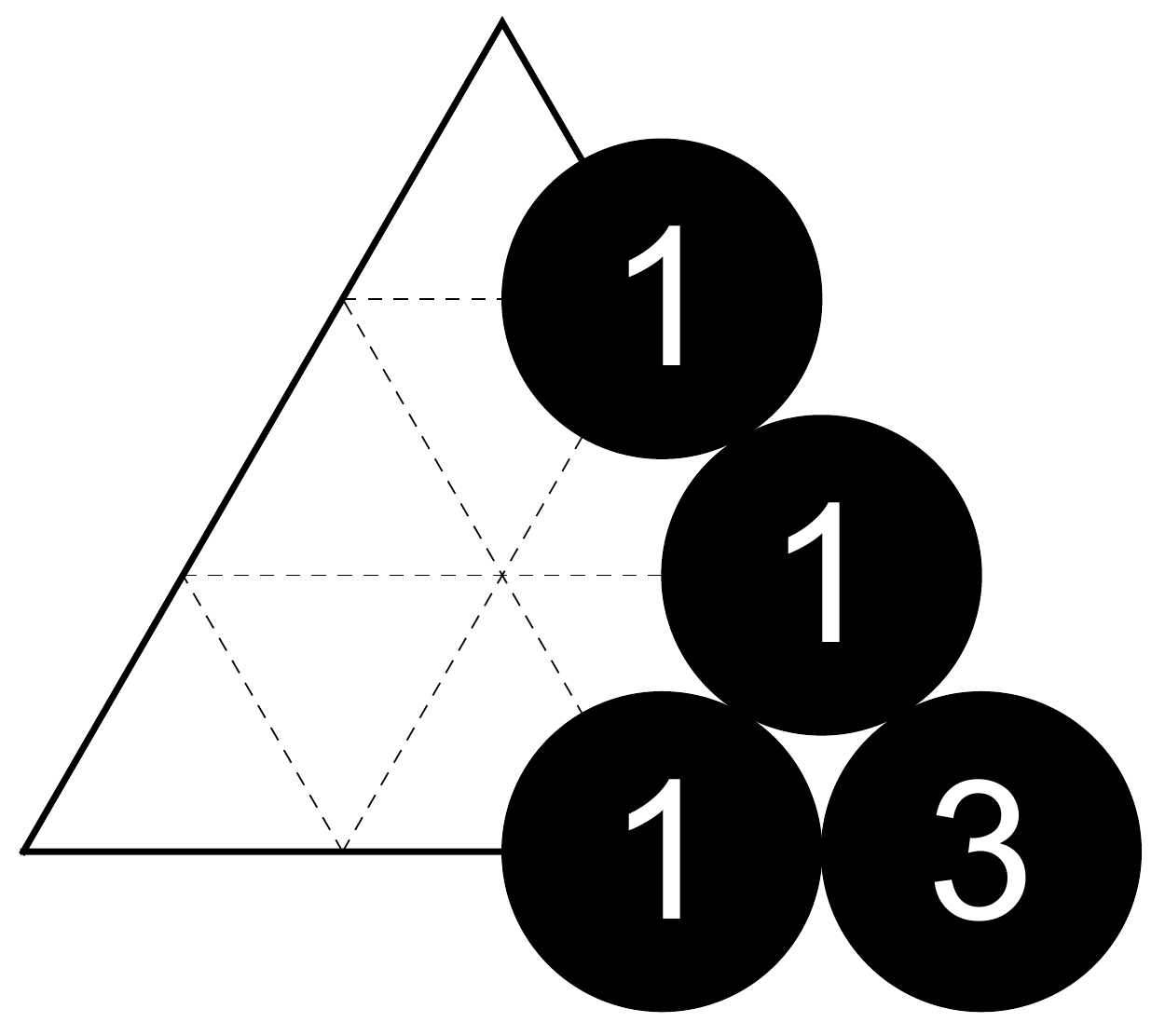}} \quad
\subfloat[7]{\includegraphics[width=1.5cm]{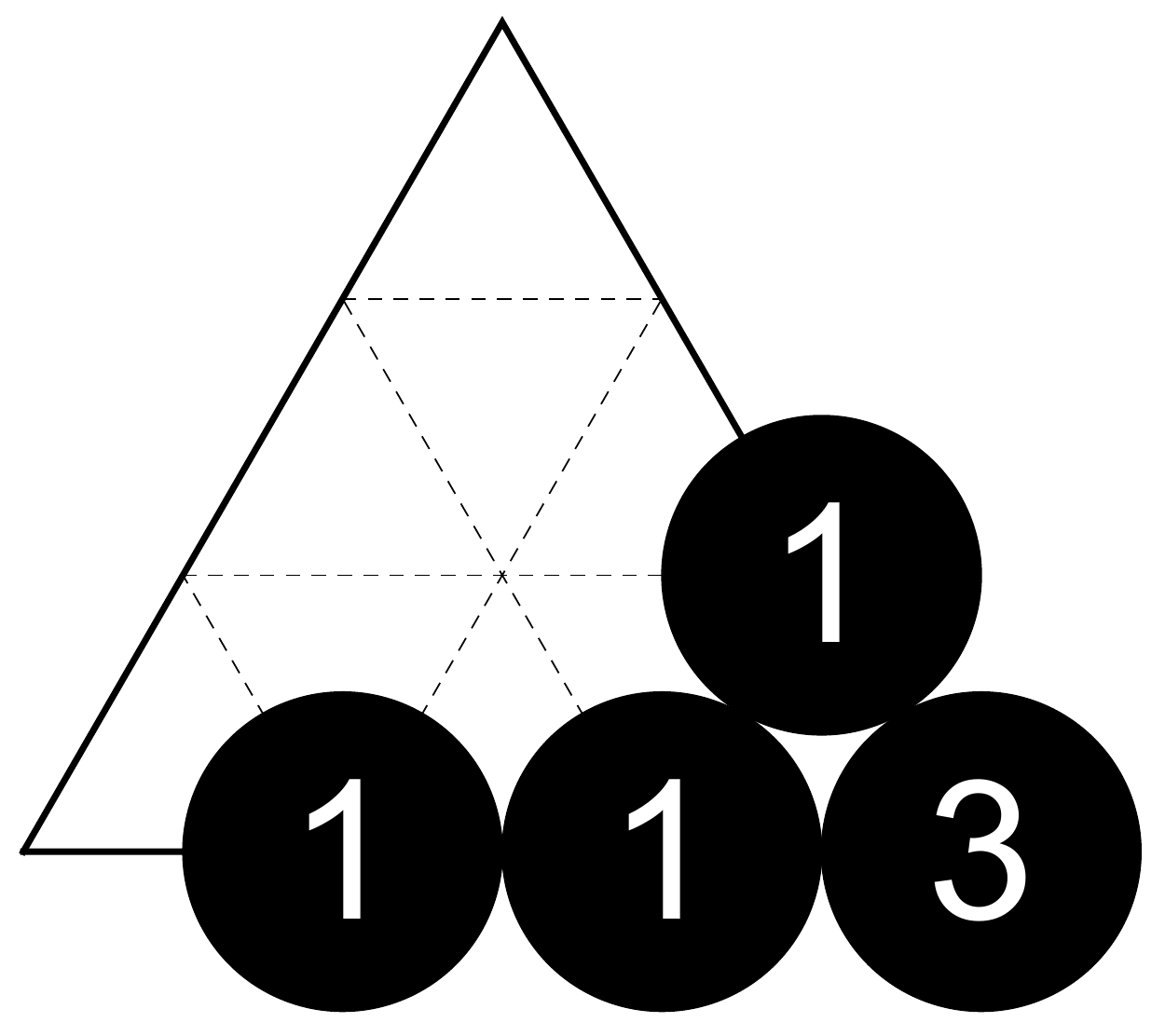}}
\\
\subfloat[8]{\includegraphics[width=1.35cm]{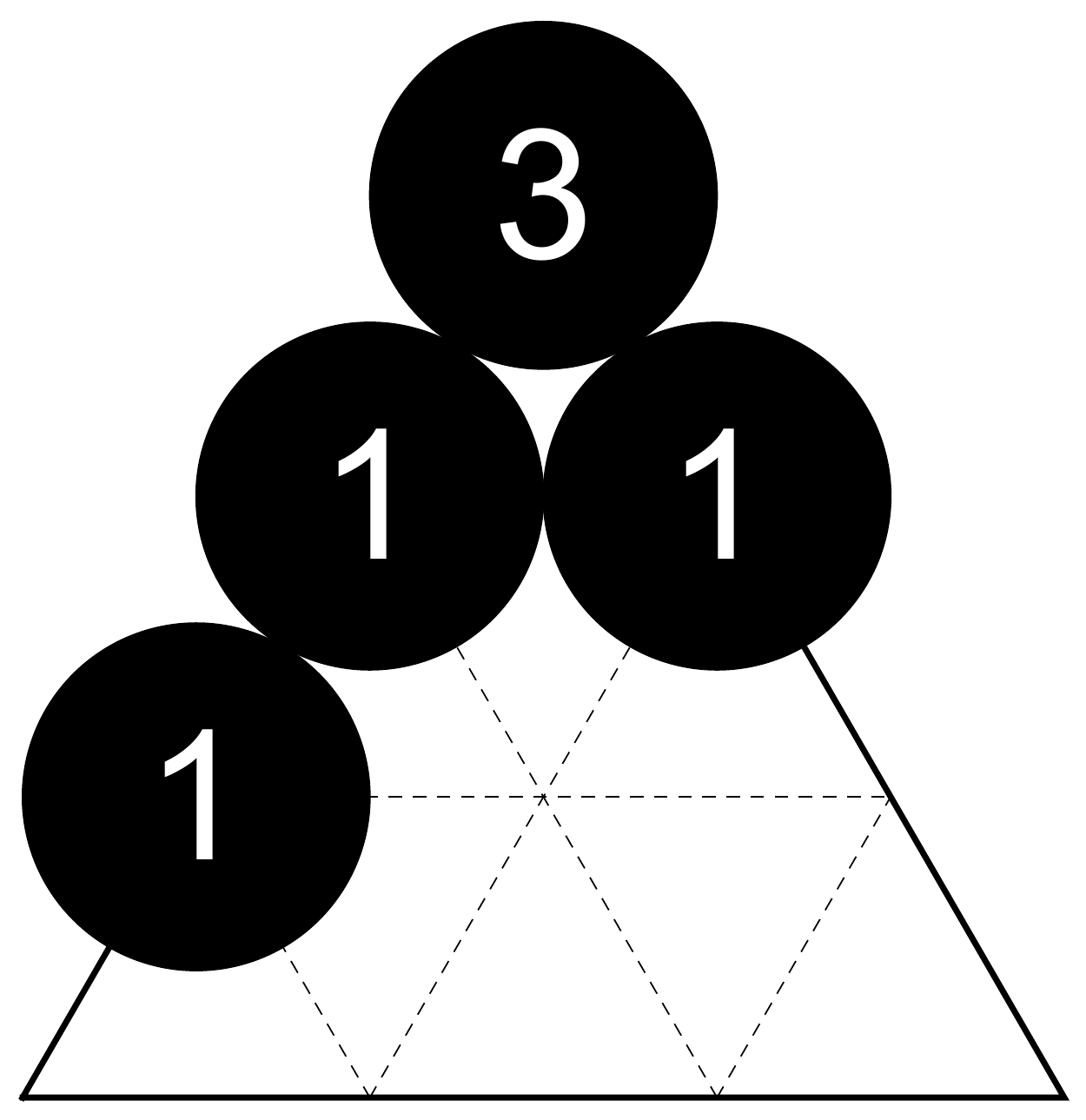}} \quad
\subfloat[9]{\includegraphics[width=1.35cm]{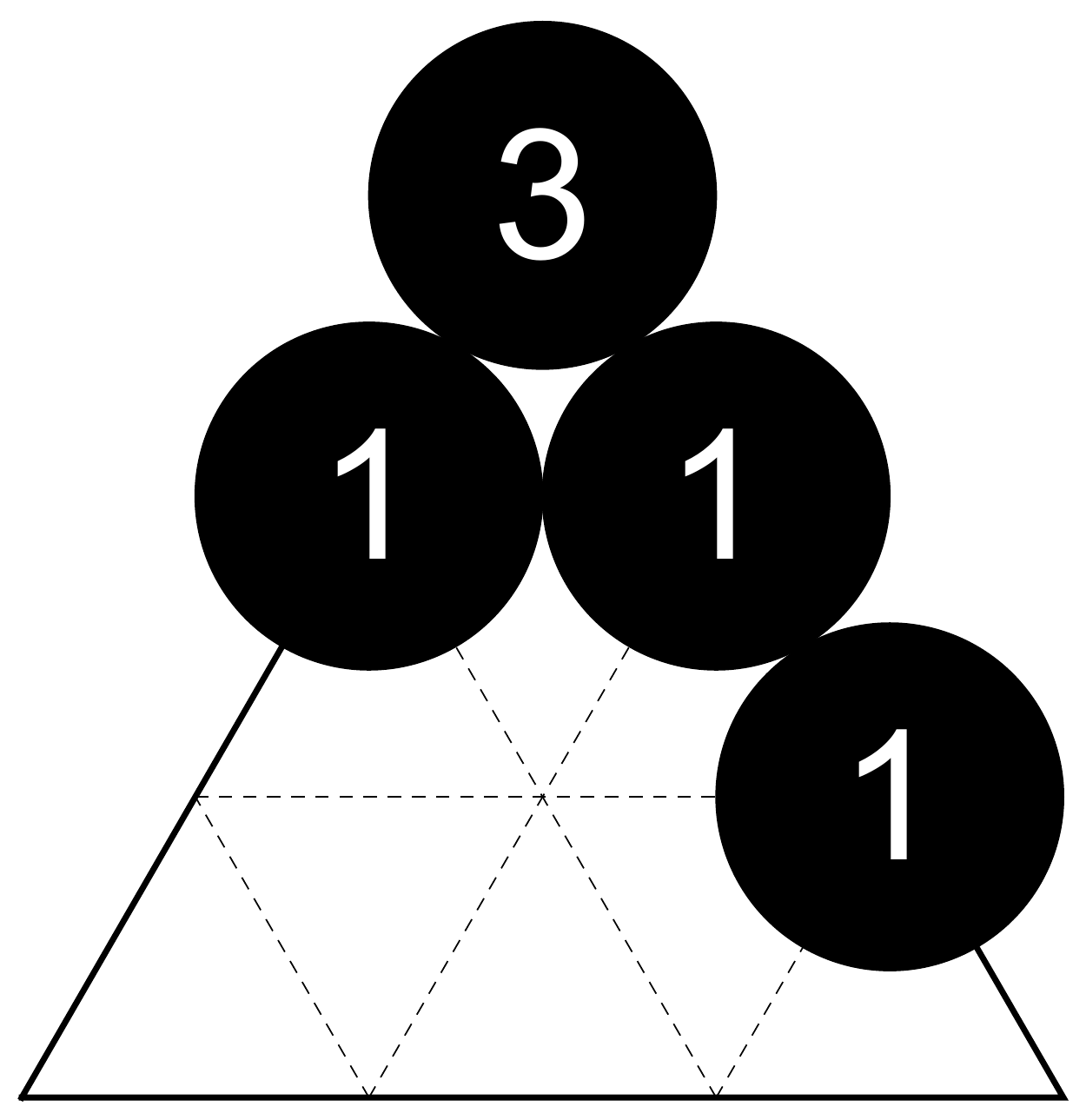}} \quad
\subfloat[10]{\includegraphics[width=1.7cm]{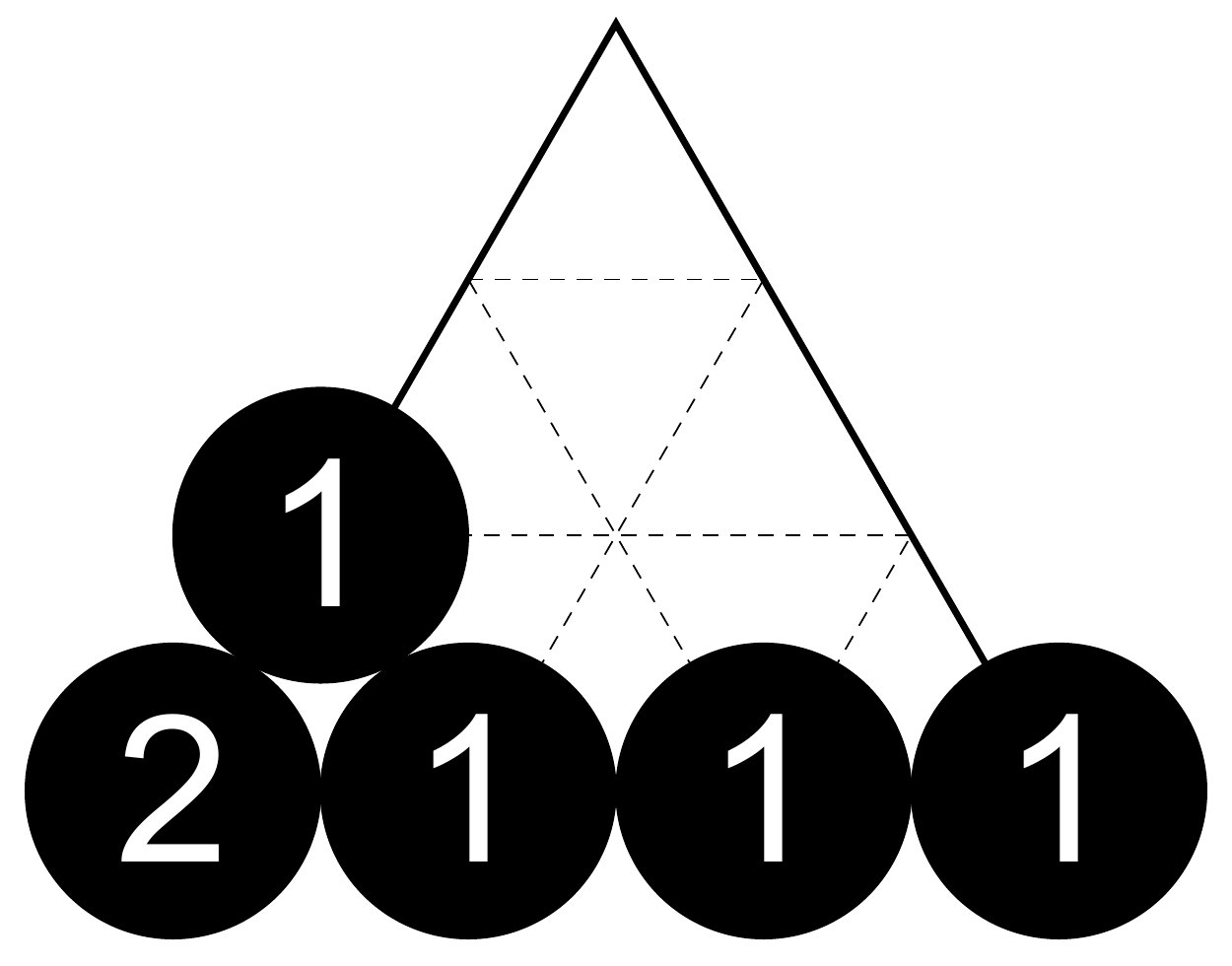}} \quad
\subfloat[11]{\includegraphics[width=1.5cm]{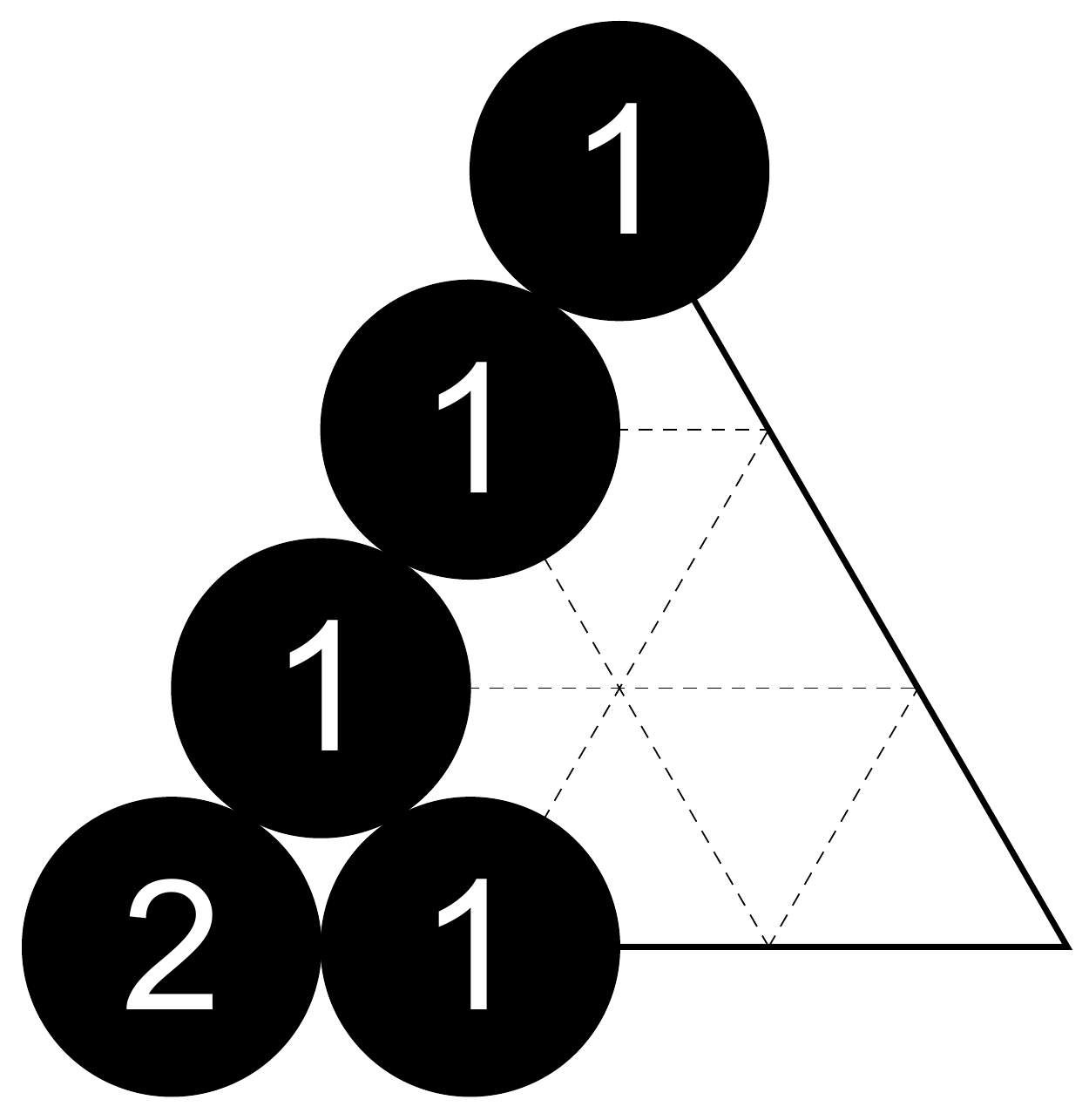}} \quad
\subfloat[12]{\includegraphics[width=1.5cm]{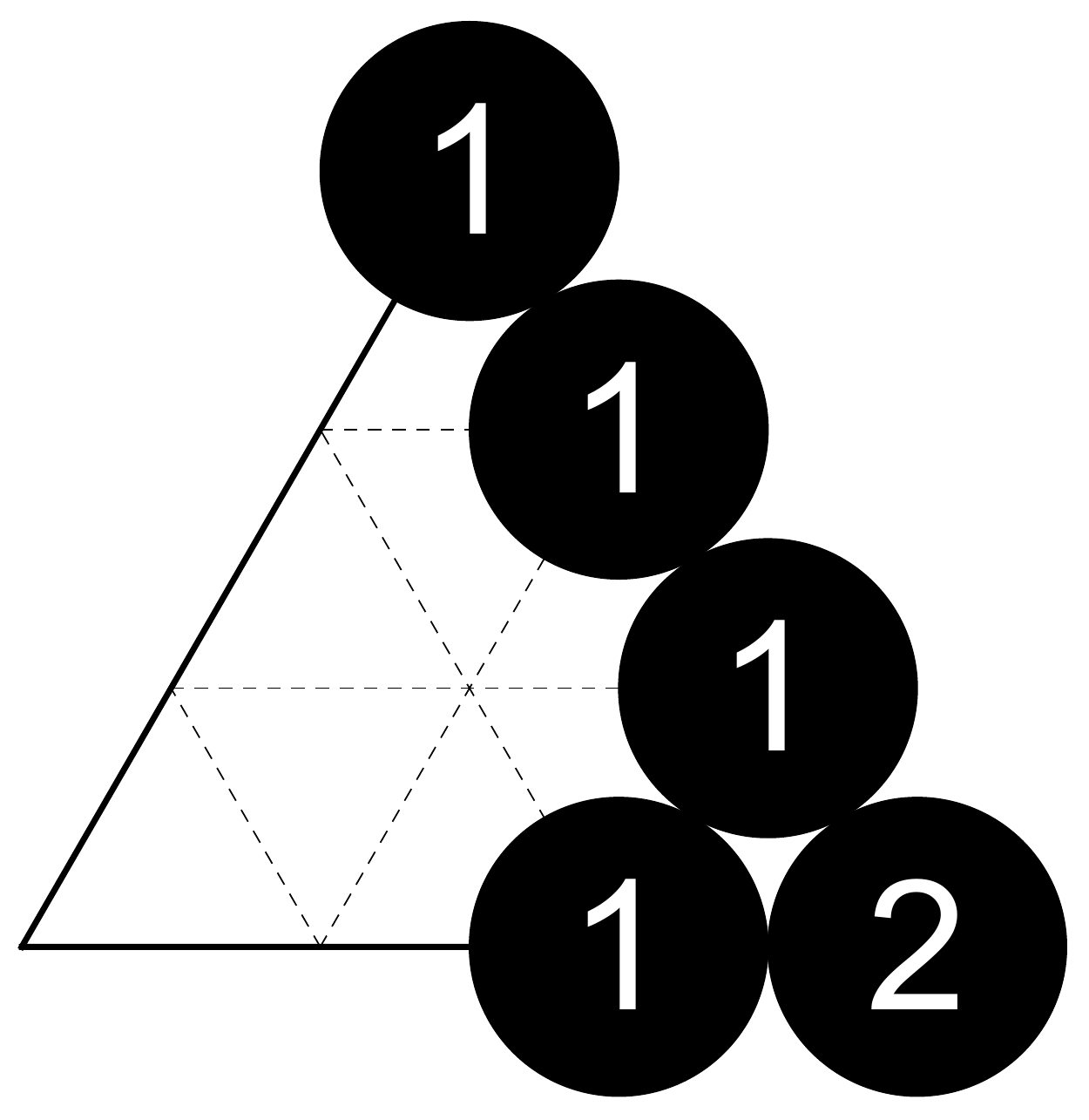}} \quad
\subfloat[13]{\includegraphics[width=1.5cm]{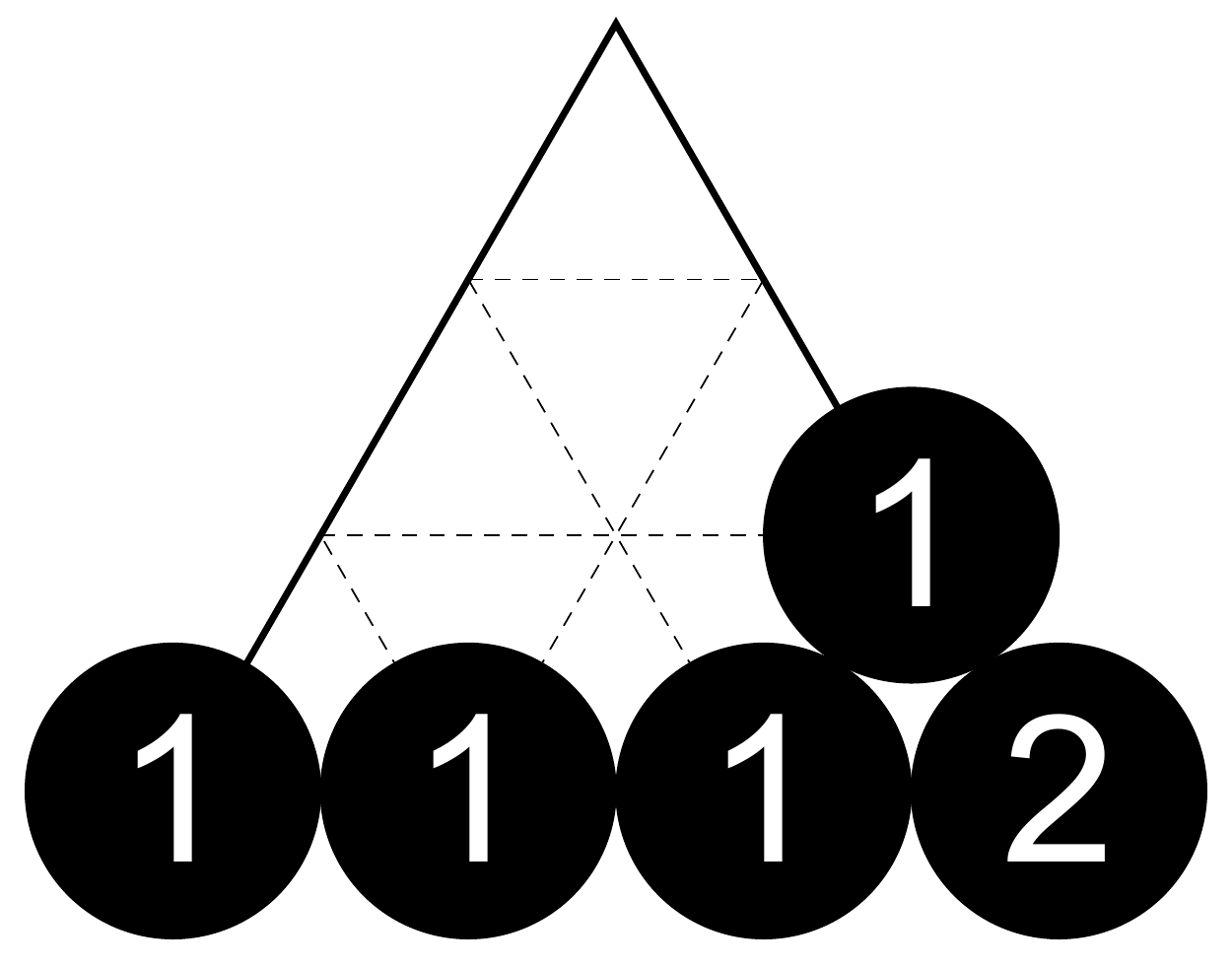}} \quad
\subfloat[14]{\includegraphics[width=1.5cm]{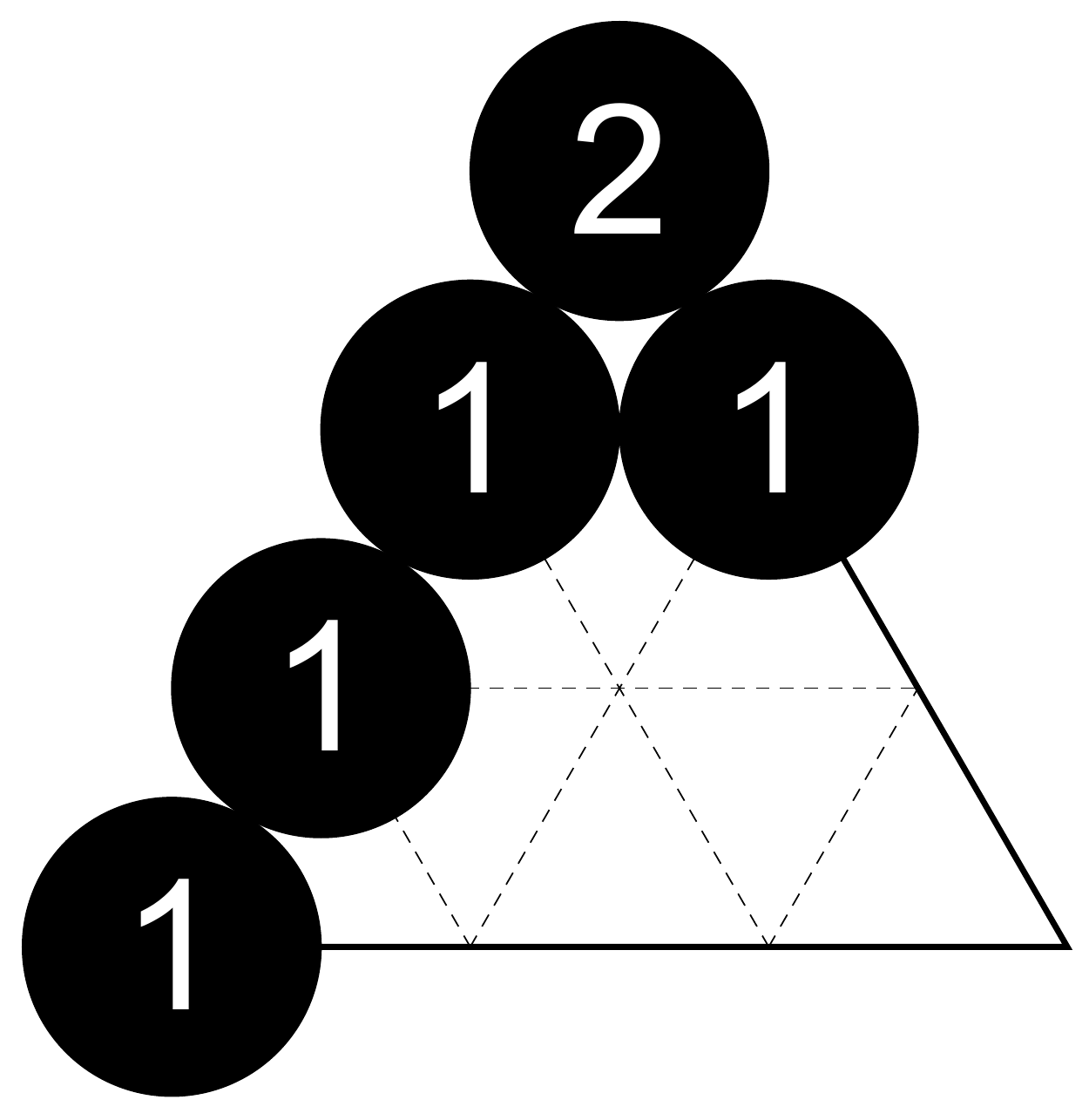}}
\\
\subfloat[15]{\includegraphics[width=1.5cm]{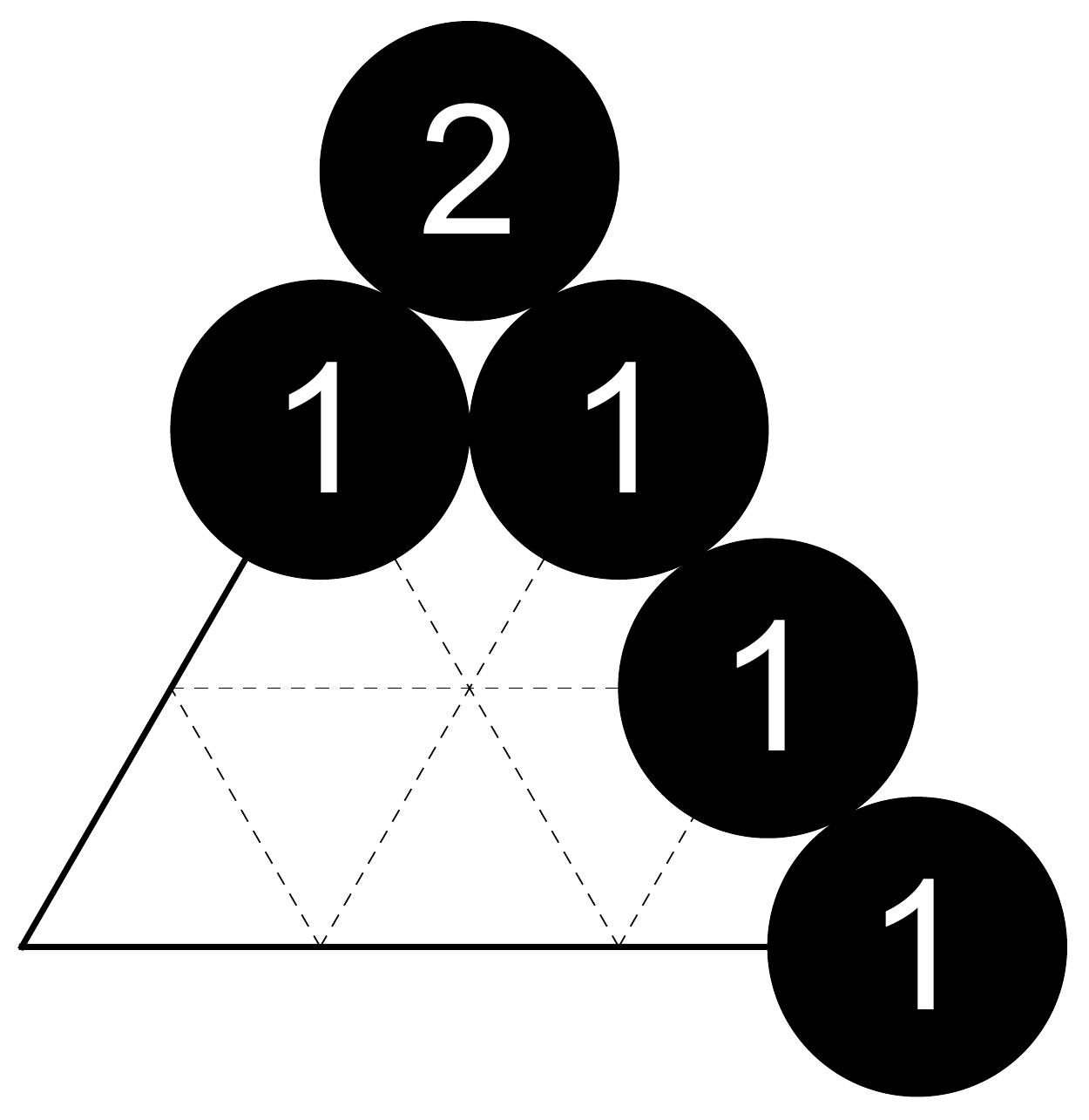}} \quad
\subfloat[16]{\includegraphics[width=1.35cm]{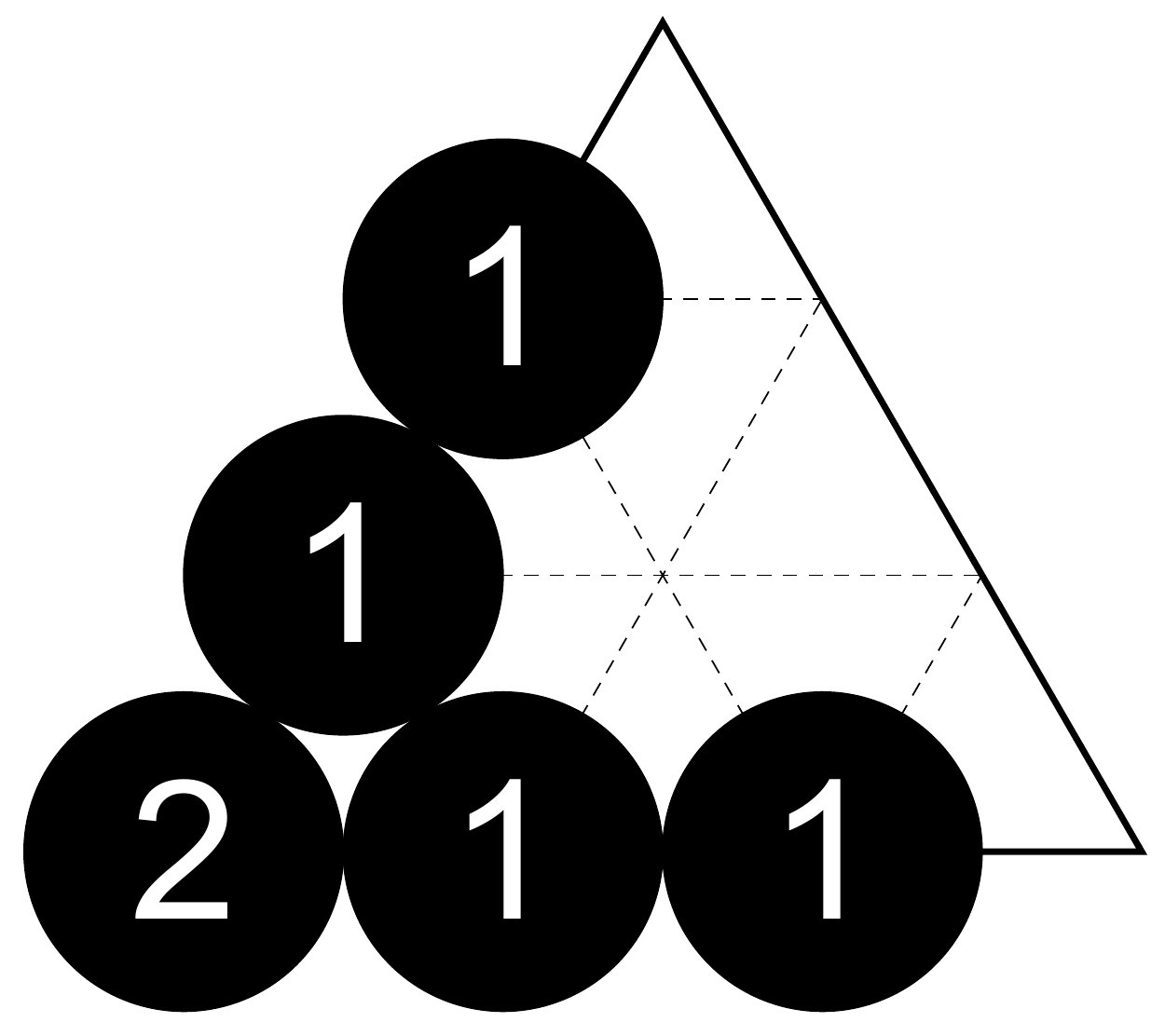}} \quad
\subfloat[17]{\includegraphics[width=1.5cm]{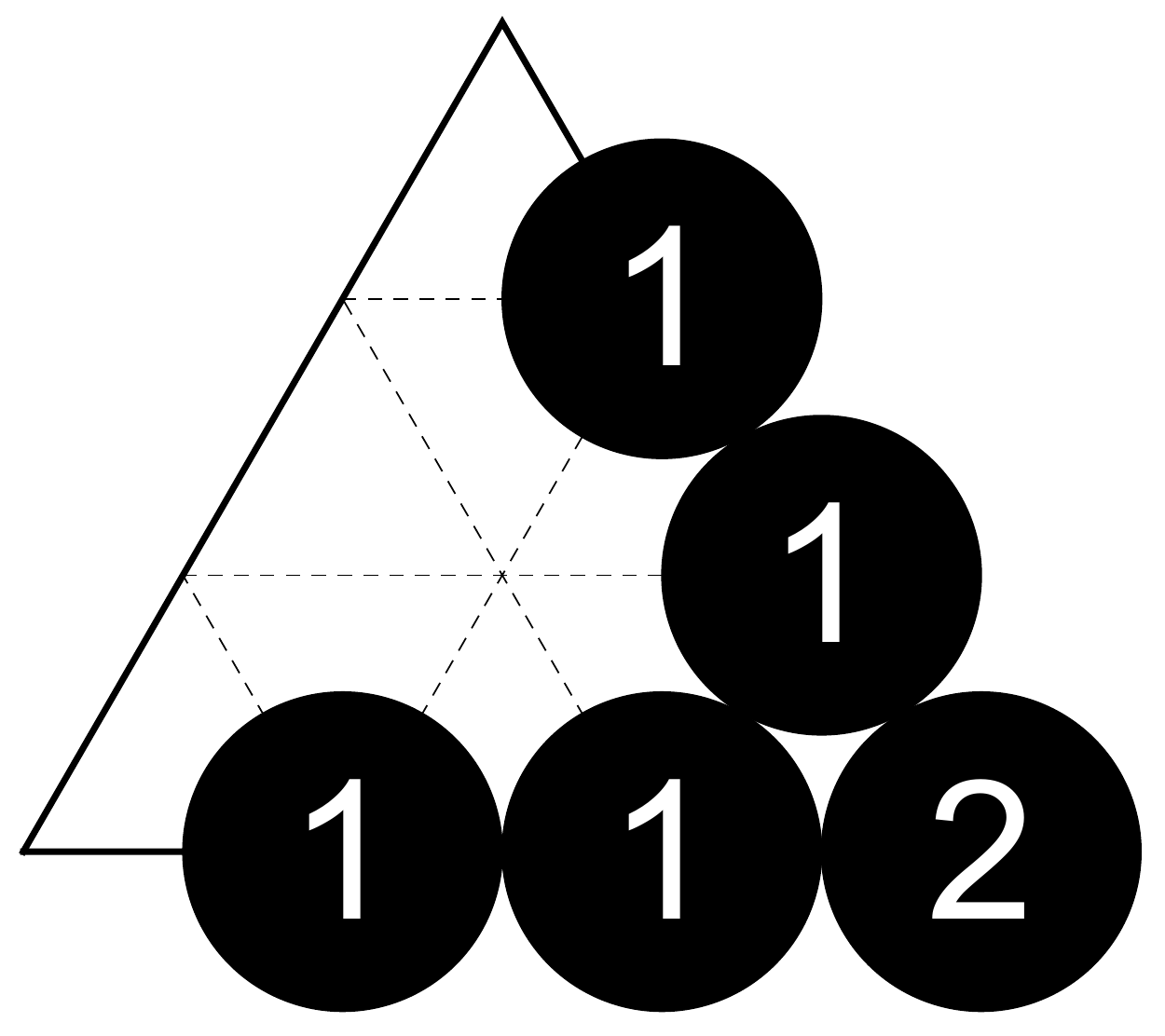}} \quad
\subfloat[18]{\includegraphics[width=1.4cm]{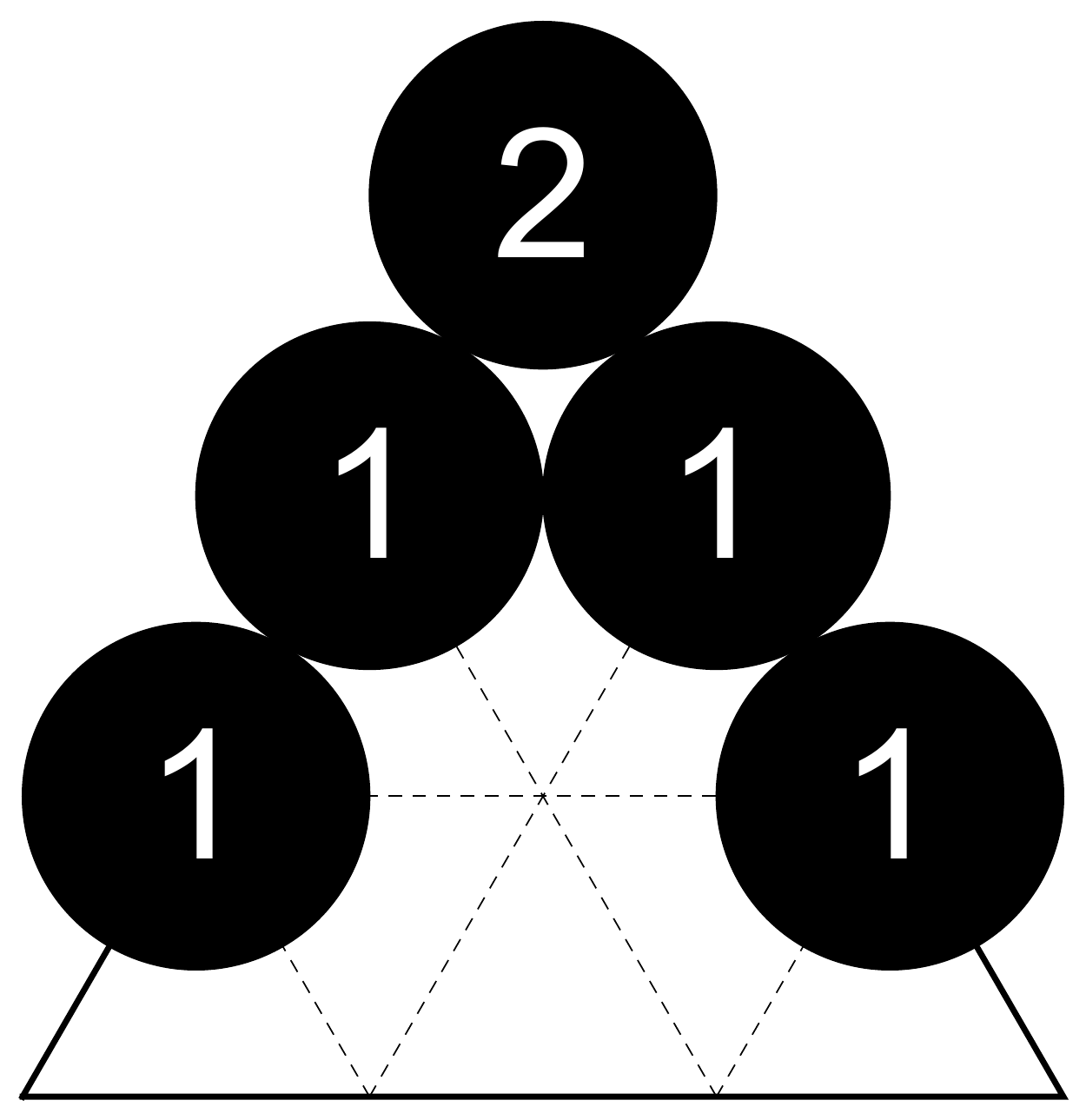}} \quad
\subfloat[19]{\includegraphics[width=1.6cm]{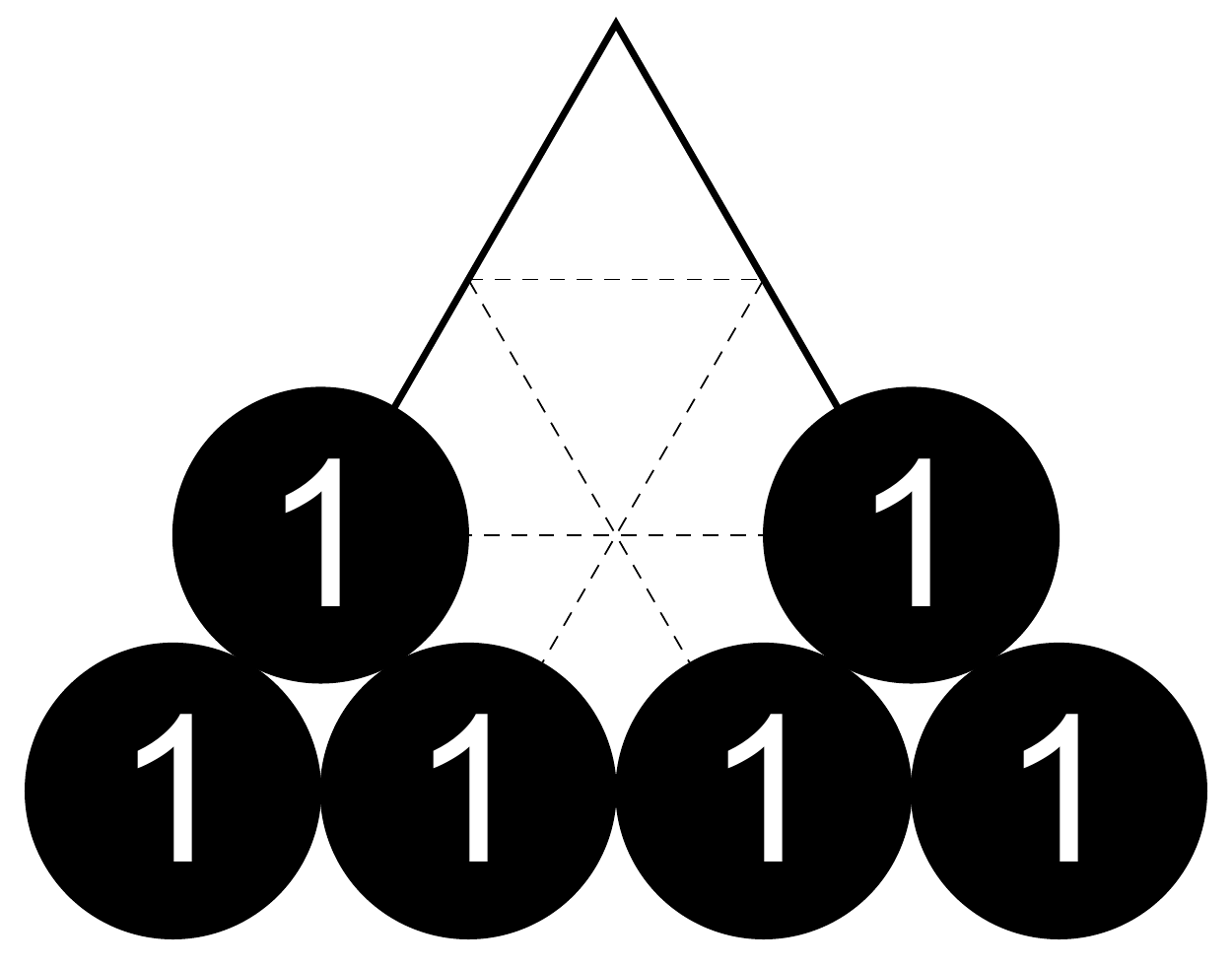}} \quad
\subfloat[20]{\includegraphics[width=1.5cm]{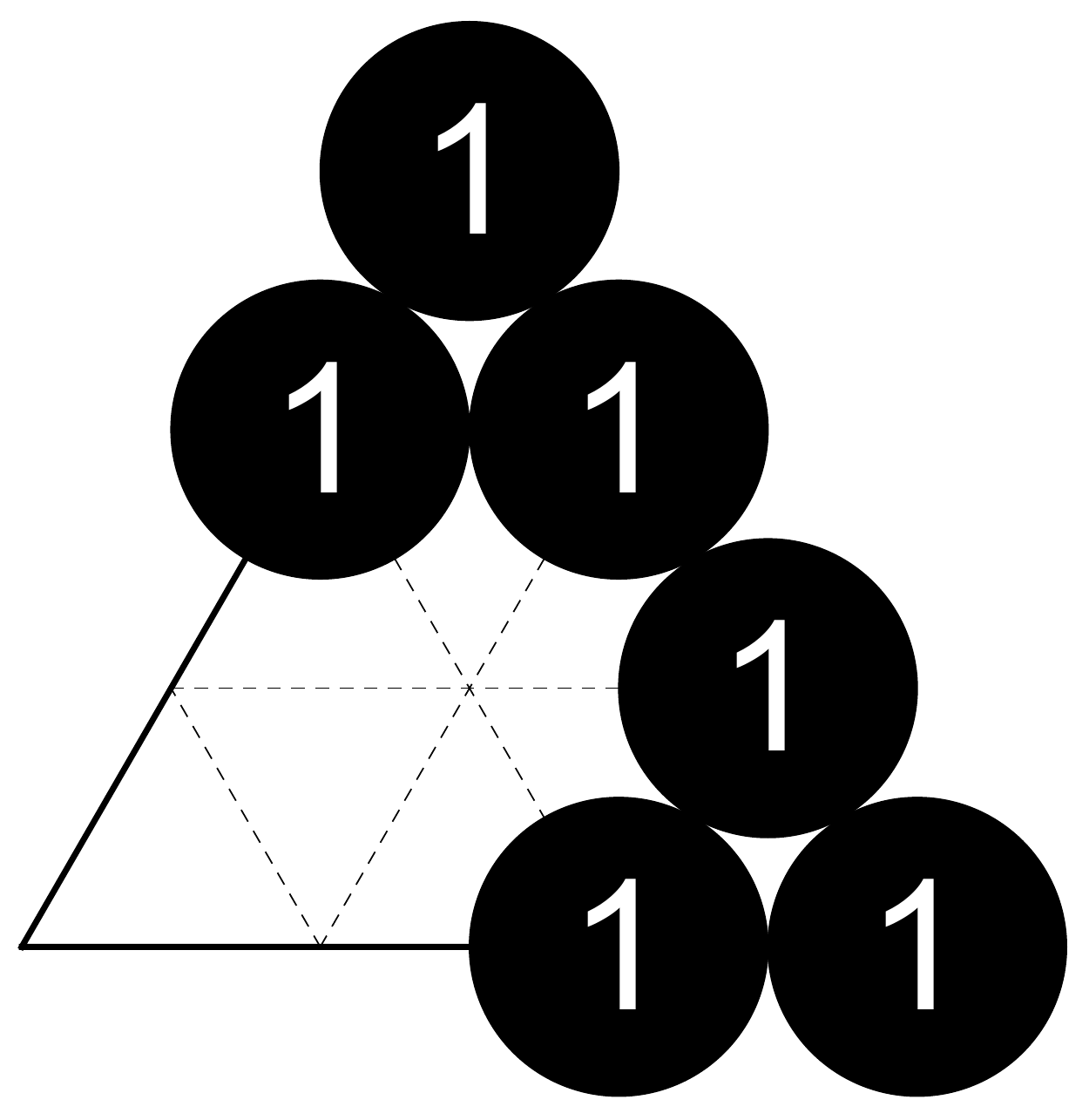}} \quad
\subfloat[21]{\includegraphics[width=1.5cm]{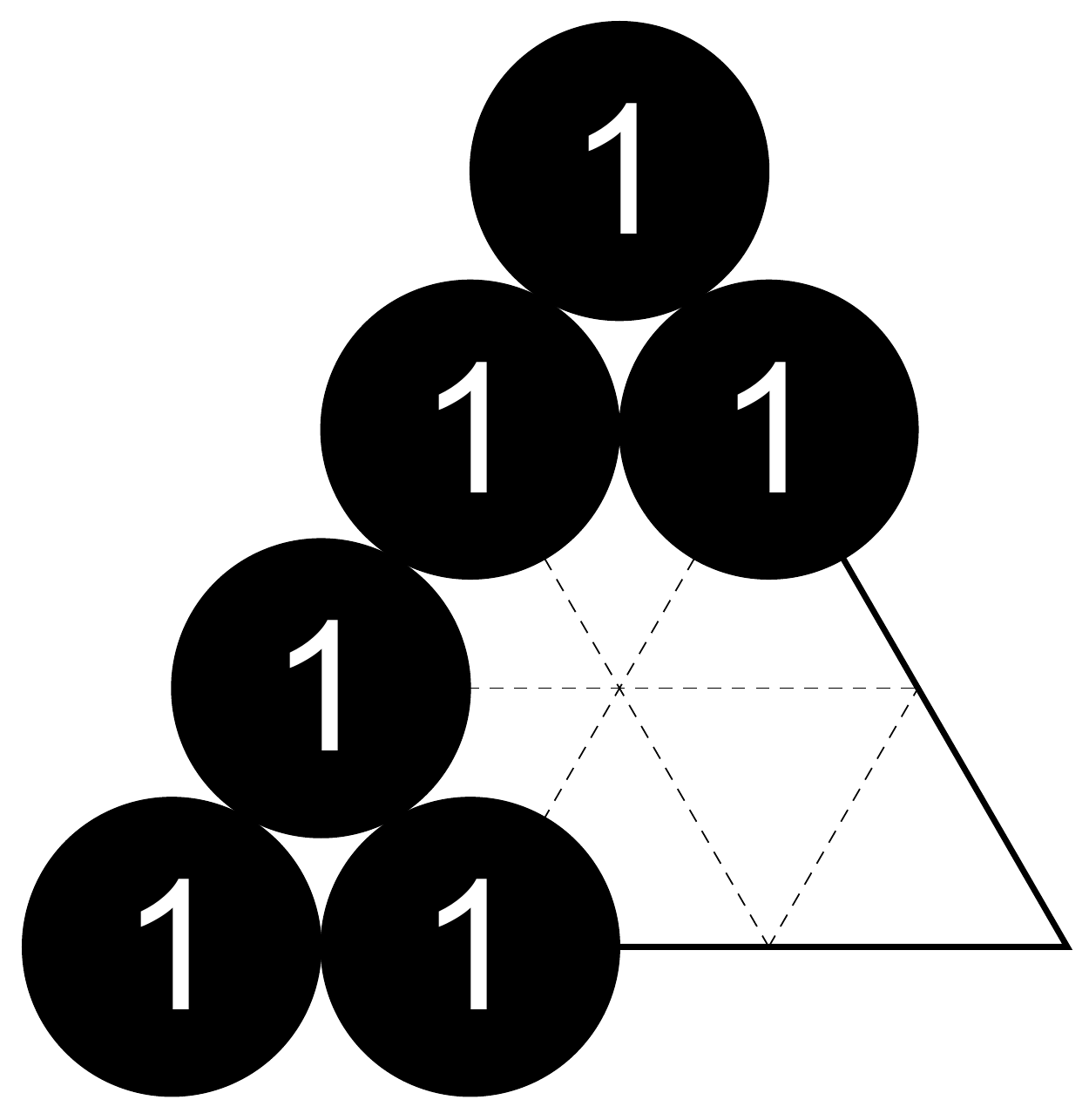}}
\\
\subfloat[22]{\includegraphics[width=1.5cm]{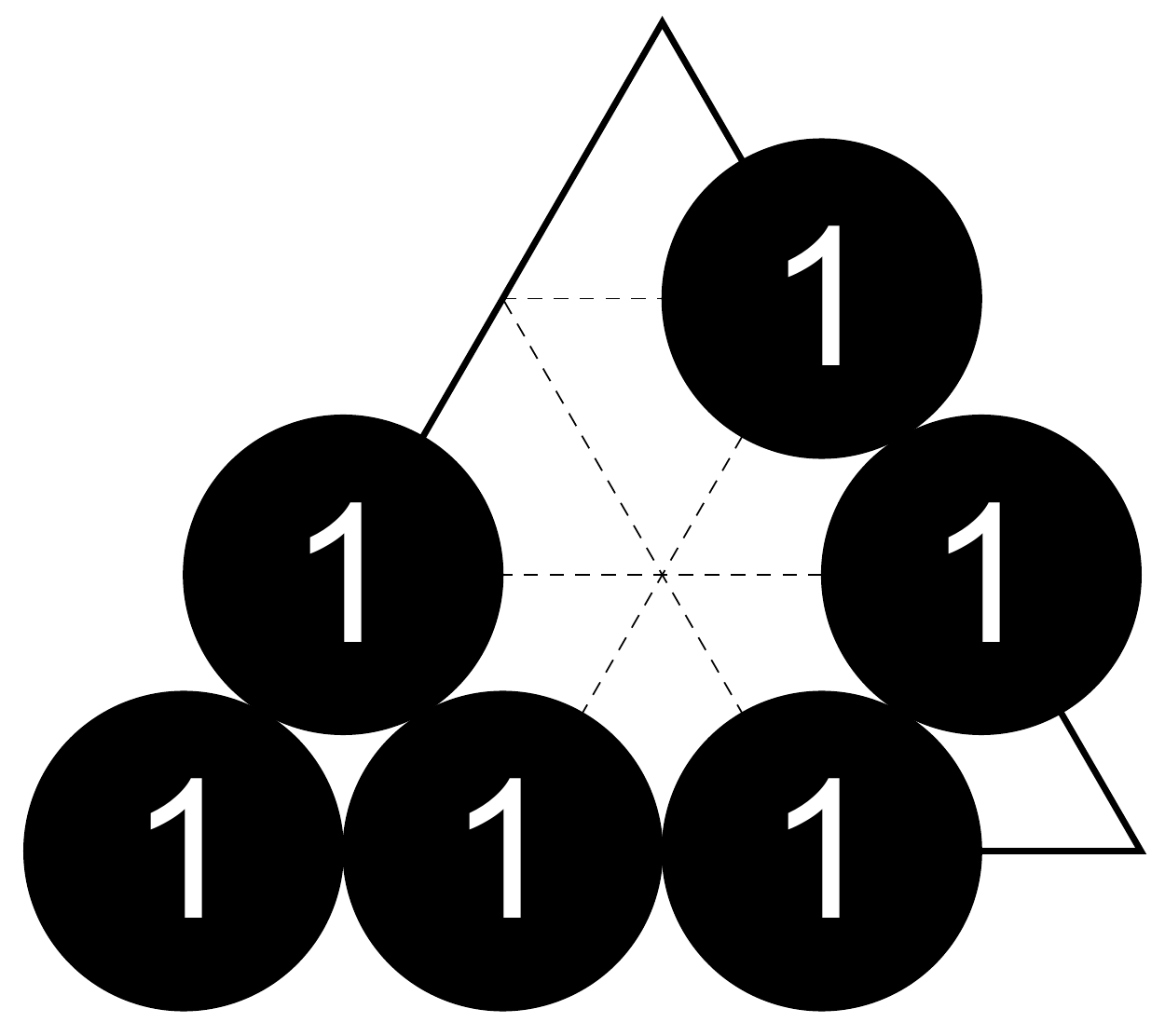}} \quad
\subfloat[23]{\includegraphics[width=1.5cm]{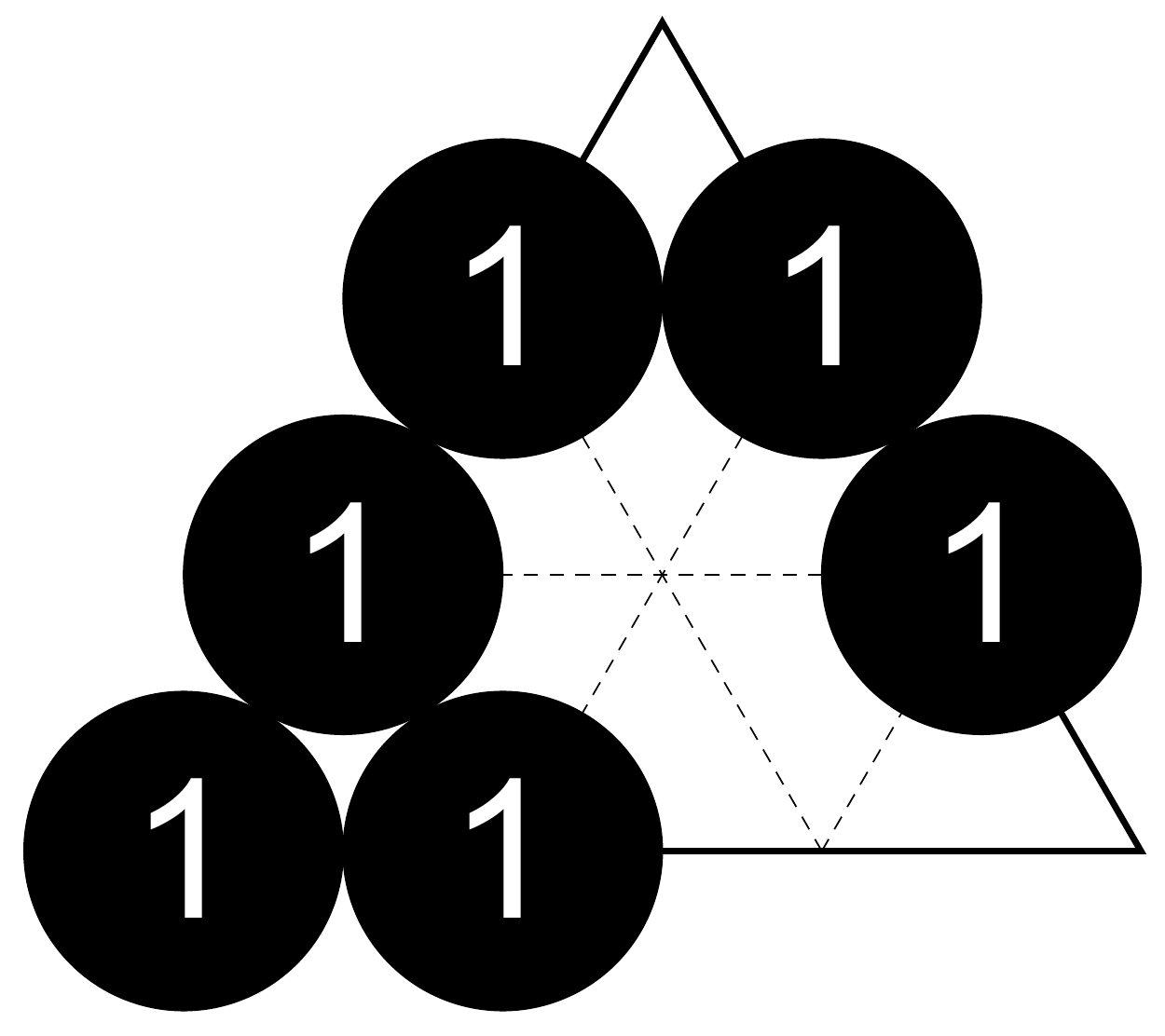}} \quad
\subfloat[24]{\includegraphics[width=1.35cm]{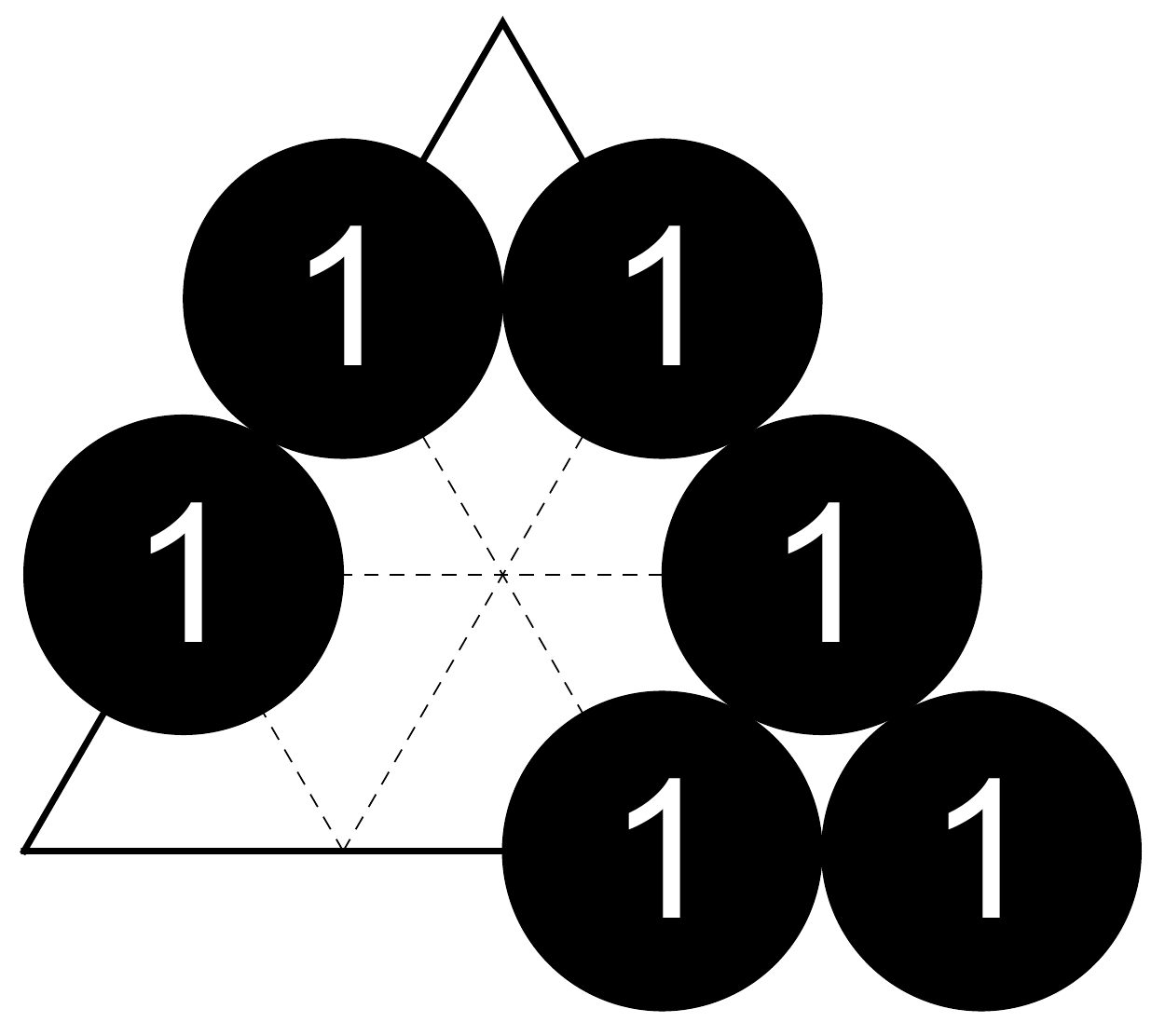}} \quad
\subfloat[25]{\includegraphics[width=1.5cm]{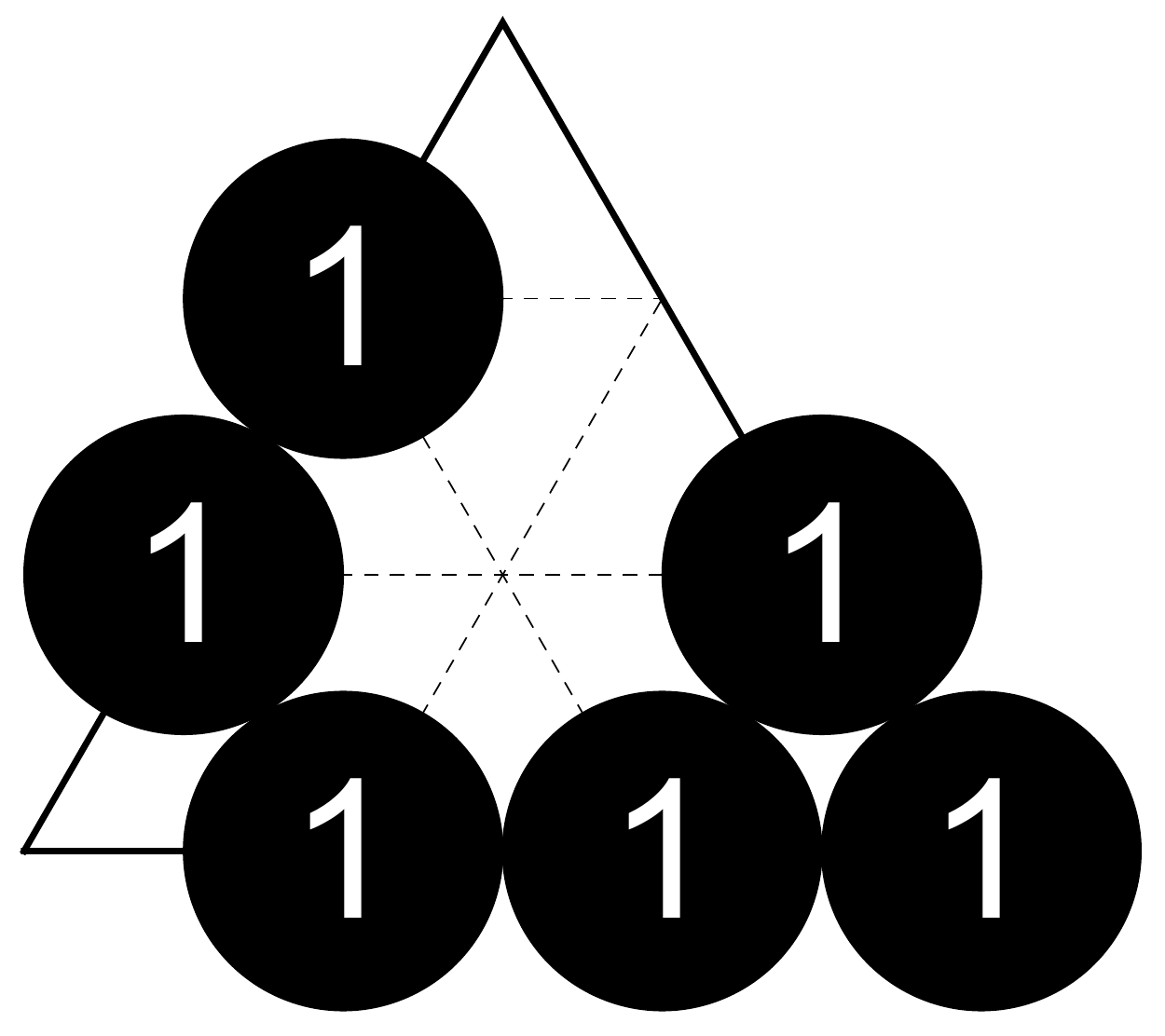}} \quad
\subfloat[26]{\includegraphics[width=1.35cm]{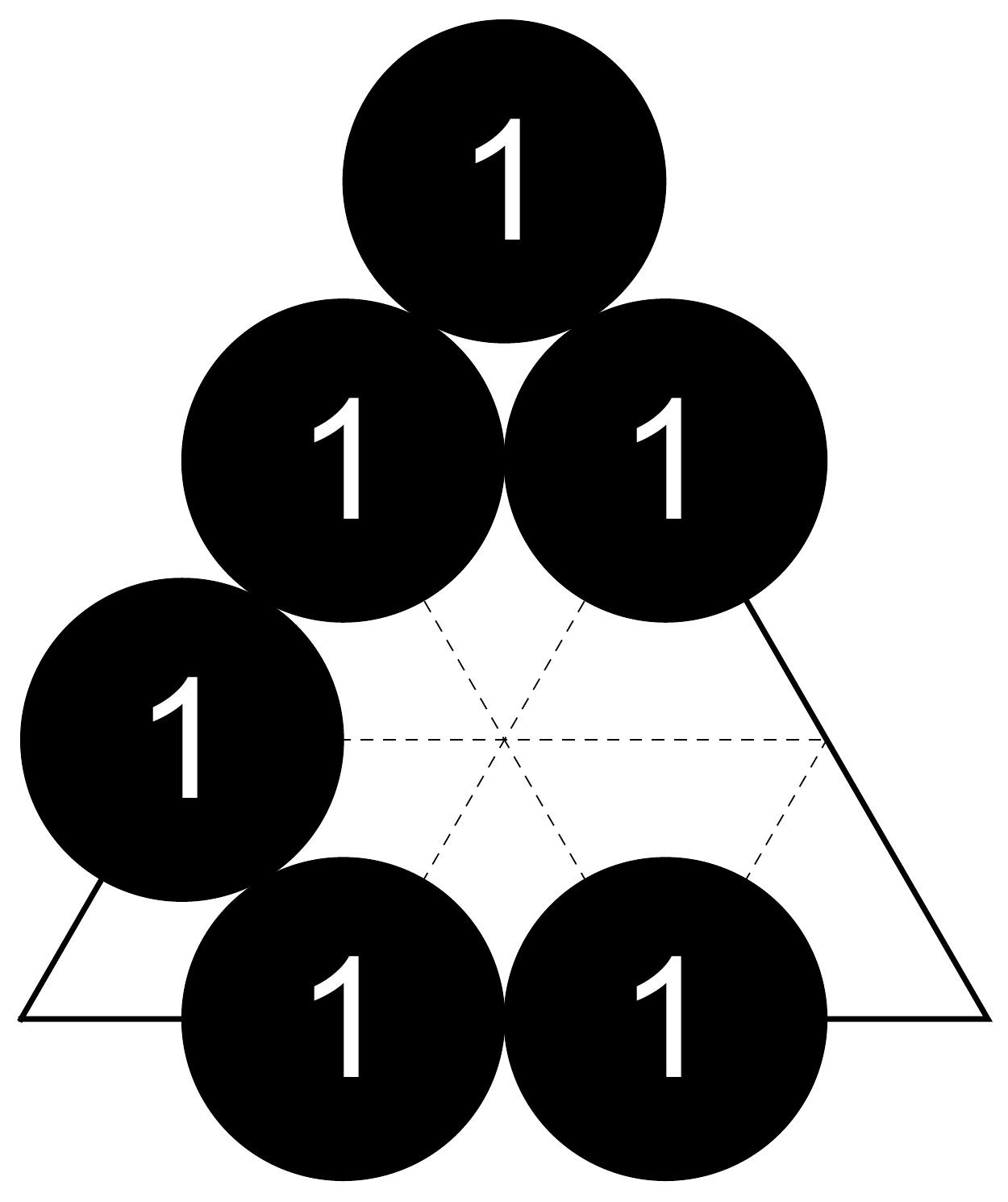}} \quad
\subfloat[27]{\includegraphics[width=1.3cm]{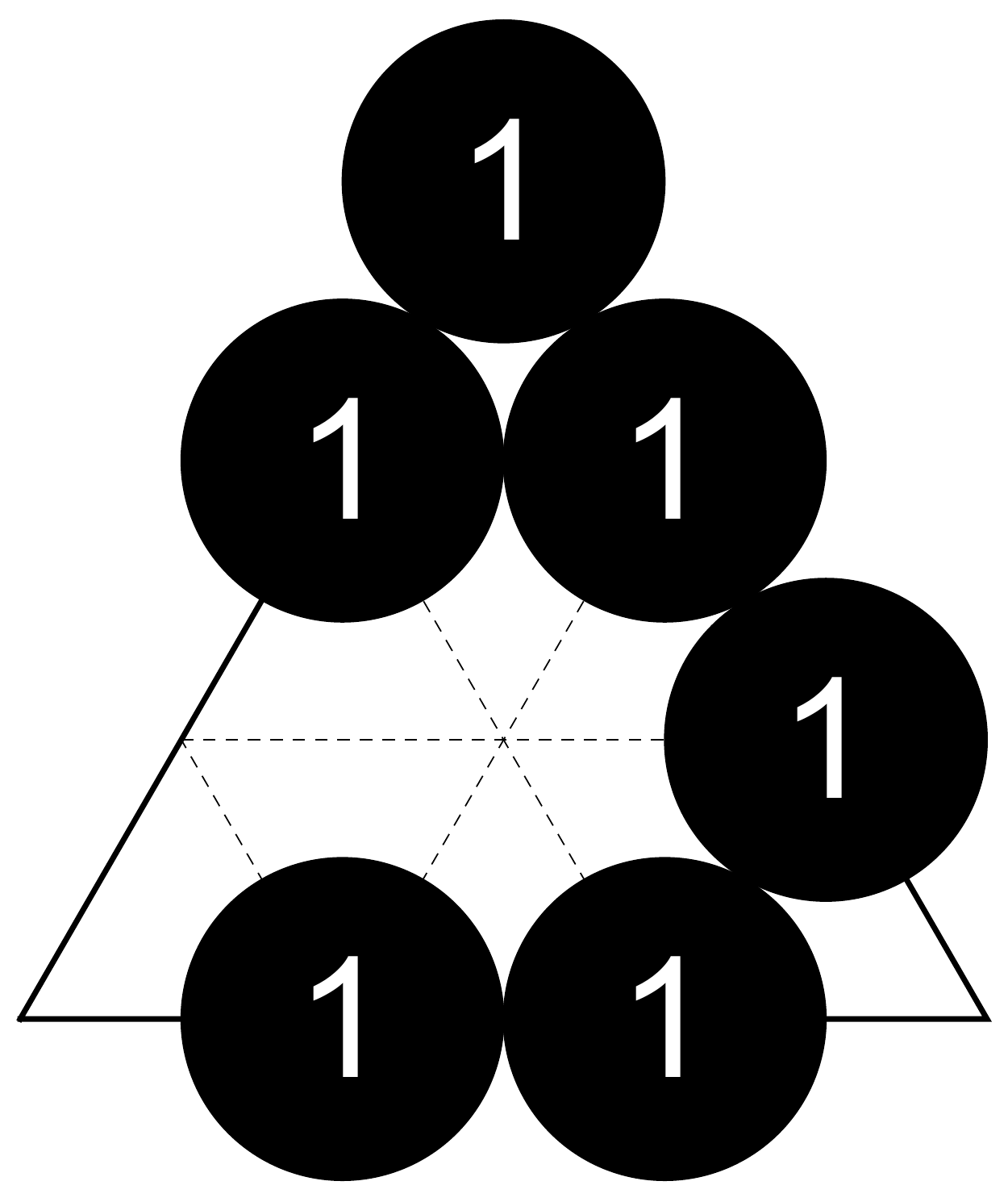}} \quad
\subfloat[28]{\includegraphics[width=1.3cm]{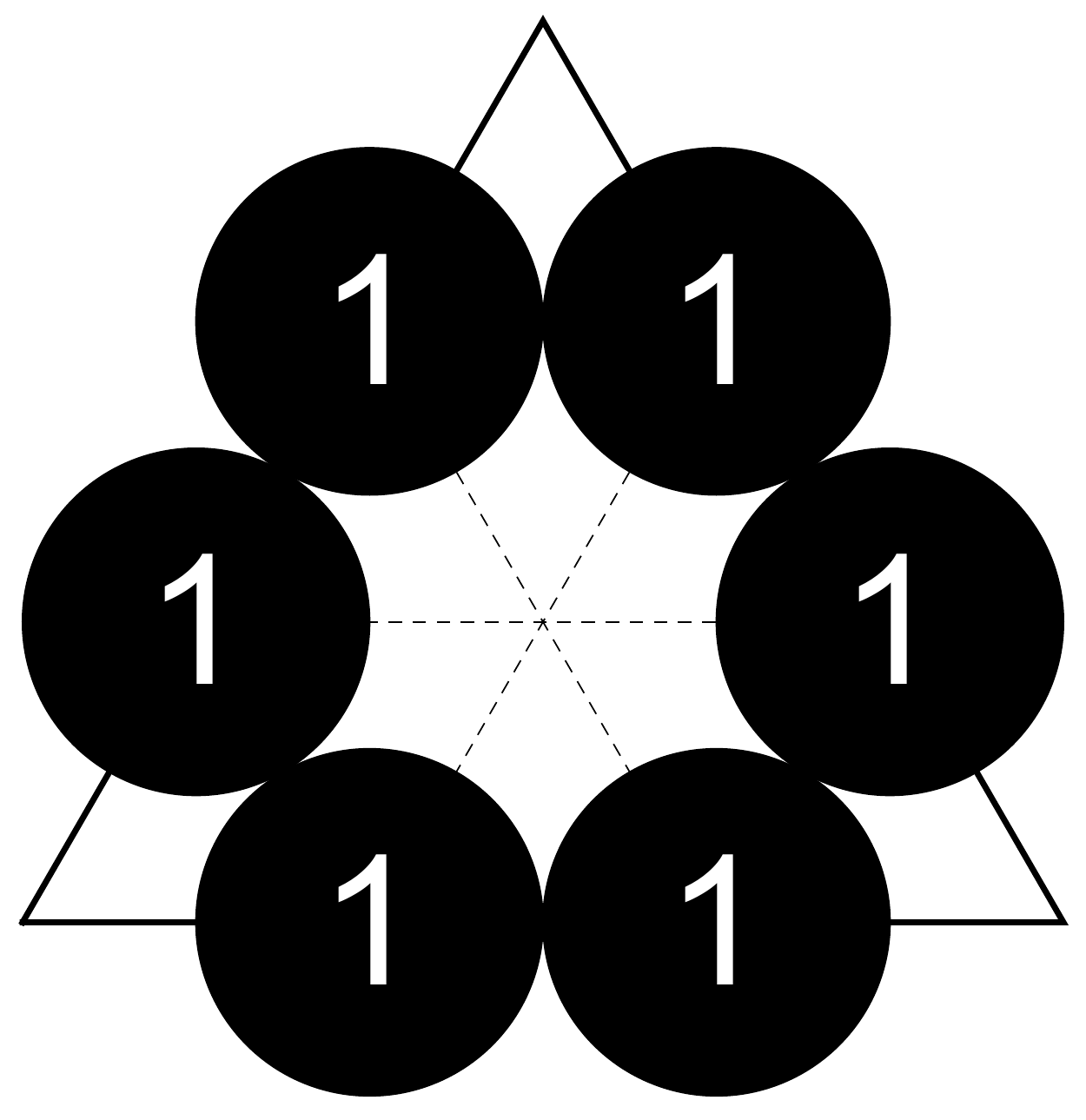}}
\caption{Sequences of knots for an alternative set of simplex spline basis functions for $\spS_3^2(\Delta_{\WS})$. Each black disc shows the position of a knot and the number inside indicates its multiplicity.}
\label{fig:allknots2}
\end{figure}

\subsection{An alternative simplex spline basis}\label{sec:simplex-basis-2}
An alternative basis for the space $\spS_3^2(\Delta_{\WS})$ is provided by the simplex splines identified by the knot sequences in Figure~\ref{fig:allknots2} and scaled to form a partition of unity. We denote this basis by
\begin{equation}\label{eq:basis-2}
\cBtilde=\{\Btilde_1,\ldots,\Btilde_{28}\}.
\end{equation}
It can be checked that
\begin{equation}\label{eq:alternative-basis}
\begin{alignedat}{2}
 \Btilde_i&= B_i, \quad i=1,\ldots, 21,  \\
\Btilde_{22}&=2B_{22}-B_{23}+\frac{1}{3}B_{28}, &\quad
\Btilde_{23}&=2B_{23}-B_{22}+\frac{1}{3}B_{28},
\\
\Btilde_{24}&=2B_{24}-B_{25}+\frac{1}{3}B_{28},  &
\Btilde_{25}&=2B_{25}-B_{24}+\frac{1}{3}B_{28},
\\
\Btilde_{26}&=2B_{26}-B_{27}+\frac{1}{3}B_{28},  &
\Btilde_{27}&=2B_{27}-B_{26}+\frac{1}{3}B_{28},
\\
\Btilde_{28}&=-B_{28}.
\end{alignedat}
\end{equation}
From \eqref{eq:alternative-basis} we easily deduce that the domain points $\vgrevilletilde_i$ associated with the set of functions in \eqref{eq:basis-2} are
\begin{equation}\label{eq:domain-points-bary-2}
\begin{alignedat}{3}
 \vgrevilletilde_i&=\vgreville_i,\quad i=1,\ldots,21,\\
 \vgrevilletilde_{22}&=\frac{2}{3}\vgreville_{22}+\frac{1}{3}\vgreville_{23} : \left(\frac{15}{27},\frac{7}{27},\frac{5}{27}\right), &\quad
 \vgrevilletilde_{23}&=\frac{2}{3}\vgreville_{23}+\frac{1}{3}\vgreville_{22} : \left(\frac{15}{27},\frac{5}{27},\frac{7}{27}\right), \\
 \vgrevilletilde_{24}&=\frac{2}{3}\vgreville_{24}+\frac{1}{3}\vgreville_{25} : \left(\frac{5}{27},\frac{15}{27},\frac{7}{27}\right), &
 \vgrevilletilde_{25}&=\frac{2}{3}\vgreville_{25}+\frac{1}{3}\vgreville_{24} : \left(\frac{7}{27},\frac{15}{27},\frac{5}{27}\right), \\
 \vgrevilletilde_{26}&=\frac{2}{3}\vgreville_{26}+\frac{1}{3}\vgreville_{27} : \left(\frac{7}{27},\frac{5}{27},\frac{15}{27}\right), &
 \vgrevilletilde_{27}&=\frac{2}{3}\vgreville_{27}+\frac{1}{3}\vgreville_{26} : \left(\frac{5}{27},\frac{7}{27},\frac{15}{27}\right),\\
 \vgrevilletilde_{28}&=\frac{1}{3}(\vgreville_{22}+\cdots+\vgreville_{27})-\vgreville_{28} : \left(\frac{1}{3},\frac{1}{3},\frac{1}{3}\right).
\end{alignedat}
\end{equation}
These points are depicted in Figure~\ref{fig:domainpoints} (right).
Note that Theorem~\ref{thm:basis} in combination with \eqref{eq:alternative-basis} confirms that the functions $\Btilde_1,\ldots,\Btilde_{28}$ are linearly independent and form a partition of unity. On the other hand, they are not all nonnegative. For subsequent use, some Hermite data of these basis functions are collected in the appendix (Tables~\ref{tab:hermiteB} and \ref{tab:hermiteBtilde-extra}).

Any spline $\spline\in \spS_3^2(\Delta_{\WS})$ can be represented in terms of this alternative basis, so
\begin{equation}
\label{eq:double-rep}
\spline=\sum_{i=1}^{28}b_iB_i=\sum_{i=1}^{28}\btilde_i\Btilde_i.
\end{equation}
Note that $\vbtilde={\vC}\vb$,
where $\vC$ is the conversion matrix already used to obtain \eqref{eq:domain-points-bary-2}. It can be easily checked that
$\|\vC\|_{\infty}=\|\vC^{-1}\|_{\infty}=3$.
Therefore, from \eqref{eq:local-stability} and \eqref{eq:double-rep} we immediately deduce
\begin{equation}
\label{eq:local-stability-2}
\frac{1}{111}\|\vbtilde\|_\infty\leq\biggl\|\sum_{i=1}^{28}\btilde_i\Btilde_i\biggr\|_\infty \leq 3\|\vbtilde\|_\infty,
\end{equation}
and the condition number can be bounded as $\kappa_\infty(\cBtilde)<333$.
We can also formulate the analogue of Proposition~\ref{prop:distance-cp} as well as a Marsden-like identity for the basis \eqref{eq:basis-2}. We omit the details for the sake of brevity.

The simplex spline basis \eqref{eq:basis} forms a convex partition of unity and so it is particularly useful for geometric modeling. On the other hand, as we will show in Section~\ref{sec:smootheness-cond}, the simplex spline basis \eqref{eq:basis-2} is more suited to handle $C^2$ smoothness conditions between spline functions on adjacent macro-triangles.
Of course, there are many more alternative sets of simplex spline basis functions. One could, for instance, take the 10 cubic Bernstein polynomials defined on $\Delta$ (they are special simplex splines; see \cite{Prautsch.Boehm.Paluszny02}) and enrich them with 18 more simplex splines that are linearly independent.

\section{$C^2$ cubic splines on the $\WS$ refinement of a triangulation}
\label{sec:spline-space}
In the previous section, we have provided simplex spline bases for
the spline space $\spS_3^2(\Delta_{\WS})$ of cubic $C^2$ splines on the $\WS$ split of a given triangle $\Delta$.
Let $\cT$ be a triangulation of a polygonal domain $\Omega$ and let $\cT_{\WS}$ denote its refinement obtained by taking the $\WS$ split of each of its triangles. In this section, we consider the spline space of $C^2$ cubic splines on $\cT_{\WS}$, i.e.,
$$
\spS_3^2(\cT_{\WS}):=\{\spline\in C^2(\Omega), \ \spline_{|\tau}\in \spP_3, \ \tau \text{ is polygon in } \cT_{\WS} \}.
$$
The unisolvency of the Hermite interpolation problem stated in Corollary~\ref{cor:hermite}
implies that the dimension of the space only depends on combinatorial properties of the triangulation, and so it is stable. From the corollary we directly deduce that (see also \cite{Wang.90})
\begin{equation}
\label{eq:dim-space}
\dim(\spS_3^2(\cT_{\WS}))=6n_V+3n_E+n_T,
\end{equation}
where $n_V$, $n_E$, $n_T$ are the number of vertices, edges, and triangles of $\cT$, respectively.
Moreover, any spline function of $\spS_3^2(\cT_{\WS})$ can be locally constructed on each (macro-)triangle $\Delta$ of $\cT$ via the Hermite data, and the corresponding spline piece on $\Delta$ can be represented in the form \eqref{eq:double-rep}.
Conversely, any function, which is represented locally in the form \eqref{eq:double-rep} on each $\Delta$ of $\cT$, is $C^2$ smooth over each $\Delta$ of $\cT$ but not necessary $C^2$ smooth across the edges of $\cal T$. First we derive conditions on the local spline coefficients to ensure global $C^2$ smoothness, and then we describe a stable global basis with local support for $\spS_3^2(\cT_{\WS})$.

\subsection{Smoothness conditions}
\label{sec:smootheness-cond}

Let $\cT$ be a triangulation of a polygonal domain $\Omega\subset\RR^2$.
We seek conditions on the local spline coefficients in \eqref{eq:double-rep} to guarantee $C^r$ smoothness across a common edge of two adjacent triangles of $\cT$ for $r=0,1,2$.

\begin{theorem}\label{thm:C1}
Suppose the triangles $\Delta^L:=\langle\vp_1,\vp_2,\vp_3\rangle$ and $\Delta^R:=\langle\vp_1,\vp_2,\vp_4\rangle$ share the common edge with vertices $\vp_1,\vp_2$, and let
\begin{equation}\label{eq:V4}
\vp_4=\eta_1\vp_1+\eta_2\vp_2+\eta_3\vp_3, \quad
\eta_1+\eta_2+\eta_3=1.
\end{equation}
Let
$\{B^L_i, \ i=1,\ldots,28\}$ and $\{B^R_i, \ i=1,\ldots,28\}$
be the scaled simplex spline basis defined by the knot sequences in Figure~\ref{fig:allknots} on $\Delta^L$ and $\Delta^R$, respectively. We assume the numbering of the basis functions in agreement with Figure~\ref{fig:allknots}.
Let us consider the spline functions
$$
\spline^L:=\sum_{i=1}^{28} b^L_iB^L_i, \quad \spline^R:=\sum_{i=1}^{28} b^R_iB^R_i.
$$
We have
\begin{itemize}
\item $\spline^L, \spline^R$ join $C^0$ across the common edge if and only if
\begin{equation}
\label{eq:C0}
b^R_i=b^L_i, \quad i=1,2,4,7,10,13;
\end{equation}
\item $\spline^L, \spline^R$ join $C^1$ across the common edge if and only if they join $C^0$ and in addition
\begin{equation}\label{eq:C1}
\begin{aligned}
b^R_5&=\eta_1b^L_1+\eta_2b^L_4+\eta_3b^L_5,\\
b^R_{16}&=\left(\eta_1+\frac{\eta_2}{2}\right)b^L_4+\frac{\eta_2}{2}b^L_{10}+\eta_3b^L_{16},\\
b^R_{19}&=\left(\frac{3}{5}\eta_1+\frac{2}{5}\eta_2\right)b^L_{10}+\left(\frac{2}{5}\eta_1+\frac{3}{5}\eta_2\right)b^L_{13}+\eta_3b^L_{19},\\
b^R_{17}&=\left(\frac{\eta_1}{2}+\eta_2\right)b^L_7+\frac{\eta_1}{2}b^L_{13}+\eta_3b^L_{17},\\
b^R_6&=\eta_1b^L_7+\eta_2b^L_2+\eta_3b^L_6.\\
\end{aligned}
\end{equation}
\end{itemize}
\end{theorem}
\begin{proof}
Let us first discuss $C^0$ smoothness. Along the common edge $\vp_1\vp_2$ the two functions $\spline^L$ and $\spline^R$ are univariate cubic $C^2$ splines with (interior) knots at the points $\vp_{3,1}$ and $\vp_{3,2}$. Considering the restriction onto the edge of the basis functions
$\{B_i^L, \ i=1,\ldots,28\}$ and $\{B_i^R, \ i=1,\ldots,28\}$, we obtain that the only nonzero elements are
$$
(B_{i}^R)_{|\vp_1\vp_2}=(B_i^L)_{|\vp_1\vp_2}, \quad i=1,2,4,7,10,13.
$$
Since they are linearly independent, $C^0$ smoothness is equivalent to agreement of the corresponding coefficients. This proves \eqref{eq:C0}.

We now discuss $C^1$ smoothness across the common edge.
It suffices to prove that along the edge $\vp_1\vp_2$ the functions
$D_{\vq_3\vp_3}\spline^L$ and $D_{\vq_3\vp_3}\spline^R$
agree. These functions are univariate $C^1$ quadratic splines with (interior) knots at the points $\vp_{3,1}$ and $\vp_{3,2}$. Therefore, each of them is uniquely determined by its value and first derivative at the two endpoints of the edge and by the value at the midpoint $\vq_3$.
From \eqref{eq:V4} we obtain
$$
\vp_3=\frac{1}{\eta_3}(\vp_4-\eta_1\vp_1-\eta_2\vp_2),
$$
and so
$$
\vq_3\vp_3=\vp_1\vp_3-\frac{1}{2}\vp_1\vp_2=\frac{1}{\eta_3}\vp_1\vp_4-\frac{\eta_3+2\eta_2}{2\eta_3}\vp_1\vp_2.
$$
Then, by employing the $C^0$ smoothness conditions and the values in Table~\ref{tab:hermiteB} we get
\begin{align*}
D_{\vq_3\vp_3}\spline^L(\vp_1) &= -9b^L_1+9b^L_5-\frac{1}{2}(-9b^L_1+9b^L_4), \\
D_{\vq_3\vp_3}\spline^R(\vp_1) &= \frac{1}{\eta_3}(-9b^L_1+9b^R_5)-\frac{\eta_3+2\eta_2}{2\eta_3}(-9b^L_1+9b^L_4).
\end{align*}
Equating the above expressions results in the first condition of \eqref{eq:C1}. With the same line of arguments we deduce the remaining four conditions.
\end{proof}

From the relations in \eqref{eq:alternative-basis} it is clear that the conditions in \eqref{eq:C0} and \eqref{eq:C1} also ensure $C^0$ and $C^1$ smoothness, respectively, for local spline representations in the alternative basis \eqref{eq:basis-2}.

\begin{corollary}\label{cor:C1-geom}
Consider the same assumptions as in Theorem~\ref{thm:C1}. 
The $C^0$ smoothness conditions for the control points can be written as
\begin{equation*}
(\vgrevilleR_i,b^R_i)=(\vgrevilleL_i,b^L_i), \quad i=1,2,4,7,10,13;
\end{equation*}
and the $C^1$ smoothness conditions for the control points can be written as
\begin{equation*}
\begin{aligned}
(\vgrevilleR_5,b^R_5)&=\eta_1(\vgrevilleL_1,b^L_1)+\eta_2(\vgrevilleL_4,b^L_4)+\eta_3(\vgrevilleL_5,b^L_5),\\
(\vgrevilleR_{16},b^R_{16})&=\left(\eta_1+\frac{\eta_2}{2}\right)(\vgrevilleL_4,b^L_4)+\frac{\eta_2}{2}(\vgrevilleL_{10},b^L_{10})+\eta_3(\vgrevilleL_{16},b^L_{16}),\\
(\vgrevilleR_{19},b^R_{19})&=\left(\frac{3}{5}\eta_1+\frac{2}{5}\eta_2\right)(\vgrevilleL_{10},b^L_{10})+\left(\frac{2}{5}\eta_1+\frac{3}{5}\eta_2\right)(\vgrevilleL_{13},b^L_{13})
+\eta_3(\vgrevilleL_{19},b^L_{19}),\\
(\vgrevilleR_{17},b^R_{17})&=\left(\frac{\eta_1}{2}+\eta_2\right)(\vgrevilleL_7,b^L_7)+\frac{\eta_1}{2}(\vgrevilleL_{13},b^L_{13})+\eta_3(\vgrevilleL_{17},b^L_{17}),\\
(\vgrevilleR_6,b^R_6)&=\eta_1(\vgrevilleL_7,b^L_7)+\eta_2(\vgrevilleL_2,b^L_2)+\eta_3(\vgrevilleL_6,b^L_6).\\
\end{aligned}
\end{equation*}
\end{corollary}
\begin{proof}
The statements follow by direct computation from \eqref{eq:C1} and from the expressions of the domain points in \eqref{eq:domain-points-bary}.
\end{proof}

The $C^1$ smoothness conditions in Corollary~\ref{cor:C1-geom} have a nice geometric interpretation. There are five sets of four control points that need to be coplanar.
In terms of our control net configuration in Figure~\ref{fig:domainpoints} (left), that means that the five triangles in both control nets along the common edge must be all pairwise coplanar. This is illustrated in Figure~\ref{fig:C1smoothnes}.

\begin{figure}[t!]
\centering
{\includegraphics[trim=40 50 40 40,clip,width=7cm]{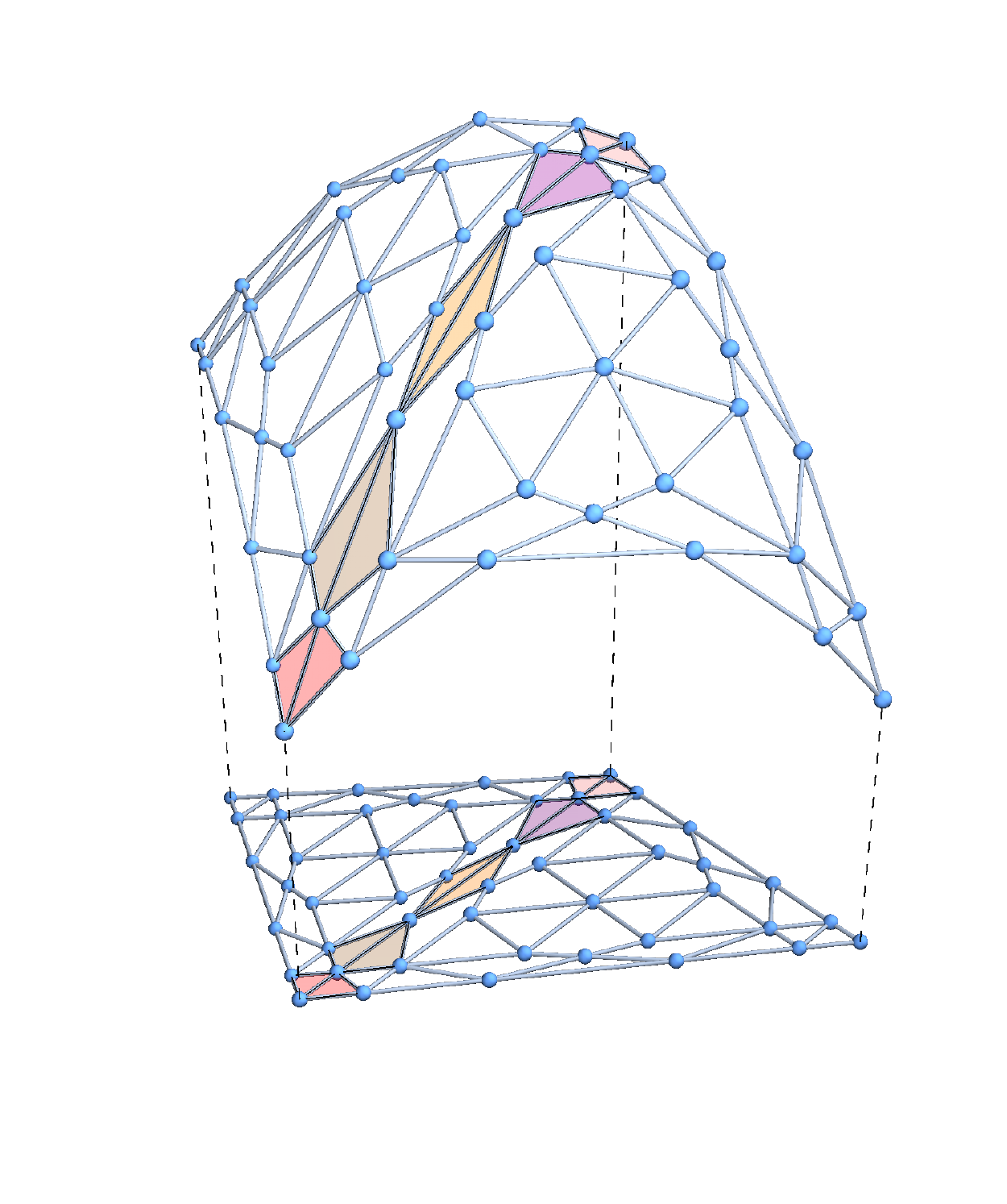}}\qquad
{\includegraphics[trim=40 50 40 40,clip,width=7cm]{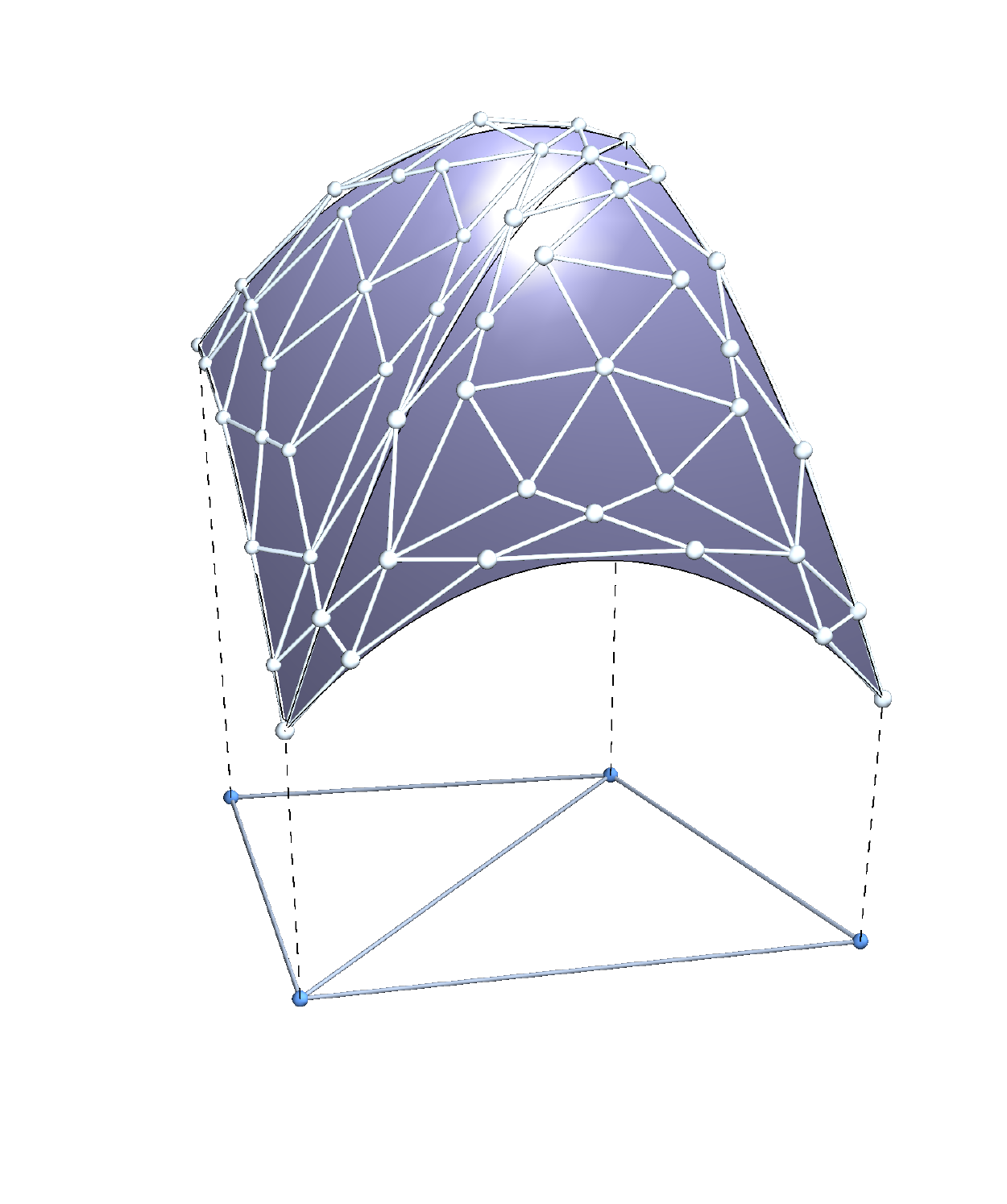}}
\caption{A $C^1$ spline surface on two adjacent domain triangles. The five pairs of triangles in the control nets that must be coplanar according to the smoothness conditions are colored.}
\label{fig:C1smoothnes}
\end{figure}

\begin{theorem}\label{thm:C2}
Consider the same assumptions as in Theorem~\ref{thm:C1}.
Let
$\{\Btilde^L_i, \ i=1,\ldots,28\}$ and $\{\Btilde^R_i, \ i=1,\ldots,28\}$
be the spline bases defined by \eqref{eq:alternative-basis} on $\Delta^L$ and $\Delta^R$, respectively.
Then, the spline functions
$$
\spline^L:=\sum_{i=1}^{28} b^L_iB^L_i = \sum_{i=1}^{21} b^L_i B^L_i+\sum_{i=22}^{27}\btilde^L_i \Btilde^L_i+ \btilde^L_{28}B^L_{28}
$$
and
$$
\spline^R:=\sum_{i=1}^{28} b^R_iB^R_i = \sum_{i=1}^{21} b^R_i B^R_i+\sum_{i=22}^{27}\btilde^R_i \Btilde^R_i+ \btilde^R_{28}B^R_{28}
$$
join $C^2$ across the common edge if and only if they join $C^1$ and in addition
\begin{equation}\label{eq:C2}
\begin{aligned}
b^R_{11} &= \eta_1(\eta_1-\eta_2-\eta_3)b^L_1+\eta_2(3\eta_1-\eta_3)b^L_4+\eta_3(3\eta_1-\eta_2)b^L_5+\eta_2^2b^L_{10}+\eta_3^2b^L_{11}+4\eta_2\eta_3b^L_{16},
\\
b^R_{12} &= \eta_2(\eta_2-\eta_1-\eta_3)b^L_2+\eta_1(3\eta_2-\eta_3)b^L_7+\eta_3(3\eta_2-\eta_1)b^L_6+\eta_1^2b^L_{13}+\eta_3^2b^L_{12}+4\eta_1\eta_3b^L_{17},
\\
\btilde^R_{22} &= \frac{1}{6}(\eta_1-\eta_3)(2\eta_1+\eta_2)b^L_4+
\left({\frac {5}{18}}\eta_{2}+{\frac {7}{18}}{\eta_{{2}}}^{2}+\frac{2}{3}
\eta_{{1}}+\frac{2}{3}\eta_{{2}}\eta_{{1}}\right)b^L_{10}\\
&\quad +\frac{1}{9} \left( 2\,\eta_{{1}}+3\,\eta_{{2}} \right)  \left( \eta_{2}-2\eta_{3} \right)b^L_{13} +\frac{1}{3}\eta_3(3\eta_1+\eta_2)b^L_{16}+\frac{10}{9}\eta_3(2\eta_2+\eta_1)b^L_{19}+\eta_3^2\btilde^L_{22},
\\
\btilde^R_{25}&=\frac{1}{6}(\eta_2-\eta_3)(2\eta_2+\eta_1)b^L_7+
\left({\frac {5}{18}}\,\eta_{{1}}+{\frac {7}{18}}\,{\eta_{{1}}}^{2}+\frac{2}{3}\,
\eta_{{2}}+\frac{2}{3}\eta_{{2}}\eta_{{1}}\right)b^L_{13} \\
&\quad +\frac{1}{9} \left( 3\eta_{1}+2\eta_{2} \right)  \left( \eta_{{1}}-2\eta_{3}\right)b^L_{10}+\frac{1}{3}\eta_3(3\eta_2+\eta_1)b^L_{17}+\frac{10}{9}\eta_3(2\eta_1+\eta_2)b^L_{19}+\eta_3^2\btilde^L_{25}.
\end{aligned}
\end{equation}
\end{theorem}
\begin{proof}
Assume $\spline^L$ and $\spline^R$ join $C^1$ across the common edge $\vp_1\vp_2$. To prove $C^2$ smoothness across the same edge, it suffices to prove that along $\vp_1\vp_2$ the functions
$
D_{\vp_1\vp_3}^2\spline^L $ and $D_{\vp_1\vp_3}^2\spline^R
$
agree. Along the edge $\vp_1\vp_2$, these functions are univariate $C^0$ linear splines with (interior) knots at the points $\vp_{3,1}$ and $\vp_{3,2}$. Therefore, each of them is uniquely determined by its value at the two endpoints of the edge and by the value at the points $\vp_{3,1}$ and $\vp_{3,2}$.
From \eqref{eq:V4} we know that
$$
\vp_1\vp_3=\frac{1}{\eta_3}\vp_1\vp_4-\frac{\eta_2}{\eta_3}\vp_1\vp_2,
$$
so that
$$
D_{\vp_1\vp_3}^2=\left(\frac{1}{\eta_3}D_{\vp_1\vp_4}-\frac{\eta_2}{\eta_3}D_{\vp_1\vp_2}\right)^2.
$$
Then, by employing the values in Table~\ref{tab:hermiteB} we get
\begin{align*}
D_{\vp_1\vp_3}^2\spline^L(\vp_1) &= 54b^L_1-81b^L_5+27b^L_{11}, \\
D_{\vp_1\vp_3}^2\spline^R(\vp_1) &= \frac{1}{\eta_3^2}(54b^R_1-81b^R_5+27b^R_{11})-\frac{2\eta_2}{\eta_3^2}(54b^R_1-54b^R_4-54b^R_5+54b^R_{16}) \\
&\quad +\frac{\eta_2^2}{\eta_3^2}(54b^R_1-81b^R_5+27b^R_{10}).
\end{align*}
By equating the above expressions and taking into account the $C^1$ smoothness conditions, we obtain the first condition of \eqref{eq:C2}. 
With the same line of arguments, taking into account \eqref{eq:alternative-basis} and the additional values in Table~\ref{tab:hermiteBtilde-extra}, we deduce the remaining three conditions.
\end{proof}

\begin{corollary}\label{cor:C2-geom}
Consider the same assumptions as in Theorem~\ref{thm:C2}.
The $C^2$ smoothness conditions for the control points can be written
as
\begin{equation*}
\begin{aligned}
(\vgrevilleR_{11},b^R_{11}) &= \eta_1(\eta_1-\eta_2-\eta_3)(\vgrevilleL_1,b^L_1)+\eta_2(3\eta_1-\eta_3)(\vgrevilleL_4,b^L_4)\\
&\quad +\eta_3(3\eta_1-\eta_2)(\vgrevilleL_5,b^L_5)+\eta_2^2(\vgrevilleL_{10},b^L_{10})+\eta_3^2(\vgrevilleL_{11},b^L_{11})+4\eta_2\eta_3(\vgrevilleL_{16},b^L_{16}),
\\
(\vgrevilleR_{12},b^R_{12}) &= \eta_2(\eta_2-\eta_1-\eta_3)(\vgrevilleL_2,b^L_2)+\eta_3(3\eta_2-\eta_1)(\vgrevilleL_6,b^L_6)\\
&\quad +\eta_1(3\eta_2-\eta_3)(\vgrevilleL_7,b^L_7)+\eta_1^2(\vgrevilleL_{13},b^L_{13})+\eta_3^2(\vgrevilleL_{12},b^L_{12})+4\eta_1\eta_3(\vgrevilleL_{17},b^L_{17}),
\\
(\vgrevilletildeR_{22},\btilde^R_{22}) &= \frac{1}{6}(\eta_1-\eta_3)(2\eta_1+\eta_2)(\vgrevilleL_4,b^L_4)+\left({\frac{5}{18}}\eta_{2}+{\frac{7}{18}}{\eta_{2}}^{2}+\frac{2}{3}
\eta_{1}+\frac{2}{3}\eta_{2}\eta_{1}\right)(\vgrevilleL_{10},b^L_{10})\\
&\quad +\frac{1}{9} \left( 2\eta_{1}+3\eta_{2} \right) \left( \eta_{2}-2\eta_{3} \right)(\vgrevilleL_{13},b^L_{13})+\frac{1}{3}\eta_3(3\eta_1+\eta_2)(\vgrevilleL_{16},b^L_{16})\\
&\quad +\frac{10}{9}\eta_3(2\eta_2+\eta_1)(\vgrevilleL_{19},b^L_{19})+\eta_3^2(\vgrevilletildeL_{22},\btilde^L_{22}),
\\
(\vgrevilletildeR_{25},\btilde^R_{25}) &= \frac{1}{6}(\eta_2-\eta_3)(2\eta_2+\eta_1)(\vgrevilleL_7,b^L_7)+\left({\frac{5}{18}}\eta_{{1}}+{\frac{7}{18}}{\eta_{1}}^{2}+\frac{2}{3}\eta_{2}+\frac{2}{3}\eta_{2}\eta_{1}\right)(\vgrevilleL_{13},b^L_{13})\\
&\quad +\frac{1}{9} \left( 3\eta_{1}+2\eta_{2} \right)\left( \eta_{1}-2\eta_{3}\right)(\vgrevilleL_{10},b^L_{10})+\frac{1}{3}\eta_3(3\eta_2+\eta_1)(\vgrevilleL_{17},b^L_{17})\\
&\quad +\frac{10}{9}\eta_3(2\eta_1+\eta_2)(\vgrevilleL_{19},b^L_{19})+\eta_3^2(\vgrevilletildeL_{25},\btilde^L_{25}).
\end{aligned}
\end{equation*}
\end{corollary}
\begin{proof}
The statements follow by direct computation from \eqref{eq:C2} and from the expressions of the domain points in \eqref{eq:domain-points-bary} and \eqref{eq:domain-points-bary-2}.
\end{proof}

Using the relations between the domain points $\vgreville_i$ and $\vgrevilletilde_i$ in \eqref{eq:domain-points-bary-2}, we can immediately rewrite the $C^2$ smoothness conditions solely in terms of the control points of the basis \eqref{eq:basis}. For instance, the third condition in Corollary~\ref{cor:C2-geom} reads as
\begin{align*}
\frac{2}{3}(\vgrevilleR_{22},b^R_{22})&+\frac{1}{3}(\vgrevilleR_{23},b^R_{23}) \\
&= \frac{1}{6}(\eta_1-\eta_3)(2\eta_1+\eta_2)(\vgrevilleL_4,b^L_4)+\left({\frac{5}{18}}\eta_{2}+{\frac{7}{18}}{\eta_{2}}^{2}+\frac{2}{3}
\eta_{1}+\frac{2}{3}\eta_{2}\eta_{1}\right)(\vgrevilleL_{10},b^L_{10})\\
&\quad +\frac{1}{9} \left( 2\eta_{1}+3\eta_{2} \right)  \left( \eta_{2}-2\eta_{3} \right)(\vgrevilleL_{13},b^L_{13})+\frac{1}{3}\eta_3(3\eta_1+\eta_2)(\vgrevilleL_{16},b^L_{16})\\
&\quad +\frac{10}{9}\eta_3(2\eta_2+\eta_1)(\vgrevilleL_{19},b^L_{19})+\eta_3^2\left(\frac{2}{3}(\vgrevilleL_{22},b^L_{22})+\frac{1}{3}(\vgrevilleL_{23},b^L_{23})\right),
\end{align*}
and the fourth condition as
\begin{align*}
\frac{2}{3}(\vgrevilleR_{25},b^R_{25})&+\frac{1}{3}(\vgrevilleR_{24},b^R_{24}) \\
&= \frac{1}{6}(\eta_2-\eta_3)(2\eta_2+\eta_1)(\vgrevilleL_7,b^L_7)+\left({\frac{5}{18}}\eta_{1}+{\frac{7}{18}}{\eta_{1}}^{2}+\frac{2}{3}\eta_{2}+\frac{2}{3}\eta_{2}\eta_{1}\right)(\vgrevilleL_{13},b^L_{13})\\
&\quad +\frac{1}{9} \left( 3\eta_{1}+2\eta_{2} \right)\left( \eta_{1}-2\eta_{3}\right)(\vgrevilleL_{10},b^L_{10})+\frac{1}{3}\eta_3(3\eta_2+\eta_1)(\vgrevilleL_{17},b^L_{17})\\
&\quad +\frac{10}{9}\eta_3(2\eta_1+\eta_2)(\vgrevilleL_{19},b^L_{19})+\eta_3^2\left(\frac{2}{3}(\vgrevilleL_{25},b^L_{25})+\frac{1}{3}(\vgrevilleL_{24},b^L_{24})\right).
\end{align*}

The smoothness conditions in Corollary~\ref{cor:C2-geom} show a structural similarity with the $C^2$ join of two adjacent triangular Bernstein--B\'ezier patches. We refer to \cite[Example~2]{Lai97} for a geometric interpretation.

\subsection{Stable bases for $\spS_3^2(\cT_{\WS})$}
In Sections~\ref{sec:simplex-basis-1} and~\ref{sec:simplex-basis-2}, we have constructed two simplex spline bases for the space of $C^2$ cubic splines on the $\WS$ split of a given triangle $\Delta$. We have also shown that these bases enjoy similar properties to the Bernstein polynomial basis defined on a triangle.
Here, we consider the space $\spS_3^2(\cT_{\WS})$ and we take the similarity between these bases one step further: we extend the concept of minimal determining sets developed for the Bernstein polynomial basis \cite{Lai.Schumaker07} to our simplex spline bases, with the aim of constructing stable bases with local support for the space $\spS_3^2(\cT_{\WS})$. For the sake of simplicity we focus on the simplex spline basis \eqref{eq:basis-2}.

To stress their dependence on a specific triangle $\Delta$, from now on we denote the simplex spline basis \eqref{eq:basis-2} by
$\{\Btilde_{i,\Delta}, \ i=1, \ldots, 28\}$.
Any spline $\spline\in\spS_3^2(\cT_{\WS})$ can be identified by its coefficients $\{\btilde_{i,\Delta}, \ i=1,\ldots, 28\}$ with respect to the above basis over any triangle $\Delta$ of $\cal T$, i.e.,
\begin{equation}
\label{eq:repr-piecewise}
\spline_{|\Delta}=\sum_{i=1}^{28}\btilde_{i,\Delta}\Btilde_{i,\Delta},
\end{equation}
where the coefficients in \eqref{eq:repr-piecewise} must satisfy the smoothness conditions derived in Section~\ref{sec:smootheness-cond} to ensure the $C^2$ joins across the edges of $\cT$.
Following \cite[Chapter~5]{Lai.Schumaker07}, for any triangle $\Delta$ of $\cT$ we denote by $\cDtilde_{\Delta}$ the set of domain points specified in \eqref{eq:domain-points-bary-2}. Then, we define a (minimal) determining set as follows.
\begin{definition}
Assume a set 
$\cDtilde \subseteq (\cup_{\Delta\in \cT}\cDtilde_{\Delta})$
is such that if $\spline\in\spS_3^2(\cT_{\WS})$ has all the coefficients corresponding to elements in $\cDtilde$ equal to zero then $s\equiv 0$. Then, $\cDtilde$ is a determining set for $\spS_3^2(\cT_{\WS})$.
A determining set is a minimal determining set in case it has the smallest possible cardinality.
\end{definition}
By using the same line of arguments as the proof of \cite[Theorem~5.13]{Lai.Schumaker07} we infer that the cardinality of a minimal determining set for $\spS_3^2(\cT_{\WS})$ equals the dimension of the space.
With the aim of specifying such a minimal determining set, let us first introduce some terminology regarding the domain points in a triangle $\Delta=\langle \vp_1, \vp_2, \vp_3\rangle$ of $\cT$: 
\begin{itemize}
\item the domain points associated with the vertex $\vp_i$ are the six points in \eqref{eq:domain-points-bary-2} with the $i$-th barycentric coordinate $\geq 2/3$, $i=1,2,3$ (for instance, the domain points associated with $\vp_1$ are $\vgrevilletilde_1, \vgrevilletilde_4, \vgrevilletilde_5, \vgrevilletilde_{10}, \vgrevilletilde_{11}, \vgrevilletilde_{16}$);
\item the domain points associated with the edge $\vp_1\vp_2$ are the three points $\vgrevilletilde_{19}, \vgrevilletilde_{22}, \vgrevilletilde_{25}$ in \eqref{eq:domain-points-bary-2};
\item the domain points associated with the edge $\vp_1\vp_3$ are the three points $\vgrevilletilde_{21}, \vgrevilletilde_{23}, \vgrevilletilde_{26}$ in \eqref{eq:domain-points-bary-2};
\item the domain points associated with the edge $\vp_2\vp_3$ are the three points $\vgrevilletilde_{20}, \vgrevilletilde_{24}, \vgrevilletilde_{27}$ in \eqref{eq:domain-points-bary-2};
\item the domain point associated with the triangle $\Delta$ is the point $\vgrevilletilde_{28}$ in \eqref{eq:domain-points-bary-2}.
\end{itemize}
We can construct a minimal determining set for $\spS_3^2(\cT_{\WS})$ as follows; see also Figure~\ref{fig:min-det-set}.

\begin{figure}[t!]
\centering
{\includegraphics[trim=0 0 0 0,clip,width=8cm]{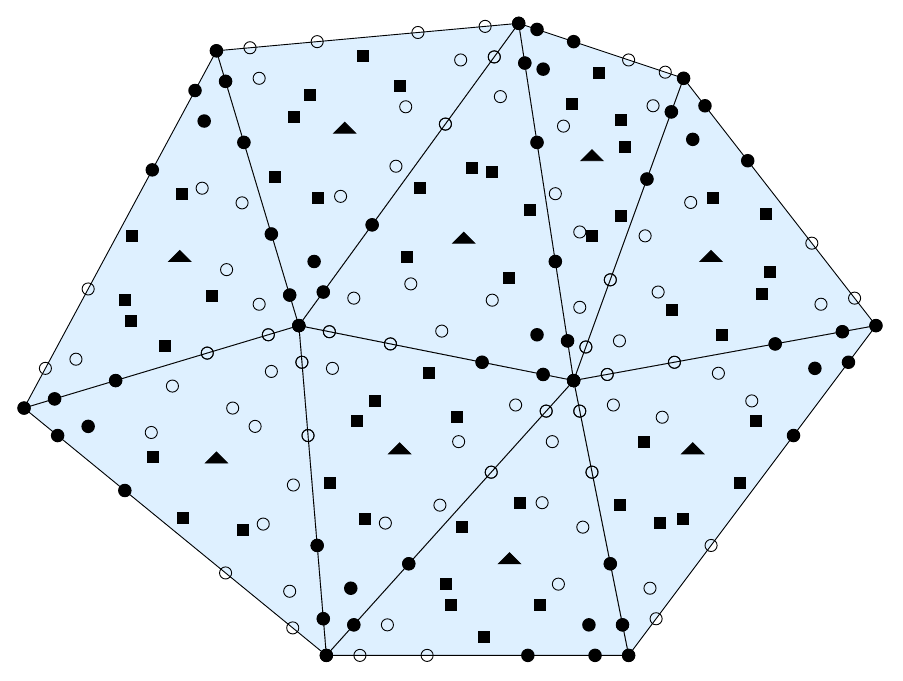}}
\caption{A minimal determining set for $\spS_3^2(\cT_{\WS})$. Filled circles: the domain points associated with a vertex. Filled squares: the domain points associated with an edge. Filled triangle: the domain points associated with a triangle. Empty circles: the domain points associated with coefficients obtained from those in the minimal determining set by imposing the smoothness conditions.}
\label{fig:min-det-set}
\end{figure}

\begin{theorem}
\label{thm:minimal-det-set}
For a given triangulation $\cal T$, let $\cMtilde$ be the set consisting of the following domain points:
\begin{itemize}
\item for each vertex $\vp$ of $\cT$, choose a triangle $\Delta$ of $\cT$ such that $\vp$ is a vertex of $\Delta$ and select the six domain points in $\Delta$ associated with $\vp$;
\item for each edge of $\cal T$, choose a triangle $\Delta$ of $\cT$ sharing this edge and select the three domain points in $\Delta$ associated with the edge;
\item for each triangle $\Delta$ of $\cT$, select the domain point associated with it.
\end{itemize}
Then, $\cMtilde$ is a minimal determining set for $\spS_3^2(\cT_{\WS})$.
\end{theorem}
\begin{proof}
Let $s\in \spS_3^2(\cT_{\WS})$.
Assume that all its coefficients associated with the domain points in $\cMtilde$ are set equal to $0$.
Let $\vp$ be any vertex of $\cT$. All the coefficients of $\spline$ corresponding to the domain points associated with $\vp$ (in any triangle of $\cT$
surrounding $\vp$) are zero because either they belong to $\cMtilde$ or they are uniquely determined by the conditions in \eqref{eq:C0}, the first two conditions in \eqref{eq:C1} and the first condition in \eqref{eq:C2} for $C^2$ smoothness across the edges emanating from $\vp$.
Let $\vp\vq$ be any edge of $\cT$. All the coefficients of $\spline$ corresponding to the domain points associated with the edge (in any of the two triangles of $\cT$ sharing the edge $\vp\vq$) are zero because either they belong to $\cMtilde$ or they are uniquely determined by the third condition in \eqref{eq:C1} and the third and fourth conditions in \eqref{eq:C2} for $C^2$ smoothness across the edge.
Finally, the domain point associated with any triangle in $\cT$ belongs to $\cMtilde$ and so the corresponding coefficient is $0$.
Hence, for any triangle $\Delta$ of $\cT$ all the coefficients in \eqref{eq:repr-piecewise} are $0$ and so $\spline\equiv 0$, i.e., $\cMtilde$ is a determining set.
Moreover, the cardinality of $\cMtilde$ clearly equals the dimension of the space, see \eqref{eq:dim-space}, and so $\cMtilde$ is a minimal determining set.
\end{proof}

Given a minimal determining set $\cMtilde$ for $\spS_3^2(\cT_{\WS})$, suppose we assign values to all the coefficients corresponding to the domain points in it. The proof of Theorem~\ref{thm:minimal-det-set} shows that these coefficients uniquely identify a spline function of $\spS_3^2(\cT_{\WS})$ because all the remaining coefficients in the representation \eqref{eq:repr-piecewise} can be deduced from the smoothness conditions.
Therefore, any minimal determining set enables us to built a basis for $\spS_3^2(\cT_{\WS})$. Let us consider the set of functions
\begin{equation}
\label{eq:basis-space-2}
\{B_{\vgrevilletildeindex,\cT},\ \vgrevilletildeindex\in \cMtilde\},
\end{equation}
where $B_{\vgrevilletildeindex,\cT}$ is the unique function of $\spS_3^2(\cT_{\WS})$ obtained by 
zeroing all the coefficients corresponding to the domain points in $\cMtilde$ except the one related to $\vgrevilletildeindex$ which is set equal to $1$. 
By construction, the functions in \eqref{eq:basis-space-2} are clearly linearly independent (see also \cite[Theorem~5.20]{Lai.Schumaker07}) and their number agrees with the dimension of the space because $\cMtilde $ is a minimal determining set. Thus, \eqref{eq:basis-space-2} is a basis for $\spS_3^2(\cT_{\WS})$.
The following proposition ensures that the elements of the above basis are uniformly bounded and have local support.
\begin{proposition}
\label{prop:loc-stab-basis}
For a given triangulation $\cal T$, let $\cMtilde$ be a minimal determining set for $\spS_3^2(\cT_{\WS})$ as specified in Theorem~\ref{thm:minimal-det-set}. 
For all $\vgrevilletildeindex\in \cMtilde$, the support of $B_{\vgrevilletildeindex,\cT}$ is contained in
\begin{enumerate}
\item [a)] the union of the triangles of $\cT$ sharing the vertex $\vp$, if $\vgrevilletildeindex$ is a domain point in $\cMtilde$ associated with the vertex $\vp$;
\item [b)] the union of the two triangles of $\cT$ sharing the edge $\vp\vq$, if $\vgrevilletildeindex$ is a domain point in $\cMtilde$ associated with the edge $\vp\vq$;
\item [c)] the triangle $\Delta$ of $\cT$, if $\vgrevilletildeindex$ is the domain point in $\cMtilde$ associated with the triangle $\Delta$.
\end{enumerate}
Moreover, there exists a constant $\Ktilde$ only depending on the minimal angle of $\cT$ such that
\begin{equation}\label{eq:bounded-basis}
\| B_{\vgrevilletildeindex,\cT}\|_{\infty}\leq \Ktilde, \quad \vgrevilletildeindex\in \cMtilde.
\end{equation}
\end{proposition}
\begin{proof}Let $\vgrevilletildeindex$ be fixed.
A direct inspection of the smoothness conditions in Theorems~\ref{thm:C1} and~\ref{thm:C2} immediately gives that the coefficients in \eqref{eq:repr-piecewise} for $B_{\vgrevilletildeindex,\cT}$ are $0$ whenever $\Delta$ is not a triangle listed in the items a)-b)-c) and so $B_{\vgrevilletildeindex,\cT}$ vanishes outside the union of those triangles.
In order to prove \eqref{eq:bounded-basis}, we first note that the number of triangles surrounding a vertex of $\cT$ and the (absolute value of the) barycentric coordinates of a point of a triangle with respect to an adjacent triangle are bounded in terms of the minimum angle of $\cT$ (see \cite[proof of Lemma~2.29]{Lai.Schumaker07}). Let $\Delta$ be a triangle of $\cT$ belonging to the support of $B_{\vgrevilletildeindex,\cT}$. We have that
$(B_{\vgrevilletildeindex,\cT})_{|\Delta}$
can be represented in the form \eqref{eq:repr-piecewise} where the coefficients $\btilde_{i,\Delta}$ are obtained from the value $1$ corresponding to the domain point $\vgrevilletildeindex \in \cMtilde$ by applying the smoothness conditions in Theorems~\ref{thm:C1} and~\ref{thm:C2}; these conditions consist of linear or quadratic relations involving barycentric coordinates of points in adjacent triangles. Therefore, denoting by $\vbtilde_{\Delta}$ the vector of these $28$ coefficients, we get
$\|\vbtilde_{\Delta}\|_\infty\leq K'$,
where $K'$ is a constant only depending on the minimum angle of $\cT$.
Hence, from \eqref{eq:local-stability-2} we get
$$
 \|(B_{\vgrevilletildeindex,\cT})_{|\Delta}\|_{\infty}\leq 3 \|\vbtilde_{\Delta}\|_\infty\leq 3K'.
$$
Taking the maximum over all the triangles of $\cT$ we arrive at \eqref{eq:bounded-basis} with $\Ktilde=3K'$.
\end{proof}

A basis with properties as in Proposition~\ref{prop:loc-stab-basis} is called a stable basis with local support. For such a basis, using the same line of arguments as the proof of \cite[Theorem~5.22]{Lai.Schumaker07}, we can show that for all $\vctilde:=(\ctilde_{\vgrevilletildeindex}\in\RR:\vgrevilletildeindex\in \cMtilde)^T$,
\begin{equation}
\label{eq:global:stable:basis}
K^{-}\| \vctilde\|_\infty \leq \biggl\|\sum_{\vgrevilletildeindex\in \cMtilde} \ctilde_{\vgrevilletildeindex}B_{\vgrevilletildeindex,\cT}\biggr\|_\infty\leq K^{+} \|\vctilde\|_\infty,
\end{equation}
where $K^{-}$ and $K^{+}$ are positive constants depending only on the smallest angle of $\cT$.
The inequalities in \eqref{eq:global:stable:basis} extend the local stability result in \eqref{eq:local-stability-2} to the full spline space $\spS_3^2(\cT_{\WS})$. Similar stability results in any $L_q$-norm for (a properly scaled version of) the basis can also be achieved; see again the proof of \cite[Theorem~5.22]{Lai.Schumaker07}.
Furthermore, from the partition of unity property of the local basis in \eqref{eq:basis-2}, we directly deduce that the global basis in \eqref{eq:basis-space-2} forms a partition of unity as well.

Note that the determining set in Theorem~\ref{thm:minimal-det-set} is a stable local determining set in the sense of \cite[Definition~5.16]{Lai.Schumaker07}. This feature is, roughly speaking, the key ingredient in the proof of Proposition~\ref{prop:loc-stab-basis}. Similarly to \cite[Section~5.7]{Lai.Schumaker07}, it also ensures that the full spline space $\spS_3^2(\cT_{\WS})$ has optimal approximation power.
More precisely, Theorems~5.18 and~5.19 in \cite{Lai.Schumaker07} hold true for $\spS_3^2(\cT_{\WS})$.

The global basis in \eqref{eq:basis-space-2} can be easily expressed over any triangle $\Delta$ with respect to the local basis \eqref{eq:basis} through the conversion \eqref{eq:alternative-basis}. Contrarily to this local basis, the functions in \eqref{eq:basis-space-2} are in general not nonnegative. However, paraphrasing \cite[Section~5.8]{Lai.Schumaker07}, we observe that the explicit basis in \eqref{eq:basis-space-2} has mainly a theoretical interest. For computation with splines belonging to $\spS_3^2(\cT_{\WS})$ it is more convenient to work directly with the local representations provided by the bases \eqref{eq:basis} or \eqref{eq:basis-2}, rather than with the basis for the full spline space.

\section{Concluding remarks}
\label{sec:conclusion}
In \cite{Cohen.Lyche.Riesenfeld13}, a simplex spline basis was described for the $C^1$ quadratic spline space on the Powell--Sabin 12 split, which is the quadratic member of the Wang--Shi split family. In this paper, we have addressed the $C^2$ cubic case and constructed two simplex spline bases for the $\WS$ split.
The characteristics of the $C^2$ cubic simplex spline bases make it unnecessary to consider separate polynomial representations on each of the numerous polygonal regions of the partitioned macro-triangle.
This paves the path for a practical construction of globally $C^2$ cubic splines on any triangulation by extending the concept of minimal determining sets. 

In the following, we outline some implementation aspects and we identify few problems where the provided simplex spline bases 
may be prosperous, in order to complement the theoretical interest of our investigation with an application-oriented perspective. We end with a discussion on a higher-order extension of the construction.
 
\subsection{Implementation aspects}
For computation with splines belonging to $\spS_3^2(\cT_{\WS})$, it is convenient to work directly with the local representations provided by the bases \eqref{eq:basis} or \eqref{eq:basis-2}.
On the one hand, evaluation of the simplex spline basis functions can be achieved by applying the recurrence relation (B-recurrence) of simplex splines; see Section~\ref{sec:summary_simplex}. On the other hand, it might be more convenient to use the explicit expressions of the simplex spline basis functions \eqref{eq:basis} given in Table~\ref{tab:polynomial-expressions} in the appendix. The alternative basis functions \eqref{eq:basis-2} can be immediately deduced from the previous ones by means of the linear relation \eqref{eq:alternative-basis}. 

Having at our disposal such a table, evaluation of any spline in $\spS_3^2(\Delta_{\WS})$ can be efficiently performed by combining a lookup-table process with a search algorithm based on boolean vectors. Given (the barycentric coordinates of) any point $\vp$ in $\Delta$, the values of the simplex spline basis functions \eqref{eq:basis} at $\vp$ can be directly obtained from Table~\ref{tab:polynomial-expressions} once we have figured out which polygonal region of the macro-triangle the evaluation point belongs to. 
Since the $\WS$ split is a cross-cut partition, any polygonal region in $\Delta_{\WS}$ is uniquely identified by the sign of the linear expressions of the 18 interior lines in the split. These signs can be interpreted as binary digits of an integer belonging to $[0, 2^{18}-1]$. Therefore, in order to detect which polygonal region of the macro-triangle a given point $\vp$ belongs to, it suffices to evaluate all the 18 interior lines at $\vp$, to collect the resulting signs in a boolean vector, and to interpret such a vector as binary digits of an integer. A similar search algorithm has been described in \cite[Algorithm~1.1]{Cohen.Lyche.Riesenfeld13}.

It is important to remark that the selection of the different polynomial pieces is just an implementation aspect.
Thanks to the characteristics of the simplex spline representation, there is a single control net to facilitate control and early visualization of a spline function over each element $\Delta$ in $\cT$. This single control net makes that the complex geometry of the $\WS$ split (consisting of $75$ polygons including triangles, quadrilaterals, and pentagons) is transparent to the user.
In this perspective, an interesting topic of possible future research is to investigate whether the control net introduced in the paper can give rise to a de Casteljau/de Boor-type algorithm for evaluation of splines in $\spS_3^2(\Delta_{\WS})$.

\subsection{Application areas} \label{sec:applications}
Splines on (refined) triangulations are valuable in several application areas.
When dealing with bivariate/multivariate problems, the straightforward approach is to rely on tensor-product structures, and in particular tensor-product splines. Tensor-product structures offer several advantages, mainly the simplicity of their use and the inheritance of univariate properties. Major drawbacks, however, are the lack of adequate local refinement and the struggle to represent geometries with complicated shapes. Although there are several appealing extensions of tensor-product splines towards local refinement (see, e.g., \cites{Dokken13,Giannelli12,Sederberg03}) and complex geometries (see, e.g., \cites{Bercovier17,Peters08,Reif97}), splines on triangulations emerge as the natural tool to efficiently deal with problems where local features has to be detected, modeled, or simulated.

As mentioned in the introduction, low degree splines are preferable due to their stable behavior and their low computational complexity. In particular, univariate $C^2$ cubic splines are one of the most used tools in modeling, approximation, and simulation. Constructing $C^2$ splines of low degree on triangulations is a difficult task, but their interest remains unquestionable in the bivariate setting. We limit ourselves to mention two important application areas: computer aided surface modeling and numerical simulation. 

In computer aided design/manufacturing (CAD/CAM) high quality free-form surfaces are of utmost importance. The quality of the surfaces can be checked by different techniques, such as the well-established isophotes \cite{Poeschl84}, to detect irregularities of intrinsic measures of surface smoothness like the Gaussian curvature or the distribution of the surface normals. For milling surfaces by five axis machines, second derivatives should not jump too much across edges and $C^2$ smoothness is desirable.
In this context, our $C^2$ cubic simplex spline representations on triangulations could be beneficial. The spline surfaces could be constructed by direct (interactive) modeling via the control net or by data fitting using quasi-interpolation schemes based on the Marsden-like identity, similar to \cite{Speleers15a}.
In CAD/CAM systems it is common to rely on general parametric surfaces; in our case such surfaces are specified on each macro-triangle by a control net consisting of triangles and quadrilaterals.
See also \cite{Zhang.17} for the use of simplex splines in the context of parametric surface reconstruction.
As a possible future work, it is of interest to investigate the interplay with
tensor-product (piecewise) bicubic parametric surfaces in Bernstein--B\'ezier (or B-spline) form, which are ubiquitous in industrial applications.
In particular, an important question is whether one can blend standard bicubic Bernstein--B\'ezier patches with parametric triangular patches whose components are $C^2$ cubic splines represented in terms of the simplex spline bases introduced in the paper.

Isogeometric analyis (IgA) is a numerical simulation paradigm that extends finite element analysis (FEA) by providing a true design-through-analysis methodology \cite{Hughesbook}. The isogeometric paradigm has some important advantages over traditional FEA. The geometry of the physical domain is exactly described, so the interaction with the CAD system during any further refinement process in the analysis phase is eliminated. Moreover, the discretization spaces possess an inherent higher smoothness (with respect to the polynomial degree) than classical FEA spaces, leading to a higher accuracy per degree of freedom \cites{Bressan19,Sande20}. The success of IgA roots in the above two properties, the latter being even more relevant. 
Besides the use of spline spaces based on (local) tensor-product structures and rather involved multipatch constructions (see, e.g., \cites{Blidia20,Chan18,Kapl19,Sangalli16,Toshniwal17}), a powerful IgA formulation has been obtained by considering spline spaces on triangulations (see, e.g., \cites{Cao19,Jaxon14,Speleers12,Speleers15c,Wang18}).
In particular, spline representations obtained from local Bernstein representations by means of minimal determining sets have been profitably applied and efficiently implemented via B\'ezier-extraction \cite{Jaxon14}.
In this context, the space of $C^2$ cubic splines defined on the $\WS$ refinement of a given triangulation is appealing because it combines low degree and high smoothness. Our simplex spline bases are the natural counterpart of Bernstein polynomials to define stable global bases by means of minimal determining sets (see Section~\ref{sec:spline-space}), and allow for a straightforward extension of the B\'ezier-extraction procedure for practical implementation.
Of course, in order to efficiently exploit the potential of the space $\spS_3^2(\cT_{\WS})$ and its local representations in terms of simplex spline bases in the context of IgA, several steps are still missing, for instance, the need for tailored quadrature rules.

\subsection{Higher-order extension of the basis}
Besides the application-oriented investigations mentioned in the previous subsection, an interesting follow-up work would be the generalization of the simplex spline construction to $C^{d-1}$ spline spaces $\spSd(\Delta_{\WSd})$ on the general $\WSd$ split of a triangle $\Delta:=\langle \vp_1,\vp_2,\vp_3\rangle$ for degree $d>3$. Under the assumption of Theorem~\ref{thm:dimension}, the dimension can be written as
\begin{align*}
\dim(\spSd(\Delta_{\WSd})) &= \frac{(d+2)(d+1)}{2}+3d(d-1) \\
&= \frac{3(d+1)d}{2}+\frac{3d(d-1)}{2}+\frac{(d-1)(d-2)}{2}.
\end{align*}
Then, similar to Corollary~\ref{cor:hermite}, we may formulate the following Hermite interpolation problem to characterize the space $\spSd(\Delta_{\WSd})$: for given data $f_{k,\alpha,\beta}$, $g_{k,\alpha,l}$, and $h_{\alpha,\beta}$, there is a unique spline $\spline\in\spSd(\Delta_{\WSd})$ such that
\begin{alignat*}{3}
D_x^\alpha D_y^\beta \spline(\vp_k) &= f_{k,\alpha,\beta}, &\quad &0\leq \alpha+\beta\leq d-1,\quad k=1,2,3, \\
D_{\vn_k}^\alpha \spline(\vq_{k,\alpha,l}) &= g_{k,\alpha,l}, &\quad &\alpha=1,\ldots,d-1, \quad l=1,\ldots,\alpha,\quad k=1,2,3, \\
D_x^\alpha D_y^\beta \spline(\vq) &= h_{\alpha,\beta}, &\quad &0\leq \alpha+\beta\leq d-3,
\end{alignat*}
where
$$
\vq:=\frac{\vp_1+\vp_2+\vp_3}{3}, \quad \vq_{k,\alpha,l}:=\frac{l\vp_{\mods{k}{3}+1}+(\alpha-l+1)\vp_{\mods{(k+1)}{3}+1}}{\alpha+1},
$$
and $\vn_k$ is the normal direction of the edge opposite to vertex $\vp_k$.
Given a general triangulation $\cT$, this scheme can be used to construct a globally $C^{d-1}$ spline of degree $d$ on $\cT$ where every triangle is refined with the $\WSd$ split. Such a construction is local, in the sense that the spline can be built on each macro-triangle $\Delta$ of $\cT$ separately, and the simplex spline basis would then be useful to represent the corresponding spline piece on $\Delta$, without considering explicitly the complicated geometry in the $\WSd$ split.

\section*{Acknowledgements}
This work was supported 
by the Beyond Borders Programme of the University of Rome Tor Vergata through the project ASTRID (CUP E84I19002250005) and 
by the MIUR Excellence Department Project awarded to the Department of Mathematics, University of Rome Tor Vergata (CUP E83C18000100006).
The authors are grateful to the Mathematisches Forschungsinstitut Oberwolfach for the Research in Pairs support (R1926).
C. Manni and H. Speleers are members of Gruppo Nazionale per il Calcolo Scientifico, Istituto Nazionale di Alta Matematica.

\appendix
\section{Appendix}

In this appendix, we collect data related to the spline basis functions $B_1,\ldots,B_{28}$ in \eqref{eq:basis} that might be useful for practical computations. We also provide few data for the alternative spline basis functions $\Btilde_1,\ldots,\Btilde_{28}$ in \eqref{eq:basis-2}.


\begin{figure}[t!]
\centering
\includegraphics[trim=0 0 0 10,width=6cm]{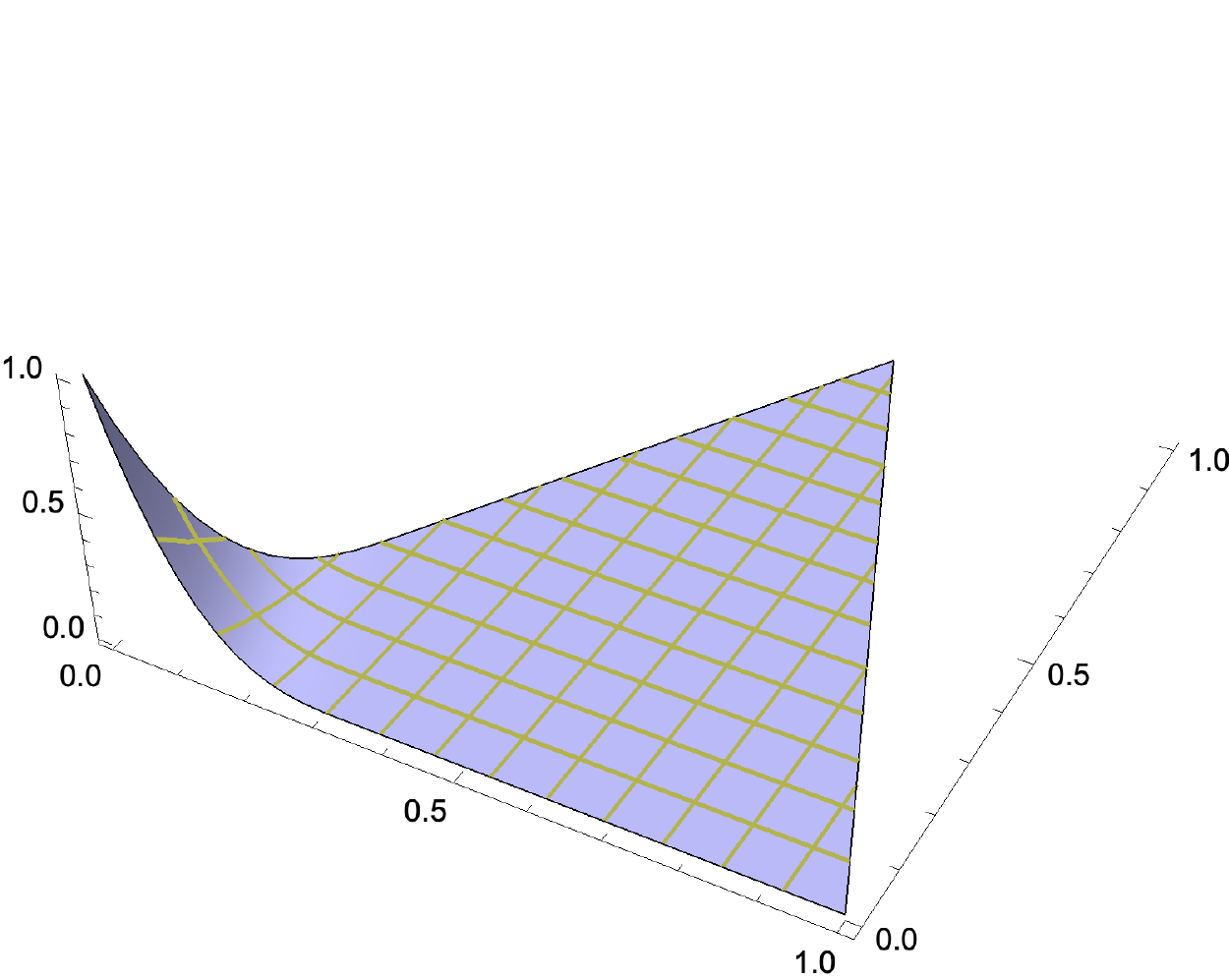}
\includegraphics[width=4cm]{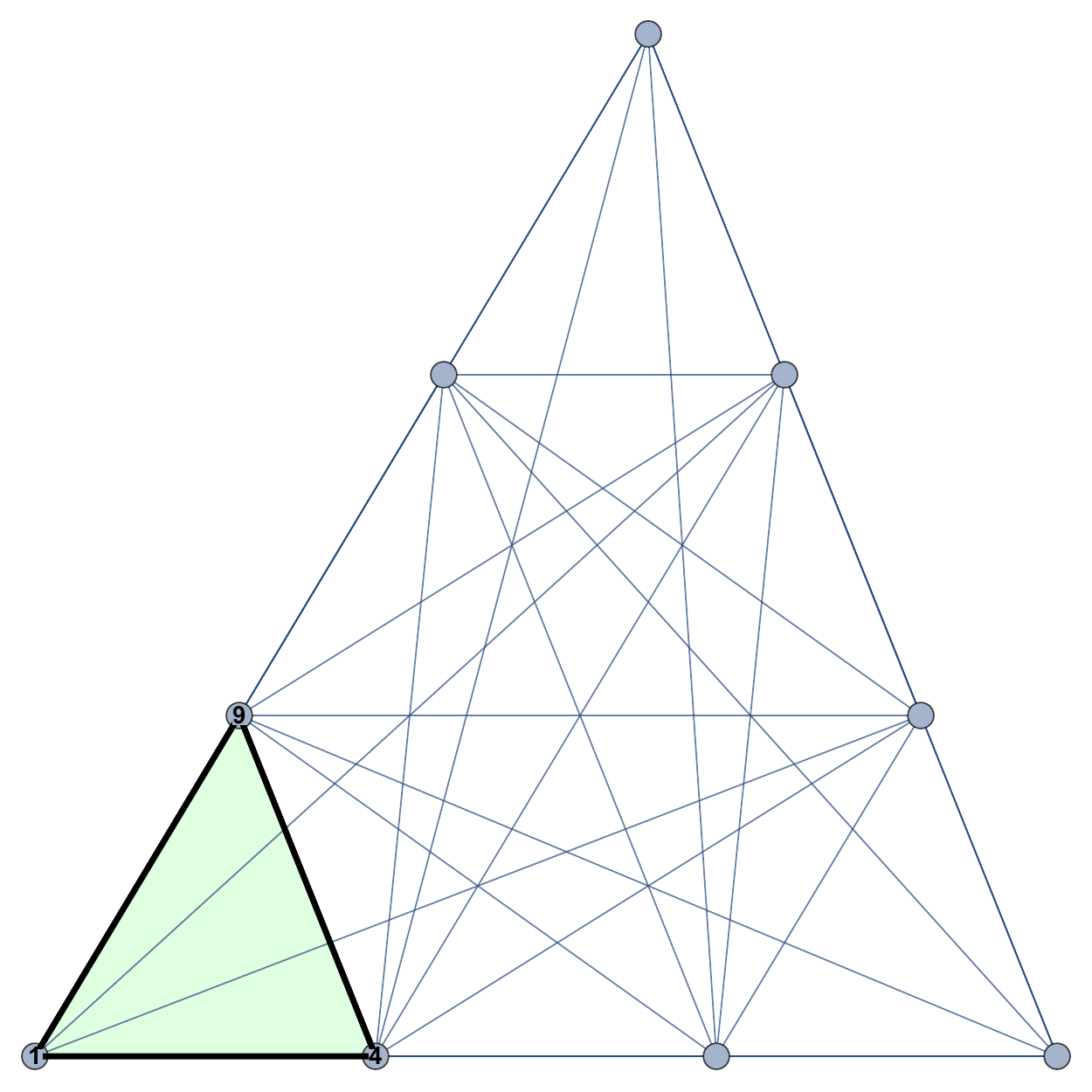}
\caption{The simplex spline basis function $B_1$ and its support.}
\label{fig:B1}
\end{figure}
\begin{figure}[t!]
\centering
\includegraphics[trim=0 0 0 10,width=6cm]{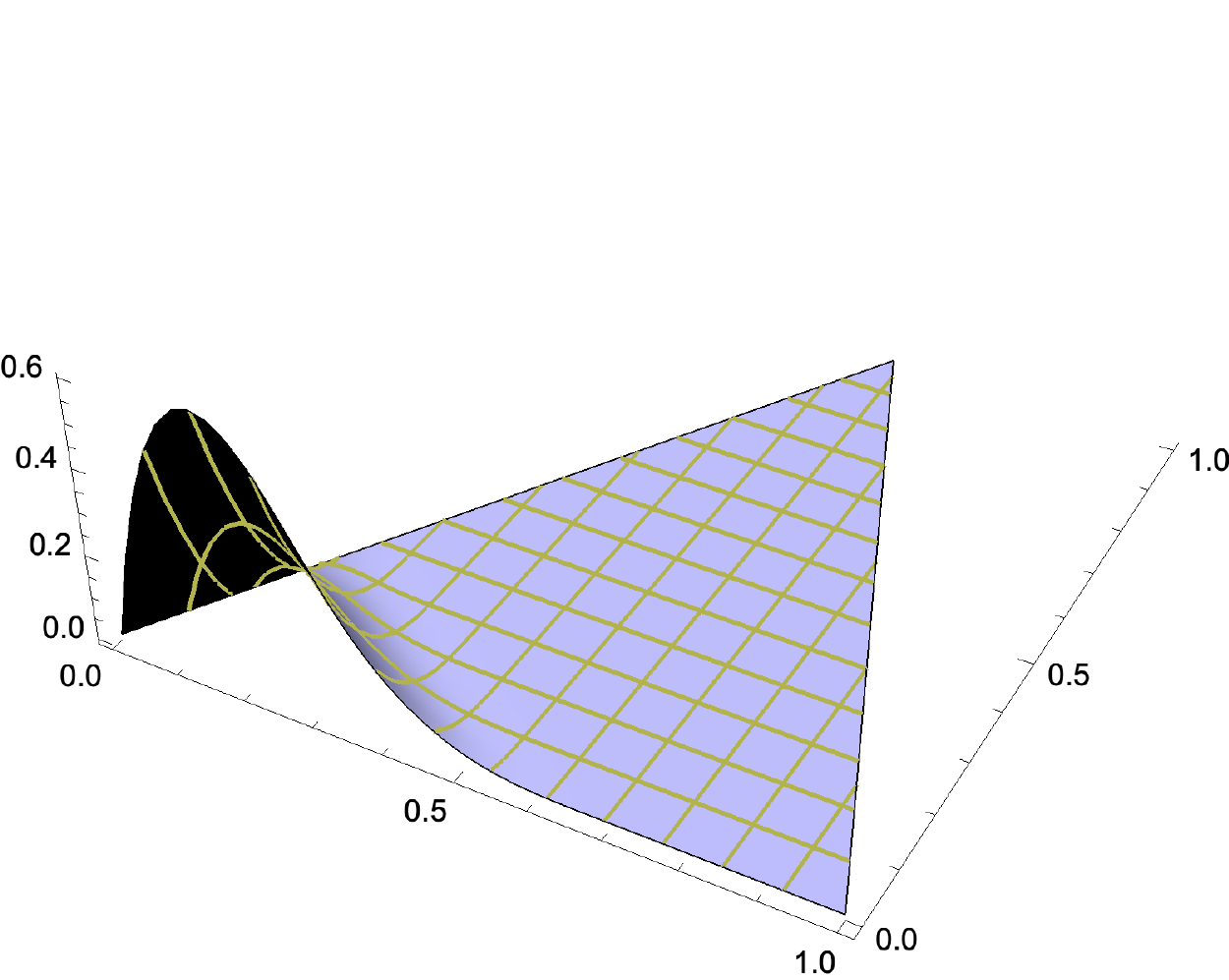}
\includegraphics[width=4cm]{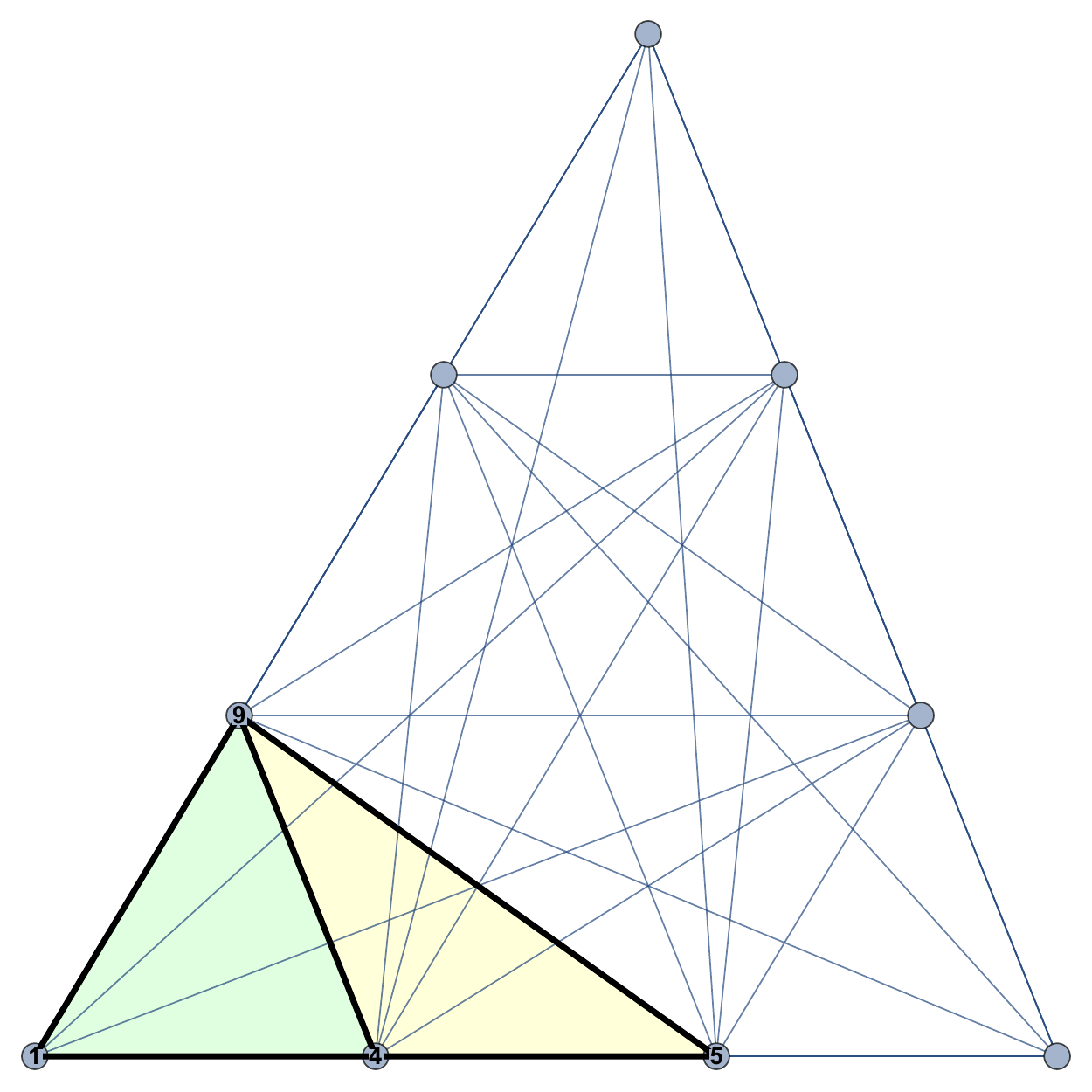}
\caption{The simplex spline basis function $B_4$ and its support.}
\label{fig:B4}
\centering
\includegraphics[trim=0 0 0 10,width=6cm]{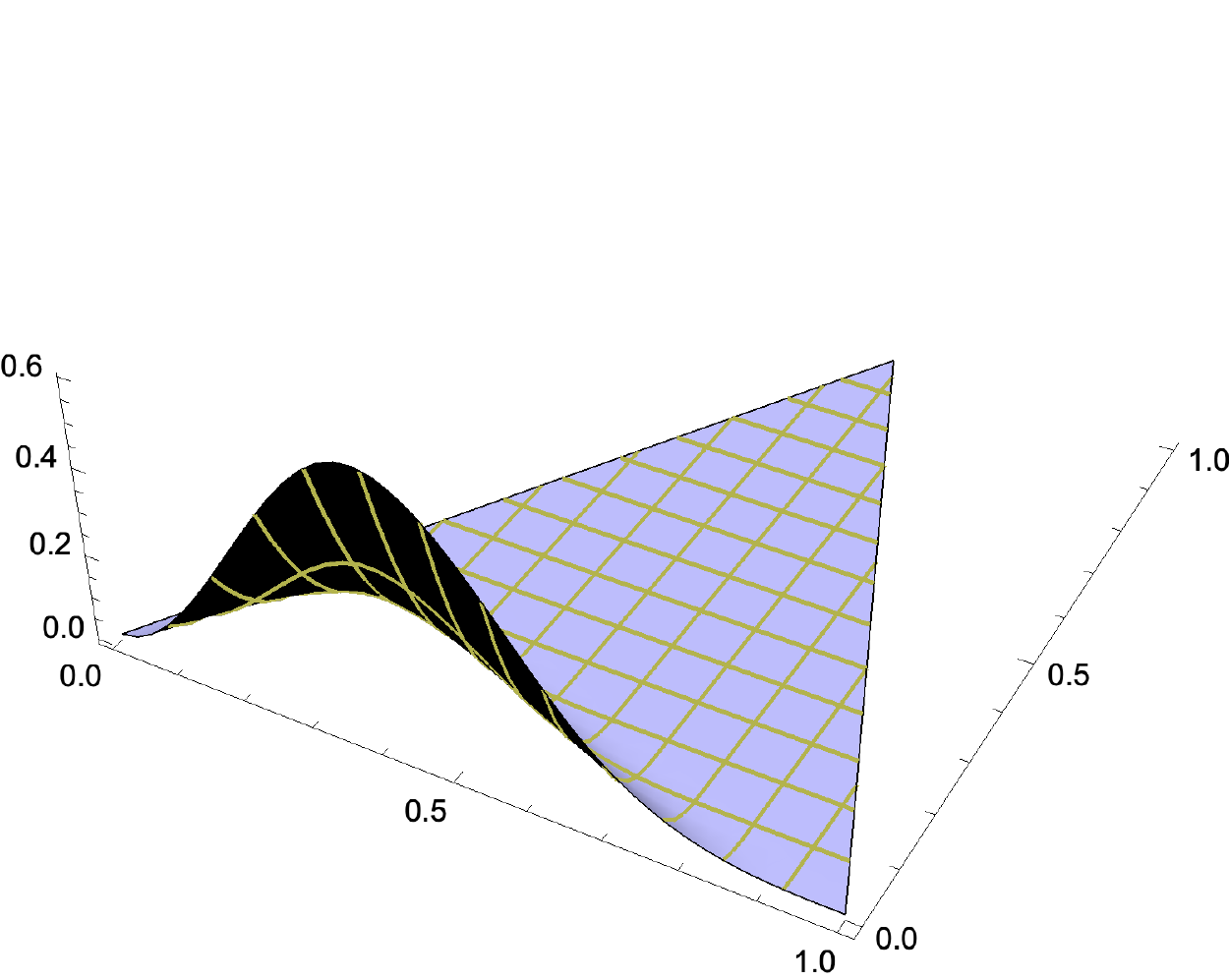}
\includegraphics[width=4cm]{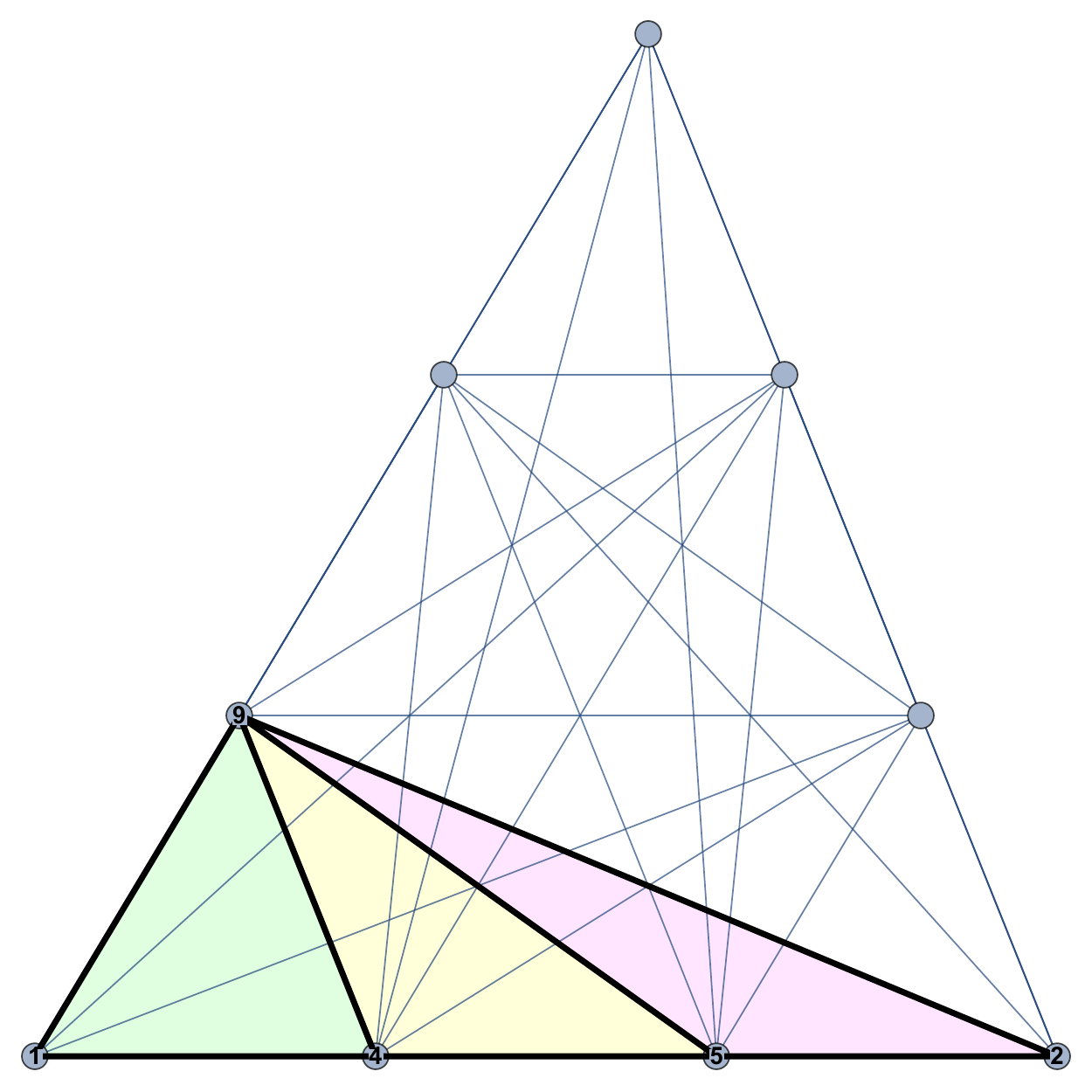}
\caption{The simplex spline basis function $B_{10}$ and its support.}
\label{fig:B10}
\centering
\includegraphics[trim=0 0 0 10,width=6cm]{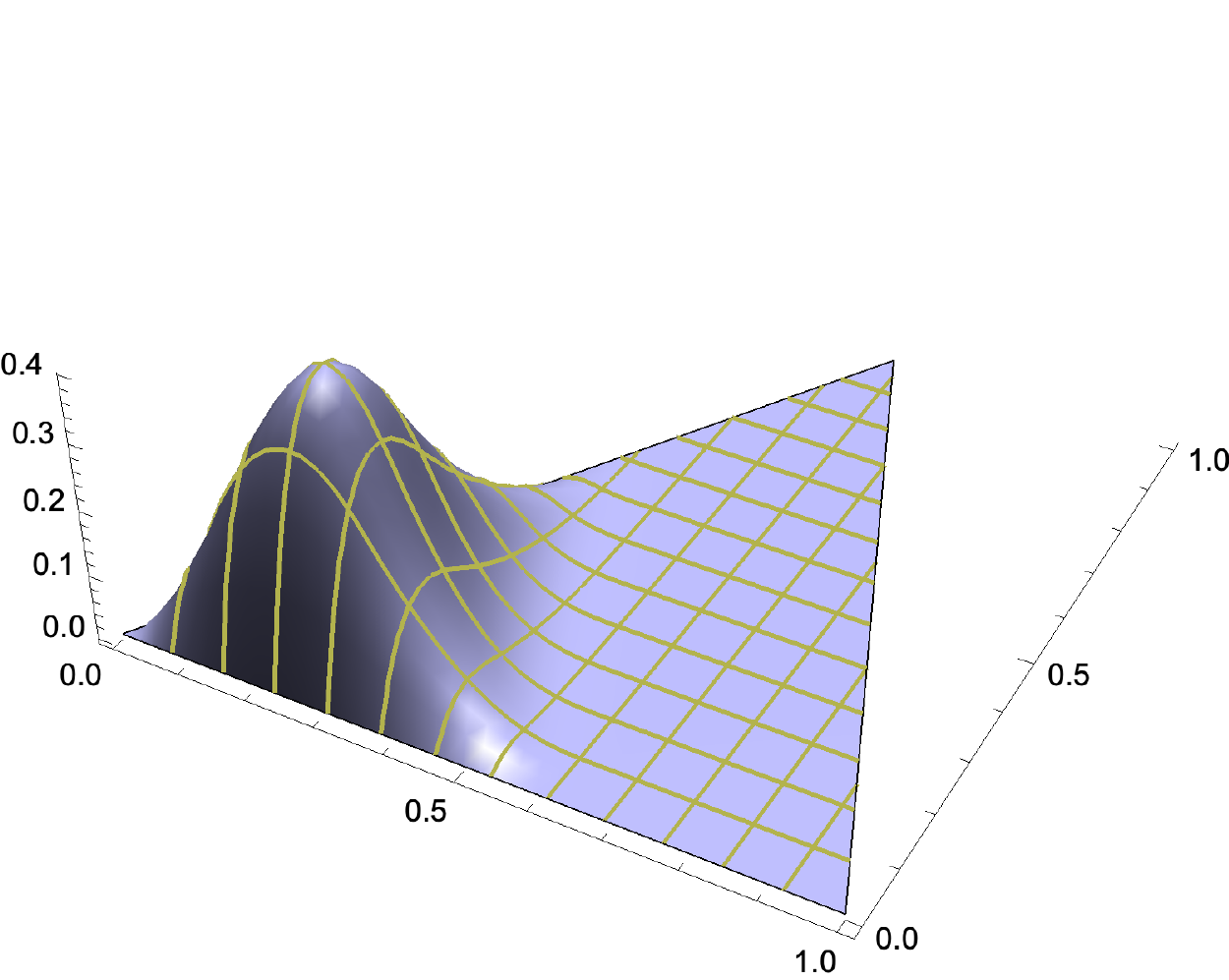}
\includegraphics[width=4cm]{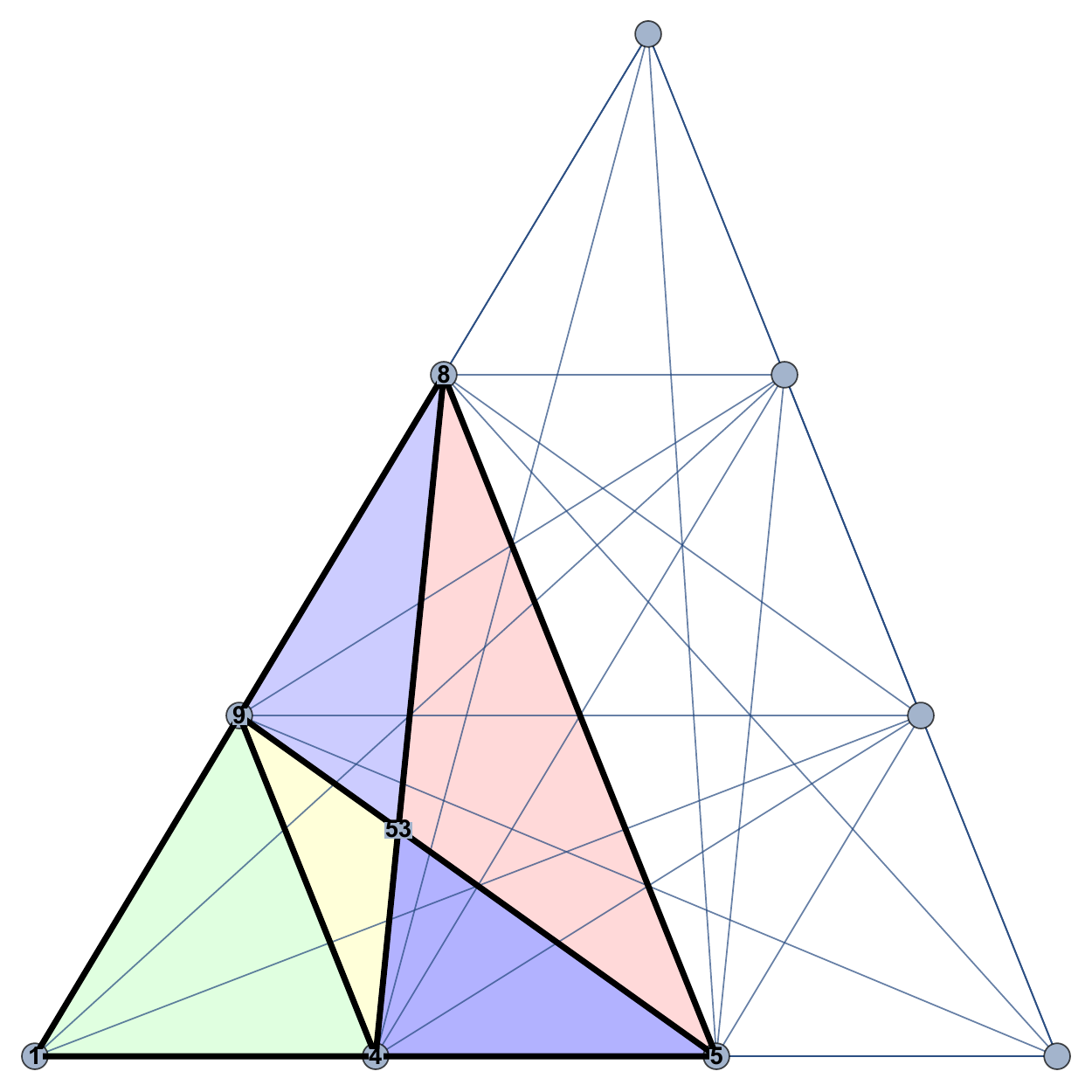}
\caption{The simplex spline basis function $B_{16}$ and its support.}
\label{fig:B16}
\centering
\includegraphics[trim=0 0 0 10,width=6cm]{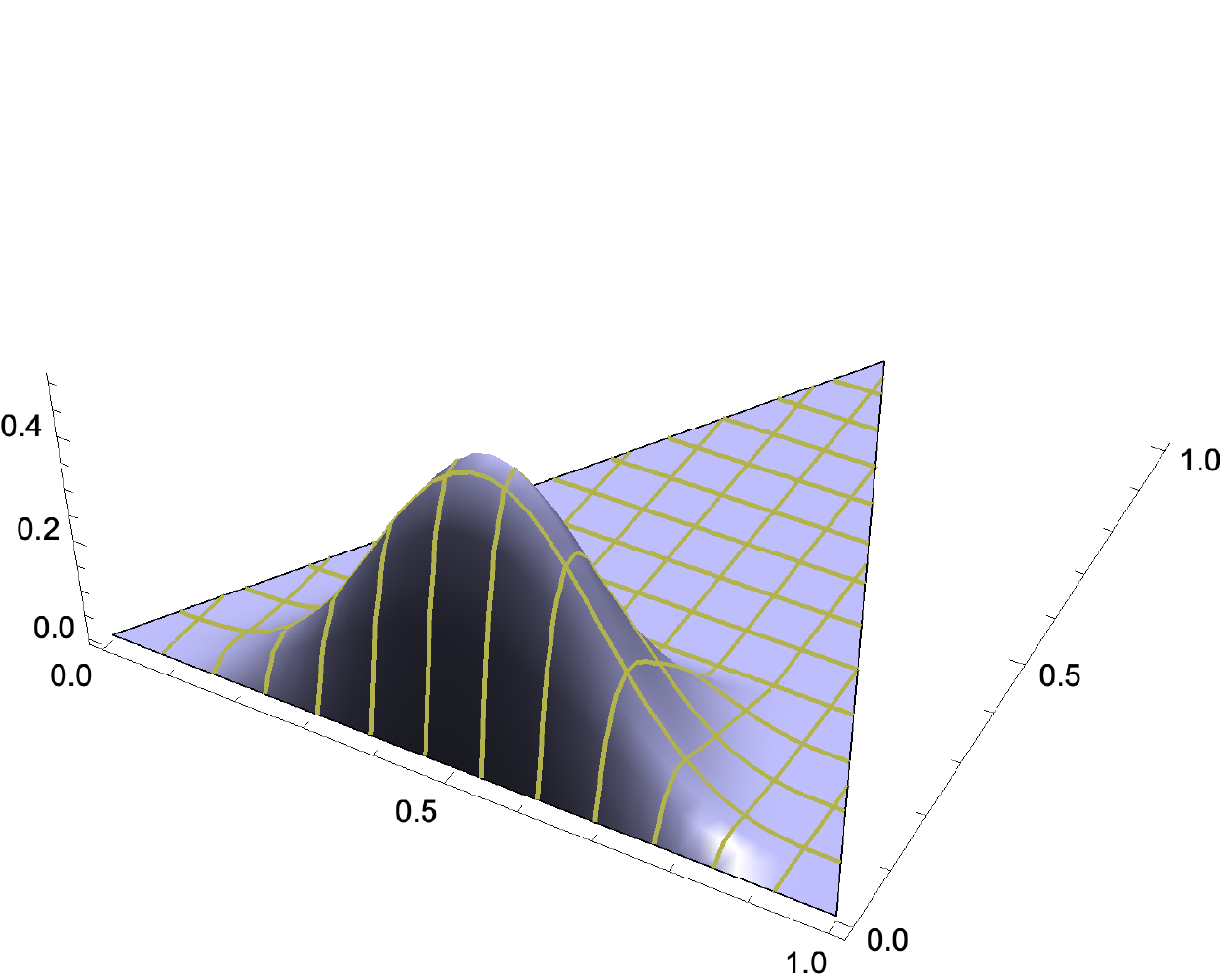}
\includegraphics[width=4cm]{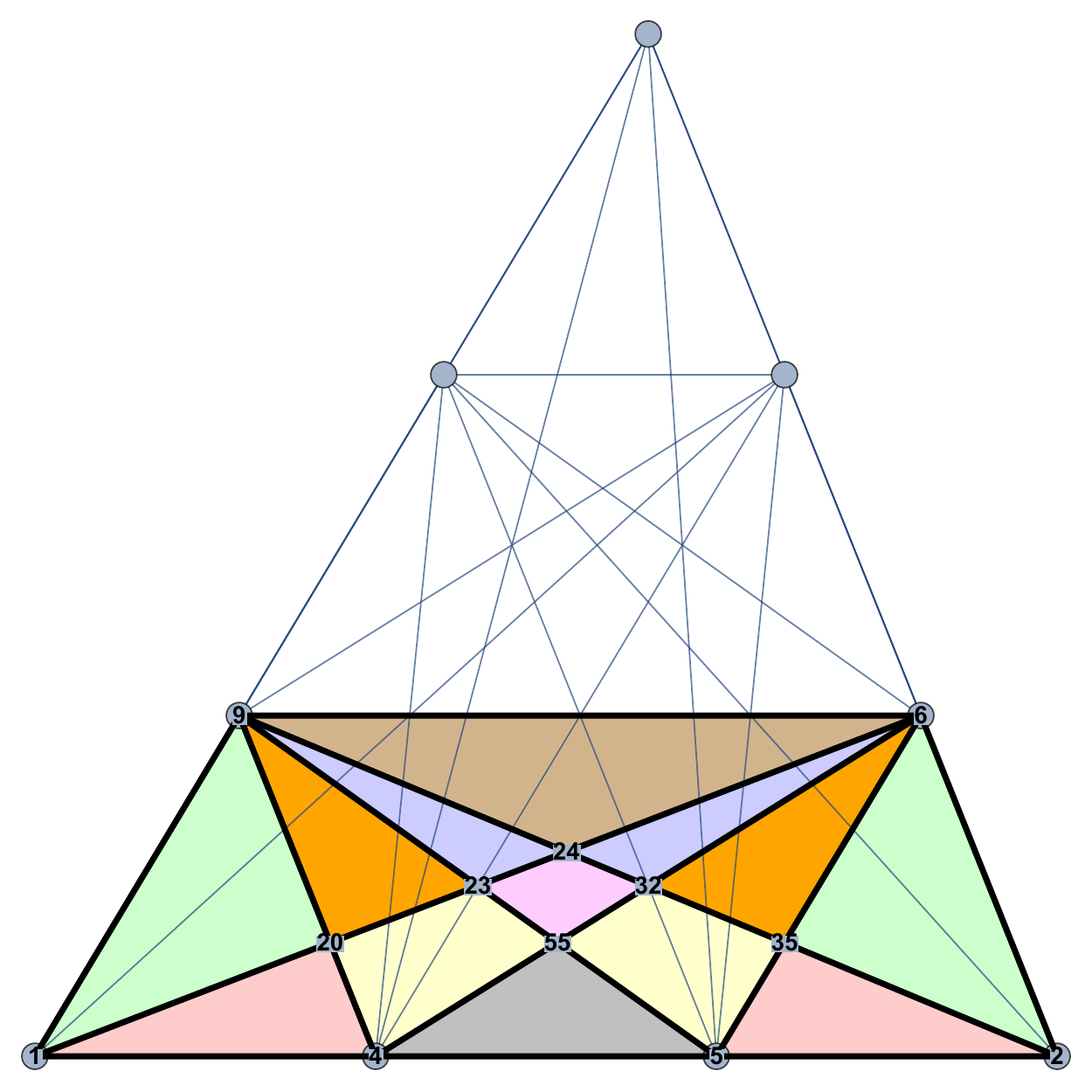}
\caption{The simplex spline basis function $B_{19}$ and its support.}
\label{fig:B19}
\end{figure}
\begin{figure}[t!]
\centering
\includegraphics[trim=0 0 0 10,width=6cm]{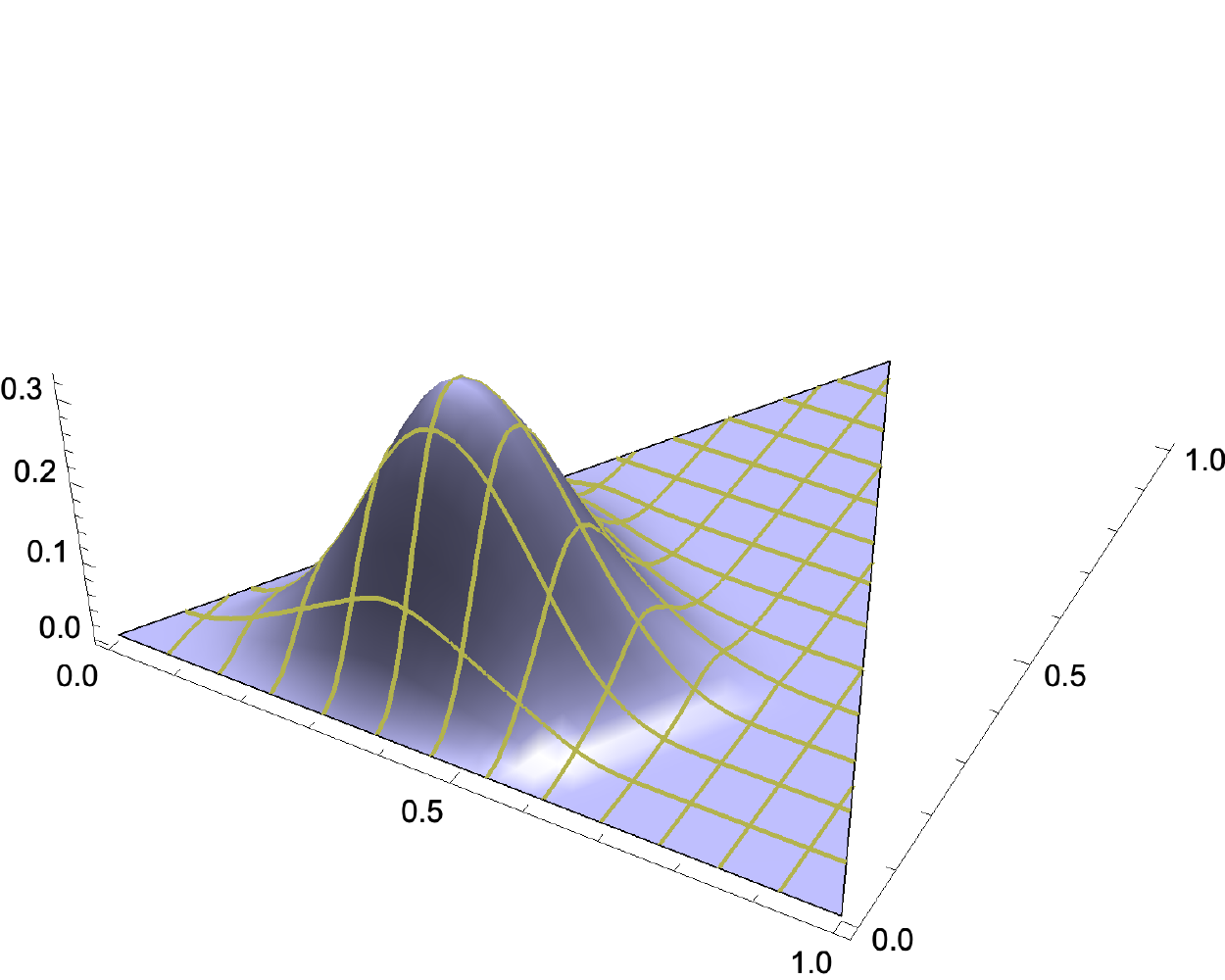}
\includegraphics[width=4cm]{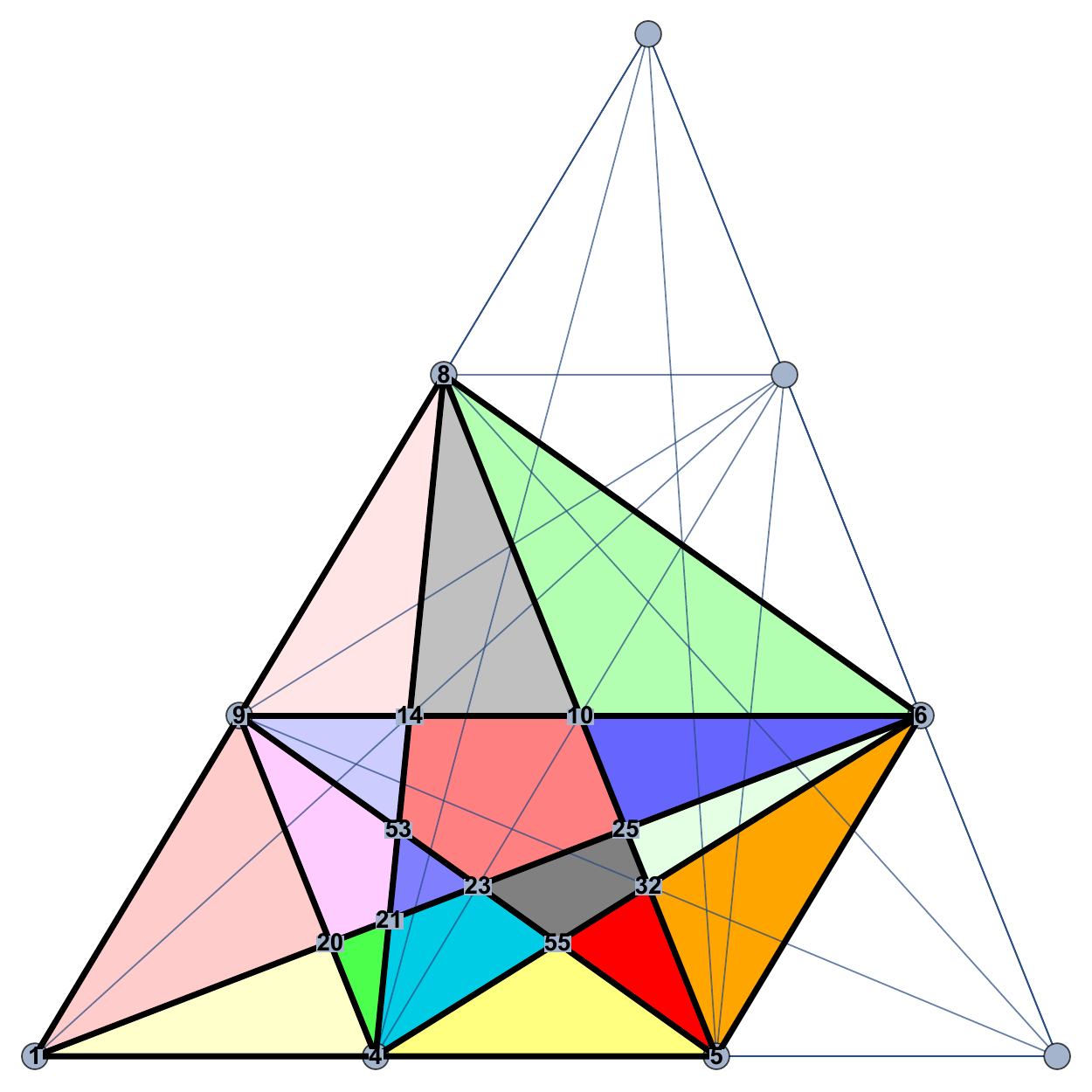}
\caption{The simplex spline basis function $B_{22}$ and its support.}
\label{fig:B22}
\centering
\includegraphics[trim=0 0 0 10,width=6cm]{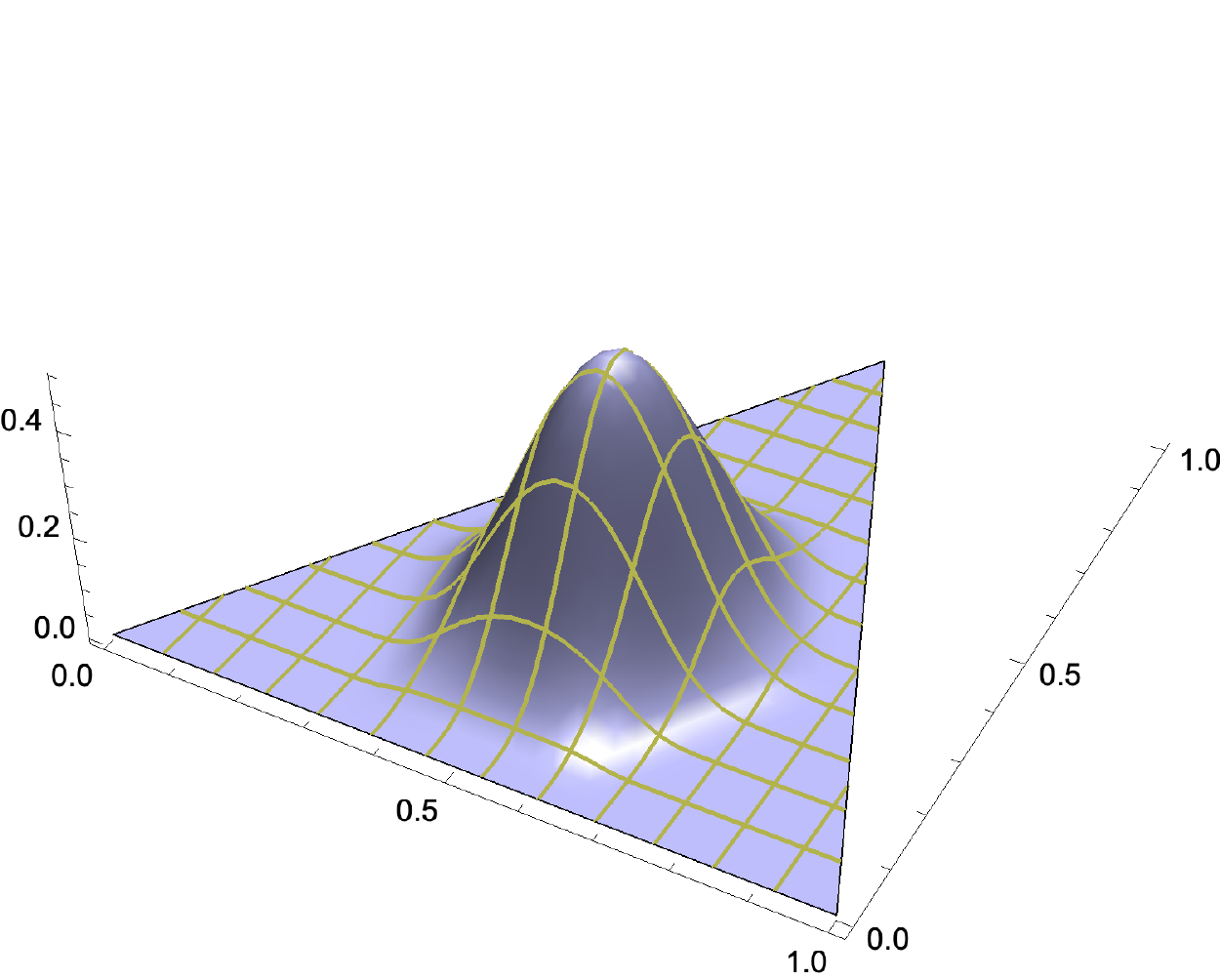}
\includegraphics[width=4cm]{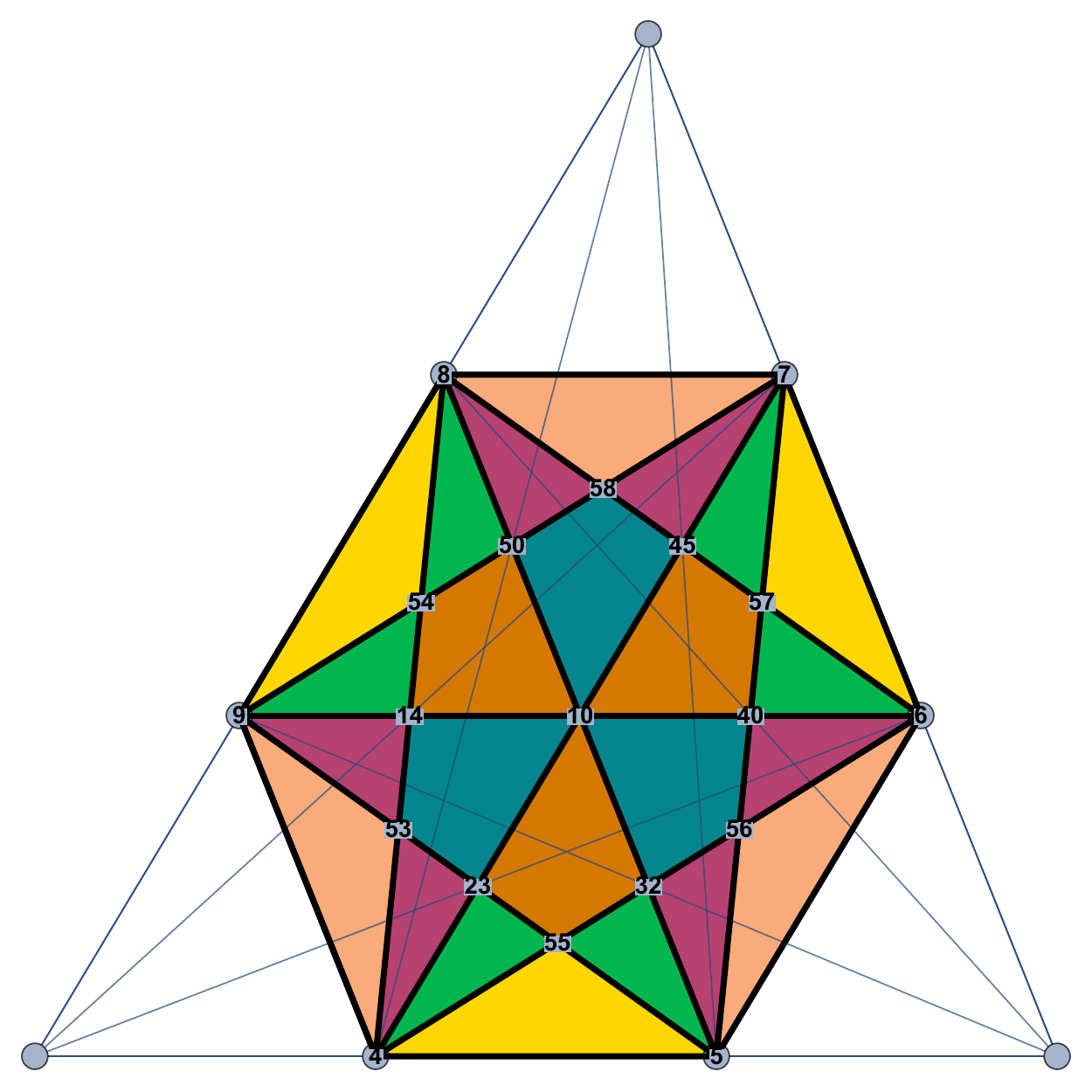}
\caption{The simplex spline basis function $B_{28}$ and its support.}
\label{fig:B28}
\end{figure}


\subsection{Visualization of basis functions}

Each simplex spline basis function $B_i$ is a piecewise polynomial of degree three on the partition formed by the complete graph of its knots. The different types of basis functions are depicted in Figures~\ref{fig:B1}--\ref{fig:B28}. For each basis function, its support is indicated in the figure and all polynomial pieces are marked by different colors.

\subsection{Explicit expressions of basis functions} \label{sec:appendix-expression}

\begin{figure}[t!]
\centering
\includegraphics[width=9.5cm]{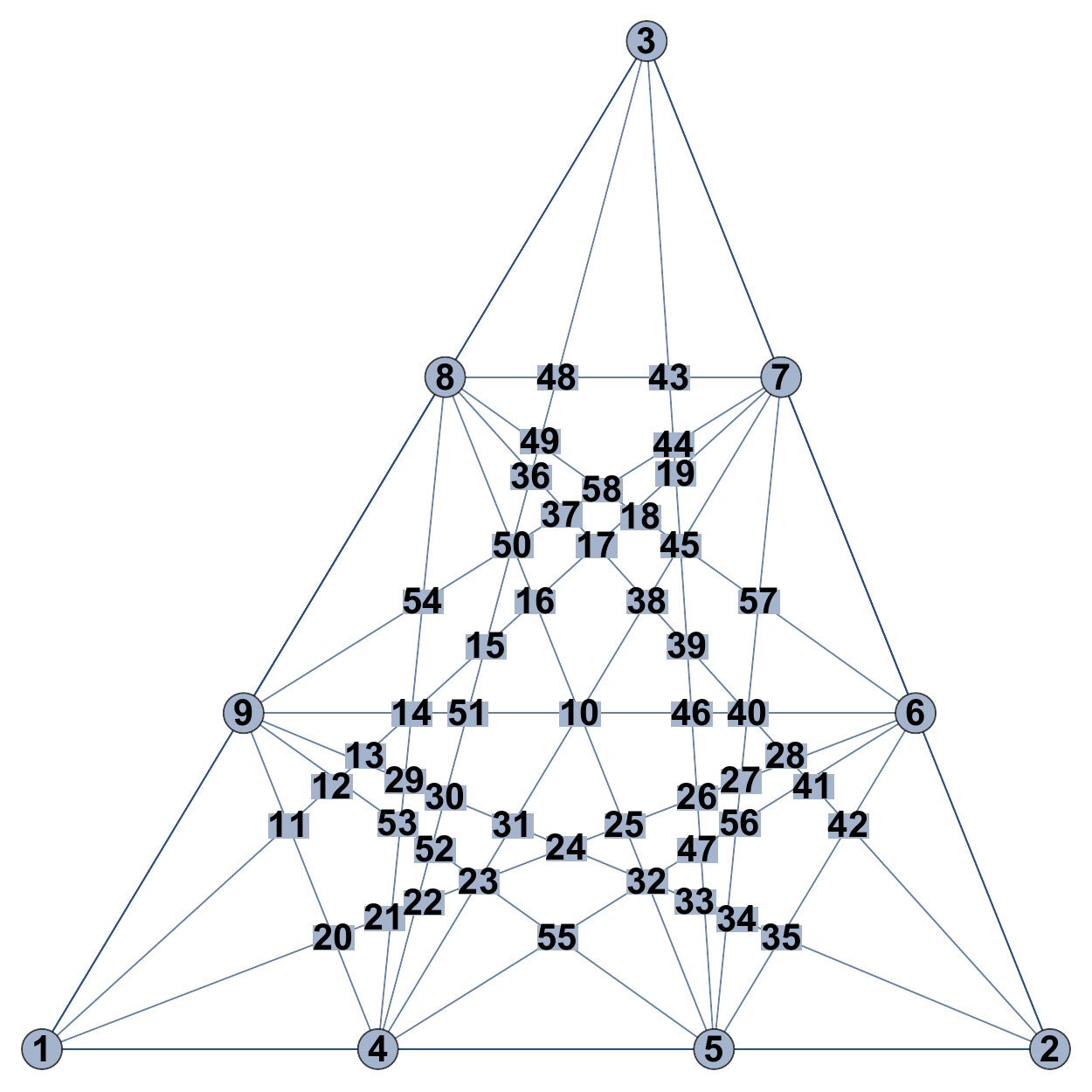}
\caption{Numbering of the 58 intersection points $\vv_k$ in the $\WS$ split.}
\label{fig:vertexnumbering}
\end{figure}

Here we provide the polynomial expressions of the basis functions $B_i$, $i=1,\ldots,28$, up to symmetries. These polynomials can be expressed in terms of the barycentric coordinates $(\beta_1,\beta_2,\beta_3)$ with respect to the macro-triangle $\Delta$. To this end, we define the following polynomials:
\begin{equation}\label{eq:l}
\begin{alignedat}{3}
l_{1,2} &:= \beta_3^3, &\quad
l_{1,3} &:= \beta_2^3, &\quad
l_{1,6} &:= (2\beta_3-\beta_2)^3, \\
l_{2,3} &:= \beta_1^3, &
l_{2,9} &:= (\beta_1-2\beta_3)^3, &
l_{4,6} &:= (3\beta_2-3\beta_3-1)^3, \\
l_{4,8} &:= (3\beta_1-3\beta_2-1)^3, &\quad
l_{4,9} &:= (3\beta_1-2)^3, &\quad
l_{5,6} &:= (3\beta_2-2)^3, \\
l_{5,7} &:= (3\beta_2-3\beta_1-1)^3, &
l_{5,8} &:= (3\beta_1-1)^3, &
l_{5,9} &:= (3\beta_1-3\beta_3-1)^3, \\
l_{6,8} &:= (3\beta_3-3\beta_1-1)^3, &
l_{6,9} &:= (3\beta_3-1)^3, &
l_{7,8} &:= (3\beta_3-2)^3, \\
l_{7,9} &:= (3\beta_3-3\beta_2-1)^3.
\end{alignedat}
\end{equation}
Let $\vv_k$ be the vertices of $\Delta_{\WS}$ visualized in Figure~\ref{fig:vertexnumbering}.
The equation $l_{i,j}=0$ represents the cubic power of the straight line connecting the points $\vv_i$ and $\vv_j$. Then, the polynomial pieces of the basis functions are described in Table~\ref{tab:polynomial-expressions}, up to symmetries. The corresponding regions are specified as the convex hull of the points $\vv_k$.

\subsection{Hermite data of basis functions}

With the aim of showing linear independence of the spline functions $B_1,\ldots,B_{28}$, we have set up a Hermite interpolation problem in the proof of Theorem~\ref{thm:basis}. The Hermite data is computed through the operators $\rho_1,\ldots,\rho_{28}$; they are defined in \eqref{eq:rho-1}, \eqref{eq:rho-2}, and \eqref{eq:rho-3}. The values of these operators applied to the $B_i$'s are collected in Table~\ref{tab:hermiteB}. The table also provides those values for the spline functions $\Btilde_{22},\ldots,\Btilde_{28}$. The other values are obtained through the identity $\Btilde_i=B_i$, $i=1,\ldots,21$.
Finally, we collect some additional second derivative values of the $\Btilde_i$'s in Table~\ref{tab:hermiteBtilde-extra}, where
\begin{equation} \label{eq:rho-4}
\begin{alignedat}{4}
\rho_{29}(f) &:= D^2_{\vp_1\vp_2}f(\vp_{3,1}), &\quad
\rho_{30}(f) &:= D^2_{\vp_1\vp_3}f(\vp_{3,1}), &\quad
\rho_{31}(f) &:= D_{\vp_1\vp_3}D_{\vp_1\vp_2}f(\vp_{3,1}), \\
\rho_{32}(f) &:= D^2_{\vp_1\vp_2}f(\vp_{3,2}), &
\rho_{33}(f) &:= D^2_{\vp_1\vp_3}f(\vp_{3,2}), &
\rho_{34}(f) &:= D_{\vp_1\vp_3}D_{\vp_1\vp_2}f(\vp_{3,2}).
\end{alignedat}
\end{equation}
These are useful to prove Theorem~\ref{thm:C2}.

\begin{landscape}
\begin{table}[t!]
\centering {\small
\begin{tabular}{|c|c|c|}
\hline
Basis & Region & Expression \\
\hline
$B_1$ & 1,4,9 & $l_{4,9}$\\
\hline
$B_4$ & 4,5,9 & $\frac{1}{4} l_{5,9}$\\
   & 1,4,9 & $\frac{1}{4} l_{5,9}-2 l_{4,9}$\\
\hline
$B_{16}$ & 5,8,53 & $\frac12l_{5,8}$\\
   & 4,5,53 & $\frac12l_{5,8}-\frac12l_{5,9}$\\
   & 8,9,53 & $\frac12l_{5,8}-\frac12l_{4,8}$\\
   & 1,4,9 & $\frac{27}{2} \beta_2\beta_3 (9 \beta_1-5 )$\\
   & 9,4,53 & $\frac{27}{2} \beta_2\beta_3 (9 \beta_1-5 )-4l_{4,9}$\\
\hline
 $B_{28}$ & 4,5,55 & $\frac{81}{2} l_{1,2}$\\
   & 6,7,57 & $\frac{81}{2}  l_{2,3}$\\
   & 8,9,54 & $\frac{81}{2}  l_{1,3}$\\
   & 5,6,56 & $-\frac{3}{2}l_{5,6}$\\
   & 7,8,58 & $-\frac{3}{2} l_{7,8}$\\
   & 4,53,9 & $-\frac{3}{2}l_{4,9}$\\
   & 4,55,23 & $\frac{81}{2}l_{1,2}+\frac12 l_{4,6}$\\
   & 5,32,55 & $\frac{81}{2}l_{1,2}+\frac12l_{5,9}$\\
   & 5,56,32 & $-\frac{3}{2}l_{5,6}+\frac12l_{5,7}$\\
   & 6,57,40 & $\frac{81}{2}l_{2,3}+\frac12l_{6,8}$\\
   & 8,54,50 & $\frac{81}{2}l_{1,3}+\frac12l_{4,8}$\\
   & 9,14,54 & $\frac{81}{2}l_{1,3}+\frac12l_{7,9}$\\
   & 4,23,53 & $-\frac{3}{2}l_{4,9}+\frac12l_{4,8}$\\
   & 6,40,56 & $-\frac{3}{2}l_{5,6}+\frac12l_{4,6}$\\
   & 7,58,45 & $-\frac{3}{2}l_{7,8}+\frac12l_{7,9}$\\
   & 8,50,58 & $-\frac{3}{2}l_{7,8}+\frac12l_{6,8}$\\
   & 7,45,57 & $\frac{81}{2}l_{2,3}+\frac12l_{5,7}$\\
   & 9,53,14 & $-\frac{3}{2}l_{4,9}+\frac12l_{5,9}$\\
   & 14,53,23,10 & $-\frac32 l_{4,9} + \frac12( l_{4,8} + l_{5,9})$\\
   & 23,55,32,10 & $\frac{81}{2} l_{1,2} + \frac12( l_{4,6} + l_{5,9})$\\
   & 32,56,40,10 & $-\frac32 l_{5,6} + \frac12(  l_{5,7} + l_{4,6})$\\
   & 40,57,45,10 & $\frac{81}{2} l_{2,3} + \frac12 ( l_{6,8} + l_{5,7})$\\
   & 45,58,50,10 & $-\frac{3}{2} l_{7,8} + \frac12 ( l_{6,8}+ l_{7,9}) $\\
   & 50,54,14,10 & $\frac{81}{2} l_{1,3} + \frac12 ( l_{4,8} + l_{7,9})$\\
\hline
&&\\
\hline
\end{tabular}
\hspace{1cm}
\begin{tabular}{|c|c|c|}
\hline
Basis & Region & Expression \\
\hline
$B_{10}$ & 1,4,9 & $\frac{9}{4} \beta_2^2 (6
   \beta_1-5 \beta_2-12 \beta_3)$\\
   & 4,5,9 & $\frac92 l_{2,9}-\frac34 l_{5,9}$\\
   & 5,2,9 & $\frac{9}{2} l_{2,9}$\\
\hline
$B_{19}$ & 1,20,9 & $\frac{45}{4}l_{1,3}$\\
   & 2,6,35 & $\frac{45}{4} l_{2,3}$\\
   & 6,9,24 & $-\frac{5}{2} l_{6,9}$\\
   & 1,4,20 & $\frac{45}{4} (l_{1,3}+l_{1,6})$\\
   & 5,2,35 & $\frac{45}{4} (l_{2,3}-l_{2,9})$\\
   & 6,24,32 & $-\frac{5}{2}l_{6,9}+\frac{45}{4} l_{1,6}$\\
   & 9,23,24 & $-\frac{5}{2} l_{6,9}-\frac{45}{4} l_{2,9}$\\
   & 9,20,23 & $\frac{45}{4}l_{1,3}+\frac54 l_{4,9} $\\
   & 6,32,35 & $\frac{45}{4} l_{2,3}+\frac54 l_{5,6} $\\
   & 4,55,23,20 & $\frac{45}{4} (l_{1,3}+l_{1,6})+\frac{5}{4} l_{4,9}$\\
   & 5,35,32,55 & $\frac{45}{4} (l_{2,3}-l_{2,9})+\frac{5}{4} l_{5,6}$\\
   & 4,5,55 & $\frac{45}{2}\beta_3(\beta_3^2+6\beta_1\beta_2-1)$\\
   & 55,32,24,23 & $-\frac{5}{2} l_{6,9}+\frac{45}{4}(l_{1,6} -l_{2,9})$\\
\hline
$B_{22}$ & 5,6,32 & $-\frac{1}{2} l_{5,6}$\\
   & 5,32,55 & $-\frac12 l_{5,6} - l_{5,8} $\\
   & 4,5,55 & $27 \beta_3^2 (3 \beta_1-1)$\\
   & 4,55,23,21 & $27 \beta_3^2 (3\beta_1-1)+\frac23 l_{4,6}$\\
   & 1,4,20 & $27 \beta_3^2 (3\beta_2-2 \beta_3)$\\
   & 4,21,20 & $27 \beta_3^2 (3\beta_2-2 \beta_3)+2 l_{4,9} $\\
   & 6,32,25 & $-\frac12l_{5,6}+\frac23l_{4,6}$\\
   & 55,32,25,23 & $-\frac12( 3\beta_2-2)^3 + \frac23 l_{4,6} - l_{5,8}$\\
   & 6,8,10 & $-\frac{1}{12}l_{6,8}$\\
   & 8,14,10 & $-\frac{1}{12}l_{6,8}-l_{5,8}$\\
   & 8,9,14 & $\frac{27}{4} \beta_2^2 (7 \beta_1+ \beta_3-3)$\\
   & 9,53,14 & $\frac{27}{4}\beta_2^2 (7\beta_1 +\beta_3 - 3) + 2 l_{6,9}$\\
   & 1,20,9 & $\frac{27}{4} \beta_2^2
   (6 \beta_3-\beta_2)$\\
   & 9,20,21,53,9 & $ \frac{27}{4} \beta_2^2 (6 \beta_3 - \beta_2) + 2 l_{4,9}$\\
   & 6,10,25 & $-\frac{1}{12}l_{6,8}+2l_{6,9}$\\
   & 23,25,10,14,53 & $-\frac{1}{12}l_{6,8}+2l_{6,9}-l_{5,8}$\\
   & 21,23,53 & $\frac{27}{4}\beta_2^2 (6\beta_3 -\beta_2) + 2 l_{4,9} - \frac13l_{4,8}$\\
\hline
\end{tabular}}
\caption{Explicit expressions of the polynomial pieces of the basis functions $B_i$ in terms of the functions defined in \eqref{eq:l}. The regions of the polynomial pieces are specified as the convex hull of the points $\vv_k$ numbered in Figure~\ref{fig:vertexnumbering}.}\label{tab:polynomial-expressions}
\end{table}
\end{landscape}

\begin{landscape}
\begin{table}[t!]
\centering{\scriptsize
\setlength{\tabcolsep}{5pt}
\begin{tabular}{|c|ccc|cccccc|cccccc|ccc|ccc|cccccc|c|}
\hline
& $\rho_1$ & $\rho_2$ & $\rho_3$ & $\rho_4$ & $\rho_5$ & $\rho_6$ & $\rho_7$ & $\rho_8$ & $\rho_9$ & $\rho_{10}$
& $\rho_{11}$ & $\rho_{12}$ & $\rho_{13}$ & $\rho_{14}$ & $\rho_{15}$ & $\rho_{16}$ & $\rho_{17}$ & $\rho_{18}$
& $\rho_{19}$ & $\rho_{20}$ & $\rho_{21}$ & $\rho_{22}$ & $\rho_{23}$ & $\rho_{24}$ & $\rho_{25}$ & $\rho_{26}$ & $\rho_{27}$ & $\rho_{28}$
 \\
\hline
$B_{1}$ & 1 & 0 & 0 & -9 & -9 & 0 & 0 & 0 & 0 & 54 & 54 &
   0 & 0 & 0 & 0 & 54 & 0 & 0 & 0 & 0 & 0 & 0 & 0 & 0 & 0 & 0 & 0 & 0 \\
$B_{2}$ & 0 & 1 & 0 & 0 & 0 & -9 & -9 & 0 & 0 & 0 & 0 &
   54 & 54 & 0 & 0 & 0 & 54 & 0 & 0 & 0 & 0 & 0 & 0 & 0 & 0 & 0 & 0 & 0 \\
$B_{3}$ & 0 & 0 & 1 & 0 & 0 & 0 & 0 & -9 & -9 & 0 & 0 & 0
   & 0 & 54 & 54 & 0 & 0 & 54 & 0 & 0 & 0 & 0 & 0 & 0 & 0 & 0 & 0 & 0  \\
   \hline
$B_{4}$ & 0 & 0 & 0 & 9 & 0 & 0 & 0 & 0 & 0 & -81 & 0 & 0
   & 0 & 0 & 0 & -54 & 0 & 0 & -$\frac{27}{32}$ & 0 & 0
   & $\frac{75}{2}$ & 0 & 0 & 0 & 0 & 0 & 0 \\
$B_{5}$ & 0 & 0 & 0 & 0 & 9 & 0 & 0 & 0 & 0 & 0 & -81 & 0
   & 0 & 0 & 0 & -54 & 0 & 0 & 0 & 0 & -$\frac{27}{32}$
   & 0 & $\frac{75}{2}$ & 0 & 0 & 0 & 0 & 0 \\
$B_{6}$ & 0 & 0 & 0 & 0 & 0 & 9 & 0 & 0 & 0 & 0 & 0 & -81
   & 0 & 0 & 0 & 0 & -54 & 0 & 0 & -$\frac{27}{32}$ & 0
   & 0 & 0 & $\frac{75}{2}$ & 0 & 0 & 0 & 0 \\
$B_{7}$ & 0 & 0 & 0 & 0 & 0 & 0 & 9 & 0 & 0 & 0 & 0 & 0
   & -81 & 0 & 0 & 0 & -54 & 0& -$\frac{27}{32}$ & 0 & 0
   & 0 & 0 & 0 & $\frac{75}{2}$ & 0 & 0 & 0 \\
$B_{8}$ & 0 & 0 & 0 & 0 & 0 & 0 & 0 & 9 & 0 & 0 & 0 & 0
   & 0 & -81 & 0 & 0 & 0 & -54 & 0 & 0 & -$\frac{27}{32}$
   & 0 & 0 & 0 & 0 & $\frac{75}{2}$ & 0 & 0 \\
$B_{9}$ & 0 & 0 & 0 & 0 & 0 & 0 & 0 & 0 & 9 & 0 & 0 & 0
   & 0 & 0 & -81 & 0 & 0 & -54 & 0 & -$\frac{27}{32}$ & 0
   & 0 & 0 & 0 & 0 & 0 & $\frac{75}{2}$ & 0 \\
   \hline
$B_{10}$ & 0 & 0 & 0 & 0 & 0 & 0 & 0 & 0 & 0 & 27 & 0 & 0
   & 0 & 0 & 0 & 0 & 0 & 0 & -$\frac{189}{32}$ & 0 & 0
   & $\frac{31}{2}$ & 0 & 0 & 49 & 0 & 0 & 0 \\
$B_{11}$ & 0 & 0 & 0 & 0 & 0 & 0 & 0 & 0 & 0 & 0 & 27 & 0
   & 0 & 0 & 0 & 0 & 0 & 0 & 0 & 0 & -$\frac{189}{32}$ & 0
   & $\frac{31}{2}$ & 0 & 0 & 49 & 0 & 0 \\
$B_{12}$ & 0 & 0 & 0 & 0 & 0 & 0 & 0 & 0 & 0 & 0 & 0 & 27
   & 0 & 0 & 0 & 0 & 0 & 0 & 0 & -$\frac{189}{32}$ & 0 & 0
   & 0 & $\frac{31}{2}$ & 0 & 0 & 49 & 0 \\
$B_{13}$ & 0 & 0 & 0 & 0 & 0 & 0 & 0 & 0 & 0 & 0 & 0 & 0
   & 27 & 0 & 0 & 0 & 0 & 0 & -$\frac{189}{32}$ & 0 & 0
   & 49 & 0 & 0 & $\frac{31}{2}$ & 0 & 0 & 0 \\
$B_{14}$ & 0 & 0 & 0 & 0 & 0 & 0 & 0 & 0 & 0 & 0 & 0 & 0
   & 0 & 27 & 0 & 0 & 0 & 0 & 0 & 0 & -$\frac{189}{32}$
   & 0 & 49 & 0 & 0 & $\frac{31}{2}$ & 0 & 0 \\
$B_{15}$ & 0 & 0 & 0 & 0 & 0 & 0 & 0 & 0 & 0 & 0 & 0 & 0
   & 0 & 0 & 27 & 0 & 0 & 0 & 0 & -$\frac{189}{32}$ & 0
   & 0 & 0 & 49 & 0 & 0 & $\frac{31}{2}$ & 0 \\
   \hline
$B_{16}$ & 0 & 0 & 0 & 0 & 0 & 0 & 0 & 0 & 0 & 0 & 0 & 0
   & 0 & 0 & 0 & 54 & 0 & 0 & $\frac{9}{8}$ & 0
   & $\frac{9}{8}$ & -63 & -63 & 0 & 0 & 0 & 0 & 0 \\
$B_{17}$ & 0 & 0 & 0 & 0 & 0 & 0 & 0 & 0 & 0 & 0 & 0 & 0
   & 0 & 0 & 0 & 0 & 54 & 0 & $\frac{9}{8}$ & $\frac{9}{8}$
   & 0 & 0 & 0 & -63 & -63 & 0 & 0 & 0 \\
$B_{18}$ & 0 & 0 & 0 & 0 & 0 & 0 & 0 & 0 & 0 & 0 & 0 & 0
   & 0 & 0 & 0 & 0 & 0 & 54 & 0 & $\frac{9}{8}$
   & $\frac{9}{8}$ & 0 & 0 & 0 & 0 & -63 & -63 & 0 \\
   \hline
$B_{19}$ & 0 & 0 & 0 & 0 & 0 & 0 & 0 & 0 & 0 & 0 & 0 & 0
   & 0 & 0 & 0 & 0 & 0 & 0 & $\frac{45}{4}$ & 0 & 0
   & -120 & 0 & 0 & -120 & 0 & 0 & 0 \\
$B_{20}$ & 0 & 0 & 0 & 0 & 0 & 0 & 0 & 0 & 0 & 0 & 0 & 0
   & 0 & 0 & 0 & 0 & 0 & 0 & 0 & $\frac{45}{4}$ & 0 & 0
   & 0 & -120 & 0 & 0 & -120 & 0 \\
$B_{21}$ & 0 & 0 & 0 & 0 & 0 & 0 & 0 & 0 & 0 & 0 & 0 & 0
   & 0 & 0 & 0 & 0 & 0 & 0 & 0 & 0 & $\frac{45}{4}$ & 0
   & -120 & 0 & 0 & -120 & 0 & 0 \\
   \hline
$B_{22}$ & 0 & 0 & 0 & 0 & 0 & 0 & 0 & 0 & 0 & 0 & 0 & 0
   & 0 & 0 & 0 & 0 & 0 & 0 & 0 & 0 & 0 & 54 & 27 & 0 & 0
   & 0 & 0 & $\frac{1}{12}$ \\
$B_{23}$ & 0 & 0 & 0 & 0 & 0 & 0 & 0 & 0 & 0 & 0 & 0 & 0
   & 0 & 0 & 0 & 0 & 0 & 0 & 0 & 0 & 0 & 27 & 54 & 0 & 0
   & 0 & 0 & $\frac{1}{12}$ \\
$B_{24}$ & 0 & 0 & 0 & 0 & 0 & 0 & 0 & 0 & 0 & 0 & 0 & 0
   & 0 & 0 & 0 & 0 & 0 & 0 & 0 & 0 & 0 & 0 & 0 & 54 & 27
   & 0 & 0 & $\frac{1}{12}$ \\
$B_{25}$ & 0 & 0 & 0 & 0 & 0 & 0 & 0 & 0 & 0 & 0 & 0 & 0
   & 0 & 0 & 0 & 0 & 0 & 0 & 0 & 0 & 0 & 0 & 0 & 27 & 54
   & 0 & 0 & $\frac{1}{12}$ \\
$B_{26}$ & 0 & 0 & 0 & 0 & 0 & 0 & 0 & 0 & 0 & 0 & 0 & 0
   & 0 & 0 & 0 & 0 & 0 & 0 & 0 & 0 & 0 & 0 & 0 & 0 & 0
   & 54 & 27& $\frac{1}{12}$ \\
$B_{27}$ & 0 & 0 & 0 & 0 & 0 & 0 & 0 & 0 & 0 & 0 & 0 & 0
   & 0 & 0 & 0 & 0 & 0 & 0 & 0 & 0 & 0 & 0 & 0 & 0 & 0
   & 27 & 54 & $\frac{1}{12}$ \\
   \hline
$B_{28}$ & 0 & 0 & 0 & 0 & 0 & 0 & 0 & 0 & 0 & 0 & 0 & 0
   & 0 & 0 & 0 & 0 & 0 & 0 & 0 & 0 & 0 & 0 & 0 & 0 & 0
   & 0 & 0 & $\frac{1}{2}$ \\
   \hline\hline
$\Btilde_{22}$ & 0 & 0 & 0 & 0 & 0 & 0 & 0 & 0 & 0 & 0 & 0
   & 0 & 0 & 0 & 0 & 0 & 0 & 0 & 0 & 0 & 0 & 81 & 0 & 0
   & 0 & 0 & 0 & $\frac{1}{4}$ \\
$\Btilde_{23}$ & 0 & 0 & 0 & 0 & 0 & 0 & 0 & 0 & 0 & 0 & 0
   & 0 & 0 & 0 & 0 & 0 & 0 & 0 & 0 & 0 & 0 & 0 & 81 & 0
   & 0 & 0 & 0 & $\frac{1}{4}$ \\
$\Btilde_{24}$ & 0 & 0 & 0 & 0 & 0 & 0 & 0 & 0 & 0 & 0 & 0
   & 0 & 0 & 0 & 0 & 0 & 0 & 0 & 0 & 0 & 0 & 0 & 0 & 81
   & 0 & 0 & 0 & $\frac{1}{4}$ \\
$\Btilde_{25}$ & 0 & 0 & 0 & 0 & 0 & 0 & 0 & 0 & 0 & 0 & 0
   & 0 & 0 & 0 & 0 & 0 & 0 & 0 & 0 & 0 & 0 & 0 & 0 & 0
   & 81 & 0 & 0 & $\frac{1}{4}$ \\
$\Btilde_{26}$ & 0 & 0 & 0 & 0 & 0 & 0 & 0 & 0 & 0 & 0 & 0
   & 0 & 0 & 0 & 0 & 0 & 0 & 0 & 0 & 0 & 0 & 0 & 0 & 0 & 0
   & 81 & 0 & $\frac{1}{4}$ \\
$\Btilde_{27}$ & 0 & 0 & 0 & 0 & 0 & 0 & 0 & 0 & 0 & 0 & 0
   & 0 & 0 & 0 & 0 & 0 & 0 & 0 & 0 & 0 & 0 & 0 & 0 & 0 & 0
   & 0 & 81 & $\frac{1}{4}$ \\
\hline
 $\Btilde_{28}$ & 0 & 0 & 0 & 0 & 0 & 0 & 0 & 0 & 0 & 0 & 0
   & 0 & 0 & 0 & 0 & 0 & 0 & 0 & 0 & 0 & 0 & 0 & 0 & 0 & 0
   & 0 & 0 & -$\frac{1}{2}$ \\
\hline
\end{tabular}}
\caption{Values of $\rho_j(B_i)$ and $\rho_j(\Btilde_i)$ for $i,j=1,\ldots,28$, where $\rho_j$ is defined in \eqref{eq:rho-1}, \eqref{eq:rho-2}, and \eqref{eq:rho-3}. Note that $\Btilde_i=B_i$ for $i=1,\ldots,21$.}
\label{tab:hermiteB}
\end{table}
\end{landscape}

\begin{table}[t!]
\centering{\scriptsize
\setlength{\tabcolsep}{5pt}
\begin{tabular}{|c|cccccc|}
\hline
& $\rho_{29}$ & $\rho_{30}$ & $\rho_{31}$ & $\rho_{32}$ & $\rho_{33}$ & $\rho_{34}$
\\
\hline
$\Btilde_{1}$ & 0 & 0 & 0 & 0 & 0 & 0 \\
$\Btilde_{2}$ & 0 & 0 & 0 & 0 & 0 & 0 \\
$\Btilde_{3}$ & 0 & 0 & 0 & 0 & 0 & 0 \\
\hline
$\Btilde_{4}$ & $\frac{27}{2}$ & 54 & 27 & 0 & 0 & 0 \\
$\Btilde_{5}$ & 0 & 0 & 0 & 0 & 0 & 0 \\
$\Btilde_{6}$ & 0 & 0 & 0 & 0 & 0 & 0 \\
$\Btilde_{7}$ & 0 & 0 & 0 & $\frac{27}{2}$ & $\frac{27}{2}$ & -$\frac{27}{2}$ \\
$\Btilde_{8}$ & 0 & 0 & 0 & 0 & 0 & 0 \\
$\Btilde_{9}$ & 0 & 0 & 0 & 0 & 0 & 0 \\
\hline
$\Btilde_{10}$ & -$\frac{45}{2}$ & 0 & -27 & 9 & 81 & 27 \\
$\Btilde_{11}$ & 0 & 0 & 0 & 0 & 0 & 0 \\
$\Btilde_{12}$ & 0 & 0 & 0 & 0 & 0 & 0 \\
$\Btilde_{13}$ & 9 & 36 & -18 & -$\frac{45}{2}$ & $\frac{63}{2}$ & $\frac{9}{2}$ \\
$\Btilde_{14}$ & 0 & 0 & 0 & 0 & 0 & 0 \\
$\Btilde_{15}$ & 0 & 0 & 0 & 0 & 0 & 0 \\
\hline
$\Btilde_{16}$ & 0 & -81 & -27 & 0 & 0 & 0 \\
$\Btilde_{17}$ & 0 & 0 & 0 & 0 & -27 & 27 \\
$\Btilde_{18}$ & 0 & 0 & 0 & 0 & 0 & 0 \\
\hline
$\Btilde_{19}$ & 0 & -90 & 45 & 0 & -180 & -45 \\
$\Btilde_{20}$ & 0 & 0 & 0 & 0 & 0 & 0 \\
$\Btilde_{21}$ & 0 & 0 & 0 & 0 & 0 & 0 \\
\hline
$\Btilde_{22}$ & 0 & 81 & 0 & 0 & 0 & 0 \\
$\Btilde_{23}$ & 0 & 0 & 0 & 0 & 0 & 0 \\
$\Btilde_{24}$ & 0 & 0 & 0 & 0 & 0 & 0 \\
$\Btilde_{25}$ & 0 & 0 & 0 & 0 & 81 & 0 \\
$\Btilde_{26}$ & 0 & 0 & 0 & 0 & 0 & 0 \\
$\Btilde_{27}$ & 0 & 0 & 0 & 0 & 0 & 0 \\
\hline
$\Btilde_{28}$ & 0 & 0 & 0 & 0 & 0 & 0 \\
\hline
\end{tabular}}
\caption{Values of $\rho_j(\Btilde_i)$ for $i=1,\ldots,28$ and $j=29,\ldots,34$, where $\rho_j$ is defined in \eqref{eq:rho-4}.}
\label{tab:hermiteBtilde-extra}
\end{table}

\end{document}